\documentclass[a4paper,11pt,twoside]{article}

\usepackage{xcolor}
\usepackage{fontenc}
\usepackage{amsthm,amsmath,amsfonts,amssymb}
\usepackage{bm}
\RequirePackage[colorlinks,citecolor=purple,urlcolor=black]{hyperref}
\RequirePackage{epsfig, graphicx, color}
\RequirePackage{latexsym}

\usepackage{a4wide}
\usepackage{epsf,epsfig,subfigure}
\usepackage{natbib}
\usepackage{booktabs}
\usepackage{multirow}
\usepackage{graphics,graphicx}
\usepackage{enumitem}
\usepackage[]{algorithm2e}

\usepackage{subfigure}
\usepackage{bbm}

\usepackage{color}
\usepackage{dsfont}

\newtheorem{theorem}{Theorem}

\newtheorem{lemma}[theorem]{Lemma}

\newtheorem{corollary}[theorem]{Corollary}
\newtheorem{proposition}[theorem]{Proposition}

\theoremstyle{definition}
\newtheorem{definition}{Definition}

\newtheorem*{example*gplm}{Example 1}
\newtheorem*{example*gpliv}{Example 2}
\newtheorem*{example*anchor}{Example 3}
\newtheorem*{example*huber}{Example 4}
\newtheorem*{example*survival}{Example 5}
\newtheorem{assumption}{Assumption}

\theoremstyle{remark}

\newcommand{\R}{\mathbb{R}}
\newcommand{\PP}{\mathbb{P}}
\newcommand{\E}{\mathbb{E}}

\newcommand{\Vol}{\mathrm{Vol}\,}
\newcommand{\sgn}{\mathrm{sgn}}

\newcommand{\indist}{\stackrel{d}{\to}}
\newcommand{\TV}{\mathrm{TV}}

\newcommand{\eqdist}{\stackrel{d}{=}}

\newcommand{\norm}[1]{\left\lVert #1 \right\rVert}

\newcommand{\ind}{\mathbbm{1}}

\newcommand{\interior}{\mathrm{int}\,}
\newcommand{\floor}[1]{\left\lfloor #1 \right\rfloor}

\newcommand{\Cov}{\mathrm{Cov}}

\newcommand{\Var}{\mathrm{Var}}

\newcommand{\cX}{\mathcal{X}}
\newcommand{\cZ}{\mathcal{Z}}
\newcommand{\cY}{\mathcal{Y}}

\newcommand{\cG}{\mathcal{G}}

\newcommand{\rose}{\mathrm{(rose)}}
\newcommand{\roseplus}{\mathrm{(rose+)}}

\newcommand{\loceff}{\mathrm{(loceff)}}
\newcommand{\SL}{\mathrm{SL}}
\newcommand{\Strain}{\hat{w}_1}

\newcommand{\SIkc}{S_{\cI_k^c}}

\newcommand{\cL}{\mathcal{L}}
\newcommand{\cR}{\mathcal{R}}

\newcommand{\cP}{\mathcal{P}}

\newcommand{\cI}{\mathcal{I}}
\newcommand{\supP}{\sup_{P\in\cP}}
\newcommand{\given}{\,|\,}
\newcommand{\biggiven}{\,\big|\,}
\newcommand{\Biggiven}{\,\Big|\,}
\newcommand{\bigggiven}{\,\bigg|\,}
\newcommand{\Bigggiven}{\,\Bigg|\,}

\newcommand{\hr}{\bar{r}^{\perp}}

\newcommand{\ball}[1]{{\mathcal{B}_{\epsilon^*}(#1)}}

\newcommand{\B}[1]{B_{\RN{#1}}}

\newcommand{\diam}{\mathrm{diam}}

\theoremstyle{plain}

\newcommand{\mb}{\mathbf}

\newcommand{\vertiii}[1]{{\left\vert\kern-0.25ex\left\vert\kern-0.25ex\left\vert #1 
    \right\vert\kern-0.25ex\right\vert\kern-0.25ex\right\vert}}
    
\newcommand\independent{\protect\mathpalette{\protect\independenT}{\perp}}
    \def\independenT#1#2{\mathrel{\rlap{$#1#2$}\mkern2mu{#1#2}}}

\newcommand{\RN}[1]{
  \textup{\uppercase\expandafter{\romannumeral#1}}
}

\makeatletter
\newcommand{\mylabel}[2]{#2\def\@currentlabel{#2}\label{#1}}
\makeatother

\title{ROSE Random Forests for Robust Semiparametric Efficient Estimation}

\date{December 2024}
\usepackage{authblk}%
\author[1]{Elliot Young}
\author[1]{Rajen D.\ Shah}
\affil{Statistical Laboratory, University of Cambridge, UK}

\begin{document}

\maketitle

\begin{abstract}
It is widely recognised that semiparametric efficient estimation can be hard to achieve in practice: estimators that are in theory efficient may require unattainable levels of accuracy for the estimation of complex nuisance functions. As a consequence, estimators deployed on real datasets are often chosen in a somewhat ad hoc fashion, and may suffer high variance. We study this gap between theory and practice in the context of a broad collection of semiparametric regression models that includes the generalised partially linear model. We advocate using estimators that are robust in the sense that they enjoy $\sqrt{n}$-consistent uniformly over a sufficiently rich class of distributions characterised by certain conditional expectations being  estimable by user-chosen machine learning methods. We show that even asking for locally uniform estimation within such a class narrows down possible estimators to those parametrised by certain weight functions. Conversely, we show that such estimators do provide the desired uniform consistency and introduce a novel random forest-based procedure for estimating the optimal weights. We prove that the resulting estimator recovers a notion of $\textbf{ro}$bust $\textbf{s}$emiparametric $\textbf{e}$fficiency (ROSE) and provides a practical alternative to semiparametric efficient estimators. We demonstrate the effectiveness of our ROSE random forest estimator in a variety of semiparametric settings on simulated and real-world data.
\end{abstract}

\section{Introduction}\label{sec:intro}
In a variety of statistical problems, the target of interest is a low-dimensional parameter. When estimating such a parameter, rather than postulating a fully parametric model for the data at hand, which may be vulnerable to misspecification, it is often desirable to formulate a model relying on as few assumptions as possible. Such models typically involve nonparametric nuisance functions which can only be estimated at relatively slow rates. Semiparametric statistics~\citep{bickel, vandervaart, tsiatis, kosorok} provides a rich theory for how to construct estimators for the parameter of interest that can retain the attractive features of estimates constructed within parametric models; namely $\sqrt{n}$-consistency and asymptotic normality. These ideas form the basis of targeted learning~\citep{targeted-learning} and debiased machine learning~\citep{chern}: two highly popular frameworks that leverage the predictive capabilities of user-chosen flexible regression (machine learning) methods to estimate the infinite or high-dimensional
nuisance functions, and deliver estimators for the target of interest that enjoy parametric rates of convergence.

A key concept in semiparametric statistics is that of an influence function, which encapsulates how the target parameter changes when the underlying distribution of the data varies within the postulated semiparametric model. Strictly semiparametric models, that is models which do not include all possible data distributions (but nevertheless are highly flexible through the inclusion of nonparametric nuisance functions) give rise to a family of influence functions, each of which may in principle be used to construct an estimator for the target of interest. Among influence functions, the efficient influence function is that which results in the estimator with the smallest asymptotic variance: in fact it yields the smallest asymptotic variance among all regular estimators (see e.g.~\citet[\S25]{vandervaart}), and is therefore a natural choice to use in a given problem. However, while in theory such estimators are optimal, many authors tend not to recommend them for practical use, citing poor empirical performance (see for example \citet[\S25.9]{vandervaart}, \citet[Rem.~2.3~and~\S2.2.4]{chern}, \citet[\S4.6]{tsiatis}). We illustrate some of the issues involved in the case of the partially linear model.

\subsection{Motivation: the partially linear model}\label{sec:PLMmodel}
Given a triple $S:=(Y,X,Z)\in\R\times\R\times \mathcal{Z}$ the partially linear model posits the regression model
\begin{equation}\label{eq:PLM}
    Y = X\theta_0 + f_0(Z) + \varepsilon,
\end{equation}
where the function $f_0:\mathcal{Z}\to\R$ is unknown, and where the error term $\varepsilon$ satisfies $\E\left[\varepsilon\given X,Z\right]=0$ and $v_0(X,Z) := \Var\left(\varepsilon\given X,Z\right)<\infty$ almost surely. The target of inference is the parameter $\theta_0 \in \R$ which characterises the contribution of the covariate $X$ to the response after accounting for the additional predictors $Z$. This model has been studied in great detail \citep{engle, green, rice, chen, speckman, robinson, schick-plr, liang-zhou, hardle, ma, you, chern} and  the efficient influence function $\psi_{\text{eff}}$ is proportional to 
\begin{gather}\label{eq:EIF}
    \psi_{\text{eff}}(S;\theta,\eta_0) := \frac{1}{v_0(X,Z)}\left(X - h_0(Z)\right)\left(Y - X\theta - f_0(Z)\right), \\
\label{eq:h}
    h_0(Z) := \E\left[\frac{1}{v_0(X,Z)}\,\bigg|\,Z\right]^{-1}\E\left[\frac{X}{v_0(X,Z)}\,\bigg|\,Z\right],
    \qquad
    \eta_0 = (h_0,f_0,v_0);
\end{gather}
see \citet{ma} for a derivation. Suppose we are given i.i.d.\ data $S_i:=(Y_i, X_i, Z_i)$, $i=1,\ldots,n$, where $S_1 \eqdist S$. A corresponding estimator $\hat{\theta}_{\text{eff}}$  that can in principle asymptotically achieve the minimum attainable variance among regular estimators is (under additional conditions) then obtained through solving for $\theta$ the estimating equation
\begin{equation} \label{eq:est_eqn}
\sum_{i=1}^n \psi(S_i;\theta,\hat{\eta})=0;
\end{equation}
see \citet[\S25]{vandervaart}. 
Here $\psi = \psi_{\text{eff}}$ and $\hat{\eta} = (\hat{h}, \hat{f}, \hat{v})$ collects together estimates of the nuisance functions $(h_0, f_0, v_0)$. Moreover, the efficient influence function satisfies what is known as a double robustness property a key consequence of which is that the primary component of the bias of the corresponding estimator will be negligible when
\[
\E\left[\big(\hat{h}(Z) - h_0(Z)\big)^2 \,\big|\, \hat{h} \right] \E\left[ \big(\hat{f}(Z) - f_0(Z)\big)^2 \,\big|\, \hat{f} \right] = o_P(n^{-1}).
\]
When this is satisfied we may expect $\sqrt{n}(\hat{\theta}_{\text{eff}} - \theta_0) \indist \mathcal{N}(0, \sigma_{\text{eff}}^2)$, for a minimal attainable variance $\sigma_{\text{eff}}^2$.
Crucially, the appearance of a product of errors on the left-hand side permits relatively slow rates for each of these. 
So far, so good. 
However, while $f_0(Z) = \E(Y \given Z) - \theta_0 (X - \E(X \given Z))$ is an interpretable quantity directly linked to the target $\theta_0$ and whose estimation may be informed by domain knowledge, $h_0$ depends intricately on the distribution of $X$ given $Z$ and the conditional variance $v_0$, and can be challenging to estimate well, even with fairly large sample sizes. Figure~\ref{fig:surface} shows that $h_0$ can be prohibitively complicated even when $\E(Y \given X, Z)$, $\E(X \given Z)$ and $v_0$ are all very simple, and Figure~\ref{fig:EIF_performing_poorly} demonstrates the impact $h_0$ being too complex can have on the estimation quality of $\hat{\theta}_{\text{eff}}$; full details of the simulation setup are given in Section~\ref{sec:sim1}.

\begin{figure}[ht]
   \centering
   \includegraphics[width=1\textwidth]{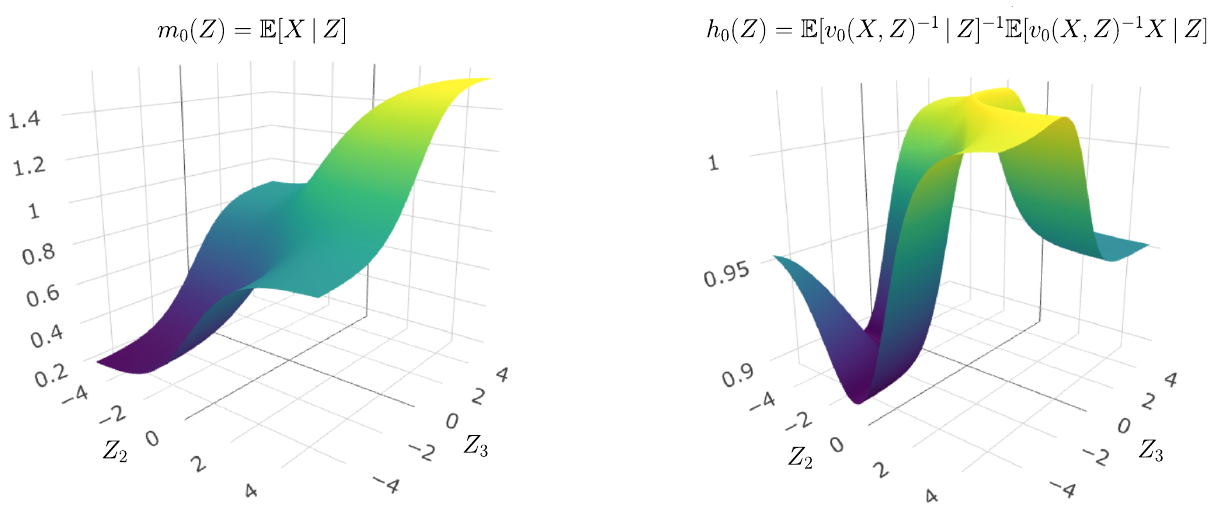}
   \caption{Surface plots of the functions $(z_2,z_3) \mapsto m_0(-1,z_2,z_3)$ and $(z_2, z_3) \mapsto h_0(-1,z_2,z_3)$ for the setting considered in Figure~\ref{fig:EIF_performing_poorly}b (see also Section~\ref{sec:sim3}).}
    \label{fig:surface}
\end{figure}

To address this, \citet{chern} for example instead consider solving \eqref{eq:est_eqn} with $\psi$ given by the influence function (up to a constant of proportionality)
\begin{equation}\label{eq:unwIF}
    \psi_{\mathrm{unw}}(S;\theta,\eta_0(Z)) := \left(X - m_0(Z)\right)\left(Y - X\theta - f_0(Z)\right),
    \qquad
    \eta_0 = (f_0, m_0),
\end{equation}
where $m_0(z) := \E_P\left[X\given Z=z\right]$. We refer to the corresponding estimator $\hat{\theta}_{\text{unw}}$ as the \textbf{unw}eighted estimator, for reasons that will become clear in the following. 
In contrast to $\hat{\theta}_{\text{eff}}$, the unweighted estimator $\hat{\theta}_{\text{unw}}$ enjoys $\sqrt{n}$-consistency under the more benign condition that
\begin{equation} \label{eq:unweighted_criterion}
    \E\left[\big(\hat{m}(Z)-m_0(Z)\big)^2\,\Big|\,\hat{m}\right]
    \E\left[\big(\hat{f}(Z)-f_0(Z)\big)^2\,\Big|\,\hat{f}\right]
    = o_P(n^{-1}).
\end{equation}
Figure~\ref{fig:EIF_performing_poorly} demonstrates how this can sacrifice efficiency, but may be a safer choice compared to the semiparametric efficient estimator $\hat{\theta}_{\text{eff}}$.

\begin{figure}[ht]
   \centering
   \includegraphics[width=1.0\textwidth]{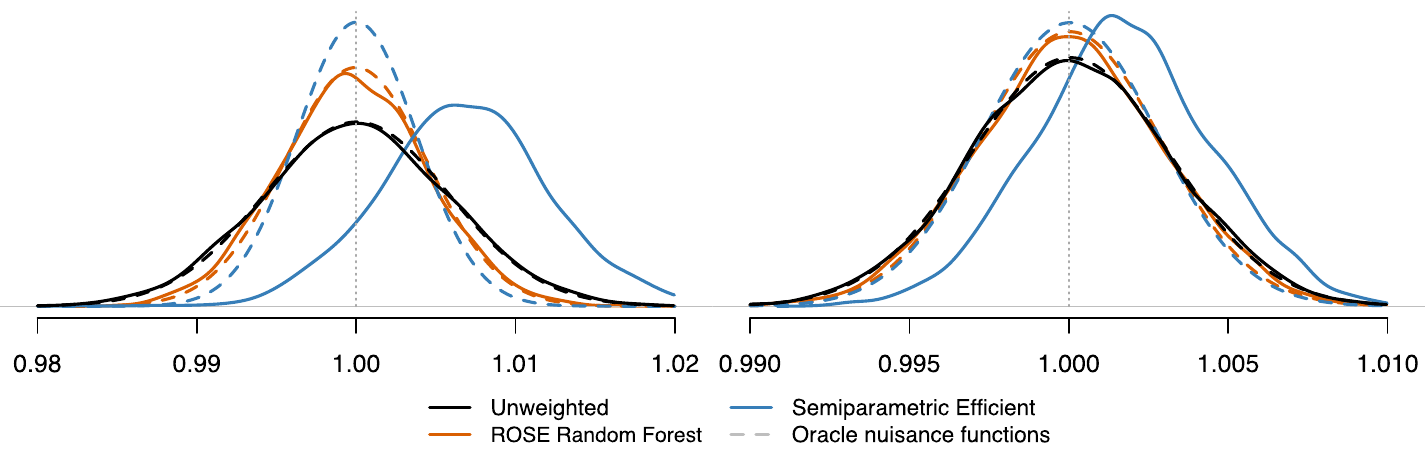}
    \caption{Kernel density estimates for the $\theta$-estimators: the unweighted estimator~\eqref{eq:unwIF}; the semiparametric efficient estimator~\eqref{eq:EIF}; and the robust semiparametric efficient (ROSE) estimator for a partially linear model ($4000$ simulations, each of sample size $n=80\,000$). Figure 2a (left) is an adapted version of the simulation in Section~\ref{sec:sim1} (see Section~\ref{appsec:Fig2a}) and Figure 2b (right) corresponds to the simulation in Section~\ref{sec:sim3} for which some associated nuisance functions are plotted in Figure~\ref{fig:surface}. Dashed curves correspond to versions of the associated estimators using oracle nuisance functions.}
    \label{fig:EIF_performing_poorly}
\end{figure}

\subsection{Our contributions and organisation of the paper}\label{sec:contributions}
As we have discussed, choosing an estimator based on the notion of efficiency, while theoretically sound, is often not well-founded from a practical perspective. On the other hand, the estimator $\hat{\theta}_{\text{unw}}$, which is apparently more robust in a certain sense, seems somewhat ad hoc. These issues raise the following questions which we seek to address in this work:
\begin{enumerate}[label=(\roman*)]
    \item How can we formalise the notion of robustness afforded by $\hat{\theta}_{\text{unw}}$ more generally?
    \item What is the maximal class of estimators that retain this form of robustness?
    \item How can we pick the optimal estimator from among this class?
\end{enumerate}
We study these questions in the context of a broad collection of conditional moment models for the triple $(Y, X, Z)$ which includes the class of generalised partially linear models. The target of inference is a scalar parameter $\theta_0$, and $Z$ is permitted to contribute to $\E(Y \given X, Z)$ through an unknown nonparametric function $f_0$. Our final goal then is to arrive at a principled alternative to choosing estimators based on efficient influence functions within such models, that bridges the divide between practical performance and semiparametric efficiency.

Our answer to (i) involves recognising that it is typically desirable to understand the performance of an estimator \emph{uniformly} over a class of plausible data-generating distributions. 
No estimator can retain good performance uniformly over the class of \emph{all} distributions in our setting. Indeed, with such an estimator we would then  be able to test the null hypothesis $\theta_0=0$, but this contains the smaller conditional independence null $Y \independent X \given Z$, for which \citet{shahpeters} prove there exist no non-trivial tests when $Z$ is a continuous random variable. Thus the use of any estimator necessarily involves an implicit choice of subsets of distributions, over which the estimator does have good performance.
In particular, to distinguish the contribution to the response of $X$ from the remaining predictors $Z$ requires that at least some aspects of the conditional distribution of $X$ given $Z$ are estimable. 
We recommend specifying that certain conditional moments $\E[M_j(X) \given Z=z]$, for user-chosen functions $M_j$ for $j=1,\ldots,J$, are relatively simple and can hence be estimated sufficiently well. For instance, in the case of the partially linear model we may take $J=1$ and $M_1(X) = X$ and ask that $\E(X \given Z=z)$ be estimated well, as also required by the unweighted estimator. 
 
Regarding question (ii), we argue that even the requirement of good performance of the $\theta$ estimator \emph{locally} uniformly over a class of distributions of the form given above narrows down the set of corresponding influence functions to a subset  parametrised by weight functions  $w_j:\mathcal{Z} \to (0, \infty)$ for $j=1,\ldots,J$. For instance, in the case of the partially linear model and taking $J=1$ and $M_1$ as above, we obtain the set of influence functions taking the form
 \begin{equation}\label{eqn:rose-IF}
    \psi_{w}(S;\theta,\eta) := w(Z)(X-m(Z))(Y-X\theta-f(Z)),
    \qquad
    \eta = (f, m),
\end{equation}
where the  weight function $w$ is to be chosen; comparing to \eqref{eq:unwIF} explains the tagline of $\hat{\theta}_{\text{unw}}$ as the `unweighted' estimator. We show that estimators derived through solving \eqref{eq:est_eqn} with $\psi= \psi_{w}$ for some $w$ enjoy an attractive robustness property: the primary condition guaranteeing their $\sqrt{n}$-consistency is precisely that required by the unweighted estimator \eqref{eq:unweighted_criterion}.

Similarly to how a semiparametric efficient estimator arises through picking from among all influence functions, that which yields the minimal variance, we can think of the estimator based on picking among influence functions of the form \eqref{eqn:rose-IF} as a
`\textbf{ro}bust \textbf{s}emiparametric \textbf{e}fficient' (ROSE) estimator. We introduce a new random forest-based procedure (ROSE forests) that we prove
estimates the corresponding optimal weight function $w^*$ (or weight functions in the case where $J>1$) consistently. As a consequence, the resulting estimator of $\theta$ achieves the minimal asymptotic variance among those estimators our theory identifies as robust in the sense described above. 
Figure~\ref{fig:EIF_performing_poorly} shows the performance of our ROSE random forest estimator compared to a semiparametric efficient estimator and an unweighted estimator in the two settings, including the simulation in Section~\ref{sec:sim1} referred to earlier. We see in particular that although the underlying influence function is doubly robust, even at large sample sizes, the semiparametric estimator can suffer from substantial bias. Both the unweighted at ROSE random forest estimators have negligible bias, but the latter improves on the variance of the former. Note that this improvement comes essentially for free, as both estimators lie in the same robustness class only requiring estimation of the nuisance function $\E(X \given Z=z)$ in addition to $f_0$. 

The rest of the paper is organised as follows. After reviewing some related work and introducing relevant notation in Sections~\ref{sec:literature} and \ref{sec:notation} respectively, in Section~\ref{sec:robust-class} we formally set out the class of conditional moment models we consider in this work. Given user-chosen functions $M_1, \ldots, M_J$, we derive the subset of influence functions with the robustness properties sketched out above, showing in particular that they may be parametrised in terms of weight functions. 
In Section~\ref{sec:robust-eff-est} we set out a general scheme for estimating $\theta_0$ based on a (potentially) data-driven choice of weight functions and show that the resulting estimators are uniformly asymptotically Gaussian. Section~\ref{sec:rrfs} introduces our ROSE forest estimator in the case $J=1$ and proves that using this to estimate the weight function(s) results in an estimator with minimal asymptotic variance among those considered to be robust. We further demonstrate that our approach can outperform an unweighted estimator, which in turn can outperform a locally efficient estimator (see Section~\ref{sec:literature} and Section~\ref{appsec:div-ratio-gplm}), by arbitrary magnitudes. The supplementary material contains the proofs of all results presented in the main text, additional theoretical results, an extension of the ROSE random forest estimator to the case where $J>1$ and further details on the examples and numerical experiments. ROSE random forests are implemented in the \texttt{R} package~\texttt{roseRF}\footnote{\url{https://github.com/elliot-young/roseRF/}}.

\subsection{Related literature} \label{sec:literature}

As discussed earlier, a number of authors have informally noted that efficient semiparametric estimators may suffer from robustness issues and not be computationally practical~\citep{chern, chen-odds, tchetgen, tan, liu}. In a similar vein, a number of authors have recognised that the relaxed rates certain nuisance functions are commonly asked to achieve are often unattainable~\citep{dukes, vansteelandt-comments}. 
The classical notion of \emph{local efficiency}, which refers to an estimator being efficient in a working submodel of the semiparametric model being considered, and a consistent estimator outside this submodel, represents one attempt to justify the use of estimators that do not use the efficient influence function~\citep{tsiatis, locallyeff1, locallyeff2, targeted-learning}. This idea is commonly used in the context of generalised estimating equations for clustered data~\citep{liangzeger, geehardin, ziegler} in the form of a `working (co)variance' model under which estimators are locally efficient. Recent work by~\citet{sandwich-boosting} highlights how in the setting of the partially linear model, while locally efficient estimators may perform well when postulated working submodel holds, there may be arbitrary suboptimal in the rest of the semiparametric model. This work further proposes to use a weighted least squares approach in the context with weights determined through minimising a sandwich estimate of the variance of the final estimator. We build on these ideas here,  with our ROSE random forest method providing a concrete weight estimation procedure that we show can estimate optimal weights in the broader set of semiparametric models we consider here.

Our work also connects to a literature on random forests, which have received a lot of interest in recent years due to their success when deployed on modern datasets. Since their inception~\citep{randomforest} a number of variants have been proposed for different statistical problems~\citep{meinshausen,grfs,distrandfor} that use alternative splitting and evaluation procedures. Another favourable property of random forests are the theoretical guarantees they possess, notably those of consistency~\citep{meinshausen, biau} and asymptotic normality~\citep{wagerwalther,grfs,grfannals}; our result (Theorem~\ref{thm:rose_forest_consistency}) for ROSE random forests draws on some of these ideas.

\subsection{Notation}\label{sec:notation}
We denote $\Phi$ as the cumulative distribution of a standard Gaussian distribution. We will also use the shorthand $[n]:=\{1,\ldots,n\}$ for $n\in\mathbb{N}$. We write $\R_+$ for the set of positive real numbers. For the uniform convergence results we will present it will be helpful to write, for a distribution $P$ governing the distribution of a random variable $U\in\R$, $\E_PU$ for its expectation, and $\PP_P(U\in B)=:\E\ind_B(U)$ for any measurable $B\in\R$. Further, given a family of probability distributions $\cP$ and a sequence of families of real-valued random variables $(A_{P,n})_{P\in\cP, n\in\mathbb{N}}$, we write $A_{P,n}=o_\cP(1)$ if $\lim_{n\to\infty}\sup_{P\in\cP}\PP_P(|A_{P,n}|>\epsilon)=0$ for all $\epsilon>0$, $A_{P,n}=o_\cP(f(n))$ for a given function $f:(0,\infty)\to(0,\infty)$ if $f(n)^{-1}A_{P,n}=o_\cP(1)$, and $A_{P,n}=O_\cP(1)$ if for any $\epsilon>0$ there exists $M_\epsilon,N_\epsilon>0$ such that $\sup_{n\geq N_\epsilon}\sup_{P\in\cP}\PP_P(|A_{P,n}|\geq M_\epsilon)<\epsilon$. In several places we consider the expected estimation errors of regression or other nuisance function estimates on new data points, conditional on the data used to train them, and any additional randomness in the estimation procedure. For such an estimate $\hat{\eta}$, we will denote the relevant conditional expectation by $\E( \cdot \given \hat{\eta})$.

\section{Robust influence functions}\label{sec:robust-class}
In this section, we introduce a broad family of semiparametric models where we argue that in order for an estimator of the target parameter to be robust in the sense sketched out in Section~\ref{sec:contributions}, its influence function should take a particular form. (In fact our results and methodology in Sections~\ref{sec:robust-eff-est} and~\ref{sec:rrfs} are applicable more broadly to collections of models including instrumental variables models for example; we discuss this further in Appendix~\ref{appsec:additional-models}.)

\subsection{Partially parametric models}
Consider a semiparametric model $\mathcal{P}$ for the random triple $S:=(X, Y, Z) \in \mathcal{X} \times \mathcal{Y} \times \mathcal{Z}\subseteq \R \times \R \times \R^d$ characterised by the conditional mean specification
\begin{equation}\label{eq:semipara-model}
	\E_P[Y\given X,Z] = \mu(X,Z,\theta_P,f_P(Z)).
\end{equation}
Here $\mu$ is a known function, $\theta_P \in \Theta \subseteq \R$ is the parameter of interest and $f_P: \mathcal{Z}\to \R$ is an unknown nuisance function.   We refer to such a model as a \emph{partially parametric restricted moment model} or simply a \emph{partially parametric model} for brevity, following the naming scheme of~\citet{tsiatis}, who consider the conditional mean model taking a fully parametric form. In contrast, the `partial' aspect emphasises the additional flexibility given by the nuisance function $f_P$. We will require $\mathcal{P}$ to be such that $\theta_P$ and $f_P$ are identifiable, i.e.\ for each $P \in \mathcal{P}$ there exists a unique $\theta_P$ and measurable function $f_P$ such that \eqref{eq:semipara-model} holds. Typically this will be ensured by an additional requirement that $\E \Var( X \given Z) >0$, though identifiability will come as a consequence of other conditions we will impose in later results.

Our leading example of partially parametric models will be the class of generalised partially linear models specified by
\begin{equation}\label{eq:gplm}
	g\big(\E_P[Y\given X,Z]\big) = X\theta_P + f_P(Z),
\end{equation}
where $g$ is a known, strictly increasing link function (note that we do not constrain the conditional variance). This includes the partially linear model discussed in Section~\ref{sec:PLMmodel} which takes the link $g$ to be the identity function. However the family of partially parametric models also includes nonlinear models of the form
$
\E_P[Y\given X,Z] = e^{X\theta_P} + f_P(Z)
$
for example. 

As the partially parametric model is a strictly semiparametric model, there exists a family of influence functions. The following result shows that this family takes a particular form and may be parametrised by the collection of functions $\phi : \mathcal{X}\times \mathcal{Z}\to \R$ that have mean zero conditional on $Z$.
\begin{theorem}[Influence functions of the partially parametric model]\label{thm:IF}
    Consider the partially parametric model~\eqref{eq:semipara-model} satisfying Assumption~\ref{ass:IF} (see Appendix~\ref{appsec:semipara}). Each influence function $\psi$ of this model takes the form
    \begin{equation}\label{eq:IF}
        \psi(y,x,z;\theta_P,f_P,\phi_P) =
        \frac{\big(y-\mu(x,z;
        \theta_P,f_P(z))\big)\phi_P(x,z)}{\frac{\partial}{\partial\gamma}\big|_{\gamma=f_P(z)}\mu(x,z;\theta_P,\gamma)},
    \end{equation}
    for some $\phi_P:\cX\times\cZ\to\R$ satisfying $\E_P[\phi_P(X,Z)\given Z=z]=0$ for all $z\in\cZ$.
\end{theorem}
All influence functions of the form~\eqref{eq:IF} satisfy a Neyman orthogonality~\citep{neyman, neyman2} property, corresponding to an insensitivity in expectation of the function about first order perturbations of the nuisance functions (in this case $\phi_P$ and $f_P$); see Appendix~\ref{appsec:NO}. In our case, this amounts to the fact that for any estimators $(\hat{f},\hat{\phi})$ of $(f_P,\phi_P)$,
\begin{equation*}
	\E_P\bigg[\frac{\partial}{\partial t}\Big|_{t=0} \psi\big(Y,X,Z;\theta_P, f_P+t(\hat{f}-f_P), \phi_P+t(\hat{\phi}-\phi_P)\big)\bigggiven \hat{f},\hat{\phi}\bigg] = 0.
\end{equation*}
Therefore the dominant term in the expansion of \begin{equation*}
	\E_P\big[\psi(Y,X,Z;\theta_P,\hat{f},\hat{\phi})-\psi(Y,X,Z;\theta_P,f_P,\phi_P)\biggiven \hat{f},\hat{\phi}\big],
\end{equation*}
is quadratic in the estimation errors of the nuisance functions $f_P$ and $\phi_P$ (rather than linear) and may be controlled by the mean squared errors $\E_P[(f_P(Z) - \hat{f}(Z))^2 \given \hat{f}]$ and $\E_P[(\phi_P(X, Z) - \hat{\phi}(X, Z))^2 \given \hat{\phi}]$. Thus
\begin{equation*} 
	\sqrt{n}\,\E_P\big[\psi(Y,X,Z;\theta_P,\hat{f},\hat{\phi})-\psi(Y,X,Z;\theta_P,f_P,\phi_P)\given\hat{f},\hat{\phi}\big] = o_P(1),
\end{equation*}
if these mean squared errors converge at faster than relatively slow $n^{-1/2}$ rates. 
The significance of this for estimating $\theta_P$ is that then, given i.i.d.\ copies $S_{1}, \ldots, S_n$ of $S$, the solution in $\theta$ to the estimating equation
\begin{equation} \label{eq:estimating_eqn}
\sum_{i=1}^n \psi(S_i; \theta, \hat{f}, \hat{\phi}) = 0
\end{equation}
will, at the $\sqrt{n}$ scale, be asymptotically equivalent to the solution of an oracular version of the above with estimators $\hat{f}$ and $\hat{\phi}$ replaced by their true values $f_P$ and $\phi_P$. Such an estimator will therefore enjoy asymptotic normality and $\sqrt{n}$-consistency (see Theorem~\ref{thm:theta_est_asymp_norm} to follow). However, we argue below that the crucial sufficiently fast  rates for estimating the nuisance functions $f_P$ and $\phi_P$ that allow for this should not be taken for granted.

\subsection{Slow rates for estimating nuisance functions}\label{sec:hardness}
As discussed in Section~\ref{sec:contributions}, the required rates on the mean squared errors in estimating $f_P$ and $\phi_P$ cannot be guaranteed to hold over all data-generating distributions $P$, and $\sqrt{n}$-consistency of an estimator of $\theta_P$ can only hold uniformly over some subset $\mathcal{P}$ of distributions. The rate requirement in estimating $f_P$ with a specified regression procedure implicitly determines some constraints on $\mathcal{P}$. On the other hand, the additional constraints to be imposed by the estimation of $\phi_P$ are not set in stone as we have a choice in which $\phi_P$ to target. 
Note that each function of $(X,Z)$ that is mean zero conditional on $Z$ may be written as 
\begin{equation}\label{eq:phi-UW}
	W(X,Z)\bigg(U(X,Z)-\frac{\E_P[W(X,Z)U(X,Z)\given Z]}{\E_P[W(X,Z)\given Z]}\bigg).
\end{equation}
Thus a rather general way of constructing an estimate $\hat{\phi}$ for use in \eqref{eq:estimating_eqn} is to pick some functions $U$ and $W$ and estimate the conditional mean $h_P$
\begin{equation*}
	h_P(z) := \frac{\E_P[W(X,Z)U(X,Z)\given Z=z]}{\E_P[W(X,Z)\given Z=z]}.
\end{equation*}
For example, the unweighted estimator \eqref{eq:unwIF} is obtained by taking $W(X,Z)=1$ and $U(X,Z)=X$. The choices of $U$ and $W$ for the $\phi_P$ resulting in the efficient estimator are given in Appendix~\ref{appsec:semi-eff-class}.
While these latter choices may be appealing in theory, their complexity suggests that the resulting constraints on $\mathcal{P}$ may be impractical in some settings.
We instead advocate imposing modelling constraints on $\mathcal{P}$ by first specifying a class functions $(M_j)_{j=1}^J$ of $X$ whose conditional means given $Z$ are sufficiently regular, and then considering only those $\phi_P$ which can be estimated well, given these assumptions.

In the following, we formalise these ideas and arrive at a resulting set of functions $\phi_P$. Suppose that $\mathcal{X} \subset \R$ and $\mathcal{Z} \subset \R^d$ are open sets. Let $\mathcal{P}_{XZ}$ be a set of distributions $P$ for $(X, Z)$ such that for given measurable functions $M_j :\mathcal{X}\to \R$ ($j=1,\ldots,J$) and associated (non-empty) classes $\mathcal{M}_j$ of functions from $\mathcal{Z} \to \R$ we have the following:
\begin{enumerate}[label=(\roman*)]
	\item $m_{P,j}(z) := \E_P[M_j(X)\given Z=z]$ exists and $m_{P,j} \in \mathcal{M}_j$ for $j=1,\ldots,J$.
	\item $P$ has a density $p$ with respect to Lebesgue measure and writing $p_Z$ and $p_{X|Z}$ for the corresponding marginal density for $Z$ and conditional density for $X$ given $Z$ respectively, $0 < \inf_{z \in \mathcal{Z}} p_Z(z), \, \inf_{x \in \mathcal{X}, z\in \mathcal{Z}} p_{X|Z}(x | z)$ and $\sup_{z \in \mathcal{Z}} p_Z(z), \, \sup_{x \in \mathcal{X}, z\in \mathcal{Z}} p_{X|Z}(x| z) < \infty$. Further suppose that the  functions $z \mapsto p_{X|Z}(x|z)$ for $x \in \mathcal{X}$ are equicontinuous.
\end{enumerate}
Condition (i) should be regarded as the primary restriction on $\mathcal{P}_{XZ}$. The function classes $\mathcal{M}_j$ can be thought of smoothness classes such as H\"{o}lder classes, but can also be singletons for example. The additional mild regularity conditions in (ii) are made mainly for convenience. Now consider $(X_1, Z_1), (X_2, Z_2),\ldots$ formed of i.i.d.\ copies of the tuple $(X, Z)$  and a sequence of estimators $(\hat{\phi}_n)_{n \in \mathbb{N}}$ of $\phi_P$ where each $\hat{\phi}_n$ is trained on the data $(X_i, Z_i)_{i=1}^n$. Theorem~\ref{thm:slow-rates} below shows that even locally uniformly consistent estimation of $\phi_P$ within $\mathcal{P}_{XZ}$ is not possible unless $\phi_P$ takes a particular form dictated by the conditional mean conditions in (i) above. Recall that the total variation distance between distributions  $P_1$ and $P_2$ with densities $p_1$ and $p_2$ with respect to Lebesgue measure on $\mathcal{X}\times \mathcal{Z}$ is given by
\[
{\TV}(P_1,P_2) := \frac{1}{2}\int_{\mathcal{X} \times \mathcal{Z}} |p_1(x, z) - p_2(x,z)| \,dx dz. 
\]

\begin{theorem}[Slow rates for estimating nuisance functions]\label{thm:slow-rates}
Suppose $U, W :\cX\times\cZ\to\R\times\R_+$ are bounded functions and $\inf_{x \in \mathcal{X}, z \in \mathcal{Z}} W(x, z) > 0$. Suppose further that the families of functions $\{z \mapsto U(x, z) : x \in \mathcal{X}\}$ and $\{z \mapsto W(x, z) : x \in \mathcal{X}\}$ are equicontinuous. For each $P \in \mathcal{P}_{XZ}$, define
\begin{equation*}
	\phi_P : (x,z) \mapsto W(x,z)\bigg(U(x,z) - \frac{\E_P[W(X,Z)U(X,Z)\given Z=z]}{\E_P[W(X,Z)\given Z=z]}\bigg).
\end{equation*}
If there exists a $P^*\in\cP_{XZ}$ such that
\begin{equation*}
    \phi_{P^*}\not\in\bigg\{(x,z)\mapsto \sum_{j=1}^J w_j(z)\big(M_j(x)-\E_{P^*}[M_j(X)\given Z=z]\big) \;:\; w_j:\cZ\to\R \bigg\},
\end{equation*}
then for every positive sequence $(a_n)_{n \in \mathbb{N}}$ converging to $0$ and every sequence of estimators $(\hat{\phi}_n)_{n \in \mathbb{N}}$ of $\phi_P$, we have that there exists a sequence of distributions $(P_n)_{n \in \mathbb{N}} \subset \mathcal{P}_{XZ}$ such that
\begin{equation*}\label{eq:slow-rates-2}
	\TV(P_n, P^*) \to 0 \qquad \text{and} \qquad  \underset{n\to\infty}{\limsup} \frac{\E_{P_n}\Big[\big(\hat{\phi}_n(X,Z)-\phi_{P_n}(X,Z)\big)^2\Big]}{a_n}\geq1.
\end{equation*}
\end{theorem}
In words, Theorem~\ref{thm:slow-rates} says the following: suppose that (essentially) all we are willing to assume about the joint distribution of $(X, Z)$ is that it lies in a class $\mathcal{P}_{XZ}$ where certain conditional means of functions $M_j$ of $X$ given $Z$ are estimable, for $J=1,\ldots,J$; here `estimable' may be defined in any user-chosen way and can mean that the conditional means are sufficiently smooth, or obey any chosen structural conditions such as being additive or linear, for example. Then any choice of $\phi_P(X, Z)$ that is not a linear combination of $M_j(X) - \E_P(M_j(X) \given Z)$ weighted by some user-chosen weight functions $w_j(Z)$ may not result in sufficiently good estimation of $\phi_P$: for any $P^* \in \mathcal{P}_{XZ}$ (however favourable for estimation of $\phi_{P^*}$) there exist distributions $P$ in $\mathcal{P}_{XZ}$ that are arbitrarily close to $P^*$ in total variation distance where estimation of $\phi_P$ fails. In summary, unless  $\phi_P$ takes the form
\begin{equation}\label{eq:w-M-m}
	\phi_P : (x,z) \mapsto \sum_{j=1}^J w_{j}(z)\big(M_j(x)-m_{P,j}(z)\big),
\end{equation}
for some weight functions $w_{1},\ldots,w_{J} : \cZ\to\R$ (potentially depending on $P$), estimation of $\theta_P$ cannot be expected to be `robust' in the sense indicated above. Conversely, we explain in the next section that taking $\phi_P$ of this form and calculating $\hat{\theta}$ as the solution of the corresponding estimating equation \eqref{eq:est_eqn} with
\begin{equation}\label{eq:robust-estimating-eqns}
	\psi(S;\theta,\eta) = \sum_{j=1}^J w_j(Z)\big(M_j(X)-m_j(Z)\big)\frac{Y-\mu(X,Z;\theta,f(Z))}{\frac{\partial}{\partial\gamma}|_{\gamma=f(Z)}\mu(X,Z;\theta,f(Z))},
\end{equation}
does yield an estimator that is $\sqrt{n}$-consistent, uniformly over a class of distributions where the primary restrictions on the joint distribution of $(X, Z)$ are of the form given above. Here $\eta=(f,m_1,\ldots,m_J)$ collects the relevant nuisance functions. 

Clearly, there is a choice to be  made in the selection of the weight functions $w_j$ above. In Section~\ref{sec:rrfs} we set out a scheme for choosing these in a data-driven way to minimise the variance of the resulting $\hat{\theta}$ among all estimators of the form given above. We call such an estimator a \emph{\textbf{ro}bust \textbf{s}emiparametric \textbf{e}fficient} (ROSE) estimator.

We remark that this estimation strategy bears some similarities with the the generalised method of moments~\citep{gmm}, a popular practical methodology to estimate a finite dimensional parameter characterised by a set of moment equations. The method involves the choice of a weight matrix that is somewhat analogous to our weight functions above. The semiparametric model~\eqref{eq:semipara-model} we study instead specifies a finite dimensional parameter $\theta_P$ and an infinite dimensional parameter $f_P$ through a single \emph{conditional} moment equation, which is equivalent to an infinite-dimensional class of moment equations. It can be shown that solving estimating equations as in~\eqref{eq:robust-estimating-eqns} can be interpreted as an infinite dimensional extension of the generalised method of moments; we discuss this further in Appendix~\ref{appsec:GMMs}.

\section{Robust estimation}\label{sec:robust-eff-est}
The previous section motivated the use of a particular set of estimating equations involving user-chosen weight functions \eqref{eq:w-M-m} for estimating $\theta_P$ in the partially parametric model through a negative result indicating that other forms of estimating equations do not lead to robust estimates. In this section we prove a positive result demonstrating that these estimating equations do indeed lead to robust estimates of $\theta_P$ under mild conditions on the way the weights are determined based on the data. 
In Section~\ref{sec:theta-algorithms} we present our estimation strategy and in Section~\ref{sec:asymp_norm} we prove that it results in the desired uniform asymptotic properties. 
Section~\ref{sec:rrfs} then introduces a practical approach for estimating asymptotically optimal weights using a random forest-based procedure. 
We continue to present our results and methodology in the context of the partially parametric model~\eqref{eq:semipara-model}, though we explain in Appendix~\ref{appsec:additional-models} that they are applicable in a broader class of models that includes instrumental variables models, for example.

\subsection{Uniformly robust \texorpdfstring{$\theta_P$}{TEXT} estimators}\label{sec:theta-algorithms}
As indicated in Section~\ref{sec:hardness}, in order to estimate the parameter $\theta_P$ in the partially parametric model we can aim to solve a set of estimating equations of the form \eqref{eq:estimating_eqn} with each summand based on \eqref{eq:robust-estimating-eqns}. Here the nuisance functions to be estimated are $f_P$, and $(m_{P,j})_{j\in[J]}$, which we collect together in $\eta_P$, and weight
functions $w := (w_1,\ldots, w_J)$ are to be chosen. 
We work with the DML2 approach of~\citet{chern} which allows us to circumvent Donsker condition assumptions that restrict the learners that can be used to estimate $\eta_P$~\citep[\S25.8]{vandervaart}, though alternative approaches including DML1~\citep{chern} or rank transformed subsampling~\citep{guo} are also possible. This works as follows. We partition $[n]$ into $K$ folds $(\cI_k)_{k\in[K]}$ each of size at least $\floor{n/K}$, and for each $k \in [K]$, compute nuisance function estimates $\hat{\eta}^{(k)}$ and weights $\hat{w}^{(k)} = (\hat{w}^{(k)}_1,\ldots,\hat{w}^{(k)}_J)$ using data indexed by $\mathcal{I}_k^c$. We then solve for $\theta$ the estimating equation
\begin{gather}
    \sum_{k=1}^K \sum_{i\in\cI_k}\psi\big(S_i;\theta,\hat{\eta}^{(k)}(Z),\hat{w}^{(k)}(Z)\big) = 0, 
    \label{eq:DML-psi}
\end{gather}
where
\begin{gather}
    \psi(S;\theta,\eta(Z),w(Z)) := \sum_{j=1}^J w_j(Z)\psi_j(S;\theta,\eta(Z)), 
    \label{eq:w-psi}
    \\
    \psi_j(S;\theta,\eta(Z)) := \big(M_j(X)-m_j(Z)\big)\varepsilon(S;\theta,f(Z)),
    \label{eq:psi-j}
    \\
    \varepsilon(S;\theta,f(Z)) := \Big(\frac{\partial}{\partial\gamma}\Big|_{\gamma=f(Z)}\mu(X,Z;\theta,\gamma)\Big)^{-1} \big(Y-\mu(X,Z;\theta,f(Z))\big), \label{eq:eps-rest-mom-model}
\end{gather}
and $\eta = (m_1,\ldots,m_J,f)$ collects the relevant nuisance functions. 
It will also help to define the function $\partial_\theta\psi_j : \mathcal{S}\times\Theta\times\R^p$ to be any function (sufficiently regular; see Assumption~\ref{ass:data} to follow) satisfying
\begin{equation*}
    \E_P\big[\partial_\theta\psi_j(S;\theta_P,\eta_P(Z))\given X,Z\big] = \E_P\big[\nabla_{\theta}\psi_j(S;\theta_P,\eta_P(Z))\given X,Z\big].
\end{equation*}
A trivial choice would be $\partial_\theta=\nabla_\theta$, though alternatives may yield finite sample improvements and in our numerical experiments presented in Section~\ref{sec:numerical-results}, we use
\begin{equation} \label{eq:partial_theta_def}
    \partial_\theta\psi_j(S;\theta,\eta(Z)) = -\big(M_j(X)-m_j(Z)\big) \frac{\frac{\partial}{\partial\theta}\mu(X,Z;\theta,f(Z))}{\frac{\partial}{\partial\gamma}\big|_{\gamma=f(Z)}\mu(X,Z;\theta,\gamma)}.
\end{equation}
In practice, we solve the estimating equation above using Fisher scoring with the update
\[
\hat{\theta} \leftarrow \hat{\theta} - \Big(\sum_{k=1}^K\sum_{i\in\cI_k}\partial_\theta\psi(S_i;\hat{\theta},\hat{\eta}^{(k)},\hat{w}^{(k)})\Big)^{-1} \Big(\sum_{k=1}^K\sum_{i\in\cI_k}\psi(S_i;\hat{\theta},\hat{\eta}^{(k)},\hat{w}^{(k)})\Big).
\]
The procedure to calculate the estimator $\hat{\theta}$ is summarised in Algorithm~\ref{alg:main}, which also gives  constructions for a sandwich estimate of the variance of $\hat{\theta}$ and a two-sided confidence interval for $\theta_P$. Note that the algorithm is for a generic method for estimating the weight functions; in Section~\ref{sec:rrfs} we give a concrete procedure that we show is asymptotically optimal.

{
\RestyleAlgo{ruled}
\begin{algorithm}[htbp]\label{alg:main}
\KwIn{Data $(S_i)_{i\in[n]}$ associated with the $\psi$ function $\psi$~\eqref{eq:w-psi}; number of cross-fitting folds $K$; regression method(s) for estimating nuisance function(s) $\eta_P$; method for computing weight functions; significance level $\alpha \in (0, 1)$.}
{\bf Notation:} For notational convenience we will omit the dependence of $Z_i$ in $(\eta(Z_i),w(Z_i))$ i.e.~we will notate $\psi(S_i;\theta,\eta(Z_i),w(Z_i))$ as $\psi(S_i;\theta,\eta,w)$ etc.
\vspace{0.4em}
\\
Partition $[n]$ into $K$ disjoint sets $(\cI_k)_{k\in[K]}$ of approximately equal size.

\For{$k\in[K]$}{
    Compute estimator(s) $\hat{\eta}^{(k)}$ for nuisance function(s) $\eta_P$, using data $(S_i)_{i\in\cI_k^c}$.
    
Compute weight function(s) $\hat{w}^{(k)}$ using data $(S_i)_{i\in\cI_k^c}$.
        
        
    }
    Calculate 
    $\hat{\theta}$ as the solution to the equation $ \sum_{k=1}^K\sum_{i\in\cI_k}\psi(S_i;\theta,\hat{\eta}^{(k)},\hat{w}^{(k)})=0$.
    \\
    Calculate $\hat{V} := n\left(\sum_{k=1}^K\sum_{i\in\cI_k}\partial_\theta\psi(S_i;\hat{\theta},\hat{\eta}^{(k)},\hat{w}^{(k)})\right)^{-2}\left(\sum_{k=1}^K\sum_{i\in\cI_k}\psi^2(S_i;\hat{\theta},\hat{\eta}^{(k)},\hat{w}^{(k)})\right)$.
    \\
    Calculate $\hat{C}(\alpha) := \big[\hat{\theta}-n^{-\frac{1}{2}}\hat{V}^{\frac{1}{2}}\Phi^{-1}\big(1-\frac{\alpha}{2}\big),
    \;\; 
    \hat{\theta}+n^{-\frac{1}{2}}\hat{V}^{\frac{1}{2}}\Phi^{-1}\big(1-\frac{\alpha}{2}\big)\big]$.

\KwOut{Estimator $\hat{\theta}$ for $\theta_P$; sandwich estimator of its asymptotic variance $\hat{V}$; $(1-\alpha)$-level confidence interval $\hat{C}(\alpha)$ for $\theta_P$.}
\caption{Construction of a robust weighted $\theta$ estimator}
\end{algorithm}
}

\subsection{Asymptotic properties}\label{sec:asymp_norm}
We now show that the estimator $\hat{\theta}$ in Algorithm~\ref{alg:main} is asymptotically Gaussian uniformly over a class of distributions whose primary restriction is that the nuisance functions contained in $\eta_P$ are estimated sufficiently well. This sort of result is relatively standard in the semiparametric literature, though one interesting feature here is that the weight functions computed only need to converge in probability to some population level quantities, rather than having to estimate any pre-defined population level quantity at some rate.

We make use of the following assumptions on the model (implied through the function $\psi$ in~\eqref{eq:w-psi}) and the estimators of the nuisance functions. We also assume for simplicity that $K$ divides $n$ which permits us to use the shorthand $\hat{\eta}$ and $\hat{w}$ to denote any of the $K$-fold nuisance function and weight estimators respectively (that is any of $\big(\hat{\eta}^{(k)}\big)_{k\in[K]}$ and $\big(\hat{w}^{(k)}\big)_{k\in[K]}$), which will all have the same distribution.

\begin{assumption}[Model assumptions and quality of nuisance estimators]\label{ass:data}
The following assumptions are made on the function $\psi$~\eqref{eq:w-psi} (assumed to be twice differentiable with respect to the target and nuisance arguments). Denote by $S=(Y, X, Z)$ an observation independent of our data, generated from a law $P$ contained within the family of probability distributions $\cP$ that satisfies the properties below.
\begin{itemize}
\item[(\mylabel{ass:unique-solution}{M1})] (\emph{Unique Solution of $\psi$}) 
    For some compact target parameter space $\Theta \subset \R$, the solution $\theta_P\in\interior\Theta$ satisfies $\E_P[\psi_j(S;\theta_P,\eta_P(Z))]=0$ for each $j\in[J]$. Further, this solution is unique in the sense that for all $\epsilon>0$,
    \begin{equation*}
        \inf_{P\in\cP}\inf_{\underset{|\theta-\theta_P|\geq\epsilon}{\theta\in\Theta}}|\E_P[\psi_j(S;\theta,\eta_P(Z))]| > 0.
    \end{equation*}
\item[(\mylabel{ass:consistency-and-DR}{M2})] (\emph{Nuisance Function Estimation}) The nuisance function estimator $\hat{\eta}$ of $\eta_P$ (for each fold, trained on data independent of $Z$) converges at the rates
\begin{equation*}
    \E_P\Big[\|\hat{\eta}(Z)-\eta_P(Z)\|_2^2\biggiven\hat{\eta}\Big] = O_\cP(n^{-\lambda}), \tag{\emph{Consistency}}
\end{equation*}
for some $\lambda>0$, and for each $j \in [J]$,
\begin{equation*}
    \sup_{r\in[0,1]}\E_P\Big[\{\hat{\eta}(Z)-\eta_P(Z)\}^\top  \hat{H}_{P,j}(Z;r) \{\hat{\eta}(Z)-\eta_P(Z)\}\biggiven\hat{\eta}\Big] = o_\cP\big(n^{-\frac{1}{2}}\big), \tag{\emph{DR}}
\end{equation*}
where the double-robust rate Hessian matrix is given by
\begin{equation}\label{eq:hessian}
    \hat{H}_{P,j}(Z;r) := \E_P\big[\nabla_\gamma^2\psi_j(S;\theta_P,\gamma)|_{\gamma=\eta_P(Z)+r\{\hat{\eta}(Z)-\eta_P(Z)\}}\given 
    \hat{\eta}, Z\big].
\end{equation}
\item[(\mylabel{ass:bounded-moments}{M3})] (\emph{Bounded Moments}) There exists some $\delta>0$ such that for each  $j \in [J]$
\begin{equation*}
    \sup_{\theta\in\Theta}\sup_{r\in[0,1]}\E_P\Big[\big\|D\psi_j(S;\theta,\eta_P(Z)+r\{\hat{\eta}(Z)-\eta_P(Z)\})\big\|_2^{2+\delta}\biggiven\hat{\eta}\Big] = O_\cP(1),
\end{equation*}
for operators $D\in\{\text{id}, \nabla_\theta, \nabla_\theta^2, \partial_\theta\}$.
\item[(\mylabel{ass:bounded-w}{M4})] (\emph{Bounded Weights}) The estimated weights $\hat{w}$ satisfy $\norm{\hat{w}_j}_{\infty}\leq c_w$ almost surely for some finite constant $c_w>0$ and each $j \in [J]$.
\end{itemize}
\end{assumption}

Note that the regularity assumption~\ref{ass:bounded-moments} is stronger than required but presented for notational ease; see Appendix~\ref{appsec:regularity-conditions} for weaker regularity conditions. 
A sufficient condition for the double-robust rates condition of Assumption~\ref{ass:consistency-and-DR} to hold is for both
\begin{equation}\label{eq:C_DR}
    \sup_{r\in[0,1]}\Lambda_{\max}\Big\{\E_P\Big[\nabla_\gamma^2\psi_j(S;\theta_P,\gamma)|_{\gamma=\eta_P(Z)+r\{\hat{\eta}(Z)-\eta_P(Z)\}}\biggiven \hat{\eta},Z\Big]\Big\} \leq C_{\text{DR}},
\end{equation}
almost surely for some constant $C_{\text{DR}}>0$ for each $j \in [J]$, and the nuisance function estimators $\hat{\eta}$ to converge at the relaxed rates
\begin{equation}\label{eq:1/4-rates}
    \E_P\big[\|\hat{\eta}(Z)-\eta_P(Z)\|_2^2\given\hat{\eta}\big] = o_\cP\big(n^{-\frac{1}{2}}\big).
\end{equation}
In some settings, such as the partially linear model discussed in Section~\ref{sec:PLMmodel}, the Hessian~\eqref{eq:hessian} takes a relatively simple form (e.g.~certain diagonal terms being precisely zero), allowing for weaker rate requirements than in~\eqref{eq:1/4-rates}; see Appendix~\ref{appsec:DR} for further details on interpreting Assumption~\ref{ass:consistency-and-DR} in these contexts.

Equipped with Assumption~\ref{ass:data} defining the family $\cP$ of probability distributions our convergence results will be uniform over, Theorem~\ref{thm:theta_est_asymp_norm} presents the asymptotic properties of our estimator.

\begin{theorem}[Uniform asymptotic Gaussianity of the weighted estimator]\label{thm:theta_est_asymp_norm}
    Let Assumption~\ref{ass:data} hold for a class of distributions  $\mathcal{P}$.
    Suppose that for each $j \in [J]$,
    \[
    \E_P\big[(\hat{w}_j(Z) - w^*_{P,j}(Z))^2 \given \hat{w}\big] = o_{\mathcal{P}}(1)
    \]
    for some measurable functions $w^*_{P,1},\ldots,w^*_{P,J} : \mathcal{Z}\to \R$
    satisfying the non-singularity condition
    \begin{equation} \label{eq:invertible}
    	\inf_{P\in\cP}\sum_{j=1}^J\E_P[w_j^*(Z)\nabla_\theta\psi_j(S;\theta_P,\eta_P(Z))]>0.
    \end{equation}
    Then the estimator $\hat{\theta}$ (Algorithm~\ref{alg:main}) is uniformly asymptotically Gaussian:
    \begin{equation*}
        \lim_{n\to\infty}\supP\sup_{t\in\R}\left|\PP_P\left( \sqrt{n/V^*_P}\big(\hat{\theta}_{\hat{w}}-\theta_P\big) \leq t \right)-\Phi(t)\right| = 0,
    \end{equation*}
    where the asymptotic variance is given by
    \begin{equation}\label{eq:V(w)}
        V^*_P := 
        \frac{\E_P\Big[\big(\sum_{j=1}^J w^*_{P,j}(Z)\psi_j(S;\theta_P,\eta_P(Z))\big)^2\Big]}{\big(\sum_{j=1}^J\E_P\big[w^*_{P,j}(Z)\partial_\theta\psi_j(S;\theta_P,\eta_P(Z))\big]\big)^2}.
    \end{equation}
    Moreover, the above holds if $V^*_P$ is replaced by its estimator $\hat{V}$ in Algorithm~\ref{alg:main}.
\end{theorem}

Theorem~\ref{thm:theta_est_asymp_norm} justifies the uniform asymptotic validity of the confidence interval construction in Algorithm~\ref{alg:main}. 
The variance $V^*_P$ is in terms of the target $w^*_P$ of the weight functions. Recall from Section~\ref{sec:hardness} that a ROSE estimator is one which attains the minimal possible such variance. In the next section, we introduce a data-driven choice of weight functions that achieve this asymptotically.

\section{ROSE random forests}\label{sec:rrfs}

\subsection{Robust weight estimation}

Theorem~\ref{thm:theta_est_asymp_norm} derives in particular the form of the asymptotic variance of the estimator $\hat{\theta}$ and reveals how this depends on the weight functions (specifically their probabilistic limits) used in its construction.
We can in fact seek to minimise this variance over weight functions, thereby viewing it as a sort of loss function over the weights:
\begin{equation}\label{eq:sand-loss}
	L_{P,\SL}(w) 
	:=\; \frac{\E_P\left[\big(\sum_{j=1}^J w_{j}(Z)\psi_j(S;\theta_P,\eta_P(Z))\big)^2\right]}{\big(\sum_{j=1}^J\E_P\big[w_{j}(Z)\partial_\theta\psi_j(S;\theta_P,\eta_P(Z))\big]\big)^2}.
\end{equation}
This quantity, which is a population version of the sandwich estimator of the variance~\citet{huber}, when viewed as a function of $w$ is termed the \emph{sandwich loss} (SL) in \citet{sandwich-boosting}. 
The weighting scheme $w_P^{\rose} = (w_{P,1}^{\rose},\ldots,w_{P,J}^{\rose})$ that minimises this quantity, is (up to an arbitrary positive constant of proportionality) given by
\begin{gather}\label{eq:w_j}
    w_{P,j}^{\rose}(Z) := e_j^\top 
    \E_P\big[\bm{\psi}(S;\theta_P,\eta_P(Z)) \bm{\psi}(S;\theta_P,\eta_P(Z))^\top \given Z\big]^{-1}
    \E_P\big[\bm{\partial_\theta\psi}(S;\theta_P,\eta_P(Z)) \given Z\big],
    \\
    \bm{\psi} = (\psi_1,\psi_2,\ldots,\psi_J),
    \quad
    \text{where} \;\;\bm{\partial_\theta\psi} = (\partial_\theta\psi_1,\partial_\theta\psi_2,\ldots,\partial_\theta\psi_J),
    \notag
\end{gather}
provided the matrix in the display is invertible. This may be seen by an application of the Cauchy--Schwarz inequality; see Appendix~\ref{appsec:rose-weights}. If we could estimate the weight functions $w_P^{\rose}$ consistently, then by Theorem~\ref{thm:theta_est_asymp_norm} the resulting estimator $\hat{\theta}$ would enjoy the minimal asymptotic variance among all of the robust estimators we are considering. One approach to doing this, taken by \citet{sandwich-boosting} and which we follow here, is by minimising an empirical version of the population sandwich loss  \eqref{eq:sand-loss}. 
In Section~\ref{sec:rose-forests} we introduce an adapted random forest algorithm that implements a novel splitting rule to directly minimise the loss~\eqref{eq:sand-loss}, to produce a ROSE estimator.
One alternative involves estimating the function $w_{P}^{\rose}$ directly. In the case where $J=1$, this can be achieved through a weighted least squares regression. Whilst asymptotically such an approach should be equivalent to minimising the empirical sandwich loss, since the latter approach targets exactly the quantity we wish to minimise, and as overall the estimation problem is challenging, it can have better performance in finite samples, as we see empirically in Section~\ref{sec:sim2}. Another approach to estimate $w_{P}^{\rose}$ that explicitly takes account of the complexity of the estimation problem involves placing modelling assumptions on $w_{P}^{\rose}$. While this is a popular approach in related settings, we explain below why this can lead to unfavourable performance.

\subsubsection{On local efficiency}\label{sec:GEE}

One modelling assumption popular in the partially linear model in the case where $J=1$ and $M_1(X)=X$ \citep{emmenegger} is that
 $\Var_P(\varepsilon(S;\theta_P,f_P(Z))\given X,Z)\allowbreak=\allowbreak\Var_P(\varepsilon(S;\theta_P,f_P(Z))\given Z)$. Under this condition, $w_P^{(\mathrm{rose})}$ collapses to the simpler form
\begin{equation}\label{eq:w-GEE}
    w_P^{(\mathrm{loceff})}(Z) := \E_P\big[\left\{\partial_\theta\psi_1(S;\theta_P,\eta_P(Z))\right\}^{-1}\{\psi_1^2(S;\theta_P,\eta_P(Z))\}\given Z\big]^{-1},
\end{equation}
which can be estimated via a single (unweighted) least squares regression~\citep[\S4.6]{tsiatis}. Under the postulated working submodel a form of efficiency is therefore maintained, and outside this submodel, the estimator is still valid (i.e.\ $\sqrt{n}$-consistent). This property is sometimes known as `local efficiency'~\citep{tsiatis, locallyeff1, locallyeff2, targeted-learning} and the approach is particularly popular in the field of generalised estimating equations~\citep{liangzeger, geehardin, tsiatis, ziegler}. While local efficiency may appear attractive at first sight, the fact that outside the working submodel, the estimator can make no claim to any form of optimality, means that performance can deteriorate in cases where the working submodel is misspecified. The result below shows in particular that in such cases this approach can be arbitrarily worse than a simple unweighted estimator.

\begin{theorem}\label{thm:div-ratio-gplm}
    Consider the class of estimators given by Algorithm~\ref{alg:main} with $J=1$ and $M_1:x\mapsto x$ in the context of the semiparametric generalised partially linear model~\eqref{eq:gplm} with $\cY=\cX=\R$. Fix measurable functions $m_0, f_0 : \mathcal{Z} \to \R$, target parameter $\theta_0 \in \R$ and arbitrarily large constant $A>0$. Then for any marginal distribution for $Z$ where $Z$ is not almost surely constant, there exists a joint distribution $P_0$ for $(Y, X, Z)$ for which $\E_{P_0}[X\given Z]=m_0(Z)$, $\E_{P_0}[Y\given X,Z] = g^{-1}(X\theta_0+f_0(Z))$, and
    \begin{equation*}
	        \frac{V_{P_0,\mathrm{loceff}}}{V_{P_0,\mathrm{unw}}} \geq A
	        \quad\text{and}\quad
	        \frac{V_{P_0,\mathrm{unw}}}{V_{P_0,\mathrm{rose}}} \geq A.
	    \end{equation*}
	    Here  $V_{P,\mathrm{unw}}:=L_{P,\SL}(1)$, $V_{P,\mathrm{rose}}:=L_{P,\SL}(w_P^{\rose})$ and $V_{P,\mathrm{loceff}}:=L_{P,\SL}(w_P^{\loceff})$ are respectively the (scaled) population level asymptotic variances of the unweighted, ROSE (with $w_P^{\rose}$ as in~\eqref{eq:w_j}) and locally efficient (with $w_P^{\loceff}$ as in~\eqref{eq:w-GEE}) estimators of $\theta_P$.
\end{theorem}
The result above motivates direct minimisation of the sandwich loss. We introduce a specific scheme for doing this in the following.

\subsection{ROSE random forests}\label{sec:rose-forests}
Here we introduce the \emph{ROSE random forest}; a practical algorithm that directly minimises an empirical estimator for~\eqref{eq:sand-loss} to learn the optimal weights~\eqref{eq:w_j}. These estimated weight functions may then be used in Algorithm~\ref{alg:main} to arrive at a ROSE estimator for the target parameter; we refer to this as the ROSE random forest estimator. 
For simplicity, in this section we will consider the case where $J=1$. We present a generalisation for $J>1$ in Appendix~\ref{appsec:rose-forest-J>1}.

In order to arrive at an empirical version of \eqref{eq:sand-loss}, we require an estimate $\hat{\eta}$ of the nuisance parameters and a pilot estimator $\hat{\theta}_{\text{init}}$ of $\theta_P$. We compute the latter using Algorithm~\ref{alg:main}, taking the weights as equal to the constant $1$. Our theory will require that $\hat{\theta}_{\text{init}}$ and $\hat{\eta}$ are independent of the data used to construct the ROSE random forest weight function. Formally, this can be achieved through sample splitting. Since Theorem~\ref{thm:theta_est_asymp_norm} only requires consistent estimation of the weight functions, this sample-splitting would not affect the asymptotic variance of the final estimator and for example no additional cross-fitting would be necessary. In practice however, this independence is not needed and no such sample-splitting is used for any of our numerical experiments.

Given initial estimators $\hat{\theta}_{\text{init}}$ and $\hat{\eta}$, we define the empirical sandwich loss (in terms of data indexed by $i\in\cI$) for the weight function $w = w_1$ as
\begin{equation*}
    \hat{L}_{\SL}(w_1) := 
        \bigg(\sum_{i\in\cI}w_1(Z_i)\partial_\theta\psi_1(S_i;\hat{\theta}_{\text{init}},\hat{\eta}(Z_i))\bigg)^{-2}
        \bigg(\sum_{i\in\cI}w_1^2(Z_i)\psi_1^2(S_i;\hat{\theta}_{\text{init}},\hat{\eta}(Z_i))\bigg).
\end{equation*}
To develop our ROSE forest procedure, let us suppose that we have partitioned the state space $\cZ$ for $Z$ into disjoint rectangular regions $R\in\cR$ i.e.~$\cZ = \uplus_{R\in\cR}R$, with the function $w_1$ being piecewise constant on this partition (i.e.~$w_1(Z_i)=w_{1,R}$ if $i\in I_R := \{i:Z_i\in R\}$). The sandwich loss above then takes the form
\begin{equation}\label{eq:emp-sand-loss}
    \hat{L}_{\SL}(w_1) 
 := \bigg(\sum_{R\in\cR}w_{1,R}\sum_{i\in I_R}\partial_\theta\psi_1(S_i;\hat{\theta}_{\text{init}},\hat{\eta}(Z_i))\bigg)^{-2}\bigg(\sum_{R\in\cR}w_{1,R}^2\sum_{i\in I_R}\psi_1^2(S_i;\hat{\theta}_{\text{init}},\hat{\eta}(Z_i))\bigg).
\end{equation}
Unlike standard CART splitting objectives does not take the form of a minimisation problem over additive goodness-of-fit metrics over disjoint regions. The weights that minimise the sandwich loss  (up to an arbitrary multiplicative scalar constant) can nevertheless be analytically calculated as
\begin{equation*}
    w_{1,R} = \bigg(\sum_{i\in I_R} \psi_1^2(S_i;\hat{\theta}_{\text{init}},\hat{\eta}(Z_i)) \bigg)^{-1}   \bigg(\sum_{i\in I_R} \partial_\theta\psi_1(S_i;\hat{\theta}_{\text{init}},\hat{\eta}(Z_i)) \bigg),
\end{equation*}
with associated minimum being
\begin{equation*}
    \underset{(w_{1,R})_{R\in\cR}}{\min} \hat{L}_{\SL}(w_1) = \Bigg(\sum_{R\in\cR}\frac{\big(\sum_{i\in I_R} \partial_\theta\psi_1(S_i;\hat{\theta}_{\text{init}},\hat{\eta}(Z_i)) \big)^2}{\sum_{i\in I_R} \psi_1^2(S_i;\hat{\theta}_{\text{init}},\hat{\eta}(Z_i)) }\Bigg)^{-1}.
\end{equation*}
This suggests taking the goodness-of-fit metric for splitting an index set $I$ into the disjoint index sets $I'$ and $I''$ to be
\begin{equation*}
    \frac{\big(\sum_{i\in I^{'}}\partial_\theta\psi_1(S_i;\hat{\theta}_{\text{init}},\hat{\eta}(Z_i)) \big)^2}{\sum_{i\in I^{'}} \psi_1^2(S_i;\hat{\theta}_{\text{init}},\hat{\eta}(Z_i)) }
    + \frac{\big(\sum_{i\in I^{''}}\partial_\theta\psi_1(S_i;\hat{\theta}_{\text{init}},\hat{\eta}(Z_i)) \big)^2}{\sum_{i\in I^{''}} \psi_1^2(S_i;\hat{\theta}_{\text{init}},\hat{\eta}(Z_i)) } 
    - \frac{\big(\sum_{i\in I}\partial_\theta\psi_1(S_i;\hat{\theta}_{\text{init}},\hat{\eta}(Z_i)) \big)^2}{\sum_{i\in I} \psi_1^2(S_i;\hat{\theta}_{\text{init}},\hat{\eta}(Z_i)) },
\end{equation*}
which can be interpreted as the increase in the reciprocal of the empirical variance of the final weighted $\theta$ estimator. Interestingly, despite the more involved form of the sandwich loss compared to standard CART minimisation objectives, since the above goodness-of-fit metric for a split on the index set $I$ is a function only of the data indexed by $I$, this splitting rule can be used to build computationally fast decision trees (Algorithm~\ref{alg:rosetree}) and random forests (Algorithm~\ref{alg:roseforest}) that have the same computational complexity as standard CART random forests~\citep{ranger, rpart}. The key point is that to compute the objective above for an adjacent split point is an $O(1)$ update as with CART trees; see also Appendix~\ref{appsec:rose-complexity} for further details. 
Note that in Algorithm~\ref{alg:rosetree} we use data indexed by $\mathcal{I}_{\text{split}}$ to find the split points minimising the goodness-of-fit metric, and that indexed by $\mathcal{I}_{\text{eval}}$ to find the optimal levels for the piecewise constant weight function given the partition defined by the split points. For our theory, we require these to be disjoint, though in practice we always take them to be equal.
\RestyleAlgo{ruled}
\begin{algorithm}[ht]
\KwIn{Data $\left(S_i\right)_{i\in\cI}$ with split and evaluation index sets $\cI_{\text{split}}, \cI_{\text{eval}}\subset\cI$; 
estimator $\hat{\eta}$ of nuisance functions; pilot estimator $\hat{\theta}_{\text{init}} = \hat{\theta}$ of $\theta_P$; 
maximum tree depth $D_{\max}$.}

    \While{(depth of tree) $< D_{\max}$} {
    
        (For higher dimensional data) Select $d^{\text{mtry}}$ variables at random from the $d$ variables. 
        
        Select the variable (among the $d^{\text{mtry}}$) and split-point pair that minimises the goodness of fit measure that splits the node corresponding to state space region $R$ into $R'\uplus R''$ 
            \begin{equation*}
                \frac{\Big(\sum \limits_{i\in I_R'}\partial_\theta\psi_1(S_i;\hat{\theta},\hat{\eta}(Z_i))\Big)^2}{\sum\limits_{i\in I_R'}\psi_1^2(S_i;\hat{\theta},\hat{\eta}(Z_i))} 
                + \frac{\Big(\sum \limits_{i\in I_R''}\partial_\theta\psi_1(S_i;\hat{\theta},\hat{\eta}(Z_i))\Big)^2}{\sum\limits_{i\in I_R''}\psi_1^2(S_i;\hat{\theta},\hat{\eta}(Z_i))}
                -\frac{\Big(\sum \limits_{i\in I_R}\partial_\theta\psi_1(S_i;\hat{\theta},\hat{\eta}(Z_i))\Big)^2}{\sum\limits_{i\in I_R}\psi_1^2(S_i;\hat{\theta},\hat{\eta}(Z_i))},
            \end{equation*}
            where $I_R := \{i\in\cI_{\text{split}} : Z_i\in R\}$, and analogously for $I_{R'}$ and $I_{R''}$.
        
        Split the node into two child nodes as per the optimal split point.
    
    }

    We notate such a `ROSE tree' $T$, with each leaf $l\in\cL$ taking the form of a rectangular subspace $l\subset\cZ$, such that for every $z\in\cZ$ there exists a unique $l(z)\in\cL$ such that $z\in l(z)$.
    
    \For{$i\in\cI$}{
        Define $\omega_{i}(z) := |\{j\in\cI_{\text{eval}} : Z_j\in l(z)\}|^{-1}\ind_{\{i\in\cI_{\text{eval}} : Z_i\in l(z)\}}$.
    }

\KwOut{ROSE decision tree weights $\omega_i : \R^{d}\to\R$ for each $i \in \mathcal{I}$.}
\caption{ROSE decision trees ($J=1$)}
\label{alg:rosetree}
\end{algorithm}

\RestyleAlgo{ruled}
\begin{algorithm}[ht]
\KwIn{Data $\left(S_i\right)_{i\in\cI}$; estimator $\hat{\eta}$ of nuisance functions; pilot estimator $\hat{\theta}_{\text{init}} = \hat{\theta}$ of $\theta_P$; $c_{\text{split}}>0$ proportion of data chosen for each tree in the forest; ROSE tree hyperparameters.}

\For{$b\in[B]$} {

Generate random subsets of the data $\cI_{\text{split},b}, \cI_{\text{eval},b}\subset\cI$ with $|\cI_{\text{split},b}|, |\cI_{\text{eval},b}| \approx c_{\text{split}}|\cI|$.

    Grow a ROSE decision tree $T_{j,b}$ as in Algorithm~\ref{alg:rosetree}, outputting weights $\{\omega_{ib}\}_{i \in \mathcal{I}}$.

    Calculate $\tau_{b}^{(1)}(z) := \sum_{i\in\cI} \omega_{ib}(z)\psi_1^2(S_i;\hat{\theta},\hat{\eta}(Z_i))$.

    Calculate $\tau_{b}^{(2)}(z) := \sum_{i\in\cI} \omega_{ib}(z)\partial_\theta\psi_1(S_i;\hat{\theta},\hat{\eta}(Z_i))$.

}

Calculate
$\hat{w}_1^{\rose}(z) := \big(\sum_{b=1}^B \tau^{(1)}_{b}(z)\big)^{-1}\big(\sum_{b=1}^B \tau_{b}^{(2)}(z)\big)$.

\KwOut{ROSE random forest weight function $\hat{w}_1^{\rose}:\R^{d}\to\R$.}
\caption{ROSE random forests ($J=1$)}
\label{alg:roseforest}
\end{algorithm}

\subsection{Consistency of ROSE random forests}\label{sec:forest-consistency}

We now present a consistency result for ROSE forests showing that they converge to the optimal weightings~\eqref{eq:w_j}. Thus using ROSE forests to compute the weights, the resulting estimator of $\theta_P$ in the partially linear model will enjoy asymptotic normality with an asymptotic  variance equal to the minimum attainable variance when restricting to our robustness class of estimators~\eqref{eq:w-psi}. 
We require the following assumptions on the data generating mechanism. 
\begin{assumption}[Additional model assumptions]\label{ass:data-2}
In addition to Assumption~\ref{ass:data}, the following assumptions are made on the function $\psi$~\eqref{eq:w-psi}. As in Section~\ref{sec:asymp_norm}, we denote by $S=(Y,X,Z)$ an observation independent of our data, generated from a law $P$ contained within the family of probability distributions $\cP$ that satisfies the properties below.
\begin{itemize}
\item[(\mylabel{ass:bounded-density}{M5})] (\emph{Bounded Density}) Suppose $\cZ=[0,1]^{d}$ and $Z$ has density bounded away from zero and infinity.
\item[(\mylabel{ass:bounded-cond-moments}{M6})] (\emph{Bounded Conditional Moments}) There exist finite constants $c_1,c_2,C_1,C_2>0$, independent of $P$, such that
\begin{gather*}
    c_1 \leq \E_P\left[\psi_1(S;\theta_P,\eta_P(Z))^2 \given Z\right] \leq C_1 ,
    \quad
    c_2 \leq \E_P\big[\partial_\theta\psi_1(S;\theta_P,\eta_P(Z)) \given Z\big] \leq C_2 ,
\end{gather*}
almost surely.
\item[(\mylabel{ass:holder}{M7})] (\emph{H\"{o}lder Continuity}) 
There exist constants $(\beta,L)\in(0,1]\times\R_+$ that do not depend on $P$ such that
\begin{gather*}
    \big|\E_P[\psi_1^2(S;\theta_P,\eta_P(Z))\given Z=z]-\E_P[\psi_1^2(S;\theta_P,\eta_P(Z))\given Z=z']\big|\leq L\|z-z'\|^{\beta},
    \\
    \big|\E_P[\partial_\theta\psi_1(S;\theta_P,\eta_P(Z))\given Z=z]-\E_P[\partial_\theta\psi_1(S;\theta_P,\eta_P(Z))\given Z=z']\big|\leq L\|z-z'\|^{\beta},
\end{gather*}
for all $z,z'\in\cZ$.
\end{itemize}
\end{assumption}
Assumptions~\ref{ass:bounded-cond-moments} and \ref{ass:holder} are relatively mild conditions on the (scaled) influence function $\psi_1$. For example, in the partially linear model with $M_1(X) = X$, if the errors from regressing $Y$ on $(X, Z)$ and $X$ on $Z$ are conditionally independent given $Z$, Assumption~\ref{ass:bounded-cond-moments} will hold when $\Var(X \given Z)$ and $\Var(Y \given X, Z)$ are bounded from above and below almost surely.
Assumption~\ref{ass:bounded-density} is restrictive but standard in the literature on random forest as are the following conditions on the generation of the random ROSE random forest \citep{biau, meinshausen, wagerwalther, grfs, distrandfor}; empirical evidence suggests they are far from necessary.
\begin{assumption}[Properties of the ROSE random forest]\label{ass:forest}
The splits for each ROSE decision tree (Algorithm~\ref{alg:rosetree}) making up ROSE random forests (Algorithm~\ref{alg:roseforest}) satisfy the following properties:
\begin{itemize}
\item[(\mylabel{P1}{P1})] (\emph{Honesty}) For each tree, the split and evaluation index sets $\cI_{\text{split}}$ and $\cI_{\text{eval}}$ are disjoint, and both are disjoint from the data used to construct the nuisance function estimate $\hat{\eta}$ and pilot estimate $\hat{\theta}_{\text{init}}$.
\item[(\mylabel{P2}{P2})] (\emph{Node Size}) Let $k(l):=\left|\left\{i\in\cI_{\text{eval}}:Z_i\in l \right\}\right|$ be the number of (evaluation labeled) training observations in a leaf $l$. 
For each tree, the proportion of observations in any given node vanishes in $n$: $n^{-1}\max_{l}k(l)=o(1)$. Also, the minimal number of observations in any given node grows in $n$: $1/\min_{l}k(l)=o(1)$.
\item[(\mylabel{P3}{P3})] (\emph{$\alpha$-regularity}) Each tree split leaves at least a fraction $\alpha>0$ of the available training data on each side.
\item[(\mylabel{P4}{P4})] (\emph{Random Split}) For each tree split, the probability that the split occurs along the $q$th feature is bounded below by $\pi/d$ for some $\pi>0$, for all $q\in[d]$.
\end{itemize}
\end{assumption}

Equipped with Assumption~\ref{ass:data}~and~\ref{ass:data-2}, defining the family $\cP$ of the probability distributions our consistency result will be uniform over, and Assumption~\ref{ass:forest} dictating the properties of the forest, Theorem~\ref{thm:rose_forest_consistency} presents asymptotic consistency of our ROSE random forests. 

\begin{theorem}[Consistency of ROSE random forests]\label{thm:rose_forest_consistency}
Let Assumptions~\ref{ass:data} and~\ref{ass:data-2} hold for a class of distributions  $\mathcal{P}$. Then, in the case $J=1$, the ROSE random forest estimator $\hat{w}_1^{\rose}$ (Algorithms~\ref{alg:rosetree} and~\ref{alg:roseforest}, satisfying Assumption~\ref{ass:forest}) is uniformly consistent in estimating the optimal weights~\eqref{eq:w_j}:
    \begin{equation*} 
        \E_P\Big[\big(\hat{w}_1^{\rose}(Z)-w_{P,1}^{\rose}(Z)\big)^2 \Biggiven \hat{w}_1^{\rose}\Big]
        = o_\cP(1).
    \end{equation*}
\end{theorem}
Combining Theorems~\ref{thm:theta_est_asymp_norm}~and~\ref{thm:rose_forest_consistency}, we see that the ROSE random forest estimator for $\theta_P$, i.e.~the estimator of Algorithm~\ref{alg:main} with ROSE random forest weights formed as in Algorithms~\ref{alg:rosetree}~and~\ref{alg:roseforest}, is uniformly asymptotically Gaussian, and achieves minimal asymptotic variance over all robust estimators given by~\eqref{eq:robust-estimating-eqns}. In Appendix~\ref{appsec:rose-forest-J>1} we extend the ROSE random forests of Algorithms~\ref{alg:rosetree} and~\ref{alg:roseforest} to accommodate $J>1$ additional nuisance functions as in~\eqref{eq:robust-estimating-eqns}.

\section{Numerical experiments}\label{sec:numerical-results}
We study ROSE random forest performance on a range of simulated and real data examples, comparing with three alternative estimators: a simple `unweighted estimator'; a semiparametric efficient estimator; and a locally efficient estimator (as defined in Section~\ref{sec:GEE}). 
For all of these estimators we employ the DML2 cross-fitting framework~\citep{chern} (as in Algorithm~\ref{alg:main}). 
Except for the case of ROSE random forests, all nuisance functions expressed as conditional expectation are fit using random forests with the~\texttt{ranger} package~\citep{ranger} (see Appendix~\ref{appsec:sims} for details) and all conditional probabilities with probability forests using the \texttt{grf} package~\citep{grf-code}. In the case of the semiparametric efficient estimator, the ratio of conditional expectations that arises (see Appendix~\ref{appsec:semi-eff-class} and~\eqref{eq:h}) is estimated via a weighted random forest fitted with the~\texttt{ranger} package. Section~\ref{sec:sim1} compares these estimators in a setting where their asymptotic properties differ. 
In Section~\ref{sec:sim2} we directly compare finite sample properties of ROSE random forests with classical CART-based random forests. Section~\ref{sec:sim3} considers a model where we consider robust estimators with $J=2$ additional nuisance functions. Finally, a real data example is explored in Section~\ref{sec:real-data}.

\subsection{Simulation 1: Asymptotic benefits of ROSE random forests}\label{sec:sim1}

We consider two generalised partially linear models. In each case we take $n=20\,000$ i.i.d.~instances of the model $Z\sim N_{10}\big({\bf 0},(0.9^{|j-k|})_{(j,k)\in[10]^2}\big)$ with:
\medskip
\\
\noindent\emph{Simulation 1a: Partially Linear Model}
\begin{gather*}
    p(Z) := \text{expit}(3Z_1) \vee 0.01,
    \quad
    B\given Z \sim \text{Ber}(p(Z)),
    \quad
    X\given (Z,B) \sim N\bigg(\sum_{j=1}^5\text{expit}(Z_j) \;,\; \frac{B}{p(Z)}\bigg),
    \\
    Y \given (X,Z,B) \sim N\bigg(\theta_0X + \sum_{j=1}^5\text{expit}(Z_j) \;,\; \frac{B}{\sqrt{p(Z)}}\bigg);
\end{gather*}
\smallskip
\\
\noindent\emph{Simulation 1b: Generalised Partially Linear Model (square root link)}
\begin{gather*}
    p(Z) := \text{expit}(3Z_1) \vee 0.01,
    \quad
    B\given Z \sim \text{Ber}(p(Z)),
    \quad
    X\given(Z,B) \sim \Gamma\bigg( \sum_{j=1}^5\text{expit}(Z_j) \;,\; \frac{B}{p(Z)} \bigg),
    \\
    Y \given (X,Z,B) \sim \Gamma\big( \mu(X,Z) \;,\; \sigma^2(X,Z,B) \big),
    \\
    \sqrt{\mu(X,Z)} := \theta_0 X + \sum_{j=1}^5\text{expit}(Z_j),
    \quad
    \sigma^2(X,Z,B) := 0.01 \mu^2(X,Z) \frac{B}{\sqrt{p(Z)}}.
\end{gather*}
Here $\Gamma(\mu,\sigma^2)$ denotes the Gamma distribution with mean $\mu$ and variance $\sigma^2$, and $\theta_0=1$ is the target parameter of interest in both cases.

Table~\ref{tab:sim1} presents the bias--variance decomposition of the mean squared error for each estimator, alongside coverage of nominal $95\%$ confidence intervals, and Figure~\ref{fig:sim1_app_hist} plots density function estimates for each estimator. In both simulations, the semiparametric efficient estimator is substantially biased and yields suboptimal mean squared error and poor coverage. The former issue is most pronounced in Simulation 1a, where it has a larger mean squared error than the unweighted estimator, and the latter issue is most severe in Simulation 1b, where it has only 5\% coverage with nominal 95\% confidence intervals. In contrast, the other three estimators (all of which lie within our robustness class given by \eqref{eq:robust-estimating-eqns}) have negligible bias and valid coverage. We see that the variance and mean squared error of the ROSE random forests estimator is substantially smaller than the other competing robust estimators, supporting the conclusion of Theorem~\ref{thm:rose_forest_consistency}. 

\begin{table}[ht]
	\begin{center}
		\begin{tabular}{c|cccc}
			\toprule
                \multirow{3}{*}{Method} & \multicolumn{4}{c}{Simulation 1}
			\\
                & \begin{tabular}{@{}c@{}c@{}} Squared Bias\\$(\times 10^{-4})$\end{tabular} & \begin{tabular}{@{}c@{}c@{}}Variance\\$(\times 10^{-4})$\end{tabular} &\begin{tabular}{@{}c@{}c@{}} MSE (ratio to\\unweighted) \end{tabular}
                &
                \begin{tabular}{@{}c@{}c@{}} Coverage \\($95\%$) \end{tabular}
                \\
			\midrule
            \multicolumn{5}{l}{\,\emph{Gaussian Partially Linear Model (Example 1a)}} 
            \\
            \midrule
			Unweighted &   0.023   &    1.328  & 1 & 94.5\%  \\
			{\bf ROSE Random Forest} &   0.013   &    0.894  & {\bf 0.671}  & 94.4\%  \\
			Locally Efficient Random Forest &   0.052   &   4.010 & 3.986 & 94.5\% \\
            Semiparametric Efficient & 3.039 & 1.518 & 3.373   & 64.7\%  \\
			\midrule
            \multicolumn{5}{l}{\,\emph{Gamma Generalised Partially Linear Model (Example 1b)}} 
            \\
            \midrule
			Unweighted & 0.005 & 2.532 & 1 &  91.3\%  \\
            {\bf ROSE Random Forest} & 0.001 & 0.499 & {\bf 0.197} &  94.2\%  \\
            Locally Efficient Random Forest & 0.014 & 12.235 & 4.827 & 82.1\%   \\
            Semiparametric Efficient & 1.843 & 0.135 & 0.780 & 5.0\% \\
			\bottomrule
		\end{tabular}                                                                      
		\caption{Results of Simulation 1 ($4000$ simulations). Note that in the case of Example 1b, the squared bias, variance and MSE estimators use 1\% trimmed means.}
        \label{tab:sim1}
	\end{center}
\end{table}

\begin{figure}[!ht]
   \centering
   \includegraphics[width=0.95\textwidth]{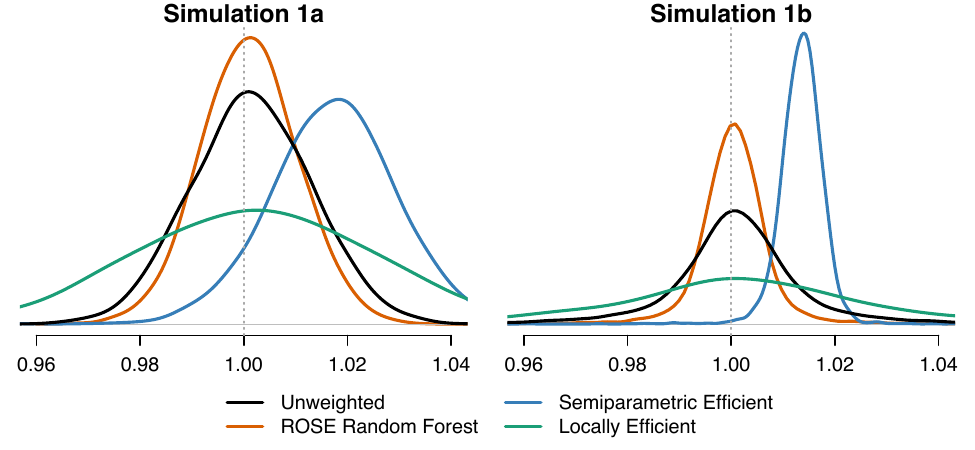}
    \caption{Kernel density estimates for the
    four $\hat{\theta}$ estimators for Simulations~1a (partially linear model) and~1b (generalised partially linear model) in Section~\ref{sec:sim1} (4000 repetitions, each of sample size $n=20\,000$). For reference we also indicate the true parameter $\theta_0=1$ with a dotted grey line.}
   \label{fig:sim1_app_hist}
\end{figure}

\subsection{Simulation 2: Finite sample benefits of ROSE random forests}\label{sec:sim2}

We explore the finite sample properties of ROSE random forests through the toy example
\begin{gather*}
    Y = X\theta_0 + \varepsilon,
    \qquad
    (Z_1,Z_2,Z_3,X)\sim N_3({\bf 0},4 I_4),
    \\
    \varepsilon\given(X,Z) \sim N(0,\sigma_Y^2(Z)),
    \qquad
    \sigma_Y(Z) = 1+2\;\text{expit}(Z_1+Z_2)+2\;\text{expit}(Z_2+Z_3),
\end{gather*}
where $\theta_0=1$ is the target parameter and $\text{expit}(z):=e^z/(1+e^z)$. As $\Var(Y\given X,Z)=\Var(Y\given Z)$ here, the locally efficient estimator that involves regressing squared residuals from regressing $Y$ on $(X,Z)$ onto $Z$ is semiparametrically efficient, and we implement this with CART-based random forests. Moreover, this estimator is asymptotically equivalent to ROSE random forests. To isolate the finite sample properties of ROSE forests, for all estimators we permit oracular knowledge of $f_0=m_0=h_0\equiv0$ (adopting the notation of Section~\ref{sec:PLMmodel}), and consequently all estimators have negligible bias. 
We take samples of size $n\in\{100\cdot 2^r \,:\,r=1,\ldots,8\}$ and employ cross-fitting (DML2) with $K=2$ folds for weight estimation. Hyperparameter tuning is used to determine an optimal tree depth with the criterion for ROSE and CART-forests being the sandwich loss and the residual sum of squares respectively; for details see Appendix~\ref{appsec:sim2}.

\begin{figure}[ht]
   \centering
   \includegraphics[width=0.85\textwidth]{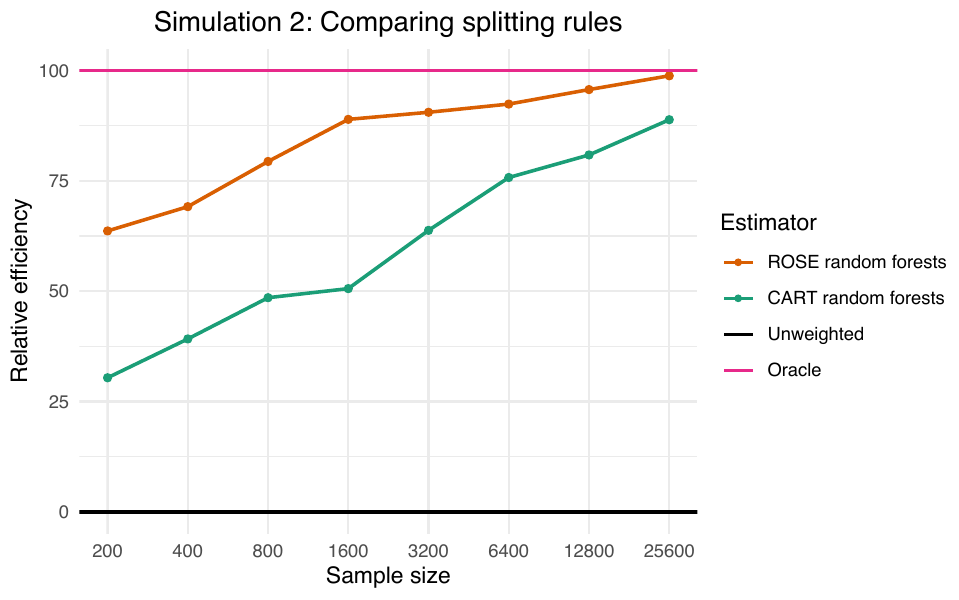}
   \caption{The ratio $\frac{\text{MSE}(\hat{\theta}_{\mathrm{unw}})-\text{MSE}(\hat{\theta})}{\text{MSE}(\hat{\theta}_{\mathrm{unw}})-\text{MSE}(\hat{\theta}_{\textrm{ora}})}$ measures the `relative accuracies' of an estimator $\hat{\theta}$ on a scale where $0\%$ reflects equivalent accuracy to the unweighted estimator $\hat{\theta}_{\text{unw}}$ and 100\% corresponds to oracle estimation $\hat{\theta}_{\textrm{ora}}$. Relative accuracies presented for ROSE and CART random forests in the empirical setting of Simulation~\ref{sec:sim2}. The MSE denotes the mean squared error over 1000 simulations.}
   \label{fig:sim2}
\end{figure}

Figure~\ref{fig:sim2} presents the performances the of ROSE and CART-based random forests estimators. We compare each of these to an oracle estimator $\hat{\theta}_{\textrm{ora}}$ given the true optimal weights by measuring their performance in terms of a relative accuracy formed through a ratio of differences in mean-squared error to the unweighted estimator.
We see that the ROSE random forest estimator outperforms the CART random forest estimator over all sample sizes, with $63.7\%$ versus $30.4\%$ `relative accuracy' at the smaller sample size of $n=200$, and $98.8\%$ versus $80.9\%$ `relative accuracy' at the larger sample size $n=25\,600$. 
Such benefits are likely a result of both the splitting criterion used in ROSE random forests and also the sandwich loss cross-validation criterion used for hyperparameter tuning. 
Note that we consider a relatively small class of hyperparameters for tuning here, but expect that performance difference between the two forms of cross-validation used may well widen the more hyperparameters are tuned for (see Appendix~\ref{appsec:sim2} for further discussions).

\subsection{Simulation 3: Effects of increasing sample size}\label{sec:sim3}

Consider $n$ i.i.d.~copies of the partially linear model
\begin{gather*}
    Z \sim N_3({\bf 0}, I_3),
    \quad
    X \given Z \sim 
    \begin{cases}
        \delta_0 &\,\text{with probability } 0.1+0.8\,\text{expit}(Z_1)
        \\
        \Gamma\big(\sum_{j=1}^3\text{expit}(Z_j) \;,\; 0.1\big) &\,\text{with probability } 0.9-0.8\,\text{expit}(Z_1)
    \end{cases},
    \\
    Y \given (X,Z) \sim \Gamma\Big( \theta_0 X + \sum_{j=1}^3\text{expit}(Z_j) \;,\; 0.9\,\ind_{\{0\}}(X) + 0.4\,\ind_{(0,1.5]}(X) + 0.1\,\ind_{(1.5,\infty)}(X) \Big),
\end{gather*}
where $\theta_0=1$ and $\delta_0$ denotes a point mass at zero, so the distribution of $X \given Z$ is `zero-inflated'. As a result, in addition to the nuisance function $\E(X\given Z)$, another natural estimable quantity is $\PP(X=0\given Z)$. We study a ROSE random forest with $J=2$ taking $M_1(X):=X$ and $M_2(X):=\ind_{\{X=0\}}$, and so $m_1(Z)=\E[X\given Z]$ and $m_2(Z) = \PP(X=0\given Z)$. We consider sample sizes $n\in\{2^r\times 10\,000 : r\in\{0,1,2,3\}\}$ i.e.~ranging from $10\,000$ to $80\,000$, and compare the unweighted estimator, the semiparametric efficient estimator and the two ROSE estimators with $J=1$ (i.e.~$M_1(X)=X$ only) and $J=2$.

\begin{table}
	\begin{center}
		\begin{tabular}{c|cccc}
			\toprule
                \multirow{3}{*}{Method} & \multicolumn{4}{c}{Simulation 3}
			\\
                & \begin{tabular}{@{}c@{}c@{}} Squared Bias\\$(\times 10^{-5})$\end{tabular} & \begin{tabular}{@{}c@{}c@{}}Variance\\$(\times 10^{-5})$\end{tabular} &\begin{tabular}{@{}c@{}c@{}} MSE (ratio to\\unweighted) \end{tabular}
                &
                \begin{tabular}{@{}c@{}c@{}} Coverage \\$(95\%)$ \end{tabular}
                \\
            \midrule
			Unweighted & 0.001  & 0.997   & 1 &  95.2\%  \\
			ROSE Random Forest ($J=1$) &  0.003    &  0.955    &  0.960 & 95.3\%  \\
			{\bf ROSE Random Forest ($J=2$)} &  0.001    &  0.876  & {\bf 0.879} & 95.4\% \\
            Semiparametric Efficient & 0.286 & 0.786 & 1.075 & 90.4\%  \\
			\bottomrule
		\end{tabular}                                                                      
		\caption{Results of Simulation 3 for sample size $n=80\,000$ (4000 simulations).}\label{tab:sim3}
	\end{center}
\end{table}

\begin{figure}[ht]
   \centering
   \includegraphics[width=1\textwidth]{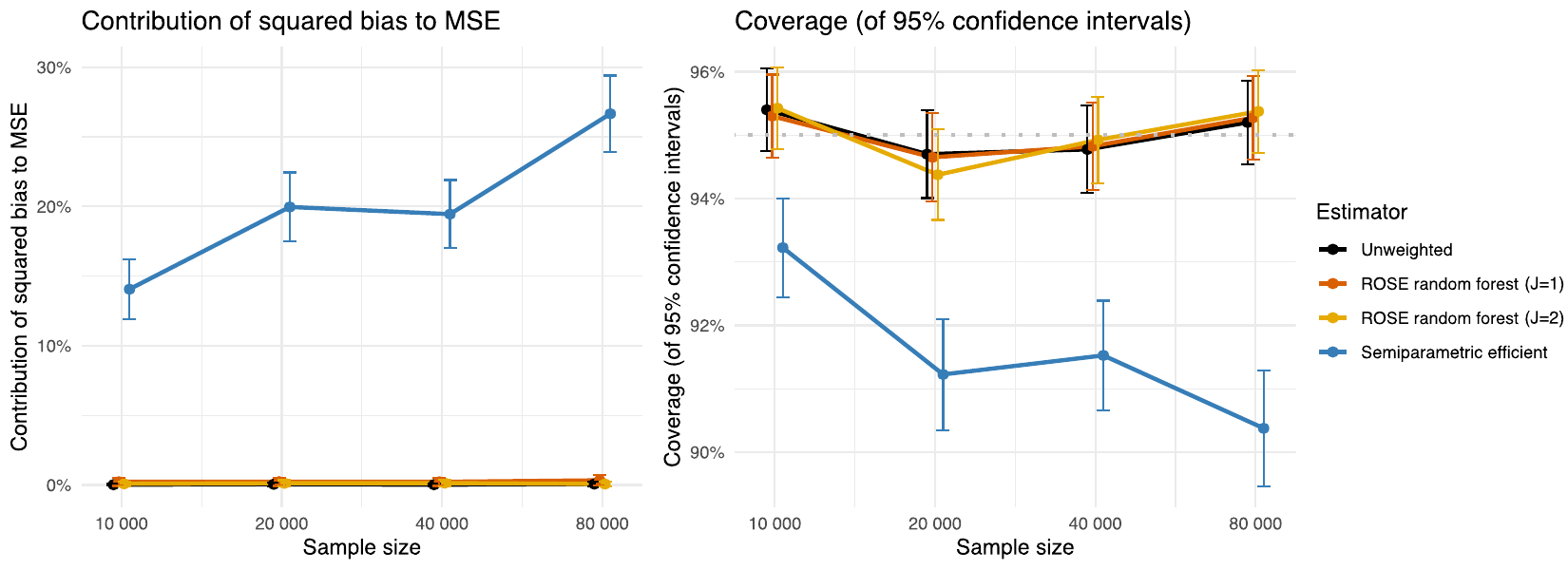}
   \caption{Contribution of squared bias to mean squared error and coverage of confidence intervals for Simulation 3.}
   \label{fig:sim3}
\end{figure}

Table~\ref{tab:sim3} gives the results in the case $n=80\,000$; for the results at other sample sizes see Appendix~\ref{appsec:sim3}. 
We see that the semiparametric efficient estimator is biased, even at this large sample size, resulting in undercoverage and mean squared error worse than the unweighted estimator. The ROSE random forest estimator with $J=1$ improves on the unweighted estimator, though only modestly. However, the $J=2$ ROSE random forest estimator gives genuine improvements, and illustrates the benefits from (correctly) assuming that we can additionally estimate $m_2(Z)=\PP(X=0\given Z)$ at sufficiently fast rates.

Figure~\ref{fig:sim3} plots for each estimator at each sample size the contribution of squared bias to the mean square error, and coverage of nominal 95\% confidence intervals. As the unweighted and two ROSE random forest estimators ($J\in\{1,2\}$) are all robust, they have negligible squared bias and valid coverage. On the other hand, the squared bias of the semiparametric efficient estimator contributes significantly to its mean squared error, and moreover this contribution tends to increase with sample size. This results in  decreasing coverage of confidence intervals, and decreasing relative efficiency (see Appendix~\ref{appsec:sim3}).

\subsection{Real-world data example: Effect of temperature on bike rental demand}\label{sec:real-data}

We apply ROSE random forests to the Seoul Bike Rental Demand dataset \citep{seoulbikes}, which can be accessed from the UCI Machine Learning Repository. The dataset consists of hourly count data of rental bike usage in Seoul over a period of one year (2018). We aim to estimate the (linear) effect $\theta_0$ of temperature ($X=\text{Temperature}$) on bike demand ($Y=\text{Bike Count}$) via the generalised partially linear model with $\log$ link
\begin{equation*}
    \log\big(\E[Y\given X,Z]\big) = \theta_0 X + f_0(Z),
\end{equation*}
for some $f_0$ and for temperatures in the range $\text{Temperature}\in[-10\text{\textdegree C},20\text{\textdegree C}]$, controlling for other weather and time effects with $Z=(\text{Time, Date, Holiday, Rain, Snow})$. The dataset consists of 5403 observations. We compare the unweighted estimator, the ROSE random forest estimator ($J=1$ with $M_1(x)=x$), and the semiparametricaly efficient estimator.
We use the DML2 cross-fitting framework with $K=10$ folds, and with all nuisance functions fit using random forests in the \texttt{ranger} package~\citep{ranger}. 

Figure~\ref{fig:bikesSplot} shows for each method, 50 estimators $\hat{\theta}_s$ and their nominal 95\% confidence intervals, each corresponding to a random partition of the data used for cross-fitting. Clearly the semiparametric efficient estimator performs poorly here; 26\% of its confidence intervals do not even overlap.
In contrast, all the robust unweighted and ROSE random forest confidence intervals overlap, with the ROSE random forest estimator enjoying a two thirds reduction in variance over the unweighted estimator.

\begin{figure}[ht]
   \centering
   \includegraphics[width=1\textwidth]{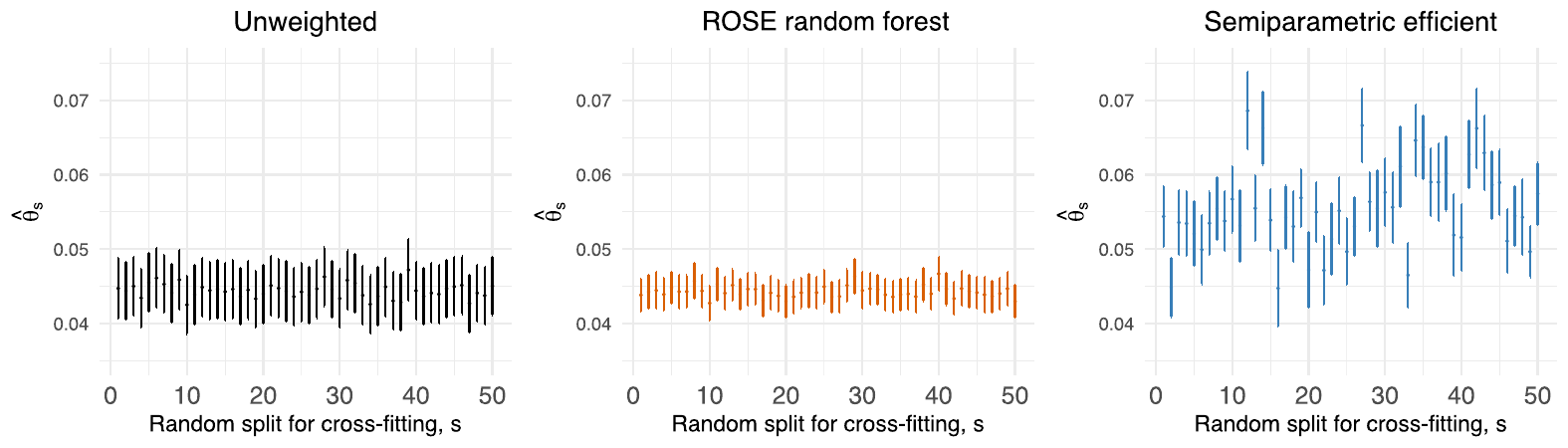}
   \caption{Confidence intervals (nominal 95\%) for the three semiparametric estimators for 50 different random $K=10$ fold cross-fitting sample splits of the data.}
   \label{fig:bikesSplot}
\end{figure}

\section{Discussion}\label{sec:discussion}
One of the great achievements of classical semiparametric theory has been the development of lower bounds for the estimation of finite-dimensional quantities in semiparametric problems. In practice however, these bounds may be unrealistic: indeed our empirical results show that the semiparametric efficient estimator may perform poorly even at large sample sizes running into the tens of thousands. In this work, we have attempted to introduce a new form of efficiency---robust semiparametric efficiency---in an attempt to close the gap between theory and practice. Our notion of robust efficiency involves requiring locally uniformly consistent estimation within a class of distributions $\mathcal{P}$ characterised  by user-specified conditional moments being estimable. We show that our ROSE random forest procedure achieves this robust efficiency in that it attains the minimal variance among the class of estimators deemed to be robust in the sense above.

The requirement of locally uniform consistency that forms the basis of our formulation of robust efficiency bears some resemblance to the imposition of regularity in (semi)parametric efficiency that excludes so-called `superefficient' estimators. 
For instance, the well-known Hodges' estimator~\citep{le-cam-thesis}, in contrast to the sample mean, enjoys smaller pointwise asymptotic mean squared error at a single point distribution within a parametric family of distributions, but its mean squared error uniformly over the model class diverges~\citep[\S8.1]{vandervaart}. In a similar vein, whilst a semiparametric efficient estimator may enjoy smaller pointwise mean squared error at a single distribution, uniformly over the class of distributions we introduce, its mean squared error can diverge.

Our work offers a number of interesting directions for future work. 
It would be interesting to extend the framework introduced to other semiparametric models such as
quantile regression models. In our current setting, we assume differentiability of $\psi$ with respect to our parameter of interest $\theta$ and consequently the form of sandwich loss estimator~\eqref{eq:sand-loss} is  relatively simple. Such an assumption is violated in the partially linear quantile regression model for example, where the $\tau$th conditional quantile $q_{Y\given(X,Z)}(\tau)$ of $Y$ given $(X, Z)$  takes the form $\theta_0 X + f_0(Z)$. 
The denominator of the corresponding sandwich loss would involve the conditional density of $Y$ given $(X,Z)$. 
How best to practically estimate such a more involved empirical sandwich loss is less clear but is certainly worth further investigation.

In the setting we have studied, the target of inference is a single scalar quantity. An interesting avenue of further work would be to consider settings where the target of interest takes the form of a function; a prototypical example being the heterogeneous partially linear model
\begin{equation*}
    \E[Y\given X,Z] = \theta_0(Z) X + f_0(Z),
\end{equation*}
where now $\theta_0$ is a target function. 
In such a setting the notion of `optimality' of an estimator within a robustness class could be that of a minimal pointwise asymptotic variance of $\theta_0$, or a minimal mean squared error over a certain distribution for $Z$, for example.

There are also a number of practical questions that naturally follow from this work. Firstly, the hardness result of Theorem~\ref{thm:slow-rates} requires a priori specification of a set of nuisance functions $(m_{P,j})_{j\in[J]}$ that are estimable at sufficiently fast polynomial rates. A richer class of such functions further reduces the variance of the optimal estimator within the resulting classes (see e.g.~Section~\ref{sec:sim3}) at a cost in robustness, exhibited through implicit assumptions on the regularity additional nuisance functions being introduced. An interesting extension would be to study the possibility of diagnosing regularisation biases that can ensue through a breakdown of robustness, ideally leading towards a data-driven way to choose a set of nuisance functions $(m_{P,j})_{j\in[J]}$ for ROSE estimation.

\newpage 

\appendix

{
\Large\bf
\begin{center}
Supplementary material for ‘ROSE Random Forests for Robust Semiparametric Efficient Estimation’, by Elliot H. Young and Rajen D. Shah
\end{center}
}

\section{Supplementary Information for Section~\ref{sec:robust-class}}

This section collects a number of supplementary remarks relating to Section~\ref{sec:robust-class}. Appendix~\ref{appsec:NO} briefly outlines Neyman orthogonality; a favourable property our influence functions satisfy. In Appendix~\ref{appsec:semi-eff-class} we derive the form of the semiparametric efficient influence function for the partially parametric model. 
Our estimation strategy is interpreted from the perspective of an infinite dimensional extension of the generalised method of moments in Appendix~\ref{appsec:GMMs}. Our robust estimation strategy is also briefly explored within a broader class of models in Appendix~\ref{appsec:additional-models}, where we study as an example instrumental variable models.

\subsection{Neyman orthogonality}\label{appsec:NO}

All influence functions of the partially parametric model as in~\eqref{eq:IF} satisfy a so-called Neyman orthogonality~\citep{app-neyman, neyman2} property of the following form:
\begin{definition}[Neyman Orthogonality]\label{def:NO}
    Consider a function $\tilde{\psi}:\mathcal{S}\times\Theta\times\R^{p_\alpha}\times\R^{p_\beta}$ where $p_\alpha,p_\beta\in\mathbb{N}$. Then the function $\tilde{\psi}$ is Neyman orthogonal at the law $P$ with respect to the nuisance functions $\rho:\cZ\to\R^{p_\alpha}$ and $\zeta:\cX\times\cZ\to\R^{p_\beta}$ if
    \begin{align*}
        \E_P\big[\nabla_\alpha\tilde{\psi}(S;\theta_P,\alpha,\zeta(X,Z))|_{\alpha=\rho(Z)} \given Z\big] &= 0,
        \\
        \E_P\big[\nabla_\beta\tilde{\psi}(S;\theta_P,\rho(Z),\beta)|_{\beta=\zeta(X,Z)} \given X,Z\big] &= 0,
    \end{align*}
    almost surely.
\end{definition}
Note in particular that in the above we adopt a definition of Neyman orthogonality in terms of Euclidean derivatives as in e.g.~\citet{app-mackey}, as opposed to G\^{a}teaux derivatives as in e.g.~\citet{app-chern}. We see functions $\psi$ in the class given by~\eqref{eq:IF} are indeed Neyman orthogonal at $P$ with respect to the nuisance functions $(f_P,\phi_P)$ for any $\psi_P:\cX\times\cZ\to\R$ satisfying $\E_P[\phi_P(X,Z)\given Z]=0$.

\subsection{On the semiparametric efficient influence function}\label{appsec:semi-eff-class}

\begin{proposition}
    The function $\phi_P$ corresponding to the efficient influence function in the partially parametric model takes the form
    \begin{equation*}
        (x,z)\mapsto
        W_P(x,z) \bigg(U_P(x,z)-\frac{\E_P[W_P(X,Z)U_P(X,Z)\given Z=z]}{\E_P[W_P(X,Z)\given Z=z]}\bigg),
    \end{equation*}
    where
    \begin{gather*}
        W_P(x,z) = \E_P[\varepsilon^2(Y,X,Z;\theta_P,f_P(Z))\given X=x,Z=z]^{-1},
        \\
        U_P(x,z) = \E_P[\nabla_\theta\varepsilon(Y,X,Z;\theta_P,f_P(Z))\given X=x,Z=z],
        \\
        \varepsilon(x,y,z;\theta_P,f_P(z)) := \frac{y-\mu(x,z;
        \theta_P,f_P(z))}{\frac{\partial}{\partial\gamma}\big|_{\gamma=f_P(z)}\mu(x,z;\theta_P,\gamma)},
    \end{gather*}
    with $W_P$ given up to a constant of proportionality.
\end{proposition}

\begin{proof}
Recall the notation $S:=(Y,X,Z)$. Then the asymptotic variance of the estimator derived from the (scaled) influence function~$\psi$ as in~\eqref{eq:IF} is
\begin{equation*}
    \frac{\E_P\left[\psi^2(S;\theta_P,f_P(Z),\phi_P(X,Z))\right]}{\left(\E_P\left[\nabla_{\theta}\psi(S;\theta_P,f_P(Z),\phi_P(X,Z))\right]\right)^2},
\end{equation*}
where
\begin{gather*}
    \psi(S;\theta_P,f_P(Z),\phi_P(X,Z)) = \varepsilon(S;\theta_P,f_P(Z))\phi_P(X,Z)
    \\
    \varepsilon(S;\theta_P,f_P(Z)) := \Big(\frac{\partial}{\partial\gamma}\Big|_{\gamma=f(Z)}\mu(X,Z;\theta,\gamma)\Big)^{-1} \big(Y-\mu(X,Z;\theta,f(Z))\big),
\end{gather*}
and $\E_P[\phi_P(X,Z)\given Z]=0$. 
Then
\begin{align*}
    &\quad\;
    \frac{\E_P\left[\psi^2(S;\theta_P,f_P(Z),\phi_P(X,Z))\right]}{\left(\E_P\left[\nabla_{\theta}\psi(S;\theta_P,f_P(Z),\phi_P(X,Z))\right]\right)^2}
    \\
    &=
    \frac{\E_P[\phi_P^2(X,Z)\varepsilon^2(S;\theta_P,f_P(Z))]}{(\E_P[\phi_P(X,Z)\nabla_\theta\varepsilon(S;\theta_P,f_P(Z))])^2}
    \\
    &=
    \frac{\E_P[\phi_P^2(X,Z)\E_P[\varepsilon^2(S;\theta_P,f_P(Z))\given X,Z]]}{(\E_P[\phi_P(X,Z)\E[\nabla_\theta\varepsilon(S;\theta_P,f_P(Z))\given X,Z]])^2}
    \\
    &=
    \frac{\E_P[\phi_P^2(X,Z)\E_P[\varepsilon^2(S;\theta_P,f_P(Z))\given X,Z]]}{(\E_P[\phi_P(X,Z)\{\E_P[\nabla_\theta\varepsilon(S;\theta_P,f_P(Z))\given X,Z]-\varphi_P(Z;\theta_P,f_P(Z))\}])^2}
    \\
    &\geq
    \E_P\left[\frac{\big(\E_P\left[\nabla_\theta\varepsilon(S;\theta_P,f_P(Z))\given X,Z\right] - \varphi_P(Z;\theta_P,f_P(Z))\big)^2}{\E_P\left[\varepsilon^2(S;\theta_P,f_P(Z))\given X,Z\right]}\right]^{-1},
\end{align*}
where
\begin{multline*}
    \varphi_P(Z;\theta_P,f_P(Z)) 
    \\
    := \E_P\left[\frac{1}{\E_P\left[\varepsilon^2(S;\theta_P,f_P(Z)) \given X,Z\right]} \,\Big|\, Z\right]^{-1} \E_P\left[\frac{\E_P\left[\nabla_\theta\varepsilon(S;\theta_P,f_P(Z))\given X,Z\right]}{\E_P\left[\varepsilon^2(S;\theta_P,f_P(Z))\given X,Z\right]} \,\Big|\, Z\right],
\end{multline*}
and recalling that $\E_P[\phi_P(X,Z)\given Z]=0$. 
Note the final inequality follows as a result of the Cauchy--Schwarz inequality. Equality in the above holds when
\begin{gather*}
    \phi_P(X,Z) = W_P(X, Z)\Big( U_P(X,Z) - \E_P\big[W_P(X,Z)\given Z\big]^{-1}\E_P\big[W_P(X,Z)U_P(X,Z)\given Z\big] \Big),
    \\
    W_P(X,Z) = C_P^{-1}\E_P\big[\varepsilon^2(S;\theta_P,f_P(Z))\given X,Z\big]^{-1},
    \qquad
    U_P(X,Z) = \E_P\big[\nabla_\theta\varepsilon(S;\theta_P,f_P(Z))\given X,Z\big],
\end{gather*}
for any constant $C_P$. 
This gives the function $\phi_P$ corresponding to the efficient influence function of the partially parametric model. Precisely the efficient influence function takes this form~\eqref{eq:IF} with specifically $C_P = \E_P\big[\Cov_P\big(\nabla_\theta\mu(X,Z;\theta_P,f_P(Z)),\,\phi_P(X,Z)\given Z\big)\big]$ in the above. 
\end{proof}

\subsection{Reparameterisations of the nuisance functions in the partially parametric model}\label{appsec:nuisance-repara}

In the partially parametric model defined through the conditional moment specification
\begin{equation*}
    \E_P[Y\given X,Z] = \mu(X,Z;\theta_P,f_P(Z)),
\end{equation*}
to estimate the nuisance function $f_P$ it may help in some settings to reparameterise the above model in terms of nuisance functions that are easier to estimate with `off-the-shelf' regression packages (e.g.~nuisance functions expressible as a single conditional expectation). For example the partially linear model
\begin{equation*}
    \E_P[Y\given X,Z] = X\theta_P + f_P(Z),
\end{equation*}
can be reparametrised as
\begin{equation*}
    \E_P[Y\given X,Z] = \E_P[Y\given Z] + \big(X-\E_P[X\given Z]\big)\theta_P,
\end{equation*}
in terms of the two nuisance functions $\big(\E_P[Y\given Z],\, \E_P[X\given Z]\big)$. Similarly, the generalised partially linear model
\begin{equation*}
    g\big(\E_P[Y\given X,Z]\big) = X\theta_P + f_P(Z),
\end{equation*}
can be reparametrised as
\begin{equation*}
    g\big(\E_P[Y\given X,Z]\big) = \E_P\big[ g\big(\E_P[Y\given X,Z]\big)\given Z\big] + \big(X-\E_P[X\given Z]\big)\theta_P,
\end{equation*}
in terms of the two nuisance functions $\big(\E_P\big[g(\E_P[Y\given X,Z])\given Z\big],\, \E_P[X\given Z]\big)$. Under such reparametrisations it is then natural in our robustness framework to take $M_1(X)=X$ as to estimate $f_P$ with this approach we would in any case hope that the regression function $\E_P(X \given Z=z)$ can be estimated well.

\subsection{Connection to the generalised method of moments}\label{appsec:GMMs}

Here we describe an alternative interpretation of the robust estimating equations~\eqref{eq:robust-estimating-eqns} we consider in terms of an infinite dimensional extension of the generalised method of moments \citep{app-gmm}.

To this end, we introduce the following notation. We take $S = (Y,X,Z)$ following law $P$, and take $T := (X,Z)$, $\xi_j(T) := M_j(X)-m_{P,j}(Z)$ and $\varepsilon(S;\theta) := \big({\frac{\partial}{\partial\gamma}\mu(T;\theta,\gamma)|_{\gamma=f_P(Z)}}\big)^{-1}\allowbreak({Y-\mu(T;\theta,f_P(Z))})$, notationally dropping the dependence in $P$ on both terms $\xi_j$ and $\varepsilon$.

The population version of the method of moments solves the minimisation problem
\begin{equation}\label{eq:GMM-pop-min}
    \underset{\theta\in\Theta}{\text{minimise}} \; \E_P[\mathbf{G}(S;\theta)]^\top  \,\mathbf{M}\, \E_P[\mathbf{G}(S;\theta)],
\end{equation}
for some $L$ dimensional vector of functions $\mathbf{G}:\mathcal{S}\times\Theta\to\R^L$ and fixed, symmetric, positive semidefinite matrix $\mathbf{M}\in\R^{L\times L}$. Now consider functions $\mathbf{G}$ of the form
\begin{align*}
    \mathbf{G}(S;\theta)
    &=
    \Big(
        \sum_{j=1}^J w_j^{(1)}(Z) \xi_j(T) \varepsilon(S;\theta)
        ,\ldots,
        \sum_{j=1}^J w_j^{(L)}(Z) \xi_j(T) \varepsilon(S;\theta)
    \Big)
    \\
    &=
    \sum_{j=1}^J \big(w_j^{(1)}(Z),\ldots,w_j^{(L)}(Z)\big) \xi_j(T)\varepsilon(S;\theta).
\end{align*}
The minimisers of~\eqref{eq:GMM-pop-min} therefore coincide with solutions to the equation
\begin{equation*}
    \sum_{l=1}^L\sum_{l'=1}^L \mathbf{M}_{ll'} \, \E_P\bigg[\sum_{j=1}^J w_j^{(l)}(Z) \xi_j(T) \varepsilon(S;\theta)\bigg] \E_P\bigg[\sum_{j'=1}^J w_{j'}^{(l')}(Z) \xi_{j'}(T) \nabla_\theta\varepsilon(S;\theta)\bigg]
    = 0,
\end{equation*}
which can be reposed as
\begin{equation}\label{eq:GMM-w-xi-ep}
    \E_P\bigg[ \sum_{j=1}^J w_j(Z) \xi_j(T) \varepsilon(S;\theta) \bigg]
    = 0,
\end{equation}
where
\begin{equation*}
    w_j(Z) := \sum_{l=1}^L \sum_{l'=1}^L  \mathbf{M}_{ll'} w_j^{(l)}(Z) 
    \underbrace{\E_P\bigg[
        \sum_{j'=1}^J w_{j'}^{(l')}(Z) \xi_{j'}(T) \nabla_\theta\varepsilon(S;\theta)
    \bigg]}_{\text{constant}}.
\end{equation*}
Note in particular that for the generalised partially linear model
\begin{align*}
    \varepsilon(S;\theta) &= g'\big(g^{-1}(X\theta+f_P(Z))\big) \big(Y-g^{-1}(X\theta+f_P(Z))\big),
    \\
    \nabla_\theta \varepsilon(S;\theta)
    &= \frac{g''\big(g^{-1}(X\theta+f_P(Z))\big)}{g'\big(g^{-1}(X\theta+f_P(Z))\big)} X \big(Y-g^{-1}(X\theta+f_P(Z))\big) - X,
    \\
    \E_P\big[\nabla_\theta\varepsilon(S;\theta)\given T\big] &= -X,
\end{align*}
the latter of which does not depend on $\theta$. The `infinite dimensional generalised method of moments' estimating equations~\eqref{eq:GMM-w-xi-ep} therefore takes the form of the estimating equations~\eqref{eq:robust-estimating-eqns} we consider.

\subsection{Additional models}\label{appsec:additional-models}

As noted in Section~\ref{sec:robust-class} the hardness result of Theorem~\ref{thm:slow-rates}, as well as the asymptotic results of Theorems~\ref{thm:theta_est_asymp_norm} and~\ref{thm:rose_forest_consistency}, extend to a broader class of models than the partially parametric model. We outline a broader class of models here.

Consider the random variables $(Y,X,Z)\in\cY\times\cX\times\cZ\subseteq\R^{d_Y}\times\R^{d_X}\times\R^{d_Z}$ for some $d_Y,d_X,d_Z\in\mathbb{N}$ and $\cY$, $\cX$ and $\cZ$ all open non-empty sets. Again, we define the tuple $S:=(Y,X,Z)\in\mathcal{S}:=\cY\times\cX\times\cZ$. We consider semiparametric models defined through the moment condition
\begin{equation}\label{eq:varep-res}
    \E_P[\varepsilon_{\text{res}}(S;\theta_P,f_P(Z))\given X,Z] = 0,
\end{equation}
for some `pseudo-residual' $\varepsilon_{\text{res}}:\mathcal{S}\times\Theta\times\R^{p_f}$ and some measurable function $f_P:\cZ\to\R^{p_f}$, with $p_f\in\mathbb{N}$, that also satisfies
\begin{equation}\label{eq:k-c}
    \E_P\big[\nabla_{\gamma}\varepsilon_{\text{res}}(S;\theta_P,\gamma)|_{\gamma=f_P(Z)}\given X,Z\big] = k(X,Z;\theta_P,f_P(Z))c(Z;\theta_P,f_P(Z)),
\end{equation}
for some $k:\cX\times\cZ\times\Theta\times\R^{p_f}\to\R$ and $c:\cZ\times\Theta\times\R^{p_f}\to\R^{p_f}$. Then the class of functions
\begin{equation}\label{eq:additional-models-NO}
    \psi(S;\theta_P,f_P,\phi_P) = \phi_P(X,Z) \big(k(X,Z;\theta_P,f_P(Z))\big)^{-1} \varepsilon_{\text{res}}(S;\theta_P,f_P(Z)).
\end{equation}
for $\phi_P$ satisfying $\E_P[\phi_P(X,Z)\given Z]=0$, are Neyman orthogonal with respect to $(f_P,\phi_P)$ (Definition~\ref{def:NO}), with the robust functions corresponding to $\phi_P$ taking the form~\eqref{eq:w-M-m}. Alternatively worded,~\eqref{eq:additional-models-NO} lies in the orthogonal complement of the nuisance tangent space of the model~\eqref{eq:varep-res}, and thus is up to a multiplicative constant an influence function of the semiparametric model~\eqref{eq:varep-res}. 

This broader model also notably includes that of instrumental variable regression, which we now discuss.

\subsubsection{Instrumental variable regression}\label{appsec:IV}

Consider the partially linear instrumental variable model, where given a response $\tilde{Y}\in\R$, covariate / treatment of interest $\tilde{X}\in\R$, set of instrumental variables $\tilde{V}\in\R^{d_{\tilde{V}}}$ and covariates $\tilde{Z}\in\R^{d_{\tilde{Z}}}$ satisfying
\begin{equation*}
    \tilde{Y} = \tilde{X}\theta_P + f_P(\tilde{Z}) + \tilde{\varepsilon},
\end{equation*}
where $\E_P[\tilde{\varepsilon}\given \tilde{V},\tilde{Z}]=0$, $f_0:\R^{d_{\tilde{Z}}}\to\R$ and $\theta_0\in\R$ is the parameter of interest, and $d_{\tilde{V}}, d_{\tilde{Z}}\in\mathbb{N}$. Then take $Y=(\tilde{Y},\tilde{X})$, $X=\tilde{V}$, $Z=\tilde{Z}$, $S=(Y,X,Z)=(\tilde{Y},\tilde{X},\tilde{V},\tilde{Z})$, and
\begin{equation*}
    \varepsilon_{\text{res}}^{(\text{IV})}(S;\theta,f(Z)) = \tilde{Y}-\tilde{X}\theta-f(\tilde{Z}).
\end{equation*}
Then~\eqref{eq:varep-res} holds, and~\eqref{eq:k-c} holds with $p_f=1$ and
\begin{equation*}
    k(X,Z;\theta,f(Z)) = 1,
    \qquad
    c(Z;\theta,f(Z)) = -1.
\end{equation*}
The condition $\E_P[\tilde{\varepsilon}\given\tilde{V},\tilde{Z}]=0$, or equivalently
\begin{equation}\label{eq:IV-condition}
    \E_P\big[ (\tilde{Y}-\theta_P\tilde{X}) - \E_P[\tilde{Y}-\theta_P\tilde{X}\given \tilde{Z}] \given \tilde{V},\tilde{Z}\big] = 0,
\end{equation}
is often the primary imposition defining the instrumental variable model; see for example~\citet{app-chern, app-chen-IV-semieff, app-liu}, among others. A sufficient condition for~\eqref{eq:IV-condition} is $(\tilde{Y},\tilde{X})\independent \tilde{V}\given\tilde{Z}$; often this stronger independence is assumed in practice when selecting and justifying valid instruments, and indeed in all empirical examples considered in~\citet{app-chen-IV-semieff, liu} the instruments used satisfy this stronger independence. The semiparametric efficient estimator under only the weaker identifiability condition~\eqref{eq:IV-condition} has (scaled) efficient influence function
\begin{gather*}
    \hspace{-12em}
    \psi_{\text{IV.eff}}(S;\theta_P,\eta_P) := \frac{1}{\varsigma_P(\tilde{V},\tilde{Z})}
    \bigg( \iota_P(\tilde{V},\tilde{Z}) - \frac{\E_P\big[\varsigma_P^{-1}(\tilde{V},\tilde{Z}) \tilde{X}\given\tilde{Z}\big]}{\E_P\big[\varsigma_P^{-1}(\tilde{V},\tilde{Z})\given\tilde{Z}\big]} \bigg)
    \\
    \hspace{22em}
    \cdot\Big(\tilde{Y}-\E_P\big[\tilde{Y}\given\tilde{Z}\big]-\theta(\tilde{X}-\E_P\big[\tilde{X}\given\tilde{Z}\big])\Big),
    \\
    \varsigma_P(\tilde{V},\tilde{Z}) := \E_P\Big[\big(\tilde{Y}-\E_P[\tilde{Y}\given\tilde{Z}] - \theta_P(\tilde{X}-\E_P[\tilde{X}\given\tilde{Z}])\big)^2 \biggiven \tilde{V},\tilde{Z}\Big] ,
    \quad
    \iota_P(\tilde{V},\tilde{Z}) := \E_P\big[\tilde{X}\given\tilde{V},\tilde{Z}\big],
\end{gather*}
where $\eta_P$ collects the relevant conditional expectation functions given above. Under the stronger independence model $(\tilde{Y}-\theta_P\tilde{X})\independent\tilde{V}\given\tilde{Z}$ the (scaled) efficient influence function $\psi_{\text{IV.eff}}$ collapses to
\begin{equation*}
    \psi_{\text{IV.eff}}(S;\theta_P,\eta_P) = w_P(\tilde{Z})
    \big( \E_P[\tilde{X}\given \tilde{V},\tilde{Z}] - \E_P[\tilde{X}\given \tilde{Z}] \big)
    \big(\tilde{Y}-\E_P[\tilde{Y}\given \tilde{Z}]-\theta_P(\tilde{X}-\E_P[\tilde{X}\given \tilde{Z}])\big),
\end{equation*}
where $w_P(\tilde{Z}):=\Var_P(\tilde{Y}-\theta_P \tilde{X}\given \tilde{Z})$. Thus, under the stronger independence model, the efficient instrumental variable estimator lies within the robust class of estimators we consider. Thus, the classically semiparametric efficient instrumental variable estimator coincides with the robust semiparaemtric efficient (ROSE) estimator in instrumental variable regression specifically under the stronger independence model $(\tilde{Y},\tilde{X})\independent\tilde{V}\given\tilde{Z}$.

\section{Influence function calculation in the partially parametric model}\label{appsec:semipara}

Here we formally define the partially parametric model~\eqref{eq:semipara-model}, and construct the class of influence functions of this model.

\begin{assumption}[The partially parametric model]\label{ass:IF}
    The semiparametric model for the tuple $S=(Y,X,Z)$ with distribution $P$ satisfies $\E_P[Y\given X,Z]=\mu(X,Z;\theta_P,f_P(Z))$ 
    for some $\mu:\cX\times\cZ\times\Theta\times\R\to\R$ and is such that:
    \begin{enumerate}[label=(\roman*)]
        \item There exists a unique $\theta_P\in\interior\Theta$ and $f_P:\cZ\to\R$ such that $\E_P[Y\given X,Z] = \mu(X,Z;\theta_P,f_P(Z))$.
        \item For all $(x,z)\in\cX\times\cZ$ the map $\gamma\mapsto\mu(x,z;\theta,\gamma)$ on the domain $\R$ is differentiable, with  $\inf_{(x,z)\in\cX\times\cZ}\big|\frac{\partial}{\partial\gamma}\mu(x,z;\theta_0,\gamma)|_{\gamma=f_0(z)}\big|>0$ and for some positive constant $\omega>0$, $\sup_{(x,z)\in\cX\times\cZ}\sup_{|\eta|\leq\omega} \big|\frac{\partial}{\partial\gamma}(x,z;\theta_0,\gamma)_{\gamma=f_0(z)+\eta}\big|<\infty$.
        \item The conditional variance function $\sigma^2(x,z):=\Var_P(Y\given X=x,Z=z)$ is bounded from above and below by finite positive constants i.e.~$\inf_{(x,z)\in\cX\times\cZ}\sigma^2 (x,z)>0$ and $\sup_{(x,z)\in\cX\times\cZ} \sigma^2 (x,z)<\infty$.
        \item The random variables $Y$ and $(X,Z)$ admit a joint density with respect to the product measure $\nu\otimes\lambda$ where $\nu$ and $\lambda$ are $\sigma$-finite measures on $\mathcal{Y}$ and $\mathcal{X} \times \mathcal{Z}$ respectively. 
    \end{enumerate}
\end{assumption}
Note that depending on the semiparametric model Assumption~\ref{ass:IF}(i) may restrict the class of semiparametric models beyond that of merely the condition~\eqref{eq:semipara-model}. For example, in the semiparametric generalised partially linear model, given a law $P$ satisfying $\E_P[\Var_P(X\given Z)]>0$ the pair $(\theta_P,f_P)$ are uniquely defined by $\theta_P=\frac{\E_P[\Cov_P(g(\E_P[Y\given X,Z]),X\given Z)]}{\E_P[\Var_P(X\given Z)]}$ and $f_P(Z)=\E_P[g(\E_P[Y\given X,Z])\given Z]-\frac{\E_P[\Cov_P(g(\E_P[Y\given X,Z]),X\given Z)]}{\E_P[\Var_P(X\given Z)]}\E_P[X\given Z]$. Also for the semiparametric generalised partially linear model Assumption~\ref{ass:IF}(ii) is satisfied if the link function $g$ is strictly increasing and differentiable.

For the remainder of this section we take $P_0$ to be the true distribution (as opposed to $P$ in the main text), with the corresponding implied parameters/ functions given by $(\theta_0,f_0):=(\theta_{P_0},f_{P_0})$. In all of this section, all expectations, (co)variances and probabilities should be understood to be under the true distribution $P_0$.

\subsection{Proof of Theorem~\ref{thm:IF}}\label{sec:proof-of-IF-thm}

\begin{proof}[Proof of Theorem~\ref{thm:IF}]

By Assumption~\ref{ass:IF}, $Y$ and $(X,Z)$ have a joint density with respect to the product measure $\nu\otimes\lambda$, where $\nu$ and $\lambda$ are $\sigma$-finite. Let $p_0(x,z)$ be the density for $(X,Z)$ with respect to $\lambda$, and let $p_0(y\given x,z)$ be the conditional density for $Y\given(X,Z)$ with respect to $\nu$, and let $p_0(y,x,z) := p_0(x,z) p_0(y\given x,z)$, with the property that
\begin{gather*}
    \int y \, p_0(y\given x,z) d\nu(y) = \mu(x,z;f_0(z)).
\end{gather*}
Note that here and for the remainder of this section we suppress the dependence of $\mu$ above on $\theta_0$, which is fixed throughout. Also let $Q_0$ be the marginal distribution of $(X,Z)$. 
We construct the collection of regular parametric submodels given by
\begin{equation}\label{eq:p_t_def}
\begin{split}
    p_t(y,x,z) &:= p_0(x,z) p_t(y\given x,z), 
    \\
p_t(y\given x,z) &:= p_0(y\given x,z)\bigg(1+\big(\mu(x,z;f_t(z))-\mu(x,z;f_0(z))\big)
    \\
    &\;\;\cdot\frac{h(y-\mu(x,z;f_0(z)))-\E[h(Y-\mu(X,Z;f_0(Z)))\given X=x,Z=z]}{\Cov(h(Y-\mu(X,Z;f_0(Z))),Y\given X=x,Z=z)}\bigg) 
    \\
    h(u) &:= u \ind_{\{|u|\leq C\}} + C \,\sgn (u) \, \ind_{\{|u|>C\}} 
    \\
    f_t(z) &:= f_0(z) + t g(z), 
    \end{split}
\end{equation}
where $C$ is as defined in Lemma~\ref{lem:valid-density}, and where $g$ satisfies $m\leq\inf_{z\in\cZ}g(z)\leq\sup_{z\in\cZ}g(z)\leq M$ for positive finite constants $M>m>0$. By Lemma~\ref{lem:valid-density} there exists some $\epsilon>0$ such that for all $t\in[-\epsilon,\epsilon]$, $p_t$ is a density in the full semiparametric model~\eqref{eq:semipara-model}, with
\begin{equation}
    \int y \, p_t(y\given x,z)d\nu(y) = \mu(x,z;f_t(z)).
\end{equation}
We now proceed to compute the tangent set of the semiparametric model corresponding to the collection of submodels of the form $\{p_t : t \in [-\epsilon, \epsilon]\}$ as the function $g$ varies among all bounded functions. Since the parameter $\theta_0$ is fixed, this is a subset of the (full) nuisance tangent space consisting of the scores of all possible submodels with $\theta_0$ fixed. Recall \citep[\S 25.4]{app-vandervaart} that each influence function is orthogonal to the nuisance tangent space. The tangent set we construct will in fact be the full nuisance tangent set, though this is inconsequential for our purposes: our goal is only to verify that any element of the orthogonal complement of the full nuisance tangent set, and hence any influence function, is of the form given in~\eqref{eq:IF}.

Now suppose $\varphi\in L_2(P_0)$ is a score function corresponding to our path, that is
\begin{equation*}
    \lim_{t\downarrow 0} \int \bigg(\frac{\sqrt{p_t(y,x,z)}-\sqrt{p_0(y,x,z)}}{t}
    - \frac{1}{2} \varphi(y,x,z) \sqrt{p_0(y,x,z)}
    \bigg)^2 d\nu(y)d\lambda(x,z) = 0.
\end{equation*}
By the Cauchy--Schwarz inequality, this implies that for all sequences of functions $H_t:\mathcal{S}\to\R$ with
\begin{equation*}
\limsup_{t \downarrow 0}\int H_t(y,x,z)^2 d\nu(y)d\lambda(x,z) < \infty,
\end{equation*}
we have
\begin{equation*}
    \lim_{t\downarrow0} \int 
    H_t(y,x,z)
    \bigg(\frac{\sqrt{p_t(y,x,z)}-\sqrt{p_0(y,x,z)}}{t} 
    -\frac{1}{2} \varphi(y,x,z)\sqrt{p_0(y,x,z)}
    \bigg)
    d\nu(y)d\lambda(x,z) 
    =0.
\end{equation*}
Now, using the shorthand $\mu'(x,z;\gamma) = \frac{\partial}{\partial\gamma}\mu(x,z;\gamma)$, take
\begin{equation}\label{eq:Ht}
    H_t(y,x,z) = \frac{\big(y-\mu(x,z;f_t(z))\big)F(x,z)}{\mu'(x,z;f_0(z))g(z)} \big(\sqrt{p_t(y,x,z)}
    +\sqrt{p_0(y,x,z)}\big),
\end{equation}
where $F\in L_2(Q_0)$. We then have
\begin{align*}
    &\qquad
    \int \sup_{t\in[-\epsilon,\epsilon]}H_t^2(y,x,z)d\nu(y)d\lambda(x,z)
    \\
    &\leq \int \frac{F^2(x,z)}{\big(\mu'(x,z;f_0(z))\big)^2\big(g(z)\big)^2} \sup_{t\in[-\epsilon,\epsilon]} \big(y-\mu(x,z;f_t(z))\big)^2 
    \\
    &\qquad\qquad\qquad\qquad\qquad\qquad
    \cdot
    \big(\sqrt{p_t(y,x,z)}+\sqrt{p_0(y,x,z)}\big)^2 d\nu(y)d\lambda(x,z)
    \\
    &\leq \frac{1}{a^2m^2} \int F^2(x,z) \sup_{t\in[-\epsilon,\epsilon]} \big(y-\mu(x,z;f_t(z))\big)^2 \big(\sqrt{p_t(y,x,z)}+\sqrt{p_0(y,x,z)}\big)^2 d\nu(y)d\lambda(x,z)
    \\
    &\leq \frac{2}{a^2m^2} \int F^2(x,z) \sup_{t\in[-\epsilon,\epsilon]} \big(y-\mu(x,z;f_t(z))\big)^2 \big(p_t(y,x,z)+p_0(y,x,z)\big) d\nu(y)d\lambda(x,z)
    \\
    &\leq \frac{2}{a^2m^2} \int F^2(x,z) p_0(x,z) \bigg\{ \underbrace{\int\sup_{t\in[-\epsilon,\epsilon]}\big(y-\mu(x,z;f_t(z))\big)^2p_t(y\given x,z)d\nu(y)}_{\text{$<\infty$ (see Lemma~\ref{lem:valid-density})}}
    \\
    &\qquad
    + \underbrace{\int\big(y-\mu(x,z;f_0(z))\big)^2p_0(y\given x,z)d\nu(y)}_{<\infty}
    \\
    &\qquad
    +  \underbrace{\sup_{t\in[-\epsilon,\epsilon]}\big(\mu(x,z;f_t(z))-\mu(x,z;f_0(z))\big)^2 }_{<\infty \text{ (see Lemma~\ref{lem:valid-density})}}
    \bigg\}
    d\lambda(x,z)
    \\
    &< \infty,
\end{align*}
where we make use of $\inf_{z\in\cZ}|g(z)|\geq m > 0$ and $\inf_{(x,z)\in\cX\times\cZ}|\mu'(x,z;f_0(z))|\geq a$ for some constant $a>0$ (which exists by Assumption~\ref{ass:IF}~(ii)). For $H_t$ as in~\eqref{eq:Ht} we have that
\begin{multline*}
    \frac{1}{2}\lim_{t\downarrow0}\int H_t(y,x,z)\varphi(y,x,z)\sqrt{p_0(y,x,z)} d\nu(y)d\lambda(x,z)
    \\
    =
    \E\bigg[\frac{\big(Y-\mu(X,Z;f_0(Z))\big)\varphi(Y,X,Z) F(X,Z)}{\mu'(X,Z;f_0(Z))g(Z)}\bigg],
\end{multline*}
where in the above we apply the dominated convergence theorem, which holds because
\begin{align*}
    \quad & \bigg( \int \sup_{t\in[-\epsilon,\epsilon]}\big|H_t(y,x,z)\varphi(y,x,z)\sqrt{p_0(y,x,z)}\big|d\nu(y)d\lambda(x,z) \bigg)^2
    \\
    &\leq \Big(\int\varphi^2(y,x,z)p_0(y,x,z)d\nu(y)d\lambda(x,z)\Big)
    \Big(\int \sup_{t\in[-\epsilon,\epsilon]}H_t^2(y,x,z)d\nu(y)d\lambda(x,z)\Big)
    \\
    &<\infty.
\end{align*}
Next, define
\begin{multline*}
    R_t(y,x,z) 
    :=
    \\
    \frac{\big(y-\mu(x,z;f_t(z))\big)\big(h(y-\mu(x,z;f_0(z))-\E[h(Y-\mu(X,Z;f_0(Z)))\given X=x,Z=z]\big)}{\Cov(h(Y-\mu(X,Z;f_0(Z))),Y\given X=x,Z=z)}.
\end{multline*}
Then by construction
\begin{equation*}
    \int p_0(y\given x,z) R_t(y,x,z) d\nu(y) = 1,
\end{equation*}
for any $(x,z)\in\cX\times\cZ$ and any $t\in[-\epsilon,\epsilon]$. 
Thus
\begin{align*}
    &\lim_{t\downarrow0} \int H_t(y,x,z) \cdot \frac{\sqrt{p_t(y,x,z)}-\sqrt{p_0(y,x,z)}}{t} d\nu(y)d\lambda(x,z)
    \\
    &= \lim_{t\downarrow0} \int \frac{\big(y-\mu(x,z;f_t(z))\big)F(x,z)}{\mu'(x,z;f_0(z))g(z)}
    \cdot \frac{p_t(y,x,z)-p_0(y,x,z)}{t} d\nu(y)d\lambda(x,z)
    \\
    &= \lim_{t\downarrow0} \int p_0(y,x,z)F(x,z)\cdot\frac{\mu(x,z;f_t(z))-\mu(x,z;f_0(z))}{t\mu'(x,z;f_0(z))g(z)}
    R_t(y,x,z) d\nu(y)d\lambda(x,z)
    \\
    &= \lim_{t\downarrow0} \int p_0(y,x,z)F(x,z)\cdot\frac{\mu(x,z;f_t(z))-\mu(x,z;f_0(z))}{t\mu'(x,z;f_0(z))g(z)}
    \\
    &\hspace{15em} \cdot \bigg\{\underbrace{\int p_0(y\given x,z)R_t(y,x,z) d\nu(y)}_{=1}\bigg\}d\lambda(x,z)
    \\
    &= \int p_0(x,z)F(x,z)\cdot \underbrace{\lim_{t\downarrow 0}\bigg\{\frac{\mu(x,z;f_t(z))-\mu(x,z;f_0(z))}{t \mu'(x,z;f_0(z))g(z)}\bigg\}}_{=1}
    d\lambda(x,z)
    \\
    &= \int p_0(x,z)F(x,z) d\lambda(x,z)
    \\
    &= \E[F(X,Z)],
\end{align*}
where again we use the dominated convergence theorem as well as Fubini's theorem; we verify the conditions for each of these now. First, with the positive finite constants $(C,c,\delta_1)$ defined in Lemma~\ref{lem:valid-density}, we have
\begin{align*}
    &\quad 
    \sup_{t\in[-\epsilon,\epsilon]}\bigg|H_t(y,x,z) \cdot \frac{\sqrt{p_t(y,x,z)}-\sqrt{p_0(y,x,z)}}{t}\bigg|
    \\
    &= 
    \sup_{t\in[-\epsilon,\epsilon]}\bigg|\frac{\big(y-\mu(x,z;f_t(z))\big)F(x,z)}{\mu'(x,z;f_0(z))tg(z)}\big(p_t(y,x,z)-p_0(x,y,z)\big)\bigg|
    \\
    &=
    \underbrace{\sup_{t\in[-\epsilon,\epsilon]}\bigg|\frac{\mu(x,z;f_t(z))-\mu(x,z;f_0(z))}{\mu'(x,z;f_0(z)) t g(z)}\bigg|}_{<\infty\text{ (see Lemma~\ref{lem:valid-density})}}
    \cdot
    \big|y-\mu(x,z;f_t(z))\big|
    \cdot
    |F(x,z)|
    \\
    &\qquad\qquad
    \cdot
    \underbrace{\bigg|\frac{h(y-\mu(x,z;f_0(z)))-\E[h(Y-\mu(X,Z;f_0(Z)))\given X=x,Z=z]}{\Cov(h(Y-\mu(X,Z;f_0(Z))),Y\given X=x,Z=z)}\bigg|}_{\leq 2Cc^{-2}\delta_1^{-1} \text{ (see Lemma~\ref{lem:valid-density})}}
    \, p_0(y,x,z)
    \\
    &\leq
    2^2Cc^{-2}\delta_1^{-1} \sup_{t\in[-\epsilon,\epsilon]}
    \big|y-\mu(x,z;f_t(z))\big|
    \cdot
    |F(x,z)| \, p_0(y,x,z).
\end{align*}
Then, by Jensen's inequality and the above inequality,
\begin{align*}
    &\quad
    \bigg(\int\sup_{t\in[-\epsilon,\epsilon]}\bigg|H_t(y,x,z) \cdot \frac{\sqrt{p_t(y,x,z)}-\sqrt{p_0(y,x,z)}}{t}\bigg|d\nu(y)d\lambda(x,z)\bigg)^2
    \\
    &\leq
    2^4C^2c^{-4}\delta_1^{-2}
    \int \sup_{t\in[-\epsilon,\epsilon]} \Big(\big(y-\mu(x,z;f_t(z))\big)^2 F^2(x,z) p_0(y,x,z)\Big) d\nu(y)d\lambda(x,z)
    \\
    &=
    2^4C^2c^{-4}\delta_1^{-2} \int \bigg\{ \underbrace{\int \big(y-\mu(x,z;f_0(z))\big)^2p_0(y\given x,z) d\nu(y)}_{<\infty} 
    \\
    &\qquad + \underbrace{\sup_{t\in[-\epsilon,\epsilon]}\big(\mu(x,z;f_t(z))-\mu(x,z;f_0(z))\big)^2}_{<\infty \text{ (see Lemma~\ref{lem:valid-density})}} \bigg\} \, F^2(x,z)p_0(x,z)  d\lambda(x,z)
    \\
    &<\infty.
\end{align*}
Thus we can apply the dominated convergence theorem in the above. To verify the conditions for Fubini's theorem, it follows immediately from the above that for any $t\in[-\epsilon,\epsilon]$
\begin{equation*}
    \int \big|H_t(y,x,z)\varphi(y,x,z)\sqrt{p_0(y,x,z)}\big|d\nu(y)d\lambda(x,z) < \infty.
\end{equation*}

Therefore the score function $\varphi$ must satisfy the condition that for all $F\in L_2(Q_0)$,
\begin{equation*}
    \E\big[F(X,Z)\big] = \E\bigg[\frac{\big(Y-\mu(X,Z;f_0(Z))\big)\varphi(Y,X,Z) F(X,Z)}{\mu'(X,Z;f_0(Z))g(Z)}\bigg],
\end{equation*}
and so
\begin{equation*}
    g(Z)\mu'(X,Z;f_0(Z)) = \E\big[\big(Y-\mu(X,Z;f_0(Z))\big)\varphi(Y,X,Z)\given X,Z\big].
\end{equation*}

Let $\Lambda_0$ be the closure of the linear span of the nuisance tangent set of the semiparametric model corresponding to the submodels $\{p_t:t\in[-\epsilon,\epsilon]\}$ as $g$ ranges over all bounded functions. Then
\begin{align*}
    \Lambda^{(g)}_0 &:= \bigg\{
    \varphi\in L_2(P_0)
    \,:\,
    \E[\varphi(Y,X,Z)] = 0,\,
    \E\bigg[\frac{(Y-\mu(X,Z;f_0(Z)))\varphi(Y,X,Z)}{\mu'(X,Z;f_0(Z))} \,\Big|\, X,Z\bigg] = g(Z)
    \bigg\}
    ,\,
    \\
    \Lambda_0 &:= \overline{\text{span}\cup_{g} \Lambda_0^{(g)}} =  \bigg\{
    \varphi\in L_2(P_0) \,:\,
    \E[\varphi(Y,X,Z)] = 0,
    \\
    &\quad\E\bigg[\frac{(Y-\mu(X,Z;f_0(Z)))\varphi(Y,X,Z)}{\mu'(X,Z;f_0(Z))} \,\Big|\, X,Z\bigg] = \E\bigg[\frac{(Y-\mu(X,Z;f_0(Z)))\varphi(Y,X,Z)}{\mu'(X,Z;f_0(Z))} \,\Big|\, Z\bigg]
    \bigg\},
\end{align*}
with orthogonal complement
\begin{multline*}
    \Lambda_0^\perp = \bigg\{
    \psi\in L_2(P_0)
    \;:\;
    \psi(Y,X,Z)=\frac{\big(Y-\mu(X,Z;f_0(Z))\big)\phi(X,Z)}{\mu'(X,Z;f_0(Z))}
    ,\,
    \\
    \phi\in L_2(Q_0)
    ,\,
    \E[\phi(X,Z)\given Z] = 0
    \bigg\}.
\end{multline*}

As each influence function is orthogonal to the nuisance tangent space~\citep[\S25.4]{app-vandervaart}, each nuisance function lies in $\Lambda_0^\perp$.

\end{proof}

\begin{lemma}\label{lem:valid-density}
    Consider the setup of Theorem~\ref{thm:IF}. There exist finite constants $C, \epsilon>0$ such that the function $p_t(y,x,z)$ given by~\eqref{eq:p_t_def} (defined in terms of $C$) is a valid density contained within the semiparametric model~\eqref{eq:semipara-model} (i.e.~Assumption~\ref{ass:IF}) for all $t\in[-\epsilon,\epsilon]$. Moreover, for such a path
    \begin{equation*}
         \sup_{(x,z)\in\cX\times\cZ} \int \sup_{t\in[-\epsilon,\epsilon]}\big(y-\mu(x,z;f_t(z))\big)^2p_t(y\given x,z)d\nu(y) < \infty.
    \end{equation*}
\end{lemma}
\begin{proof}
    By Assumption~\ref{ass:IF}(iii) there exist $c_1, C_1 >0$ such that $C_1 \geq \sup_{(x,z)\in\cX\times\cZ}\Var(Y\given X=x,Z=z) \geq \inf_{(x,z)\in\cX\times\cZ}\Var(Y\given X=x,Z=z) \geq c_1$. Then there exists constants $C>c>0$ and $\delta_1>0$ such that $\inf_{(x,z)\in\cX\times\cZ}\PP(c\leq|Y-\mu(X,Z;f_0(Z))|\leq C\given X=x,Z=z)\geq\delta_1$.

    We first show that $p_t(y|x,z)$ is non-negative for sufficiently small $t$. Recalling the definition of the function $h$~\eqref{eq:p_t_def}, we have
    \begin{align*}
        &\quad\inf_{(x,z)\in\cX\times\cZ}\Cov\big(h(Y-\mu(X,Z;f_0(Z))),Y\given X=x,Z=z\big)
        \\
        &= \inf_{(x,z)\in\cX\times\cZ}\E\big[(Y-\mu(X,Z;f_0(Z)))\, h(Y-\mu(X,Z;f_0(Z)))\given X=x,Z=z\big]
        \\
        &= \inf_{(x,z)\in\cX\times\cZ}\big(
        \E\big[(Y-\mu(X,Z;f_0(Z)))^2\ind_{\{|Y-\mu(X,Z;f_0(Z))|<c\}}\given X=x,Z=z\big]
        \\
        &\quad +   \E\big[(Y-\mu(X,Z;f_0(Z)))^2\ind_{\{c\leq |Y-\mu(X,Z;f_0(Z))|\leq C\}}\given X=x,Z=z\big]
        \\
        &\quad +   C\,\E\big[|Y-\mu(X,Z;f_0(Z))|\ind_{\{|Y-\mu(X,Z;f_0(Z))|>C\}}\given X=x,Z=z\big]
        \big)
        \\
        & \geq \inf_{(x,z)\in\cX\times\cZ}\E\big[(Y-\mu(X,Z;f_0(Z)))^2\ind_{\{c\leq |Y-\mu(X,Z;f_0(Z))|\leq C\}}\given X=x,Z=z\big]
        \\
        & \geq c^2 \inf_{(x,z)\in\cX\times\cZ} \PP\big(c\leq|Y-\mu(X,Z;f_0(Z))|\leq C \given X=x,Z=z \big)
        \\
        & \geq c^2\delta_1,
    \end{align*}
     and also
    \begin{equation*}
        \sup_{(y,x,z)\in\mathcal{S}}\big| h(y-\mu(x,z;f_0(z)))-\E[h(Y-\mu(X,Z;f_0(Z)))\given X=x,Z=z] \big| \leq 2C,
    \end{equation*}
    which together gives
    \begin{equation*}
        \sup_{(y,x,z)\in\mathcal{S}}\bigg| \frac{h(y-\mu(x,z;f_0(z)))-\E[h(Y-\mu(X,Z;f_0(Z)))\given X=x,Z=z]}{\Cov(h(Y-\mu(X,Z;f_0(Z))),Y\given X=x,Z=z)} \bigg| \leq 2Cc^{-2}\delta_1^{-1}.
    \end{equation*}

    We now proceed to bound the quantity $\sup_{(x,z)\in\cX\times\cZ}|\mu(x,z;f_t(z))-\mu(x,z;f_0(z))|$. By the mean value theorem, for each $(x,z)\in\cX\times\cZ$ there exists some $\tau(x,z)\in[-|t|,|t|]$ such that
\begin{equation*}
    \mu(x,z;f_t(z))-\mu(x,z;f_0(z)) = \mu'(x,z;f_{\tau(x,z)}(z))tg(z).
\end{equation*}
Thus for $t\neq0$ with $|t| \leq M^{-1}\omega$
\begin{align*}
    &\quad\;
    \sup_{(x,z)\in\cX\times\cZ}\big|\mu(x,z;f_t(z))-\mu(x,z;f_0(z))\big|
    \\
    &=
    \sup_{(x,z)\in\cX\times\cZ}\bigg|\frac{\mu(x,z;f_t(z))-\mu(x,z;f_0(z))}{\mu'(x,z;f_{\tau(x,z)}(z))tg(z)}\bigg|
    \cdot
    \big|\mu'(x,z;f_{\tau(x,z)}(z))\big| \cdot \big|tg(z)\big|
    \\
    &\leq
    M |t| \sup_{(x,z)\in\cX\times\cZ}\big|\mu'(x,z;f_0(z)+\tau(x,z)g(z))\big|
    \\
    &\leq
    M |t| \sup_{(x,z)\in\cX\times\cZ}\sup_{|\eta|\leq\omega} \big|\mu'(x,z;f_0(z)+\eta)\big|
    \\
    &\leq M L |t|,
\end{align*}
where $L:=\sup_{(x,z)\in\cX\times\cZ}\sup_{|\eta|\leq\omega} |\mu'(x,z;f_0(z)+\eta)|<\infty$; note that $|\tau(x,z)g(z)|\leq M |t| \leq \omega$. Thus for all $|t|\leq \delta_2 := \big(M^{-1}\omega\big) \wedge \big(2^{-1}C^{-1}c^2\delta_1 M^{-1}L^{-1}\big)$,
    \begin{align*}
        &\quad
        \sup_{(y,x,z)\in\mathcal{S}} \bigg|\big(\mu(x,z;f_t(z))-\mu(x,z;f_0(z))\big)
        \\
        &\qquad\qquad\qquad\cdot
        \frac{\big(h(y-\mu(x,z;f_0(z)))-\E[h(Y-\mu(X,Z;f_0(Z)))\given X=x,Z=z]\big)}{\Cov(h(Y-\mu(X,Z;f_0(Z))),Y\given X=x,Z=z)}\bigg|
        \\
        &\leq 2Cc^{-2}\delta_1^{-1} \sup_{(x,z)\in\cX\times\cZ}\big|\mu(x,z;f_t(z))-\mu(x,z;f_0(z))\big|
        \\
        &\leq 2Cc^{-2}\delta_1^{-1}ML|t| \leq 1.
    \end{align*}
    Thus taking any $0<\epsilon<\delta_2$, for $t\in[-\epsilon,\epsilon]$ the function $p_t(y,x,z)$ is non-negative.

    It is also straightforward to show that for such paths
    \begin{equation*}
        \int p_t(y\given x,z)d\nu(y) = 1,
        \qquad
        \int y \, p_t(y\given x,z)d\nu(y) = \mu(x,z;f_t(z)).
    \end{equation*}

Finally, to verify Assumption~\ref{ass:IF}(iii) holds for all distributions in this class,
\begin{align*}
    &\quad\; \sup_{(x,z)\in\cX\times\cZ}\int \sup_{t\in[-\epsilon,\epsilon]}\big(y-\mu(x,z;f_t(z))\big)^2 \, p_t(y\given x,z) d\nu(y)
    \\
    &\leq 2 \sup_{(x,z)\in\cX\times\cZ} \int \sup_{t\in[-\epsilon,\epsilon]} \big(y-\mu(x,z;f_t(z))\big)^2 p_0(y\given x,z) d\nu(y)
    \\
    &\leq 4 \sup_{(x,z)\in\cX\times\cZ} \int \big(y-\mu(x,z;f_0(z))\big)^2 p_0(y\given x,z) d\nu(y)
    \\
    &\qquad + 4 \sup_{(x,z)\in\cX\times\cZ} \sup_{t\in[-\epsilon,\epsilon]} \big(\mu(x,z;f_t(z))-\mu(x,z;f_0(z))\big)^2 
    \\
    &< \infty,
\end{align*}
for any $\epsilon\leq\delta_2$; note that the first inequality follows from
\begin{multline*}
    \sup_{(y,x,z)\in\mathcal{S}}\bigg|\big(\mu(x,z;f_t(z))-\mu(x,z;f_0(z))\big)
    \\
    \cdot \frac{h(y-\mu(x,z;f_0(z)))-\E[h(Y-\mu(X,Z;f_0(Z)))\given X=x,Z=z]}{\Cov(h(Y-\mu(X,Z;f_0(Z))),Y\given X=x,Z=z)}\bigg| \leq 1.
\end{multline*}
\end{proof}

\section{Proof of Theorem~\ref{thm:slow-rates} (the `slow rates' theorem)}

In this section we prove Theorem~\ref{thm:slow-rates}; a hardness result for estimating nuisance functions that do not take particular forms, and by proxy for semiparametric efficient estimation. Our proof builds on classical slow rates results for classification~\citep{app-cover, app-devroye, app-gyorfi}. Throughout this section we take $\|\cdot\|$ to denote the Euclidean norm and $B_r(z):=\{z'\in\cZ:\|z'-z\|<r\}$ to be the open ball in $\cZ\subseteq\R^{d}$ with centre $z\in\cZ$ and radius $r$. Also throughout we notate $m_j(z):=m_{P^*,j}(z)=\E_{P^*}[M_j(X)\given Z=z]$. Throughout this section all densities are to be interpreted as  with respect to Lebesgue measure (all such densities introduced will exist by the assumptions of Theorem~\ref{thm:slow-rates}). We also define the constants $\bar{\gamma}_W:=\sup_{(x,z)\in\cX\times\cZ}W(x,z)$ and $\underline{\gamma}_W:=\inf_{(x,z)\in\cX\times\cZ}W(x,z)$, both strictly positive and finite by assumption.

\begin{proof}[Proof of Theorem~\ref{thm:slow-rates}]
First note that it is sufficient to show that for all $\alpha > 0$, there exists a sequence $(P_n)_{n \in \mathbb{N}} \subset \mathcal{P}_{XZ}$ with $\TV(P_n, P^*) \leq \alpha$ for all $n \in \mathbb{N}$ and
\begin{equation} \label{eq:P_n_bad}
\underset{n\to\infty}{\limsup} \frac{\E_{P_n}\Big[\big(\hat{\phi}_n(X,Z)-\phi_{P_n}(X,Z)\big)^2\Big]}{a_n}\geq1.
\end{equation}
To see this, suppose the above holds and take any sequence $(b_n)_{n \in \mathbb{N}}$ of positive numbers converging to $0$. Suppose $(P_{m,n})_{n \in \mathbb{N}, m \in \mathbb{N}} \subset \mathcal{P}_{XZ}$ is such that for each $m \in \mathbb{N}$, the sequence $(P_{m,n})_{n \in \mathbb{N}}$ satisfies the above with $\alpha = b_m$. Set $N(0)=0$ and then  for each $m \in \mathbb{N}$ iteratively find $N(m) > N(m-1)$ such that
\[
\frac{\E_{P_{m,N(m)}}\Big[\big(\hat{\phi}_{N(m)}(X,Z)-\phi_{P_{m,N(m)}}(X,Z)\big)^2\Big]}{a_{N(m)}}\geq1.
\]
Then let $M(n) = \min \{m : N(m) \geq n\}$ and consider the sequence $(P_{M(n), n})_{n \in \mathbb{N}}$. Note that $M(n) \to \infty$ as $n \to \infty$ so $\TV(P_{M(n), n}, P^*) \leq b_{M(n)} \to 0$. Also when $n = N(m)$ for some $m \in \mathbb{N}$, we have that the last display above holds with $P_{M(n), n}$ in place of $P_{m,N(m)}$. Thus \eqref{eq:P_n_bad} holds with $P_n := P_{M(n),n}$ as well, as required. In the following therefore we fix $\alpha > 0$ and aim to find $(P_n)_{n \in \mathbb{N}} \subset \mathcal{P}_{XZ}$ with $\TV(P_n, P^*) \leq \alpha$ for all $n \in \mathbb{N}$ satisfying \eqref{eq:P_n_bad}.

Next note that without loss of generality we may take $M_1,\ldots,M_J$ and $\cdot\mapsto1$ to be linearly independent functions; otherwise we can iteratively omit each of the functions $M_j$ that can be written as a linear combination of the others (and if $\E_P[M_j(X)\given Z=z]=m_j(z)$ for the remaining functions, where $P$ is an arbitrary law, and using linearity of expectations, the omitted functions still satisfy $\E_P[M_j(X)\given Z=z]=m_j(z)$ for all omitted $M_j$). Also as $P^*\in\cP_{XZ}$, the distribution $P^*$ admits the conditional and marginal densities $p^*(x\given z)$ and $p^*(z)$ respectively. We define $\gamma:=\sup_{z\in\cZ}{p^*(z)}<\infty$, and $\iota:\cZ\to\R$ by
\begin{equation*}
    \iota(z) := \frac{\E_{P^*}[W(X,Z)U(X,Z)\given Z=z]}{\E_{P^*}[W(X,Z)\given Z=z]}.
\end{equation*}
Then by taking Lemma~\ref{lem:theta+c} with $\zeta:=\gamma^{-1}\alpha$, we see that there exists a non-empty open ball $B\subset\cZ$, constant $\kappa>0$ and conditional density $p^{\dagger}(x\given z)$ such that the following holds:
\begin{equation*}
    \frac{\int_\cX p^{\dagger}(x\given z)W(x,z)U(x,z)dx}{\int_\cX p^{\dagger}(x\given z)W(x,z)dx}-\frac{\int_\cX p^*(x\given z)W(x,z)U(x,z)dx}{\int_\cX p^*(x\given z)W(x,z)dx}=\kappa,
\end{equation*}
for all $z\in B$ and
\begin{equation*}
    \int_{B}\int_{\cX}|p^{\dagger}(x\given z)-p^*(x\given z)|dxdz \leq \zeta.
\end{equation*}
Let ${v}:=\int_{B}dz=\Vol(B)$ and let $z_0$ be the centre of the ball $B$. For each $m\in\mathbb{N}$ we construct a class of distributions on $(X,Z)$ parametrised by some $(m,c)_{m\in\mathbb{N}, c\in\{0,1\}^{\mathbb{N}}}$. 
In the following, we fix an $m\in\mathbb{N}$ and where clear omit the dependence on $m$. We will first define a marginal distribution on $Z$ given in terms of some $q_m\in\R_+^m$ that satisfies $\sum_{k=1}^m q_{m,k}=\frac{{v}}{2}$.
Given an arbitrary $q\in\R_+^{m}$ satisfying $\sum_{k=1}^{m}q_k=\frac{{v}}{2}$ we fix a partition $(S_k, S'_k)_{k\in [m]}$ of $B$, defined as follows. For $k \in [m]$, define $r_k>r'_k>0$ to be such that
\begin{align*}
\Vol B_{r_k}(z_0) = 2\sum_{l=k}^m q_l \qquad \text{and} \qquad \Vol B_{r'_k}(z_0) = q_k + 2\sum_{l=k+1}^m q_l
\end{align*}
and set $r_{m+1} :=0$.
Note that then $B_{r_1}(z_0) =B$.
Next define for $k \in [m]$ the (thick) shells
\begin{gather*}
	S'_k := B_{r_k}(z_0)\backslash B_{r'_{k}}(z_0)
	\qquad \text{and} \qquad
	S_k := B_{r'_k}(z_0)\backslash B_{r_{k+1}}(z_0).
\end{gather*}
Note then $(S_k, S'_k)_{k\in [m]}$ forms a partition of $B$ with $\Vol S_k = \Vol S'_k=q_k$.
Also define $S:=\cup_{k \in[m]} S_k$. Writing $\beta:={v}^{-1}\int_{B}p^*(z)dz$ and \[
\delta:=\frac{\alpha}{2\beta{v}}\wedge\frac{1}{2},
\] we then define the marginal density
\begin{equation}\label{eq:p_0(z)}
    p_0(z) := p^*(z)\ind_{\cZ\backslash B}(z) + \big((1-\delta)p^*(z)+\delta \beta\big)\ind_{B}(z).
\end{equation}
Now fix some $c\in\{0,1\}^{m}$ and define the conditional density
\begin{gather}
    p_0(x\given z) := p^*(x\given z) + (p^{\dagger}-p^*)(x\given z) H(z), \label{eq:p_0(x|z)}
\end{gather}
where $p^{\dagger}$ is as defined in Lemma~\ref{lem:theta+c} and
\begin{gather*}
    H(z) :=
    \begin{cases}
        0 &\;\text{if $z\in\cZ\backslash B$}
        \\
        c_k &\;\text{if $z\in S_k$}
        \\
        c_k &\;\text{if $z\in S_k'$ and $c_k=c_{k-1}$}
        \\
        \Big(1+\exp\Big(\frac{r_k+r'_k-2\|z-z_0\|}{(\|z-z_0\|-r'_k)(r_k-\|z-z_0\|)}\cdot(c_{k-1}-c_k)\Big)\Big)^{-1} &\;\text{if $z\in S_k'$ and $c_k\neq c_{k-1}$} 
    \end{cases}, \notag
\end{gather*}
where additionally we take $c_0:=0$. 
Together this defines a joint density $p_0(x,z) := p_0(x\given z)p_0(z)$, which we claim satisfies the following properties:
\begin{enumerate}[label=(\roman*)]
    \item\label{item:TV} $\frac{1}{2}\int_{\cX\times\cZ}|p_0(x,z)-p^*(x,z)|d(x,z)\leq\alpha$.
    \item\label{item:equicontinuous} The functions $z \mapsto p_0(x\given z)$ for $x\in\cX$ are equicontinuous.
    \item\label{item:boundedness} $p_0(z)$ and $p_0(x\given z)$ are both bounded away from zero and infinity. 
    \item\label{item:iota+kappa} 
    \begin{equation*}
        \frac{\int_{\cX}p_0(x\given z)W(x,z)U(x,z)dx}{\int_{\cX}p_0(x\given z)W(x,z)dx} = \iota(z) + \kappa c_k,
    \end{equation*}
    for any $k\in[m]$ and $z\in  S_k$, and where $\kappa>0$ is as defined in Lemma~\ref{lem:theta+c}.
    \item\label{item:m_j} $\int_{\cX}p_0(x\given z)M_j(x)dx=m_j(z)$ for all $z\in\cZ$.
\end{enumerate}
We verify each of these claims in turn.

To show~\ref{item:TV}, first note that
\begin{align*}
p_0(x,z)-p^*(x,z) &= \{p_0(x\given z)p_0(z)-p^*(x\given z)p^*(z)\}\ind_B(z) \\
&=  [p_0(x\given z)\{(1-\delta)p^*(z) + \delta \beta \}-p^*(x\given z)p^*(z)]\ind_B(z)\\
&=  [(1-\delta)p^*(z) \{p_0(x\given z) - p^*(x \given z)\}  - \delta p^*(x, z) +\delta \beta p_0(x \given z)]\ind_B(z).
\end{align*}
 Thus using Lemma~\ref{lem:theta+c}, we have
 \begin{align*}
     2	\text{TV}(P_0,P^*) &= \int_{\cX\times\cZ}|p_0(x,z)-p^*(x,z)|\, d(x,z)
 	\\
 	&\leq (1-\delta) \int_B p^*(z) \int_{\mathcal{X}} |p_0(x \given z) - p^*(x \given z)| \,dx \,dz 
  \\
  &\qquad\qquad\qquad + \delta \int_{B}p^*(z)\underbrace{\int_{\cX}p^*(x\given z)\,dx}_{=1}\,dz + \delta \beta {v} \\
 	&\leq (1-\delta)\gamma\underbrace{\int_B\int_{\mathcal{X}} |p_0(x \given z) - p^*(x \given z)| \,dx \,dz}_{\leq\zeta\text{ (see Lemma~\ref{lem:theta+c})}}
    + \delta \underbrace{\int_{B}p^*(z)\,dz}_{=\beta{v}} + \delta \beta {v}
    \\
    &\leq (1-\delta)\gamma\zeta + 2\beta{v}\delta
    \\
    &\leq \gamma\zeta + 2\beta{v}\delta
    \\
    &\leq 2\alpha,
 \end{align*}
with the last inequality following by definition of the constants $\zeta:=\gamma^{-1}\alpha$ and $\delta := \frac{\alpha}{2\beta{v}}\wedge\frac{1}{2}$ above, repeated here for convenience.

Claim~\ref{item:equicontinuous} follows trivially by construction of~\eqref{eq:p_0(x|z)}, noting that function $$r\mapsto \bigg(1+\exp\bigg(\frac{l_1+l_2-2r}{(r-l_1)(l_2-r)}\bigg)\bigg)^{-1}\ind_{(l_1,l_2)}(r) + \ind_{[l_2,\infty)}(r),$$
for any $0<l_1<l_2$, is continuous.

Claim~\ref{item:boundedness} also follows by construction, as $p^*(x\given z)$ and $p^{\dagger}(x\given z)$ are both bounded away from zero and infinity and $H(z)$ is non-negative and bounded by 1.

Claim~\ref{item:iota+kappa} follows as for any $k\in[m]$ and $z\in S_k$ we have $H(z)=c_k$ and so $p_0(x\given z)=p^*(x\given z)\ind_{\{c_k=0\}} + p^{\dagger}(x\given z)\ind_{\{c_k=1\}}$ and thus alongside Lemma~\ref{lem:theta+c}(i),
\begin{equation*}
    \frac{\int_{\cX}p_0(x\given z)W(x,z)U(x,z)dx}{\int_{\cX}p_0(x\given z)W(x,z)dx}
    =
    \iota(z) + \kappa c_k.
\end{equation*}

To prove~\ref{item:m_j}, note that in the proof of Lemma~\ref{lem:theta+c} we show that
\begin{equation*}
    \int_{\cX}(p^{\dagger}- p^* )(x\given z) M_j(x) dx=0,
\end{equation*}
for any $z\in B$. Noting also that $\int_{\cX}p^*(x\given z)M_j(x)dx=m_j(z)$, it then follows
\begin{equation*}
    \int_{\cX}p_0(x\given z)M_j(x)dx
    =
    \int_{\cX}p^*(x\given z)M_j(x)dx
    +
    H(z)\int_{\cX}(p^{\dagger}-p^*)(x\given z)M_j(x)dx = m_j(z),
\end{equation*}
thus proving~\ref{item:m_j}.

We now introduce some additional notation. We will denote by $\nu^{(m,q)}$ the marginal distribution with density $p_0(z)$ as in~\eqref{eq:p_0(z)} and by $\lambda^{(m,c)}$ the conditional distribution with density $p_0(x\given z)$ as in~\eqref{eq:p_0(x|z)}. Together $\nu^{(m,q)}$ and $\lambda^{(m,c)}$ define a collection of joint distributions $\mu^{(m,c,q)}$ for $(X,Z)$, parametrised by $(m,c,q)$. 
For any such $(m,c,q)$ with $\sum_{k=1}^{m}q_{m,k}=\frac{{v}}{2}$ the distribution $\mu^{(m,c,q)}\in\cP_{XZ}$, due to~\ref{item:equicontinuous}, \ref{item:boundedness} and \ref{item:m_j}, and moreover $\text{TV}(\mu^{(m,c,q)},P^*)\leq\alpha$ by~\ref{item:TV}.

For each $m\in\mathbb{N}$ we now 
explicitly define a sequence $q_m = (q_{m,k})_{k\in[m]}$ with $\sum_{k=1}^m q_{m,k}=\frac{{v}}{2}$ as follows. Define the constant $\tau:=8^{-1}\underline{\gamma}_W^2\kappa^2\delta \beta$. Without loss of generality, we may assume that the sequence $(a_n)_{n\in\mathbb{N}}$ is decreasing and satisfies $\tau^{-1} \geq a_1$. Note that by Lemma~\ref{lem:seq}, for any positive non-increasing sequence $(a_n)_{n\in\mathbb{N}}$ that converges to zero such that $a_1\leq \tau^{-1}$ there exists a sequence $(p_k)_{k\in\mathbb{N}}$ satisfying $p_k\geq0$ for all $k\in\mathbb{N}$ with $\sum_{k=1}^{\infty} p_k=\frac{{v}}{2}$ and
    \begin{equation*}
        \frac{\underline{\gamma}_W^2\kappa^2\delta \beta}{8} \sum_{k=1}^{\infty} p_k\Big(1-\frac{2p_k}{{v}}\Big)^n \geq a_n,
    \end{equation*}
    for all $n\in\mathbb{N}$. Then define $(q_{m,k})_{m\in\mathbb{N},k\in[m]}$ by
    \begin{equation*}
        q_{m,k} := \frac{{v}}{2\sum_{k'=1}^m p_{k'}}p_k.
    \end{equation*}
    Note then by construction, $\sum_{k=1}^m q_{m,k}=\frac{{v}}{2}$ for all $m\in\mathbb{N}$, and
\begin{equation}\label{eq:q}
    \sup_{m\in\mathbb{N}}\frac{\tau\sum_{k=1}^m q_{m,k}(1-2{v}^{-1}q_{m,k})^n}{a_n}\geq1,
\end{equation}
for all $n\in\mathbb{N}$.

From hereon we consider the sequences $(q_{m,k})_{k\in[m]}$ defined above~\eqref{eq:q} as fixed. 
Then,
\begin{multline*}
    \limsup_{n\to\infty}\sup_{P\in\cP_{XZ}: \text{TV}(P,P^*)\leq\alpha}\frac{\E_P\Big[\big(\hat{\phi}_n(X,Z)-\phi_P(X,Z)\big)^2\Big]}{a_n}
    \\
    \geq
    \limsup_{n\to\infty}\sup_{m\in\mathbb{N}}\sup_{c\in\{0,1\}^\mathbb{N}}\frac{\E_{\mu^{(m,c,q)}}\Big[\big(\hat{\phi}_n(X,Z)-\phi_{\mu^{(m,c,q)}}(X,Z)\big)^2\Big]}{\tau\sum_{k=1}^m q_{m,k}(1-2{v}^{-1}q_{m,k})^n}
    \cdot
    \frac{\tau\sum_{k=1}^m q_{m,k}(1-2{v}^{-1}q_{m,k})^n}{a_n},
\end{multline*}
noting that $\text{TV}(\mu^{(m,c,q)},P^*)\leq\alpha$. 
We will apply Lemma~\ref{lem:limsup} with
\begin{equation*}
    x_{n,m} = \sup_{c\in\{0,1\}^{\mathbb{N}}}\frac{\E_{\mu^{(m,c,q)}}\Big[\big(\hat{\phi}_n(X,Z)-\phi_{\mu^{(m,c,q)}}(X,Z)\big)^2\Big]}{\tau\sum_{k=1}^m q_{m,k}(1-2{v}^{-1}q_{m,k})^n},
    \,
    y_{n,m} = \frac{\tau\sum_{k=1}^m q_{m,k}(1-2{v}^{-1}q_{m,k})^n}{a_n}.
\end{equation*}
As~\eqref{eq:q} has been proved, to prove Theorem~\ref{thm:slow-rates} via Lemma~\ref{lem:limsup} it suffices to show the following: 
for any subsequence $(m_n)_{n\in\mathbb{N}}$ of the natural numbers
\begin{equation}\label{eq:c}
    \limsup_{n\to\infty}\sup_{c\in\{0,1\}^{\mathbb{N}}}\frac{\E_{\mu^{(m_n,c,q)}}\Big[\big(\hat{\phi}_n(X,Z)-\phi_{\mu^{(m_n,c,q)}}(X,Z)\big)^2\Big]}{\tau\sum_{k=1}^{m_n}q_{m_n,k}(1-2{v}^{-1}q_{m_n,k})^n}\geq1.
\end{equation}
The remainder of the proof verifies the claim~\eqref{eq:c}. Fix a subsequence $(m_n)_{n\in\mathbb{N}}$ of the natural numbers. We will show that there exists some $c\in\{0,1\}^{\mathbb{N}}$ such that
\begin{equation}\label{eq:c-sub}
    \limsup_{n\to\infty}\frac{\E_{\mu^{(m_n,c,q)}}\Big[\big(\hat{\phi}_n(X,Z)-\phi_{\mu^{(m_n,c,q)}}(X,Z)\big)^2\Big]}{\tau\sum_{k=1}^{m_n}q_{m_n,k}(1-2{v}^{-1}q_{m_n,k})^n} \geq 1,
\end{equation}
as then~\eqref{eq:c} follows as
\begin{multline*}
    \limsup_{n\to\infty}\sup_{c\in\{0,1\}^{\mathbb{N}}}\frac{\E_{\mu^{(m_n,c,q)}}\Big[\big(\hat{\phi}_n(X,Z)-\phi_{\mu^{(m_n,c,q)}}(X,Z)\big)^2\Big]}{\tau\sum_{k=1}^{m_n}q_{m_n,k}(1-2{v}^{-1}q_{m_n,k})^n} 
    \\
    \geq
    \sup_{c\in\{0,1\}^{\mathbb{N}}}\limsup_{n\to\infty}\frac{\E_{\mu^{(m_n,c,q)}}\Big[\big(\hat{\phi}_n(X,Z)-\phi_{\mu^{(m_n,c,q)}}(X,Z)\big)^2\Big]}{\tau\sum_{k=1}^{m_n}q_{m_n,k}(1-2{v}^{-1}q_{m_n,k})^n} \geq 1.
\end{multline*}
We now proceed to prove~\eqref{eq:c-sub}.

Consider an arbitrary $c\in\{0,1\}^{\mathbb{N}}$. Let $D_n :=\{(X_1,Z_1),\ldots,(X_n,Z_n)\}$ denote $n$ i.i.d.~copies of data generated from the distribution $\mu^{(m_n,c,q)}$, and let $\mu_n$ be the joint distribution of $D_n$. From hereon, we will omit the dependence on $(m_n,q)$ in the laws $\nu^{(m_n,q)}$, $\lambda^{(m_n,c)}$ and $\mu^{(m_n,c,q)}$, denoting them by $\nu, \lambda^{(c)}$ and $\mu^{(c)}$ respectively (recall here we are considering $q$ and $(m_n)_{n\in\mathbb{N}}$ as fixed). We also omit the dependence on $c$ in the joint law $\mu_n$, and omit the dependence on $m_n$ in $q_{m_n,k}$, writing $q_k$ instead where clear. We will use the shorthand $\E^\star$ for the expectation conditional on $D_n$, and similarly $\PP^\star$ and $\Var^\star$. 
Consider an arbitrary sequence of estimators $(\hat{\phi}_n)_{n\in\mathbb{N}}$ where $\hat{\phi}_n:\cX\times\cZ\to\R$ is trained on $D_n$. 
Also recall that
\begin{equation*}
    \phi_{\lambda^{(c)}}(x,z) = W(x,z)\bigg(U(x,z) - \frac{\E_{\lambda^{(c)}}[W(X,Z)U(X,Z)\given Z=z]}{\E_{\lambda^{(c)}}[W(X,Z)\given Z=z]}\bigg).
\end{equation*}

Given $(\hat{\phi}_n)_{n\in\mathbb{N}}$ define the sequences of functions $(\tilde{\varphi}_n)_{n\in\mathbb{N}}$ and $(\hat{\varphi}_n^{(c)})_{n\in\mathbb{N}}$ by
    \begin{equation*}
        \tilde{\varphi}_n(x,z) := U(x,z) - \frac{\hat{\phi}_n(x,z)}{W(x,z)},
        \quad
        \text{and}
        \quad
        \hat{\varphi}_n^{(c)}(z) := \E^\star_{\lambda^{(c)}}\left[\tilde{\varphi}_n(X,Z)\given Z=z\right],
    \end{equation*}
    for each $n\in\mathbb{N}$. Note we denote the dependence on $D_n$ in such functions through the $n$ subscript. We also define
    \begin{equation*}
        \varphi^{(c)}(z) := \frac{\E_{\lambda^{(c)}}\left[W(X,Z)U(X,Z)\given Z=z\right]}{\E_{\lambda^{(c)}}\left[W(X,Z)\given Z=z\right]}.
    \end{equation*}
    Then
    \begin{align*}
        \E^\star_{\mu^{(c)}}\left[\big(\hat{\phi}_n(X,Z)-\phi_{\lambda^{(c)}}(X,Z)\big)^2\right] &= \E^\star_{\mu^{(c)}}\left[W(X,Z)^2\big( \tilde{\varphi}_n(X,Z) - \varphi^{(c)}(Z) \big)^2\right]
        \\
        & \geq \underline{\gamma}_W^2 \, \E^\star_{\mu^{(c)}}\left[\big(\tilde{\varphi}_n(X,Z)-\varphi^{(c)}(Z)\big)^2\right]
        \\
        & \geq \underline{\gamma}_W^2 \, \E^\star_{\nu}\left[\big(\hat{\varphi}_n^{(c)}(Z)-\varphi^{(c)}(Z)\big)^2\right].
    \end{align*}
    Note the penultimate inequality holds as
    \begin{align*}
        & \E^\star_{\mu^{(c)}}\left[\big(\tilde{\varphi}_n(X,Z)-\varphi^{(c)}(Z)\big)^2\right] - \E^\star_{\nu}\left[\big(\hat{\varphi}_n^{(c)}(Z)-\varphi^{(c)}(Z)\big)^2\right]
        \\
        =\,&
        \E^\star_{\mu^{(c)}}\left[\big(\tilde{\varphi}_n(X,Z)-\hat{\varphi}_n^{(c)}(Z)\big)^2\right] + \underbrace{2\E^\star_{\nu}\left[\E^\star_{\lambda^{(c)}}\big[\tilde{\varphi}_n(X,Z)-\hat{\varphi}_n(Z)\given Z\big]\big(\varphi_n^{(c)}(Z)-\hat{\varphi}^{(c)}(Z)\big)\right]}_{=0}
        \\
        =\,& \E^\star_{\nu}\big[\Var^\star_{\lambda^{(c)}}\big(\tilde{\varphi}_n(X,Z)\given Z\big)\big]
        \\
        \geq\,& 0.
    \end{align*}
    Further
    \begin{align*}
        \E^\star_{\nu}\left[\big(\varphi^{(c)}(Z)-\hat{\varphi}_n^{(c)}(Z)\big)^2\right]
        \geq\,& \E^\star_{\nu}\left[\big(\varphi^{(c)}(Z)-\hat{\varphi}_n^{(c)}(Z)\big)^2 \ind_{S}(Z) \right] \\
        =\,& \int_{S}\big(\varphi^{(c)}(z)-\hat{\varphi}_n^{(c)}(z)\big)^2\nu(dz)
        \\
        =\,& \sum_{k=1}^{m_n} \int_{ S_k}\big(\varphi^{(c)}(z)-\hat{\varphi}_n^{(c)}(z)\big)^2\nu(dz).
    \end{align*}

    Now define $\bar{c}_n\in\{0,1\}^{m_n}$ by
    \begin{equation*}
        \bar{c}_{n,k} = \begin{cases}
            1 & \text{if } \frac{1}{q_k}\int_{ S_k} \big(\hat{\varphi}_n^{(c)}(z)-\iota(z)\big)\nu(dz) \geq \frac{\kappa}{2}
            \\
            0 & \text{otherwise}
        \end{cases}
        \qquad
        (k\in[m_n])
        .
    \end{equation*}
    Then,
    \begin{align*}
        &\E^\star_{\nu}\Big[\big(\varphi^{(c)}(Z)-\hat{\varphi}_n^{(c)}(Z)\big)^2\Big]
        \\
        \geq& \sum_{k=1}^{m_n} \int_{ S_k}\big(\varphi^{(c)}(z)-\hat{\varphi}_n^{(c)}(z)\big)^2\nu(dz),
        \\
        \geq& \sum_{k=1}^{m_n}
        \ind_{(c_k\neq\bar{c}_{n,k})}\int_{ S_k}\big(\varphi^{(c)}(z)-\hat{\varphi}_n^{(c)}(z)\big)^2 \nu(dz)
        \\
        =& \sum_{k=1}^{m_n}
        \ind_{(c_k\neq\bar{c}_{n,k})}\int_{ S_k} \big(\iota(z)-\hat{\varphi}_n^{(c)} + \kappa c_k\big)^2 \nu(dz)
        \\
        =& \sum_{k=1}^{m_n}
        \bigg(\underbrace{\ind_{(c_k\neq\bar{c}_{n,k})}\int_{ S_k} \Big(\iota(z)-\hat{\varphi}_n^{(c)} + \frac{1}{2}\Big)^2 \nu(dz)}_{\geq 0}
        + \ind_{(c_k\neq\bar{c}_{n,k})} \kappa^2 \underbrace{\Big(c_k-\frac{1}{2}\Big)^2}_{=\frac{1}{4}} \int_{ S_k} \nu(dz)
        \\
        &\qquad
        + \underbrace{\ind_{(c_k\neq\bar{c}_{n,k})}\Big(c_k-\frac{1}{2}\Big)\int_{ S_k}\Big(\iota(z)-\hat{\varphi}_n^{(c)}(z)+\frac{1}{2}\Big) \nu(dz)}_{\geq 0}\bigg)
        \\
        \geq&
        \frac{\kappa^2}{4}\sum_{k=1}^{m_n}\ind_{(c_k\neq\bar{c}_{n,k})}\int_{ S_k} \big((1-\delta)p^*(z)+\delta \beta\big)dz
        \\
        \geq&
        \frac{\kappa^2\delta \beta}{4}\sum_{k=1}^{m_n}\ind_{(c_k\neq\bar{c}_{n,k})}\int_{ S_k}dz
        \\
        =& \frac{\kappa^2\delta \beta}{4}\sum_{k=1}^{m_n}q_k \ind_{(c_k\neq \bar{c}_{n,k})},
    \end{align*}
    where in the sixth line we make use of
    \begin{multline*}
        \ind_{(c_k\neq\bar{c}_{n,k})}\Big(c_k-\frac{1}{2}\Big)\int_{ S_k}\Big(\iota(z)-\hat{\varphi}_n^{(c)}(z)+\frac{1}{2}\Big) \nu(dz)
        \\
        =
        \frac{1}{2}\ind_{(c_k\neq\bar{c}_{n,k})}\bigg|\int_{ S_k}\big(\iota(z)-\hat{\varphi}_n^{(c)}(z)\big) \nu(dz) - \frac{q_k}{2}\bigg|
        \geq
        0
        .
    \end{multline*}
    Therefore
    \begin{equation*}
        \E^\star_{\nu}\Big[\big(\varphi^{(c)}(Z)-\hat{\varphi}_n^{(c)}(Z)\big)^2\Big]
        \geq
        \frac{\kappa^2\delta \beta}{4}\sum_{k=1}^{m_n}q_k\ind_{(c_k\neq\bar{c}_{n,k})}
        \geq 
        \frac{\kappa^2\delta \beta}{4}\sum_{k=1}^{m_n}q_k\ind_{(c_k\neq\bar{c}_{n,k}, \Omega_{n,k})},
    \end{equation*}
    for the event $\Omega_{n,k} := \{Z_i\not\in  S_k \,(\forall i\in[n])\}$, and so
    \begin{equation*}
        \E_{\mu_{n}, \nu}\left[\big(\varphi^{(c)}(Z)-\hat{\varphi}_n^{(c)}(Z)\big)^2\right] \geq \frac{\kappa^2\delta \beta}{4}\sum_{k=1}^{m_n} q_k\PP_{\mu_n,\nu}\left(c_k\neq\bar{c}_{n,k}, \, \Omega_{n,k}\right).
    \end{equation*}
    Now for each $n\in\mathbb{N}$ define the function $R_n:\{0,1\}^{m_n}\to\R_{\geq0}$ by
    \begin{equation*}
        R_n(c) := \sum_{k=1}^{m_n} q_k \PP_{\mu_n,\nu}\left(c_k\neq\bar{c}_{n,k}, \, \Omega_{n,k} \right).
    \end{equation*}
    Note therefore that
    \begin{equation*}
        R_n(c) \leq \sum_{k=1}^{m_n} q_k \PP_{\mu_n}(\Omega_{n,k}) = \sum_{k=1}^{m_n} q_k\Big(1-\frac{2 q_k}{{v}}\Big)^n.
    \end{equation*}
    Let $(C_k)_{k\in[m]}$ be a sequence of i.i.d.~Bernoulli random variables independent of $D_n$ with $\PP(C_1=1)=\PP(C_1=0)=\frac{1}{2}$. Then
        \begin{align*}
        \E\left[R_n(C)\right]
        & = \sum_{k=1}^{m_n} q_k \E\left[\PP\left(C_k\neq\bar{C}_{n,k}, \Omega_{n,k} \given Z_1,\ldots,Z_n\right)\right]
        \\
        & = \sum_{k=1}^{m_n} q_k \E\left[\ind_{\Omega_{n,k}}\PP\left(C_k\neq \bar{C}_{n,k} \given Z_1,\ldots,Z_n\right)\right]
        \\
        & = \frac{1}{2}\sum_{k=1}^{m_n} q_k\PP\left(\Omega_{n,k}\right),
        \\
        & = \frac{1}{2} \sum_{k=1}^{m_n} q_k\Big(1-\frac{2 q_k}{{v}}\Big)^n,
    \end{align*}
    noting that $C_k\,\independent\,\bar{C}_{n,k}\given Z_1,\ldots,Z_n$ on the event $\Omega_{n,k}$ for all $k\in[m_n]$. Thus
    \begin{equation*}
        \frac{R_n(c)}{\E\left[R_n(C)\right]} \leq 2.
    \end{equation*}
    We can therefore apply Fatou's Lemma; 
    \begin{equation*}
        \E\left[\limsup_{n\to\infty}\frac{R_n(C)}{\E\left[R_n(C)\right]}\right] \geq \limsup_{n\to\infty}\E\left[\frac{R_n(C)}{\E\left[R_n(C)\right]}\right] = 1.
    \end{equation*}
    Recalling that $R_n$ is a non-negative function, there exists some $c\in\{0,1\}^{\mathbb{N}}$ such that
    \begin{equation*}
        \limsup_{n\to\infty}\frac{R_n(c)}{\E\left[R_n(C)\right]} \geq 1.
    \end{equation*}
    Thus for any subsequence $(m_n)_{n\in\mathbb{N}}$ of the natural numbers
    \begin{multline*}
        \limsup_{n\to\infty}\sup_{c\in\{0,1\}^{\mathbb{N}}}\frac{\E_{\mu^{(m_n,c,q)}}\Big[\big(\hat{\phi}_n(X,Z)-\phi_{\mu^{(m_n,c,q)}}(X,Z)\big)^2\Big]}{\tau\sum_{k=1}^{m_n}q_{m_n,k}(1-2{v}^{-1}q_{m_n,k})^n} 
        \\
        \geq
        \sup_{c\in\{0,1\}^{\mathbb{N}}}\limsup_{n\to\infty}\frac{\E_{\mu^{(m_n,c,q)}}\Big[\big(\hat{\phi}_n(X,Z)-\phi_{\mu^{(m_n,c,q)}}(X,Z)\big)^2\Big]}{\tau\sum_{k=1}^{m_n}q_{m_n,k}(1-2{v}^{-1}q_{m_n,k})^n} \geq 1.
    \end{multline*}
    Alongside~\eqref{eq:q} i.e.
    \begin{equation*}
        \sup_{m\in\mathbb{N}}\frac{\tau\sum_{k=1}^m q_{m,k}(1-2{v}^{-1}q_{m,k})^n}{a_n}\geq1,
    \end{equation*}
    for all $n\in\mathbb{N}$, and applying Lemma~\ref{lem:limsup} completes the proof;
    \begin{align*}
        &\limsup_{n\to\infty}\sup_{P\in\cP_{XZ}: \text{TV}(P,P^*)\leq\alpha}\frac{\E_P\Big[\big(\hat{\phi}_n(X,Z)-\phi_P(X,Z)\big)^2\Big]}{a_n}
        \\
        \geq\,&
        \limsup_{n\to\infty}\sup_{m\in\mathbb{N}}\sup_{c\in\{0,1\}^\mathbb{N}}\frac{\E_{\mu^{(m,c,q)}}\Big[\big(\hat{\phi}_n(X,Z)-\phi_{\mu^{(m,c,q)}}(X,Z)\big)^2\Big]}{\tau\sum_{k=1}^m q_{m,k}(1-2{v}^{-1}q_{m,k})^n}
        \cdot
        \frac{\tau\sum_{k=1}^m q_{m,k}(1-2{v}^{-1}q_{m,k})^n}{a_n}
        \\
        \geq\,&1.
    \end{align*}

\end{proof}

\subsection{Additional lemmas for the proof of Theorem~\ref{thm:slow-rates}}

\begin{lemma}\label{lem:theta+c}
    Adopt the setup of Theorem~\ref{thm:slow-rates} where additionally we assume the functions $M_1,\ldots,M_J$ and $\cdot\mapsto1$ are linearly independent. Given any $\zeta>0$ there exists a conditional density $p^{\dagger}(x\given z)$ on $X\given Z$, a non-empty open ball $B\subset\cZ$, and a constant $\kappa>0$, such that the following hold:
    \begin{enumerate}[label=(\roman*)]
    \item\label{lemitem:kappa} \begin{equation*}\frac{\int_\cX p^{\dagger}(x\given z)W(x,z)U(x,z)dx}{\int_\cX p^{\dagger}(x\given z)W(x,z)dx}-\frac{\int_\cX p^*(x\given z)W(x,z)U(x,z)dx}{\int_\cX p^*(x\given z)W(x,z)dx}=\kappa,\end{equation*}
    for all $z\in B$.
    \item\label{lemitem:m_j} $\int_\cX M_j(x)p^{\dagger}(x\given z)dx=\int_\cX M_j(x)p^*(x\given z)dx$ for all $z\in\cZ, j\in[J]$.
    \item\label{lemitem:cts}  $0<\inf_{x \in \mathcal{X}, z \in \mathcal{Z}}p^{\dagger}(x|z) \leq \sup_{x \in \mathcal{X}, z \in \mathcal{Z}}p^{\dagger}(x|z) < \infty$ and the functions $\{z \mapsto p^{\dagger}(x | z) : x \in \mathcal{X}\}$ restricted to the domain $B$ are equicontinuous.
    \item\label{lemitem:zeta} $\int_B\int_{\cX}|p^{\dagger}(x\given z)-p^*(x\given z)|dxdz\leq\zeta$.
    \end{enumerate}
\end{lemma}

\begin{proof}
    First note that if
    \begin{equation*}
        \phi_{P^*}\not\in\bigg\{(x,z)\mapsto \sum_{j=1}^J w_j(z)\big(M_j(x)-m_j(z)\big) : w_j:\cZ\to\R \bigg\},
    \end{equation*}
    then there exists some $z_0\in\cZ$ such that
    \begin{equation}\label{eq:w-form}
        \phi_{P^*,0} := \phi_{P^*}(\cdot,z_0)\not\in\bigg\{x\mapsto \sum_{j=1}^J w_j(z_0)\big(M_j(x)-m_j(z_0)\big) : w_j(z_0)\in\R \bigg\}.
    \end{equation}
    Let us define the functions $r:\cX\times\cZ\to\R$ and $r_0, W_0:\cX\to\R$ by
    \begin{gather*}
        r(x,z) := W(x,z)U(x,z)
        ,\quad
        r_0(x) := r(x,z_0),
        \quad
        \text{and}
        \quad
        W_0(x) := W(x,z_0).
    \end{gather*}
    Given the distribution $P^*$, let $p^*(x\given z)$ denote the corresponding conditional density for $X\given Z$ and $p^*(z)$ denote the corresponding marginal density for $Z$. 
    We also define the following vector spaces of functions on $\cX$:
    \begin{gather}
        \cG := \bigg\{x\mapsto \beta_{-1}W_0(x) + \beta_0 + \sum_{j=1}^J \beta_j\big(M_j(x)-m_j(z_0)\big) 
        \,:\,
        (\beta_{-1},\beta_0,\ldots,\beta_J)\in\R^{J+2}\bigg\},
        \label{eq:cG}
        \\
        \cG_{-} := \bigg\{x\mapsto \beta_{-1}W_0(x) + \sum_{j=1}^J \beta_j\big(M_j(x)-m_j(z_0)\big)
        \,:\,
        (\beta_{-1},\beta_1,\ldots,\beta_J)\in\R^{J+1}\bigg\}
        \label{eq:cG-}
        \\
        \cG_{0} := \bigg\{x\mapsto \beta_0 + \sum_{j=1}^J \beta_j\big(M_j(x)-m_j(z_0)\big)
        \,:\,
        (\beta_0,\beta_1,\ldots,\beta_J)\in\R^{J+1}\bigg\}
        \label{eq:cG0}
    \end{gather}

    The proof of the lemma consists of a number of parts that we briefly outline now:
    \begin{itemize}[leftmargin=*,labelindent=3em]
        \item[\bf Part 1:] We show that~\eqref{eq:w-form} implies either: Case 1: $r_0\notin\cG$; or Case 2: $r_0\in\cG\backslash\cG_{-}$ and $W_0\notin\mathcal{G}_0$.
        \item[\bf Part 2:] In each case we construct some set of $J+2$ points $x_1,\ldots,x_{J+2}\in\cX$ such that certain properties hold (loosely speaking these properties encode information about each of the above cases).
        \item[\bf Part 3:] We construct a class of densities $p_\eta^*(x\given z)$  (parametrised by a scalar $\eta\in\Lambda\subset\R$ where $0\in\interior\Lambda$) that perturbs $p_0^*(x\given z) = p^*(x\given z)$ within sufficiently small $\epsilon$-balls around the points $x_1,\ldots,x_{J+2}$ so as to preserve the $J$ conditional expectation conditions required, whilst also enforcing that for a sufficiently small ball $B\subset\cZ$
        \begin{equation*}
            \inf_{z\in B}\sup_{\eta\in\Lambda} \Bigg(\frac{\int_\cX p_\eta^*(x\given z)r(x,z)dx}{\int_\cX p_\eta^*(x\given z)W(x,z)dx}-\frac{\int_\cX p^*(x\given z)r(x,z)dx}{\int_\cX p^*(x\given z)W(x,z)dx}\Bigg)
        \end{equation*}
        is bounded away from zero.
        \item[\bf Part 4:] Using the result of Part 3, the function
        \begin{equation*}
            f(z,\eta) := \frac{\int_\cX p_\eta^*(x\given z)r(x,z)dx}{\int_\cX p_\eta^*(x\given z)W(x,z)dx}-\frac{\int_\cX p^*(x\given z)r(x,z)dx}{\int_\cX p^*(x\given z)W(x,z)dx},
        \end{equation*}
        is by construction continuous in $z$, and further we show there exists some $\kappa>0$ and $\bar{\eta}:B\to\Lambda$ with $f(z,\bar{\eta}(z))=\kappa$ for all $z\in\cZ$. Defining $p^\dagger(x\given z) = p^*_{\bar{\eta}(z)}(x\given z)$ on $(x,z)\in \cX\times B$ it then follows for all $z\in B$,
        \begin{equation*}
            \frac{\int_\cX p^{\dagger}(x\given z)r(x,z)dx}{\int_\cX p^{\dagger}(x\given z)W(x,z)dx}-\frac{\int_\cX p^*(x\given z)r(x,z)dx}{\int_\cX p^*(x\given z)W(x,z)dx} = \kappa.
        \end{equation*}
    \end{itemize}
    \noindent
    By virtue of the perturbed density constructed in Part 3, the conditions (iii) and (iv) will be shown to hold for $p^\dagger$. We proceed by showing each of the above parts in turn.

    \medskip
    \noindent{\bf Part 1:}
    We first note that~\eqref{eq:w-form} implies either (Case 1): $r_0\notin\cG$ or (Case 2) $r_0\in\cG\backslash\cG_{-}$ and $W_0\notin\cG_0$. This follows by a simple contrapositive argument; we must show that in the two cases: (A) 
    $r_0\in\cG_{-}$, or (B) $r_0\in\cG$ and $W_0\in\cG_0$, that~\eqref{eq:w-form} is violated. In the former setting if (A) holds then there exist some $\beta^{(A)}_{-1},\beta^{(A)}_1,\ldots,\beta^{(A)}_J\in\R$ such that for all $x\in\cX$,
    \begin{align*}
        r_0(x) = \beta^{(A)}_{-1} W_0(x) + \sum_{j=1}^J\beta^{(A)}_j\big(M_j(x)-m_j(z_0)\big).
    \end{align*}
    Then for all $x\in\cX$,
    \begin{align*}
        \phi_{P^*,0}(x) &= r_0(x) - \frac{W_0(x)\E_{P^*}[r_0(X)\given Z=z_0]}{\E_{P^*}[W_0(X)\given Z=z_0]}
        \\
        &= \beta^{(A)}_{-1}W_0(x) + \sum_{j=1}^J\beta^{(A)}_j\big(M_j(x)-m_j(z_0)\big) - \beta^{(A)}_{-1}W_0(x)
        \\
        &= \sum_{j=1}^J\beta^{(A)}_j\big(M_j(x)-m_j(z_0)\big),
    \end{align*}
    contradicting~\eqref{eq:w-form}.  
    In the latter setting of (B), if $r_0\in\cG$ and $W_0\in\cG_0$ then there exist some $\beta^{(B)}_{-1},\beta^{(B)}_0,\beta^{(B)}_1,\ldots,\beta^{(B)}_J,\gamma^{(B)}_0,\ldots,\gamma^{(B)}_J\in\R$ such that for all $x\in\cX$,
    \begin{gather*}
        r_0(x) = \beta^{(B)}_{-1} W_0(x) + \beta^{(B)}_0 + \sum_{j=1}^J\beta^{(B)}_j\big(M_j(x)-m_j(z_0)\big),
        \\
        W_0(x) = \gamma^{(B)}_0 + \sum_{j=1}^J\gamma^{(B)}\big(M_j(x)-m_j(z_0)\big).
    \end{gather*}
    Then for all $x\in\cX$,
    \begin{align*}
        \phi_{P^*,0}(x) &= r_0(x) - \frac{W_0(x)\E_{P^*}[r_0(X)\given Z=z_0]}{\E_{P^*}[W_0(X)\given Z=z_0]}
        \\
        &= \beta^{(B)}_{-1}W_0(x) + \beta^{(B)}_0 + \sum_{j=1}^J\beta^{(B)}_j\big(M_j(x)-m_j(z_0)\big)
        \\
        &\qquad -\beta^{(B)}_{-1}W_0(x) - \frac{W_0(x)}{\E_{P^*}[W_0(X)\given Z=z_0]}
        \beta^{(B)}_0
        \\
        &=\sum_{j=1}^J\beta^{(B)}_j\big(M_j(x)-m_j(z_0)\big) - \frac{\beta^{(B)}_0}{\gamma^{(B)}_0}\big(W_0(x)-\E_{P^*}[W_0(X)\given Z=z_0]\big)
        \\
        &=\sum_{j=1}^J\alpha_j\big(M_j(x)-m_j(z_0)\big),
    \end{align*}
    where
    \begin{equation*}
        \alpha_j := \beta^{(B)}_j - \frac{\beta^{(B)}_0\gamma^{(B)}_j}{\gamma^{(B)}_0}
        \quad
        (\forall j\in[J]),
    \end{equation*}
    which contradicts~\eqref{eq:w-form}.

    Therefore, taking the contrapositive, under~\eqref{eq:w-form} we have either: (Case 1) $r_0\notin\cG$ or (Case 2) $r_0\in\cG\backslash\cG_{-}$ and $W_0\notin\cG_0$. At certain points in the remainder of the proof we will consider these two cases separately.

    \medskip
    \noindent{\bf Part 2:}
    \noindent 
    By Lemma~\ref{lem:lin-indep} there then exist $J+1$ distinct points $x_1,\ldots,x_{J+1}\in\cX$ such that
    \begin{equation} \label{eq:M_sub_def}
    \mb M^{(0)}_{\text{sub}} := \begin{pmatrix}
    M_1(x_1) & \cdots & M_1(x_{J+1}) \\
    \vdots & \ddots & \vdots \\
    M_{J}(x_1) & \cdots & M_{J}(x_{J+1}) \\
    1 & \cdots & 1
    \end{pmatrix} \in \R^{(J+1) \times (J+1)}
    \end{equation}
   is invertible.
  
    We proceed by constructing an additional point $x_{J+2}\in\cX\backslash\{x_1,\ldots,x_{J+1}\}$ depending on the two cases above.

    \smallskip
    \noindent{\bf \emph{Case 1:} $r_0\notin\cG$.}

    \smallskip
    \noindent
    As $\mb M^{(0)}_{\text{sub}}$ \eqref{eq:M_sub_def}  is invertible, there exists a unique $\bar{r}_0\in\cG_{0}$ that interpolates $\{(x_1,r_0(x_1)),\ldots,(x_{J+1},r_0(x_{J+1}))\}$ i.e.~$\bar{r}_0(x_i)=r_0(x_i)$ for each $i\in[J+1]$; also define $\bar{\beta}$ to be the $\beta$ in~\eqref{eq:cG0} that gives rise to this $\bar{r}_0$, and define $\hr_0:=r_0-\bar{r}_0$.
    
    Now define the function 
    \begin{equation}\label{eq:Delta0'}
        \Delta_0'(x) := a_0 \bar{r}_0^\perp(x) - (\bar{\beta}_0 + c_0) W_0(x)
        +
        (\bar{\beta}_0 + c_0) \sum_{j=1}^{J+1} \bigg(\sum_{i=1}^{J+1}W_0(x_i)\big\{(\mathbf{M}^{(0)}_{\text{sub}})^{-1}\big\}_{ij}\bigg)M_j(x)
        ,
    \end{equation}
    where
    \begin{equation*}
        c_0
        :=
        \int_\cX p_0(x\given z_0) \bar{r}_0^\perp(x) dx.
    \end{equation*}
    Then there exists some $x_{J+2}\in\cX\backslash\{x_1,\ldots,x_{J+1}\}$ with $\Delta_0'(x_{J+2})\neq0$; else $\Delta_0'\equiv0$ and so by~\eqref{eq:Delta0'} it follows that $\bar{r}_0^{\perp}\in\cG$ and hence $r_0=\bar{r}_0+\bar{r}_0^{\perp}\in\cG$ contradicting Case~1.

    \smallskip
    \noindent{\bf \emph{Case 2:} $r_0\in\cG\backslash\cG_{-}$ and $W_0\notin\cG_0$.}

    \smallskip
    \noindent
    As $r_0\in\cG\backslash\cG_{-}$ there exist some $\tilde{\beta}_{-1},\tilde{\beta}_0,\tilde{\beta}_1,\ldots,\tilde{\beta}_j\in\R$, with $\tilde{\beta}_0\neq0$, such that for all $x\in\cX$
\begin{equation*}
    r_0(x) = \tilde{\beta}_{-1}W_0(x) + \tilde{\beta}_0 + \sum_{j=1}^J\tilde{\beta}_j\big(M_j(x)-m_j(z_0)\big).
\end{equation*}
Further, there exists a unique $\bar{W}_0\in\cG_0$ that interpolates $\{(x_1,W_0(x_1)),\ldots, (x_{J+1},W_0(x_{J+1}))\}$ (as $\mb M^{(0)}_{\text{sub}}$ \eqref{eq:M_sub_def}  is invertible). 
Let $\bar{W}^\perp_0 := W_0-\bar{W}_0$. 
    As $W_0\notin\cG_0$ it follows that $\bar{W}^\perp_0\not\equiv0$ i.e.~there exists some $x_{J+2}\in\cX\backslash\{x_1,\ldots,x_{J+1}\}$ where $\bar{W}^\perp_0(x_{J+2})\neq0$. 

    \medskip
    \noindent{\bf Part 3:} 
    We will construct a `perturbed density function' $p_\eta(x\given z)$ as follows. 
    Consider the equation 
    \begin{equation}\label{eq:eta-z-s}
        {\bf M}(\epsilon) {\text{{\boldmath$\eta$}}}(\epsilon)  = {\bf 0},
    \end{equation}
    where the matrix $\mathbf{M}(\epsilon)\in\R^{(J+1)\times(J+2)}$ is given for $\epsilon>0$ by
    \begin{equation*}
        ({\bf M}(\epsilon))_{ji} := \begin{cases}
            V_i(\epsilon)^{-1} \int_{\mathcal{B}_{\epsilon}(x_i)} k_i(x) M_j(x) dx & \text{if } j\leq J
            \\
            1 &\text{if } j = J+1
        \end{cases},
        \quad
        k_i(x) := e^{-\frac{\epsilon^2}{\epsilon^2-\|x-x_i\|_2^2}}\ind_{\mathcal{B}_{\epsilon}(x_i)}(x),
    \end{equation*}
    where $\mathcal{B}_{\epsilon}(x_i):=\{x'\in\cX : \|x'-x_i\|<\epsilon \}$ and $V_i(\epsilon) := \int_{\mathcal{B}_{\epsilon}(x_i)} k_i(x) dx$, and for $\epsilon=0$,
    \begin{equation*}
        (\mathbf{M}(0))_{ji} := \begin{cases}
            M_j(x_i) &\text{if }j\leq J
            \\
            1 &\text{if }j=J+1
        \end{cases}.
    \end{equation*}
    
    Let $\mathbf{M}_{\text{sub}}(\epsilon) \in \R^{(J+1) \times (J+1)}$ be the matrix consisting of the first $J+1$ columns of $\mathbf{M}(\epsilon)$, and let $\mathbf{m}(\epsilon) \in \R^{J+1}$ be the vector of its $(J+2)$th column. Note that as $\mathbf{M}_{\text{sub}}(\epsilon)$ is continuous, and $\mathbf{M}_{\text{sub}}(0) = \mathbf{M}_{\text{sub}}^{(0)}$ which is invertible, there exists some $\epsilon_1>0$ such that for all $\epsilon\in[0,\epsilon_1]$ the matrix $\mathbf{M}_{\text{sub}}(\epsilon)$ is invertible. Also define $\mathbf{d}:[0,\epsilon_1]\to\R^{J+2}$ by 
    \begin{equation*}
        \mb d(\epsilon) := \begin{pmatrix}
            -\mathbf{M}_{\text{sub}}(\epsilon)^{-1}\mathbf{m}(\epsilon) \\
            1
        \end{pmatrix} \in \R^{J+2}.
    \end{equation*}
    Then for each $i\in[J+2]$ the function $d_i$ is continuous and for any $\eta\in\R$ the vector ${\text{{\boldmath$\eta$}}}(\epsilon):=\eta\,\mathbf{d}(\epsilon)$ satisfies the matrix equation~\eqref{eq:eta-z-s}.
    
    Define $\epsilon_3 := 3^{-1}\min_{i,i'\in[J+2],i\neq i'}\|x_i-x_{i'}\|$, in addition to taking $\epsilon_2>0$ as defined by Lemma~\ref{lem:holder-ctrl}, and $\epsilon_4 := \sup\{\epsilon>0:\mathcal{B}_{\epsilon}(x_i)\subseteq\cX\;(\forall i\in[J+2])\}$, and then take $\epsilon^* := \epsilon_1\wedge\epsilon_2\wedge\epsilon_3\wedge\epsilon_4$. Then note that $V_i(\epsilon^*)$ carries no $i$ dependence; we omit this dependence from hereon. Define $\bar{d}:=\|\mathbf{d}(0)\|_\infty\vee\|\mathbf{d}(\epsilon^*)\|_\infty < \infty$. For ${\text{{\boldmath$\eta$}}}(\epsilon^*) := \eta \, \mathbf{d}(\epsilon^*)$ it then follows that $\|{\text{{\boldmath$\eta$}}}(\epsilon^*)\|_\infty \leq |\eta|\bar{d}$. Now take an open ball $B'\subseteq\cZ$ centered at $z_0$.
    
    Now define
    \begin{equation*}
        p_\eta^*(x\given z) := \begin{cases}
            p^*(x\given z) + \eta\, V(\epsilon^*)^{-1}\sum_{i=1}^{J+2}d_i(\epsilon^*)k_i(x)\ind_{\mathcal{B}_{\epsilon^*}(x_i)}(x) &\quad\text{if }z\in B'
            \\
            p^*(x\given z) &\quad\text{if } z\in \cZ\backslash B'
        \end{cases}.
    \end{equation*}
    By construction, the equation~\eqref{eq:eta-z-s} holds for ${\text{{\boldmath$\eta$}}}=\eta\mathbf{d}(\epsilon)$, and thus
    \begin{equation*}
        \int_\cX p_{\eta}^*(x\given z) dx = 1
        \,,\quad
        \int_\cX p_{\eta}^*(x\given z)M_j(x)dx = m_j(z)
        \quad
        (j\in[J]).
    \end{equation*}
    Further, as the conditional density $p^*$ is bounded away from zero, there exists a positive constant $g>0$ such that $\inf_{(x,z)\in\cX\times B'} p^*(x\given z)\geq g$. Thus for any $|\eta|<\bar{d}^{-1}g$ the function $p_\eta^*$ is strictly positive and thus corresponds to a density satisfying condition (ii). From hereon we therefore restrict $\eta\in\Lambda:=(-\bar{d}^{-1}g,\bar{d}^{-1}g)$.
    
    Now let
        \begin{align*}
        D(z) &:= \int_\cX \frac{\{p_\eta^*-p^*\}(x\given z)}{\eta}\Delta(x,z)dx
        =
        \sum_{i=1}^{J+2}d_i(\epsilon^*)V(\epsilon^*)^{-1}\int_{\mathcal{B}_{{\epsilon^*}(x_i)}}k_i(x)\Delta(x,z)dx,
        \\
        E(z) &:= \int_\cX \frac{\{p_\eta^*-p^*\}(x\given z)}{\eta}W(x,z)dx
        =
        \sum_{i=1}^{J+2}d_i(\epsilon^*)V(\epsilon^*)^{-1}\int_{\mathcal{B}_{{\epsilon^*}(x_i)}}k_i(x)W(x,z)dx,
        \\
        a(z) &:= \int_\cX p^*(x\given z)W(x,z)dx, \\
        b(z) &:= \int_\cX p^*(x\given z)r(x,z)dx, \\
        \Delta(x,z) &:= a(z)r(x,z) - b(z)W(x,z).
    \end{align*}
    Note in particular the quantities $a(z), D(z), E(z)$ carry no $\eta$ dependence. Also note that by Lemma~\ref{lem:equicontinuity} both $a$ and $b$ are continuous, and thus $\Delta$ is equicontinuous and, as $D(z)$ and $E(z)$ are expressible as integrals of equicontinuous functions over a compact domain, both $D$ and $E$ are continuous. Then for each $\eta\in\Lambda$ and $z\in B'$,
    \begin{align*}
        & \quad\;
        \frac{\int_\cX p_\eta^*(x\given z)r(x,z)dx}{\int_\cX p_\eta^*(x\given z)W(x,z) dx} - \frac{\int_\cX p^*(x\given z)r(x,z)dx}{\int_\cX p^*(x\given z)W(x,z) dx}
        \\
        &= 
        \frac{\eta\int_{\cX}\frac{\{p_\eta^*-p^*\}(x\given z)}{\eta}\big(a(z)r(x,z)-b(z)W(x,z)\big)dx}{a(z)\Big(a(z) + \eta\int_{\cX}\frac{\{p_\eta^*-p^*\}(x\given z)}{\eta}W(x,z)dx\Big)}
        \\
        &= \frac{\eta D(z)}{a(z) \big( a(z) + \eta E(z) \big)},
    \end{align*}
    and is zero when $z\in\cZ\backslash B'$.

    We now show that
    \begin{equation*}
        |D(z_0)| \geq 2^{-1}\omega,
        \qquad
        \omega := \begin{cases}
            |\Delta_0'(x_{J+2})| &\text{in Case 1}
            \\
            |\tilde{\beta}_0\bar{W}_0^\perp(x_{J+2})|&\text{in Case 2}
        \end{cases}.
    \end{equation*}

    \smallskip\noindent
{\bf \emph{Case 1:} $r_0\notin\cG$}.
\noindent
Recalling the notation from part 2, case 1,
\begin{align*}
    r_0(x) &= \bar{r}_0(x) + \bar{r}^\perp_0(x)
    \\
    &= \bar{\beta}_0 + \sum_{j=1}^J\bar{\beta}_j\big(M_j(x)-m_j(z_0)\big) + \bar{r}_0^\perp(x).
\end{align*}
Then by straightforward calculations
\begin{equation*}
    b_0 := b(z_0) =  \int_\cX p^*(x\given z_0)r_0(x) dx = \bar{\beta}_0 + c_0.
\end{equation*}
Thus writing $a_0 := a(z_0)$,
\begin{align*}
    D(z_0) &=
    \int_\cX\frac{\{p_\eta^*-p^*\}(x\given z_0)}{\eta}\big(a_0r_0(x)-b_0W_0(x)\big)dx
    \\
    &=
    \int_\cX\frac{\{p_\eta^*-p^*\}(x\given z_0)}{\eta} \big( a_0\bar{r}_0^\perp(x) - (\bar{\beta}_0+c_0)W_0(x) \big) dx
    \\
    &= \sum_{i=1}^{J+2}d_i(\epsilon^*)V(\epsilon^*)^{-1}\int_{\ball{x_i}}k_i(x)\big(a_0\bar{r}_0^\perp(x)-(\bar{\beta}_0+c_0)W_0(x)\big)dx.
\end{align*}
Defining $\Delta_0''(x):=a_0\bar{r}_0^\perp(x) - (\bar{\beta}_0+c_0)W_0(x)$, it follows that
\begin{align*}
    |D(z_0)| &=
    \left|\sum_{i=1}^{J+2}d_i(\epsilon^*)V(\epsilon)^{-1}\int_{\ball{x_i}}k_i(x)\Delta_0''(x)dx\right|
    \\
    &=
    \bigg|
    \sum_{i=1}^{J+2}d_i(0)\Delta_0''(x_i)
    +
    \sum_{i=1}^{J+2}\Big[d_i(\epsilon^*)V(\epsilon^*)^{-1}\int_{\ball{x_i}}k_i(x)\Delta_0''(x)dx - d_i(0)\Delta_0''(x_i)\Big]
    \bigg|
    \\
    &\geq 
    \underbrace{\bigg|
    \sum_{i=1}^{J+2}d_i(0)\Delta_0''(x_i)\bigg|}_{=|\Delta_0'(x_{J+2})|=\omega}
    - 
    \underbrace{\bigg|\sum_{i=1}^{J+2}\Big[d_i(\epsilon^*)V(\epsilon^*)^{-1}\int_{\ball{x_i}}k_i(x)\Delta_0''(x)dx - d_i(0)\Delta_0''(x_i)\Big]
    \bigg|}_{\leq2^{-1}\omega \text{ by Lemma~\ref{lem:holder-ctrl}}}
    \\
    &\geq 2^{-1}\omega,
\end{align*}
where in the penultimate line (and recalling notation from part 2) we make use of
\begin{align*}
    \sum_{i=1}^{J+2}d_i(0)\Delta_0''(x_i) 
    &=
    \begin{pmatrix}
        \Delta_0''(x_1) \\ \vdots \\ \Delta_0''(x_{J+2})
    \end{pmatrix}^\top 
    \begin{pmatrix}
        -(\mathbf{M}_{\text{sub}}^{(0)})^{-1} \mathbf{m}(0) \\ 1
    \end{pmatrix}
    \\
    &= 
    -(\bar{\beta}_0+c_0)\begin{pmatrix}
        W_0(x_1) \\ \vdots \\ W_0(x_{J+1})
    \end{pmatrix}^\top  (\mathbf{M}_{\text{sub}}^{(0)})^{-1}\mathbf{m}(0)
    + \Delta_0''(x_{J+2})
    \\
    &= \Delta_0'(x_{J+2}),
\end{align*}
recalling the definition~\eqref{eq:Delta0'}, and noting that for each $i\in[J+1]$ by construction $\bar{r}_0^\perp(x_i)=0$ and so $\Delta_0''(x_i) = -(\bar{\beta}_0+c_0)W_0(x_i)$.

\smallskip
\noindent{\bf \emph{Case 2:} }$r_0\in\cG\backslash\cG_{-}$ and $W_0\notin\cG_0$.

Recall that in this case
\begin{equation*}
    r_0(x) = \tilde{\beta}_{-1}W_0(x) + \tilde{\beta}_0 + \sum_{j=1}^J\tilde{\beta}_j\big(M_j(x)-m_j(z_0)\big).
\end{equation*}
Therefore
\begin{equation*}
    b_0
    =
    \int_\cX p^*(x\given z_0) r_0(x) dx
    =
    \tilde{\beta}_{-1} \int_\cX p^*(x\given z_0) W_0(x) dx + \tilde{\beta}_0
    =
    a_0 \tilde{\beta}_{-1} + \tilde{\beta}_0,
\end{equation*}
and thus
\begin{align*}
    \Delta(x,z_0) &= a_0 r_0(x) - b_0 W_0(x)
    \\
    &=
    a_0 \tilde{\beta}_0 - \tilde{\beta}_0W_0(x) + a_0 \sum_{j=1}^J \tilde{\beta}_j \big(M_j(x)-m_j(z_0)\big)
\end{align*}
thus
\begin{align*}
    |D(z_0)| &= 
    \left|\int_\cX \frac{\{p_\eta^*-p^*\}(x\given z_0)}{\eta}\Delta(x,z_0) dx\right|
    \\
    &= 
    \left|\int_\cX \frac{\{p_\eta^*-p^*\}(x\given z_0)}{\eta}\tilde{\beta}_0W_0(x) dx\right|
    \\
    &=
    \bigg|\tilde{\beta}_0\underbrace{\int_\cX \frac{\{p_\eta^*-p^*\}(x\given z_0)}{\eta}\bar{W}_0(x) dx}_{=0 \text{ as }\bar{W}_0\in\cG_0}   +   \tilde{\beta}_0\int_\cX \frac{\{p_\eta^*-p^*\}(x\given z_0)}{\eta}\bar{W}_0^\perp(x) dx\bigg|
    \\
    &=
    \Bigg|\tilde{\beta}_0\sum_{i=1}^{J+2}d_i(0)\bar{W}_0^\perp(x_i)    
    \\
    &\qquad 
    +    \sum_{i=1}^{J+2}\tilde{\beta}_0\bigg[d_i(\epsilon^*)V(\epsilon^*)^{-1}\int_{\ball{x_i}}k_i(x)\bar{W}_0^\perp(z)dx - d_i(0)\bar{W}_0^\perp(x_i)\bigg]\Bigg|
    \\
    &\geq
    \Bigg|\tilde{\beta}_0\sum_{i=1}^{J+2}d_i(0)\underbrace{\bar{W}_0^\perp(x_i)}_{=0\text{ for }i\in[J+1]}\Bigg|  
    \\
    &\qquad -    \underbrace{\Bigg|\tilde{\beta}_0\sum_{i=1}^{J+2}\bigg[d_i(\epsilon^*)V(\epsilon^*)^{-1}\int_{\ball{x_i}}k_i(x)\bar{W}_0^\perp(z)dx - d_i(0)\bar{W}_0^\perp(x_i)\bigg]\Bigg|}_{\leq2^{-1}\omega\text{ by Lemma~\ref{lem:holder-ctrl}}}
    \\
    &\geq
    \big|\tilde{\beta}_0\bar{W}_0^\perp(x_{J+2})| - 2^{-1}\omega
    \\
    &=2^{-1}\omega,
\end{align*}
noting that $d_{J+2}(0)=1$. 

Also note that in both cases
\begin{equation*}
    a(z)\big(a(z)+\eta E(z)\big) = \bigg(\int_\cX p^*(x\given z)W(x,z)dx\bigg)\bigg(\int_\cX p_\eta^*(x\given z)W(x,z)dx\bigg) \in \big[\underline{\gamma}_W^2, \bar{\gamma}_W^2\big].
\end{equation*}

\medskip

\noindent{\bf Part 4:}

\noindent
Define the function $f:B'\times\Lambda\to\R$ by
\begin{align*}
    f(z,\eta) :=&\, \frac{\int_\cX p_\eta^*(x\given z)r(x,z)dx}{\int_\cX p_\eta^*(x\given z)W(x,z)dx}-\frac{\int_\cX p^*(x\given z)r(x,z)dx}{\int_\cX p^*(x\given z)W(x,z)dx}
    \\
    =&\,
    \frac{\eta D(z)}{a(z)\big(a(z)+\eta E(z)\big)}.
\end{align*}
Then we have shown in part 3 that
\begin{equation*}
    \sup_{\eta\in\Lambda}f(z_0,\eta) \geq 1.01\kappa,
    \qquad
    \kappa := \frac{\omega g}{2.02\bar{d}\bar{\gamma}_W^2} > 0
\end{equation*}
Also note that $\eta\mapsto f(z_0,\eta)$ is continuous with $f(z_0,0)=0$, and so by the intermediate value theorem there exists some $\bar{\eta}(z_0)\in\interior\Lambda$ such that $f(z_0,\bar{\eta}(z_0))=\kappa$. 
Recall as above $a$, $D$ and $E$ are all continuous functions. Therefore, there exists an open ball $B''\subseteq B'$ centered at $z_0$ on which $\frac{\kappa a^2(z)}{\kappa a(z)E(z)-D(z)}\in\Lambda$. Further, define the open ball $B'''$ to be the open ball centered at $z_0$ of volume $\frac{\zeta}{G\bar{d}}$, and then define $B:=B''\cap B'''$. Then define $\bar{\eta}:B\to\Lambda$ by \begin{equation*}
    \bar{\eta}(z) := \frac{\kappa a^2(z)}{\kappa a(z)E(z)-D(z)},
\end{equation*}
which is itself continuous and thus measurable, with
\begin{equation*}
    f(z,\bar{\eta}(z)) = \kappa,
\end{equation*}
for all $z\in B$ by construction. 
Further
\begin{multline}\label{eq:zeta}
    \int_{B}\int_{\cX}\big|p_{\bar{\eta}(z)}^*(x\given z)-p^*(x\given z)\big|dxdz
    \\
    \leq
    \int_B|\bar{\eta}(z)|V(\epsilon^*)^{-1}\sum_{i=1}^{J+2}|d_i(\epsilon^*)|\underbrace{\int_{\mathcal{B}_{\epsilon^*}(x_i)}k_i(x)dx}_{=V(\epsilon^*)}dz
    \leq
    \sup_{z\in B}|\bar{\eta}(z)|\bar{d}\,\Vol(B)
    \leq \zeta.
\end{multline}

Now define the density
\begin{equation*}
    p^{\dagger}(x\given z) := p_{\bar{\eta}(z)}^*(x\given z) = \begin{cases}
            p^*(x\given z) + \bar{\eta}(z)\, V(\epsilon^*)^{-1}\sum_{i=1}^{J+2}d_i(\epsilon^*)k_i(x)\ind_{\mathcal{B}_{\epsilon^*}(x_i)}(x) &\quad\text{if }z\in B
            \\
            p^*(x\given z) &\quad\text{if } z\in \cZ\backslash B
        \end{cases}.
    \end{equation*}
Then for every $z\in B$
\begin{equation*}\frac{\int_\cX p^{\dagger}(x\given z)W(x,z)U(x,z)dx}{\int_\cX p^{\dagger}(x\given z)W(x,z)dx}-\frac{\int_\cX p^*(x\given z)W(x,z)U(x,z)dx}{\int_\cX p^*(x\given z)W(x,z)dx}=\kappa,\end{equation*}
proving~\ref{lemitem:kappa} of the lemma. 
By the construction in part 3~\ref{lemitem:m_j}  follows immediately, and as $\bar{\eta}$ is continuous~\ref{lemitem:cts} follows. Finally \ref{lemitem:zeta} follows by~\eqref{eq:zeta}.
Thus all conditions for the lemma have been verified, completing the proof.

\end{proof}

\begin{lemma}\label{lem:holder-ctrl}
    Consider the setup of Lemma~\ref{lem:theta+c} (alongside the definitions and notation within its proof). Then there exists some $\epsilon_2>0$ such that for all $\epsilon\in(0,\epsilon_2]$,
    \begin{align}
        &\text{\emph{For Case 1: }}\sum_{i=1}^{J+2}
        \bigg|
        d_i(\epsilon)V_i(\epsilon)^{-1}\int_{\ball{x_i}}k_i(x)\Delta_0''(x)dx
        -
        d_i(0)\Delta_0''(x_i)dx
        \bigg| \leq 2^{-1}\omega,
        \label{eq:case1a}
        \\
        &\text{\emph{For Case 2: }} \sum_{i=1}^{J+2}\left|d_i(\epsilon)V_i(\epsilon)^{-1}\int_{\ball{x_i}}k_i(x)\bar{W}_0^\perp(x)dx - d_i(0)\bar{W}_0^\perp(x_i)\right| \leq 2^{-1}\omega.
        \label{eq:case2a}
    \end{align}
\end{lemma}

\begin{proof}
    These results follow by straightforward continuity arguments, included below for completeness. 
    Recall that $\Delta_0''(x) := a_0\bar{r}_0^\perp(x) - (\bar{\beta}_0+c_0)W_0(x)$ is a continuous bounded function, and so for each $i\in[J+2]$ the map $\epsilon \mapsto V_i(\epsilon)^{-1}\int_{\mathcal{B}_{\epsilon}(x_i)}k_i(x)\Delta_0''(x)dx$ is continuous on the domain $[0,\epsilon_1]$, with $\lim_{\epsilon\downarrow0}V_i(\epsilon)^{-1}\int_{\mathcal{B}_{\epsilon}(x_i)}k_i(x)\Delta_0''(x)dx = \Delta_0''(x_i)$. Similarly the map $\epsilon\mapsto d_i(\epsilon)$ is continuous on the domain $[0,\epsilon_1]$, and thus their product is continuous. Then there exists some $\epsilon_{2a}\in(0,\epsilon_1]$ such that for each $i\in[J+2]$
    \begin{equation*}
        \bigg|
        d_i(\epsilon)V_i(\epsilon)^{-1}\int_{\ball{x_i}}k_i(x)\Delta_0''(x)dx
        -
        d_i(0)\Delta_0''(x_i)dx
        \bigg| \leq 2^{-1}(J+2)^{-1}\omega,
    \end{equation*}
    and thus~\eqref{eq:case1a} follows for all $\epsilon\in(0,\epsilon_{2a}]$. By continuity of $\bar{W}_0^\perp$ it can similarly be shown that there exists some $\epsilon_{2b}\in(0,\epsilon_1]$ such that~\eqref{eq:case2a} holds for all $\epsilon\in(0,\epsilon_{2b}]$. 
    Taking $\epsilon_2 := \epsilon_{2a}\wedge\epsilon_{2b}$ completes the proof.
\end{proof}

\begin{lemma}\label{lem:equicontinuity}
    Adopt the setup of Lemma~\ref{lem:theta+c}, specifically recalling that $W$, $U$ and $(x,z)\mapsto p^*(x\given z)$ are bounded and equicontinuous in $z$. Then the functions
    \begin{equation*}
        z \mapsto a(z) := \int_{\cX} p^*(x\given z)W(x,z)dx,
        \quad
        z \mapsto b(z) := \int_{\cX}p^*(x\given z)W(x,z)U(x,z)dx,
    \end{equation*}
    are continuous.
\end{lemma}
\begin{proof}
Fix some $z\in\cZ$ and $\epsilon>0$. By equicontinuity of $W$ there exists some $\delta_1>0$ such that $|W(x,z')-W(x,z)|\leq\frac{\epsilon}{3}$ whenever $|z'-z|\leq\delta_1$. Further, there exists a $\delta_2>0$ and compact set $\mathcal{A}$ such that $\int_{\cX\backslash\mathcal{A}}p^*(x\given z)dx\leq \frac{\epsilon}{6\|W\|_{\infty}}$ for any $z\in\{z'\in\cZ:|z'-z|<\delta_2\}$. Also define $\Vol(\mathcal{A}):=\int_{\mathcal{A}}dx$. Then by equicontinuity of $(x,z)\mapsto p^*(x\given z)$ there exists some $\delta_3>0$ such that $|p(x\given z)-p(x\given z')|\leq\frac{\epsilon}{3\|W\|_{\infty}\Vol(\mathcal{A})}$ whenever $|z'-z|<\delta_3$. Defining $\delta_4:=\delta_1\wedge\delta_2\wedge\delta_3$ it follows that if $|z'-z|<\delta_4$ then
    \begin{align*}
        |a(z')-a(z)| 
        &=
        \bigg|\int_{\cX}p^*(x\given z)\big(W(x,z)-W(x,z')\big)+W(x,z')\big(p^*(x\given z)-p^*(x\given z')\big)dx\bigg|
        \\
        &\leq
        \int_{\cX}p^*(x\given z)\big|W(x,z)-W(x,z')\big|dx
        +
        \|W\|_{\infty}\int_{\cX}\big|p^*(x\given z)-p^*(x\given z')\big|dx
        \\
        &\leq
        \frac{\epsilon}{3} + \|W\|_{\infty}\int_{\mathcal{A}}\big|p^*(x\given z)-p^*(x\given z)\big|dx + \|W\|_{\infty}\int_{\cX\backslash\mathcal{A}}p^*(x\given z)dx 
        \\
        &\hspace{20em} + \|W\|_{\infty}\int_{\cX\backslash\mathcal{A}}p^*(x\given z')dx
        \\
        &\leq
        \frac{\epsilon}{3} + \frac{\epsilon}{3} + \frac{\epsilon}{6} + \frac{\epsilon}{6}
        \\
        &=\epsilon,
    \end{align*}
    and thus $a$ is continuous. The function $b$ can be shown to be continuous by identical arguments.
\end{proof}

\begin{lemma}\label{lem:lin-indep}
    Suppose the functions $f_1,\ldots,f_n:\cX\to\R$ are linearly independent. Then there exist distinct $x_1,\ldots,x_n\in\cX$ such that the matrix
    \[
    \mb F_n:= \begin{pmatrix}
        f_1(x_1) & \cdots & f_n(x_1) \\
        \vdots & \ddots & \vdots \\
        f_1(x_n) & \cdots & f_n(x_n)
    \end{pmatrix}
    \]
    is invertible.  
\end{lemma}
\begin{proof}
We use induction on $n$. Clearly the result is true when $n=1$, as we can take $x_1$ to be any value such that $f_1(x_1) \neq 0$; $f_1$ cannot be the zero function, nor can $\mathcal{X}$ be the empty set, as this would contradict the set $\{f_1\}$ being linearly independent. Suppose now that the result holds for some $n \geq 1$; we shall show it holds at $n+1$. Given linearly independent functions  $f_1,\ldots,f_{n+1}$, we first use the induction hypothesis to pick distinct points $x_1,\ldots,x_n$ such that the corresponding matrix $\mb F_n$ is invertible. Suppose, for a contradiction, that for all $x$,
\[
\mb F_{n+1}(x) := \begin{pmatrix}
        f_1(x_1) & \cdots & f_n(x_1) & f_{n+1}(x_1)\\
        \vdots & \ddots & \vdots & \vdots\\
        f_1(x_n) & \cdots &f_n(x_n) & f_{n+1}(x_n) \\
        f_1(x) & \cdots & f_n(x) & f_{n+1}(x)
    \end{pmatrix}
\]
is singular. Then writing $\mb c := \mb F_n^{-1} \mb f_{n+1}$, where $\mb f_{n+1} = (f_{n+1}(x_1), \ldots, f_{n+1}(x_n))^\top \in \R^n$, we have that
\[
c_1f_1(x) + \cdots +c_n f_n(x) = f_{n+1}(x)
\]
for all $x \in \mathcal{X}$, contradicting the linear independence of $f_1,\ldots,f_{n+1}$. Thus there exists an $x_{n+1}$ such that $\mb F_{n+1}:= \mb F_{n+1}(x_{n+1})$ is invertible. Moreover, we must have that $x_{n+1} \notin \{x_1,\ldots,x_n\}$, as otherwise $\mb F_{n+1}$ would have two identical rows.
\end{proof}

\begin{lemma}[\citet{app-gyorfi}, Lemma 3.1, adapted]\label{lem:seq}
    Let $b\in(0,1]$ be some constant. Let $(a_n)_{n\in\mathbb{N}}$ be a non-increasing sequence of positive numbers converging to zero with $a_1\leq\alpha$ for some constant $0<\alpha<b$. Then there exists a sequence $(p_k)_{k\in\mathbb{N}}$ of non-negative numbers satisfying
    \begin{equation*}
        \sum_{k=1}^\infty p_k = b
        ,
        \quad
        \text{and}
        \quad
        \sum_{k=1}^\infty p_k \Big(1-\frac{p_k}{b}\Big)^n \geq a_n \; \text{ for all }n\in\mathbb{N}.
    \end{equation*}
\end{lemma}
\begin{proof}
    A proof for the case of $(b,\alpha)=(1,\frac{1}{2})$ is given by~\citet{app-gyorfi}, Lemma 3.1 (we require a simple extension, included for completeness). 
    Let $p_1 = b(1-\frac{a_1}{\alpha})$, and define the constant $A := \frac{2b\log(\alpha^{-1}b)}{2+\log(\alpha^{-1}b)}\in(0,2b)$. 
    Then there exist integers $(l_n)_{n\in\mathbb{N}}$ such that $1=l_1<l_2<\cdots$ and $(p_k)_{k\in\mathbb{N}}$ such that
    \begin{equation*}
        p_k \leq \frac{A}{n} ,
    \end{equation*}
    for all $k>l_n$ and
    \begin{equation*}
        \sum_{k=l_n+1}^{m_{n+1}}p_k = \frac{b}{\alpha}(a_n-a_{n+1}).
    \end{equation*}
    Thus
    \begin{equation*}
        \sum_{k=1}^\infty p_k = p_1 + \frac{b}{\alpha}\sum_{i=1}^\infty(a_i-a_{i+1}) = b,
    \end{equation*}
    and, for all $n\geq 2$,
    \begin{multline*}
        \sum_{k=1}^\infty p_k \Big(1-\frac{p_k}{b}\Big)^n 
        \geq 
        \left(1-\frac{A}{bn}\right)^n \sum_{k : p_k\leq\frac{A}{n}}p_k
        \geq
        e^{-\frac{2A}{2b-A}}\sum_{k:p_k\leq\frac{A}{n}}p_k
        \\
        \geq e^{-\frac{2A}{2b-A}}\sum_{k=l_n+1}^\infty p_k
        = \frac{b}{\alpha}e^{-\frac{2A}{2b-A}}\sum_{i=n}^\infty
        (a_i-a_{i+1})
        = \frac{b}{\alpha}e^{-\frac{2A}{2b-A}}a_n = a_n,
    \end{multline*}
    where the final equality follows as $A=\frac{2b\log(\alpha^{-1}b)}{2+\log(\alpha^{-1}b)}$.
\end{proof}

\begin{lemma}\label{lem:limsup}
    Let $(x_{n,m})_{n\in\mathbb{N},m\in\mathbb{N}}$ and $(y_{n,m})_{n\in\mathbb{N},m\in\mathbb{N}}$ be sequences of non-negative numbers. Suppose also $\sup_{m\in\mathbb{N}}y_{n,m}\geq 1$ for all $n\in\mathbb{N}$, and for any subsequence $(m_n)_{n\in\mathbb{N}}$ of the natural numbers $\limsup_{n\to\infty}x_{n,m_n}\geq1$. 
    Then 
    \begin{equation*}
        \limsup_{n\to\infty}\sup_{m\in\mathbb{N}}x_{n,m}y_{n,m} \geq 1.
    \end{equation*}
\end{lemma}
\begin{proof}
    Fix some $\epsilon\in(0,1)$. For each $n\in\mathbb{N}$ there exists some $m'_n\in\mathbb{N}$ such that
    \begin{equation*}
        y_{n,m'_n} \geq \sup_{m\in\mathbb{N}}y_{n,m}-\epsilon \geq 1-\epsilon,
    \end{equation*}
    where the final inequality follows by assumption. Thus
    \begin{equation*}
        x_{n,m'_n}y_{n,m'_n} \geq (1-\epsilon)x_{n,m'_n},
    \end{equation*}
    and so 
    \begin{equation*}
        \limsup_{n\to\infty}\sup_{m\in\mathbb{N}}x_{n,m}y_{n,m}
        \geq
        \limsup_{n\to\infty}x_{n,m'_n}y_{n,m'_n} \geq (1-\epsilon)\limsup_{n\to\infty}x_{n,m'_n} \geq
        1-\epsilon,
    \end{equation*}
    where the final inequality follows by assumption. 
    Taking $\epsilon\downarrow0$ completes the proof.
\end{proof}

\section{Supplementary Information for Section~\ref{sec:robust-eff-est}}

In Appendix~\ref{appsec:norm-proof} we prove the uniform asymptotic Gaussainity result for the weighted $\theta$ estimator (Theorem~\ref{thm:theta_est_asymp_norm}). We conclude the section by exploring the doubly robust condition of~\ref{ass:consistency-and-DR} in more detail, showing in Appendix~\ref{appsec:DR} that in some models we can accommodate slower nuisance convergence regimes than those given by~\eqref{eq:1/4-rates}.
    
\subsection{Proof of Theorem~\ref{thm:theta_est_asymp_norm}}\label{appsec:norm-proof}
\begin{proof}[Proof of Theorem~\ref{thm:theta_est_asymp_norm}]
    The proof consists of three main parts, showing: consistency of $\hat{\theta}$; asymptotic normality of $\hat{\theta}$; and consistency of the estimator $\hat{V}$ for the asymptotic variance $V^*_P$, all of which will be shown to hold uniformly over $P\in\cP$ (for the class of distributions $\cP$ defined through Assumption~\ref{ass:data}). For notational ease we will suppress the dependence of $\eta$ on the covariates $Z$, and so will notate e.g.~$\psi_j(S;\theta,\eta(Z))$ and $\partial_\theta\psi_j(S;\theta,\eta(Z))$ as $\psi_j(S;\theta,\eta)$ and $\partial_\theta\psi_j(S;\theta,\eta)$ respectively throughout. For avoidance of doubt, for any $q\in\mathbb{N}$ quantities such as $\psi_j^q(S;\theta,\eta(Z))$ should be read as $\{\psi_j(S;\theta,\eta(Z))\}^q$. We also notate $\SIkc:=(S_i)_{i\in\cI_k^c}$ and, as in Algorithm~\ref{alg:main} $\hat{\eta}$ is estimated using $\SIkc$ conditioning on $\SIkc$ can be read as equivalent to conditioning on $\hat{\eta}$

    \medskip
    \noindent{\bf Uniform Consistency:} $\hat{\theta} = \theta_P + o_\cP(1)$
    \smallskip

    By Lemma~\ref{lem:unif_cons} (a simple extension of Theorem 5.9 of~\citet{app-vandervaart} for uniform convergence over $\cP$) alongside Assumption~\ref{ass:unique-solution}, it suffices to show that
    \begin{equation}\label{eq:vdv}
        \sup_{\theta\in\Theta}\bigg|
        \frac{1}{n}\sum_{k=1}^K\sum_{i\in\mathcal{I}_k}\psi(S_i;\theta,\hat{\eta}^{(k)},\hat{w}^{(k)})
        -
        \frac{1}{n}\sum_{i=1}^n \E_P\left[\psi(S_i;\theta,\eta_P,w^*_P)\right]
        \bigg| = o_\cP(1).
    \end{equation}
    To show this, we decompose
    \begin{align*}
        &\quad\sup_{\theta\in\Theta}\bigg|
        \frac{1}{n}\sum_{k=1}^K\sum_{i\in\mathcal{I}_k}\psi(S_i;\theta,\hat{\eta}^{(k)},\hat{w}^{(k)})
        -
        \frac{1}{n}\sum_{i=1}^n \E_P\left[\psi(S_i;\theta,\eta_P,w^*_P)\right]
        \bigg|
        \\
        &\leq
        \sum_{j=1}^J \frac{1}{K}\sum_{k=1}^K \sup_{\theta\in\Theta}\bigg|\frac{K}{n}\sum_{i\in\mathcal{I}_k}\big\{\hat{w}_j^{(k)}(Z_i)\psi_j(S_i;\theta,\hat{\eta}^{(k)})-\E_P\left[w_{P,j}^*(Z_i)\psi_j(S_i;\theta,\eta_P)\right]\big\}\bigg|.
    \end{align*}
    Then
        \begin{align*}
        &\quad \E_P\left[\sup_{\theta\in\Theta}\bigg|\frac{K}{n}\sum_{i\in\mathcal{I}_k}\big\{\hat{w}_j^{(k)}(Z_i)\psi_j(S_i;\theta,\hat{\eta}^{(k)})-\E_P\left[w_{P,j}^*(Z_i)\psi_j(S_i;\theta,\eta_P)\right]\big\}\bigg|\bigggiven\SIkc\right]
        \\
        &\leq 
        \frac{K}{n} \sum_{i\in\mathcal{I}_k}\E_P\left[\big|\hat{w}_j^{(k)}(Z_i)-w_{P,j}^*(Z_i)\big|\cdot\sup_{\theta\in\Theta}|\psi_j(S_i;\theta,\eta_P)|\Biggiven\SIkc\right]
        \\
        &\qquad
        +
        \frac{K}{n}\sum_{i\in\mathcal{I}_k}\E_P\left[\hat{w}_j^{(k)}(Z_i)\sup_{\theta\in\Theta}\big|\psi_j(S_i;\theta,\hat{\eta}^{(k)})-\psi_j(S_i;\theta,\eta_P)\big|\Biggiven\SIkc\right],
        \\
        &\leq 
        \frac{K}{n} \sum_{i\in\mathcal{I}_k}\E_P\left[\big(\hat{w}_j^{(k)}(Z_i)-w_{P,j}^*(Z_i)\big)^2\biggiven\SIkc\right]^{\frac{1}{2}}\E_P\left[
        \sup_{\theta\in\Theta}\big(\psi_j(S_i;\theta,\eta_P)\big)^2\right]^{\frac{1}{2}}
        \\
        &\qquad
        +
        \frac{K c_w}{n}\sum_{i\in\mathcal{I}_k}\E_P\left[\sup_{\theta\in\Theta}\big|\psi_j(S_i;\theta,\hat{\eta}^{(k)})-\psi_j(S_i;\theta,\eta_P)\big|\Biggiven\SIkc\right]
        \\
        &= o_\cP(1)
        ,
    \end{align*}
    using the given assumption on convergence of $\hat{w}^{(k)}$, Assumption~\ref{ass:bounded-moments} and Lemma~\ref{lem:eta-consistency}. Therefore~\eqref{eq:vdv} follows, and thus
    \begin{equation*}
        \hat{\theta} = \theta_P + o_\cP(1).
    \end{equation*}

    \medskip
    \noindent{\bf Uniform Asymptotic Normality: }
    \smallskip
    
    We decompose the empirical $\psi$ function as
    \begin{equation}\label{eq:normality-decomp}
    \begin{aligned}
        \frac{1}{n}\sum_{k=1}^K\sum_{i\in\cI_k} \psi(S_i;\hat{\theta},\hat{\eta}^{(k)},\hat{w}^{(k)}) 
        =&
        \frac{1}{n}\sum_{i=1}^n\psi(S_i;\theta_P,\eta_P,w^*_P)
        \\
        &\hspace{-4em}
        +
        \sum_{j=1}^J\frac{1}{K}\sum_{k=1}^K\underbrace{\frac{K}{n}\sum_{i\in\cI_k}\hat{w}_j^{(k)}(Z_i)\big\{\psi_j(S_i;\hat{\theta},\hat{\eta}^{(k)})-\psi_j(S_i;\theta_P,\hat{\eta}^{(k)})\big\}}_{=: A_{\RN{1},j}^{(k)}}
        \\
        &\hspace{-4em}
        +\sum_{j=1}^J\frac{1}{K}\sum_{k=1}^K\underbrace{\frac{K}{n}\sum_{i\in\cI_k}\E_P\big[\hat{w}_j^{(k)}(Z_i)\big\{\psi_j(S_i;\theta_P,\hat{\eta}^{(k)})-\psi_j(S_i;\theta_P,\eta_P)\big\}\,\big|\,\SIkc\big]}_{=:A_{\RN{2},j}^{(k)}}
        \\
        &\hspace{-4em}
        + \sum_{j=1}^J\frac{1}{K}\sum_{k=1}^K\underbrace{\frac{K}{n}\sum_{i\in\cI_k}\big(\hat{w}_j^{(k)}(Z_i)-w_{P,j}^*(Z)\big)\psi_j(S_i;\theta_P,\eta_P)}_{=:A_{\RN{3},j}^{(k)}}
        \\
        &\hspace{-4em}
        -\sum_{j=1}^J\frac{1}{K}\sum_{k=1}^K\underbrace{\sqrt{\frac{K}{n}}\sum_{i\in\cI_k}\hat{\mathds{G}}_{P}^{(k)}\Big(\hat{w}_j^{(k)}(Z_i)\big\{\psi_j(S_i;\theta_P,\hat{\eta}^{(k)})-\psi_j(S_i;\theta_P,\eta_P)\big\}\Big)}_{=:A_{\RN{4},j}^{(k)}}
        ,
    \end{aligned}
    \end{equation}
    where the empirical process term is defined as
    \begin{equation*}
        \hat{\mathds{G}}_{P}^{(k)}(g(S)) = \sqrt{\frac{K}{n}}\sum_{i\in\cI_k}\Big(g(S_i)-\E_P[g(S_i)\given\SIkc]\Big),
    \end{equation*}
    The first and second terms in this decomposition will dominate the other terms, and contribute to the asymptotic Gaussianity results. We first show that the terms indexed by $\RN{2}, \RN{3}$ and $\RN{4}$ are all $o_\cP(n^{-\frac{1}{2}})$.

    For the term  $A_{\RN{2}}:=\sum_{j=1}^J\frac{1}{K}\sum_{k=1}^K A_{\RN{2},j}^{(k)}$,
        \begin{align*}
        \big|A_{\RN{2},j}^{(k)}\big|
        &=
        \bigg|\frac{K}{n}\sum_{i\in\mathcal{I}_k}\E_P\Big[\hat{w}_j^{(k)}(Z_i) \E_P\big[  \nabla_{\eta(z)}\psi_j(S_i;\theta_P,\eta_P) \given Z_i \big]^\top \big(\hat{\eta}^{(k)}(Z_i)-\eta_P(Z_i)\big)\biggiven\SIkc\Big]
        \\
        &\quad
        + \frac{K}{2n}\sum_{i\in\cI_k}\E_P\Big[\hat{w}_j^{(k)}(Z_i)\big(\hat{\eta}^{(k)}(Z_i)-\eta_P(Z_i)\big)^\top \hat{H}_j^{(k)}(Z_i;\tilde{r})\big(\hat{\eta}^{(k)}(Z_i)-\eta_P(Z_i)\big)\biggiven\SIkc\Big]\bigg|
        \\
        &\leq
        \frac{K c_w}{2n}\sum_{i\in\cI_k}\Big|\E_P\Big[\sup_{r\in[0,1]}\big(\hat{\eta}^{(k)}(Z_i)-\eta_P(Z_i)\big)^\top \hat{H}_j^{(k)}(Z_i;r)\big(\hat{\eta}^{(k)}(Z_i)-\eta_P(Z_i)\big)\biggiven\SIkc\Big]\Big|
        \\
        &= o_\cP\big(n^{-\frac{1}{2}}\big).
    \end{align*}
    for some $\tilde{r}\in[0,1]$, following by `Neyman orthogonality' i.e. $\psi_j(S;\theta_P,\eta_P)$ is orthogonal to the nuisance tangent space of the model;
    \begin{equation*}
        \E_P\big[\nabla_{\eta(z)}\psi_j(S;\theta_P,\eta_P)\given Z\big] = 0,
    \end{equation*}
    which holds under the partially parametric model (Theorem~\ref{thm:IF}), and `double robustness' (Assumption~\ref{ass:consistency-and-DR}). Recall that $$\hat{H}_j^{(k)}(Z_i;r) := \E_P\big[\nabla_{\eta(z)}^2\psi_j\big(S_i;\theta_P,\eta_P(Z_i)+r(\hat{\eta}^{(k)}(Z_i)-\eta_P(Z_i))\big)\given Z_i,\SIkc\big]$$.

    For $A_\RN{3} := \sum_{j=1}^J\frac{1}{K}\sum_{k=1}^K A_{\RN{3},j}^{(k)}$, first note that as $\E_P[\psi_j(S_i;\theta_P,\eta_P)\given Z_i]=0$ we have that $\E_P[A_{\RN{3},j}^{(k)}\given\SIkc]=0$. Also
    \begin{align*}
        &\quad\E_P\big[\big(\sqrt{n}A_{\RN{3},j}^{(k)}\big)^2\given\SIkc\big] 
        \\
        &\leq
        \frac{K^2}{n}\sum_{i\in\mathcal{I}_k}\E_P\big[ \big(\hat{w}_j^{(k)}(Z_i)-w_{P,j}^*(Z_i)\big)^2 \psi_j^2(S_i;\theta_P,\eta_P) \given\SIkc\big]
        \\
        &\leq
        \frac{K^2 2^{\frac{4}{2+\delta}}c_w^{\frac{4}{2+\delta}}}{n}\sum_{i\in\mathcal{I}_k}\E_P\big[ \big(\hat{w}_j^{(k)}(Z_i)-w_{P,j}^*(Z_i)\big)^{\frac{2\delta}{2+\delta}} \psi_j^2(S_i;\theta_P,\eta_P) \given\SIkc\big]
        \\
        &\leq
        \frac{K^2 2^{\frac{4}{2+\delta}}c_w^{\frac{4}{2+\delta}}}{n}\sum_{i\in\mathcal{I}_k}
        \underbrace{\E_P\big[ \big(\hat{w}_j^{(k)}(Z_i)-w_{P,j}^*(Z_i)\big)^2 \given\SIkc\big]^{\frac{\delta}{2+\delta}}}_{=o_\cP(1)\text{ as in the statement of the theorem}} \underbrace{\E_P\big[|\psi_j(S_i;\theta_P,\eta_P)|^{2+\delta}\big]^{\frac{2}{2+\delta}}}_{=O_\cP(1)\text{ by Assumption~\ref{ass:bounded-moments}}}
        \\
        &= o_\cP(1),
    \end{align*}
    by H\"{o}lder's inequality. Thus $A_{\RN{3},j}^{(k)}=o_\cP(n^{-\frac{1}{2}})$ and subsequently $A_{\RN{3}}=o_\cP(n^{-\frac{1}{2}})$.

    For the $A_\RN{4} := \sum_{j=1}^J\frac{1}{K}\sum_{k=1}^K A_{\RN{4},j}^{(k)}$ term, first note that $\E[A_{\RN{4},j}^{(k)}\given\SIkc]=0$, and 
    \begin{align*}
        \E_P\big[\big(\sqrt{n}A_{\RN{4},j}^{(k)}\big)^2\given\SIkc\big]
        &\leq
        \frac{K^2}{n}\sum_{i\in\mathcal{I}_k}\E_P\big[\big(\hat{w}_j^{(k)}(Z_i)\big)^2\big\{\psi_j(S_i;\theta_P,\hat{\eta}^{(k)})-\psi_j(S_i;\theta_P,\eta_P)\big\}^2\given\SIkc\big]
        \\
        &\leq
        \frac{K^2 c_w^2}{n}\sum_{i\in\mathcal{I}_k}\E_P\big[\big\{\psi_j(S_i;\theta_P,\hat{\eta}^{(k)})-\psi_j(S_i;\theta_P,\eta_P)\big\}^2\given\SIkc\big]
        \\
        &= o_\cP(1),
    \end{align*}
    thus $A_{\RN{4},j}^{(k)}=o_\cP(n^{-\frac{1}{2}})$ and so $A_{\RN{4}}=o_\cP(n^{-\frac{1}{2}})$.

    For the $A_\RN{1} := \sum_{j=1}^J\frac{1}{K}\sum_{k=1}^K A_{\RN{1},j}^{(k)}$ term we perform a second order Taylor expansion in $\theta$ about $\theta_P$,
    \begin{align*}
        A_{\RN{1},j}^{(k)}
        &=
        \frac{K}{n}\sum_{i\in\mathcal{I}_k} \hat{w}_j^{(k)}(Z_i)\big\{\psi_j(S_i;\hat{\theta},\hat{\eta}^{(k)}) - \psi_j(S_i;\theta_P,\hat{\eta}^{(k)})\big\}
        \\
        &= 
        \bigg( \frac{K}{n}\sum_{i\in\mathcal{I}_k} \hat{w}_j^{(k)}(Z_i)\nabla_\theta\psi_j(S_i;\theta_P,\hat{\eta}^{(k)}) \bigg)\big(\hat{\theta}-\theta_P\big)
        \\
        &\quad + \bigg(\frac{K}{2n}\sum_{i\in\mathcal{I}_k}\hat{w}_j^{(k)}(Z_i)\nabla_\theta^2\psi_j(S_i;\bar{\theta},\hat{\eta}^{(k)})\bigg)\big(\hat{\theta}-\theta_P\big)^2
        \\
        &= \mathcal{A}_j^{(k)}\big(\hat{\theta}-\theta_P\big),
    \end{align*}
    where
    \begin{align}
        \mathcal{A}_j^{(k)}
        &:= 
        \frac{K}{n}\sum_{i\in\mathcal{I}_k} \hat{w}_j^{(k)}(Z_i)\nabla_\theta\psi_j(S_i;\theta_P,\hat{\eta}^{(k)}) + \Big(\frac{K}{2n}\sum_{i\in\mathcal{I}_k}\hat{w}_j^{(k)}(Z_i)\nabla_\theta^2\psi_j(S_i;\bar{\theta},\hat{\eta}^{(k)})\Big)\big(\hat{\theta}-\theta_P\big) 
        \notag
        \\
        &=
        \underbrace{\frac{K}{n}\sum_{i\in\mathcal{I}_k}\E_P[w_{P,j}^*(Z_i)\nabla_\theta\psi_j(S_i;\theta_P,\eta_P)]}_{=: I_j^{(k)} > 0} 
        \label{eq:I_j^{(k)}}
        \\
        &\quad
        +
        \underbrace{\frac{K}{n}\sum_{i\in\mathcal{I}_k}\hat{w}_j^{(k)}(Z_i)\big\{\nabla_\theta\psi_j(S_i;\theta_P,\hat{\eta}^{(k)}) - \nabla_\theta\psi_j(S_i;\theta_P,\eta_P)\big\}}_{=: A_{\RN{1}a,j}^{(k)}} 
        \notag
        \\
        &\quad
        + \underbrace{\frac{K}{n}\sum_{i\in\mathcal{I}_k}\big(\hat{w}_j^{(k)}(Z_i)-w_{P,j}^*(Z_i)\big)\nabla_\theta\psi_j(S_i;\theta_P,\eta_P)}_{=: A_{\RN{1}b,j}^{(k)}} 
        \notag
        \\
        &\quad
        + \underbrace{\frac{K}{n}\sum_{i\in\mathcal{I}_k}\big\{w_{P,j}^*(Z_i)\nabla_\theta\psi_j(S_i;\theta_P,\eta_P)-\E_P[w_{P,j}^*(Z_i)\nabla_\theta\psi_j(S_i;\theta_P,\eta_P)]\big\}}_{=: A_{\RN{1}c,j}^{(k)}} 
        \notag
        \\
        &\quad + 
        \underbrace{\bigg(\frac{K}{2n}\sum_{i\in\mathcal{I}_k}\hat{w}_j^{(k)}(Z_i)\nabla_\theta^2\psi_j(S_i;\bar{\theta},\hat{\eta}^{(k)})\bigg)\big(\hat{\theta}-\theta_P\big)}_{=: A_{\RN{1}d,j}^{(k)}}, 
        \notag
    \end{align}
    for some $\bar{\theta}\in[(\hat{\theta}\wedge\theta_P) \,,\, (\hat{\theta}\vee\theta_P)]$. For the term $A_{\RN{1}a,j}^{(k)}$,
\begin{align*}
    \E_P\big[|A_{\RN{1}a,j}^{(k)}|\given\SIkc\big]
    &\leq
    \frac{K}{n}\sum_{i\in\mathcal{I}_k}\E_P\big[ |\hat{w}_j^{(k)}(Z_i)|\cdot\big|\nabla_\theta\psi_j(S_i;\theta_P,\hat{\eta}^{(k)})-\nabla_\theta\psi_j(S_i;\theta_P,\eta_P)\big| \given\SIkc\big]
    \\
    &\leq 
    \frac{K c_w}{n}\sum_{i\in\mathcal{I}_k}\E_P\big[ \big|\nabla_\theta\psi_j(S_i;\theta_P,\hat{\eta}^{(k)})-\nabla_\theta\psi_j(S_i;\theta_P,\eta_P)\big| \given\SIkc\big]
    \\
    &=o_\cP(1),
\end{align*}
by Lemma~\ref{lem:eta-consistency}. Also,
\begin{align*}
    \E_P\big[|A_{\RN{1}b,j}^{(k)}|\given\SIkc\big]
    &\leq
    \frac{K}{n}\sum_{i\in\mathcal{I}_k} \E_P\Big[\big|\hat{w}_j^{(k)}(Z_i)-w_{P,j}^*(Z_i)\big|\cdot\big|\nabla_\theta\psi_j(S_i;\theta_P,\eta_P)\big|\biggiven\SIkc\Big]
    \\
    &\leq 
    \frac{K}{n}\sum_{i\in\mathcal{I}_k}\E_P\Big[\big(\hat{w}_j^{(k)}(Z_i)-w_{P,j}^*(Z_i)\big)^2\biggiven\SIkc\Big]^{\frac{1}{2}}\E_P\Big[\big(\nabla_\theta\psi_j(S_i;\theta_P,\eta_P)\big)^2\Big]^{\frac{1}{2}}
    \\
    &=
    o_\cP(1),
\end{align*}
using the assumption in the statement of Theorem~\ref{thm:theta_est_asymp_norm} (holding e.g.~for ROSE random forest generated weights by Theorem~\ref{thm:rose_forest_consistency}), and so $A_{\RN{1}b,j}^{(k)}=o_\cP(1)$. 
Also, note that $\E_P\big[A_{\RN{1}c,j}^{(k)}\big]=0$ and
\begin{align*}
    \supP\E\big[\big(A_{\RN{1}c,j}^{(k)}\big)^2\big]
    &\leq
    \frac{K^2}{n^2}\sum_{i\in\mathcal{I}_k}\supP\E_P\Big[\big(w_{P,j}^{*}(Z_i)\big)^2\big(\nabla_\theta\psi_j(S_i;\theta_P,\eta_P)\big)^2\Big]
    \\
    &\leq
    \frac{K^2 c_w^2}{n^2}\sum_{i\in\mathcal{I}_k}\supP\E_P\Big[\big(\nabla_\theta\psi_j(S_i;\theta_P,\eta_P)\big)^2\Big]
    \\
    &=o(1).
\end{align*}
Finally, for the term $A_{\RN{1}d,j}^{(k)}$ note that
\begin{align*}
    \E_P\left[\bigg|\frac{K}{2n}\sum_{i\in\mathcal{I}_k}\hat{w}_j^{(k)}(Z_i)\nabla_\theta^2\psi_j(S_i;\bar{\theta},\hat{\eta}^{(k)})\bigg|\Bigggiven\SIkc\right]
    &\leq
    \frac{K c_w}{2n}\sum_{i\in\mathcal{I}_k}\E_P\left[\sup_{\theta\in\Theta}\big|\nabla_\theta^2\psi_j(S_i;\theta,\hat{\eta}^{(k)})\big|\Biggiven\SIkc\right]
    \\
    &=
    O_\cP(1).
\end{align*}
This, in conjunction with $\hat{\theta}-\theta_P=o_\cP(1)$ gives $A_{\RN{1}d,j}^{(k)}=o_\cP(1)$.

Combining all these results gives $A_{\RN{1}a,j}^{(k)}+A_{\RN{1}b,j}^{(k)}+A_{\RN{1}c,j}^{(k)}+A_{\RN{1}d,j}^{(k)}=o_\cP(1)$ and thus, recalling $I_j^{(k)}$ as defined in~\eqref{eq:I_j^{(k)}},
\begin{equation*}
    A_{\RN{1},j}^{(k)} = \big(I_j^{(k)} + o_\cP(1)\big) \big(\hat{\theta}-\theta_P\big),
\end{equation*}
and thus (recalling both $J$ and $K$ are finite)
\begin{align*}
    A_{\RN{1}} 
    &= 
    \bigg(\sum_{j=1}^J\frac{1}{K}\sum_{k=1}^K I_j^{(k)} + o_\cP(1)\bigg) \big(\hat{\theta}-\theta_P\big)
    \\
    &=
    \bigg(\frac{1}{n}\sum_{i=1}^n\sum_{j=1}^J\E_P\big[w_{P,j}^*(Z_i)\nabla_\theta\psi_j(S_i;\theta_P,\eta_P)\big] + o_\cP(1)\bigg) \big(\hat{\theta}-\theta_P\big)
    .
\end{align*}
Further noting that $\E_P\big[\nabla_\theta\psi_j(S;\theta_P,\eta_P)\given Z\big] \allowbreak = \allowbreak \E_P\big[\partial_\theta\psi_j(S;\theta_P,\eta_P)\given Z\big]$, the decomposition~\eqref{eq:normality-decomp} becomes
\begin{equation*}
    o_\cP(1) + \frac{1}{\sqrt{n}}\sum_{i=1}^n \psi(S_i;\theta_P,\eta_P,w^*_P) = \bigg(\sum_{j=1}^J\E_P[w_{P,j}^*(Z)\partial_\theta\psi_j(S;\theta_P,\eta_P)] + o_\cP(1)\bigg)\cdot\sqrt{n}\big(\hat{\theta}-\theta_P\big).
\end{equation*}
Upon rearranging we have
\begin{equation*}
    \sqrt{n}\big(\hat{\theta}-\theta_P\big)
    =
    \frac{\frac{1}{\sqrt{n}}\sum_{i=1}^n \psi(S_i;\theta_P,\eta_P,w^*_P) + o_\cP(1)}{\sum_{j=1}^J\E_P[w_{P,j}^*(Z)\partial_\theta\psi_j(S;\theta_P,\eta_P)] + o_\cP(1)},
\end{equation*}
noting in particular the non-singularity assumption~\eqref{eq:invertible};
\begin{equation*}
    \inf_{P\in\cP}\sum_{j=1}^J\E_P[w_{P,j}^*(Z)\partial_\theta\psi_j(S;\theta_P,\eta_P)]
    =
    \inf_{P\in\cP}\sum_{j=1}^J\E_P[w_{P,j}^*(Z)\nabla_\theta\psi_j(S;\theta_P,\eta_P)]
    >0
\end{equation*}
The result then follows by Lemmas~\ref{lem:unif_clt} and~\ref{lem:unif_slutsky};
\begin{equation*}
    \lim_{n\to\infty}\supP\sup_{t\in\R}\left|\PP_P\left( \sqrt{n/V^*_P}\big(\hat{\theta}_{\hat{w}}-\theta_P\big) \leq t \right)-\Phi(t)\right| = 0,
\end{equation*}
for $V^*_P$ as in the statement of the theorem.

\medskip
    \noindent{\bf Consistency of $\hat{V}$: }
    \smallskip
    
We finally must show that, for $\hat{V}$ as in Algorithm~\ref{alg:main},
\begin{equation*}
    \lim_{n\to\infty}\supP\sup_{t\in\R}\left|\PP_P\left( \sqrt{n/\hat{V}}\big(\hat{\theta}_{\hat{w}}-\theta_P\big) \leq t \right)-\Phi(t)\right| = 0.
\end{equation*}
By Lemma~\ref{lem:unif_slutsky} it sufficies to show
\begin{equation*}
    \hat{V} = V^*_P + o_\cP(1).
\end{equation*}
The estimator $\hat{V}$ in Algorithm~\ref{alg:main} can be decomposed as
\begin{align}\label{eq:V-decomp}
    \hat{V} &= \frac{\sum_{j=1}^J\frac{1}{K}\sum_{k=1}^K \frac{K}{n}\sum_{i\in\cI_k}\big(\hat{w}_j^{(k)}(Z_i)\big)^2\psi_j^2(S_i;\hat{\theta},\hat{\eta}^{(k)})}{\left(\sum_{j=1}^J\frac{1}{K}\sum_{k=1}^K \frac{K}{n}\sum_{i\in\cI_k}\hat{w}_j^{(k)}(Z_i)\partial_\theta\psi_j(S_i;\hat{\theta},\hat{\eta}^{(k)})\right)^2}
    \nonumber\\
    &=
    \left( \sum_{j=1}^J\E_P\left[w_{P,j}^*(Z)\partial_\theta\psi_j(S;\theta_P,\eta_P)\right] + \sum_{j=1}^J \Big( B_{\RN{1},j} + \frac{1}{K}\sum_{k=1}^K \big(B_{\RN{2},j}^{(k)}+B_{\RN{3},j}^{(k)}+B_{\RN{4},j}^{(k)}\big) \Big) \right)^{-2}
    \nonumber\\
    &\qquad
    \cdot\left(\sum_{j=1}^J\E_P\left[w^{*2}(Z)\psi_j^2(S;\theta_P,\eta_P)\right] + \sum_{j=1}^J \Big(C_{\RN{1},j} + \frac{1}{K}\sum_{k=1}^K\big(C_{\RN{2},j}^{(k)}+C_{\RN{3},j}^{(k)}+C_{\RN{4},j}^{(k)}\big)\Big) \right),
\end{align}
where
\begin{align*}
    B_{\RN{1},j} &:= \frac{1}{n}\sum_{i=1}^n\big(w_{P,j}^*(Z_i)\partial_\theta\psi_j(S_i;\theta_P,\eta_P)-\E_P\left[w^*_P(Z)\partial_\theta\psi(S;\theta_P,\eta_P)\right]\big)
    \\
    B_{\RN{2},j}^{(k)} &:= \frac{K}{n}\sum_{i\in\cI_k}\big(\hat{w}_j^{(k)}(Z_i)-w_{P,j}^*(Z_i)\big)\partial_\theta\psi_j(S_i;\theta_P,\eta_P)
    \\
    B_{\RN{3},j}^{(k)} &:= \frac{K}{n}\sum_{i\in\cI_k}\hat{w}_j^{(k)}(Z_i)\big\{\partial_\theta\psi_j(S_i;\theta_P,\hat{\eta}^{(k)})-\partial_\theta\psi_j(S_i;\theta_P,\eta_P)\big\}
    \\
    B_{\RN{4},j}^{(k)} &:= \frac{K}{n}\sum_{i\in\cI_k}\hat{w}_j^{(k)}(Z_i)\big\{\partial_\theta\psi_j(S_i;\hat{\theta},\hat{\eta}^{(k)})-\partial_\theta\psi_j(S_i;\theta_P,\hat{\eta}^{(k)})\big\}
    \\
    C_{\RN{1},j} &:= \frac{K}{n}\sum_{i=1}^n\Big(\big(w_{P,j}^{*}(Z_i)\big)^2\psi_j^2(S_i;\theta_P,\eta_P) - \E_P\left[w_{P,j}^{*2}(X)\psi_j^2(S;\theta_P,\eta_P)\right]\Big)
    \\
    C_{\RN{2},j}^{(k)} &:= \frac{K}{n}\sum_{i\in\cI_k}\Big(\big(\hat{w}_j^{(k)}(Z_i)\big)^2-\big(w_{P,j}^{*}(Z_i)\big)^2\Big)\psi_j^2(S_i;\theta_P,\eta_P)
    \\
    C_{\RN{3},j}^{(k)} &:= \frac{K}{n}\sum_{i\in\cI_k}\big(\hat{w}_j^{(k)}(Z_i)\big)^2\big\{\psi_j^2(S_i;\theta_P,\hat{\eta}^{(k)})-\psi_j^2(S_i;\theta_P,\eta_P)\big\}
    \\
    C_{\RN{4},j}^{(k)} &:= \frac{K}{n}\sum_{i\in\cI_k}\big(\hat{w}_j^{(k)}(Z_i)\big)^2\big\{ \psi_j^2(S_i;\hat{\theta},\hat{\eta}^{(k)}) - \psi_j^2(S_i;\theta_P,\hat{\eta}^{(k)}) \big\}
    .
\end{align*}
We proceed by showing all of these terms are $o_\cP(1)$ for all $(j,k)\in[J]\times[K]$. 

    For the term $B_{\RN{1},j}$, note that
    \begin{align*}
        &
        \supP\E_P\left[\big(B_{\RN{1},j}\big)^2\right]
        \\
        &\leq \frac{K^2}{n^2}\sum_{i=1}^n\supP\E_P\left[\big(w_{P,j}^*(Z_i)\partial_\theta\psi_j(S_i;\theta_P,\eta_P)-\E_P\left[w_{P,j}^*(Z_i)\partial_\theta\psi_j(S;\theta_P,\eta_P)\right]\big)^2\right]
        \\
        &\leq
        \frac{K^2}{n^2}\sum_{i=1}^n\supP\E_P\left[\big(w_{P,j}^{*}(Z_i)\big)^2\big(\partial_\theta\psi_j(S_i;\theta_P,\eta_P)\big)^2\right]
        \\
        &\leq 
        \frac{K^2 c_w^2}{n}\supP\E_P\left[\big(\partial_\theta\psi_j(S_i;\theta_P,\eta_P)\big)^2\right]
        = O(n^{-1}K) = o(1),
    \end{align*}
    by Assumptions~\ref{ass:bounded-moments}~and~\ref{ass:bounded-w}. Therefore by Markov's inequality $\B{1}=o_\cP(1)$.

    Consider the term $B_{\RN{2},j}^{(k)}$. By the triangle and Cauchy--Schwarz inequalities,
    \begin{multline*}
        \E_P\left[\big|B_{\RN{2},j}^{(k)}\big|\biggiven\SIkc\right]
        \leq
        \frac{K}{n}\sum_{i\in\cI_k}\E_P\bigg[\Big|\big(\hat{w}_j^{(k)}(Z_i)-w_{P,j}^*(Z_i)\big)\partial_\theta\psi_j(S_i;\theta_P,\eta_P)\Big|\Biggiven\SIkc\bigg]
        \\
        \leq \frac{K}{n}\sum_{i\in\cI_k}\E_P\left[\big(\hat{w}_j^{(k)}(Z_i)-w_{P,j}^*(Z_i)\big)^2\biggiven\SIkc\right]^{\frac{1}{2}}\E_P\left[\big(\partial_\theta\psi_j(S_i;\theta_P,\eta_P)\big)^2\right]^{\frac{1}{2}}
        =o_\cP(1),
    \end{multline*}
    using Assumption~\ref{ass:bounded-moments} and Theorem~\ref{thm:rose_forest_consistency}, and so  $B_{\RN{2},j}^{(k)}=o_\cP(1)$.

    For the term $B_{\RN{3},j}^{(k)}$,
    \begin{equation*}
        \E_P\left[\big|B_{\RN{3},j}^{(k)}\big|\biggiven\SIkc\right]
        \leq
        \frac{K c_w}{n}\sum_{i\in\cI_k}\E_P\left[\big|\partial_\theta\psi_j(S_i,\theta_P,\hat{\eta}^{(k)})-\partial_\theta\psi_j(S_i,\theta_P,\eta_P)\big|\biggiven\SIkc\right]
        = o_\cP(1),
    \end{equation*}
    by Lemma~\ref{lem:eta-consistency}. Again, $B_{\RN{3},j}^{(k)}=o_\cP(1)$ follows by Markov's inequality.

    For the term $B_{\RN{4},j}^{(k)}$,
    \begin{align*}
        \supP\E_P\Big[\big|B_{\RN{4},j}^{(k)}\big|\Big]
        &\leq \frac{K c_w}{n}\sum_{i\in\cI_k}\E_P\left[\big|\partial_\theta\psi_j(S_i;\hat{\theta},\hat{\eta}^{(k)})-\partial_\theta\psi_j(S_i;\theta_P,\hat{\eta}^{(k)})\big|\right]
        \\
        &\leq \frac{K c_w}{n}\sum_{i\in\cI_k}\supP\E_P\left[\Big(\sup_{\theta\in\Theta}\big|\nabla_\theta\partial_\theta\psi_j(S_i;\theta,\hat{\eta}^{(k)})\big|\Big)\cdot|\hat{\theta}-\theta_P|
        \right]
        \\
        &\leq \frac{K c_w}{n}\sum_{i\in\cI_k}\supP\E_P\left[\Big(\sup_{\theta\in\Theta}\big|\nabla_\theta\partial_\theta\psi_j(S_i;\theta,\hat{\eta}^{(k)})\big|\Big)^2\right]^{\frac{1}{2}}
        \E_P\left[\big(\hat{\theta}-\theta_P\big)^2\right]^{\frac{1}{2}}
        \\
        &= \frac{K c_w}{n^{\frac{3}{2}}}\sum_{i\in\cI_k}\supP\E_P\left[\Big(\sup_{\theta\in\Theta}\big|\nabla_\theta\partial_\theta\psi_j(S_i;\theta,\hat{\eta}^{(k)})\big|\Big)^2\right]^{\frac{1}{2}}
        \E_P\left[n\big(\hat{\theta}-\theta_P\big)^2\right]^{\frac{1}{2}}
        \\
        &= O(n^{-\frac{1}{2}}),
    \end{align*}
    by Lemma~\ref{lem:unif_root-n-theta-moments-bounded}, and thus by Markov's inequality $B_{\RN{4},j}^{(k)}=o_\cP(1)$.

We show $C_{\RN{1}}=o_\cP(1)$ separately for the two cases $\delta\geq2$ and $\delta\in(0,2)$. If $\delta\geq2$ then
\begin{multline*}
    \supP\E_P\left[C_{\RN{1},j}^2\right]
    \leq \frac{1}{n^2}\sum_{i=1}^n\supP\E_P\left[\big(w_{P,j}^{*}(Z_i)\big)^4\psi_j^4(S_i;\theta_P,\eta_P)\right]
    \\
    \leq\frac{c_w^4}{n^2}\sum_{i=1}^n\supP\E_P\left[\psi_j^4(S_i;\theta_P,\eta_P)\right] = O(n^{-1}),
\end{multline*}
by Assumptions~\ref{ass:bounded-moments}~and~\ref{ass:bounded-w}, with the result then following by Markov's Theorem. If however $\delta\in(0,2)$ then by the von--Bahr--Esseen inequality
\begin{align*}
    \supP\E_P\left[\left|C_{\RN{1},j}\right|^{1+\frac{\delta}{2}}\right] 
    &= n^{-\left(1+\frac{\delta}{2}\right)}\supP\E_P\left[\left|\sum_{i=1}^n\big(w_{P,j}^{*}(Z_i)\big)^2\psi_j^2(S_i;\theta_P,\eta_P)\right|\right]
    \\
    &\leq \left(2-n^{-1}\right)n^{-\left(1+\frac{\delta}{2}\right)}\sum_{i=1}^n\supP\E_P\left[w_{P,j}^{*(2+\delta)}(Z_i)\left|\psi_j(S_i;\theta_P,\eta_P)\right|^{2+\delta}\right]
    \\
    &\leq 2M^{4+2\delta}n^{-\left(1+\frac{\delta}{2}\right)}\sum_{i=1}^n\supP\E_P\left[\left|\psi_j(S_i;\theta_P,\eta_P)\right|^{2+\delta}\right]
    = O\big(n^{-\frac{\delta}{2}}\big).
\end{align*}
In either case, the result $C_{\RN{1},j}=o_\cP(1)$ follows for any $\delta>0$.

For the term $C_{\RN{2},j}^{(k)}$, again consider the cases $\delta\geq 2$ and $\delta\in(0,2)$ separately. If $\delta\geq 2$
\begin{align*}
    &\E_P\left[\big|C_{\RN{2},j}^{(k)}\big|\biggiven\SIkc\right]
    \\
    &\leq \frac{K}{n}\sum_{i\in\cI_k}\supP\E_P\left[\big|\hat{w}_j^{(k)}(Z_i)+w_{P,j}^*(Z_i)\big|\cdot\big|\hat{w}_j^{(k)}(Z_i)-w_{P,j}^*(Z_i)\big|\cdot\psi_j^2(S_i;\theta_P,\eta_P)\biggiven\SIkc\right]
    \\
    &\leq\frac{2 K c_w}{n}\sum_{i\in\cI_k}\E_P\left[\big|\hat{w}_j^{(k)}(Z_i)-w_{P,j}^*(Z_i)\big|\big(\psi_j(S_i;\theta_P,\eta_P)\big)^2\biggiven\SIkc\right]
    \\
    &\leq \frac{2 K c_w}{n}\sum_{i\in\cI_k}\E_P\left[\big(\hat{w}_j^{(k)}(Z_i)-w_{P,j}^*(Z_i)\big)^2\biggiven\SIkc\right]^{\frac{1}{2}}\E_P\left[\psi_j^4(S_i;\theta_P,\eta_P)\right]^{\frac{1}{2}}
    = o_\cP(1),
\end{align*}
using the triangle inequality, the Cauchy--Schwarz inequality, and the ROSE random forest consistency result of Theroem~\ref{thm:rose_forest_consistency}.
If on the other hand $\delta\in(0,2)$ then
\begin{align*}
    &\E_P\left[\big|C_{\RN{2},j}^{(k)}\big|\biggiven\SIkc\right]
    \\
    &\leq\frac{2 K c_w}{n}\sum_{i\in\cI_k}\supP\E_P\left[\big|\hat{w}_j^{(k)}(Z_i)-w_{P,j}^*(Z_i)\big|\psi_j^2(S_i;\theta_P,\eta_P)\biggiven\SIkc\right]
    \\
    &\leq \frac{2 K c_w}{n}\sum_{i\in\cI_k}\supP\E_P\left[\big|\hat{w}_j^{(k)}(Z_i)-w_{P,j}^*(Z_i)\big|^{2+\frac{2-\delta}{\delta}}\Biggiven\SIkc\right]^{\frac{\delta}{2+\delta}}\E_P\left[\left|\psi_j(S_i;\theta_P,\eta_P)\right|^{2+\delta}\right]^{\frac{2}{2+\delta}}
    \\
    &\leq \frac{K \left(2c_w\right)^{\frac{4}{2+\delta}}}{n}\sum_{i\in\cI_k}\supP\E_P\left[\big|\hat{w}_j^{(k)}(Z_i)-w_{P,j}^*(Z_i)\big|^2\biggiven\SIkc\right]^{\frac{\delta}{2+\delta}}\E_P\left[\left|\psi_j(S_i;\theta_P,\eta_P)\right|^{2+\delta}\right]^{\frac{2}{2+\delta}}
    \\
    & = o_\cP(1),
\end{align*}
by the triangle inequality, H{\"o}lder's inequality, Assumptions~\ref{ass:bounded-moments}~and~\ref{ass:bounded-w} and Theorem~\ref{thm:rose_forest_consistency}.

For the term $C_{\RN{3},j}^{(k)}$,
\begin{align*}
    &\E_P\left[\big|C_{\RN{3},j}^{(k)}\big|\biggiven\SIkc\right]
    \\
    &\leq
    \frac{K c_w^{2}}{n}\sum_{i\in\cI_k}\E_P\left[\big|\psi_j^2(S_i;\theta_P,\hat{\eta}^{(k)}) - \psi_j^2(S_i;\theta_P,\eta_P) \big|\biggiven\SIkc\right]
    \\
    &\leq
    \frac{2 K c_w^2}{n}\sum_{i\in\mathcal{I}_k}\E_P\left[\big(\psi_j(S_i;\theta_P,\hat{\eta}^{(k)})\big)^2\biggiven\SIkc\right]^{\frac{1}{2}}\E_P\left[\big\{\psi_j(S_i;\theta_P,\hat{\eta}^{(k)})-\psi_j(S_i;\theta_P,\eta_P)\big\}^2\biggiven\SIkc\right]^{\frac{1}{2}}
    \\
    &\qquad + 
    \frac{K c_w^2}{n}\sum_{i\in\mathcal{I}_k}\E_P\left[\big\{\psi_j(S_i;\theta_P,\hat{\eta}^{(k)})-\psi_j(S_i;\theta_P,\eta_P)\big\}^2\biggiven\SIkc\right]
    \\
    & = o_\cP(1),
\end{align*}
by the Cauchy--Schwarz inequality and Lemma~\ref{lem:eta-consistency}. Therefore it follows that $\E_P\big[|C_{\RN{3},j}^{(k)}|\given\SIkc\big]=o_\cP(1)$ and thus $C_{\RN{3},j}^{(k)}=o_\cP(1)$.

For the $C_{\RN{4},j}^{(k)}$ term, first note
\begin{align*}
    &\E_P\left[\big|C_{\RN{4},j}^{(k)}\big|\biggiven\SIkc\right] 
    \\
    &\leq
    \frac{K c_w^{2}}{n}\sum_{i\in\cI_k}\E_P\left[\big|\psi_j^2(S_i;\hat{\theta},\hat{\eta}^{(k)}) - \psi_j^2(S_i;\theta_P,\hat{\eta}^{(k)}) \big| \biggiven \SIkc\right]
    \\
    &\leq
    \frac{2 K c_w^2}{n}\sum_{i\in\mathcal{I}_k}\E_P\left[\psi_j^2(S_i;\theta_P,\hat{\eta}^{(k)}) \biggiven \SIkc\right]^{\frac{1}{2}}\E_P\left[\big\{\psi_j(S_i;\hat{\theta},\hat{\eta}^{(k)})-\psi_j(S_i;\theta_P,\hat{\eta}^{(k)})\big\}^2 \biggiven \SIkc
    \right]^{\frac{1}{2}}
    \\
    &\qquad + 
    \frac{K c_w^2}{n}\sum_{i\in\mathcal{I}_k}\E_P\left[\big\{\psi_j(S_i;\hat{\theta},\hat{\eta}^{(k)})-\psi_j(S_i;\theta_P,\hat{\eta}^{(k)})\big\}^2\biggiven\SIkc\right]
    ,
    \\
    & = 
    \frac{K c_w^2}{n}\sum_{i\in\mathcal{I}_k}\E_P\left[\big\{\psi_j(S_i;\hat{\theta},\hat{\eta}^{(k)})-\psi_j(S_i;\theta_P,\hat{\eta}^{(k)})\big\}^2\biggiven\SIkc\right] + o_\cP(1),
\end{align*}
by Lemma~\ref{lem:eta-consistency}. Further,
\begin{align*}
    &\quad\supP\E_P\left[\big\{\psi_j(S_i;\hat{\theta},\hat{\eta}^{(k)})-\psi_j(S_i;\theta_P,\hat{\eta}^{(k)})\big\}^2\right]
    \\
    &\leq
    \supP\E_P\left[\Big(\sup_{\theta\in\Theta}\big|\nabla_\theta\psi_j(S_i;\theta,\hat{\eta}^{(k)})\big|\Big)^2\big(\hat{\theta}-\theta_P\big)^2\right]
    \\
    &\leq
    n^{-1}\supP\E_P\left[\Big(\sup_{\theta\in\Theta}\big|\nabla_\theta\psi_j(S_i;\theta,\hat{\eta}^{(k)})\big|\Big)^{2+\delta}\right]^{\frac{1}{2+\delta}} 
    \cdot
    \supP\E_P\Big[\big(\sqrt{n}\big(\hat{\theta}-\theta_P\big)\big)^{\frac{2(2+\delta)}{\delta}}\Big]^{\frac{\delta}{2+\delta}}
    \\
    &= o(1),
\end{align*}
by H\"{o}lder's inequality, Assumptions~\ref{ass:bounded-moments} and the asymptotic normality result of Theorem~\ref{thm:theta_est_asymp_norm} alongside Lemma~\ref{lem:unif_root-n-theta-moments-bounded}.

In summary, we have shown  $B_{\RN{1},j},B_{\RN{2},j}^{(k)},B_{\RN{3},j}^{(k)},B_{\RN{4},j}^{(k)},C_{\RN{1},j},C_{\RN{2},j}^{(k)},C_{\RN{3},j}^{(k)},C_{\RN{4},j}^{(k)} = o_\cP(1)$, and thus
\begin{equation*}
    \hat{V} = V^*_P + o_\cP(1).
\end{equation*}
Applying a uniform version of Slutsky's Theorem (Lemma~\ref{lem:unif_slutsky}) to the decomposition~\eqref{eq:V-decomp} completes the proof;
    \begin{equation*}
        \lim_{n\to\infty}\supP\sup_{t\in\R}\left|\PP_P\left( \sqrt{n/\hat{V}}\big(\hat{\theta}-\theta_P\big) \leq t \right)-\Phi(t)\right| = 0,
    \end{equation*}
\end{proof}

\subsubsection{Auxiliary lemmas for the asymptotic Gaussianity results}

\begin{lemma}\label{lem:eta-consistency}
    Under Assumption~\ref{ass:data} (and adopting the notation and setup of Theorems~\ref{thm:theta_est_asymp_norm}~and~\ref{thm:rose_forest_consistency}) the following hold for each $j\in[J]$:
    \begin{enumerate}[label=(\roman*)]
        \item $\E_P\big[\sup_{\theta\in\Theta}|\psi_j(S;\theta,\hat{\eta}) - \psi_j(S;\theta,\eta_P)| \given \hat{\eta}\big] = o_\cP(1)$,
        \item $\E_P\big[\big\{\psi_j(S;\theta_P,\hat{\eta}) - \psi_j(S;\theta_P,\eta_P)|\big\}^2 \given \hat{\eta}\big] = o_\cP(1)$,
        \item $\E_P\big[\big|\nabla_\theta\psi_j(S;\theta_P,\hat{\eta}) - \nabla_\theta\psi_j(S;\theta_P,\eta_P)|\big| \given \hat{\eta}\big] = o_\cP(1)$,
        \item $\E_P\big[|\partial_\theta\psi_j(S;\theta_P,\hat{\eta}) - \partial_\theta\psi_j(S;\theta_P,\eta_P)| \given \hat{\eta}\big] = o_\cP(1)$.
    \end{enumerate}
\end{lemma}

\begin{proof}
    We prove each statement of the lemma in turn. For (i), note first that for each $\theta\in\Theta$ we have
    \begin{align*}
        &\quad
        \E_P\big[\sup_{\theta\in\Theta}|\psi_j(S;\theta,\hat{\eta}(Z))-\psi_j(S;\theta,\eta_P(Z))| \given \hat{\eta}\big]
        \\
        &\leq
        \E_P\big[ \sup_{\theta\in\Theta}\sup_{r\in[0,1]}\big|\nabla_{\eta(z)}\psi_j\big(S;\theta,\eta_P(Z)+r(\hat{\eta}(Z)-\eta_P(Z))\big)^\top  \big(\hat{\eta}(Z)-\eta_P(Z)\big) \big|  \given  \hat{\eta}\big]
        \\
        &\leq
        \Big(\E_P\Big[\sup_{\theta\in\Theta}\sup_{r\in[0,1]}\norm{\nabla_{\eta(z)}\psi_j\big(S;\theta,\eta_P(Z)+r(\hat{\eta}(Z)-\eta_P(Z))\big)}_2^2\given\hat{\eta}\Big]\Big)^{\frac{1}{2}}
        \E_P\big[\norm{\hat{\eta}(Z)-\eta_P(Z)}_2^2\given\hat{\eta}\big]^{\frac{1}{2}}
        \\
        &= o_\cP(1),
    \end{align*}
    by Assumptions~\ref{ass:consistency-and-DR}~and~\ref{ass:bounded-moments}. 

    For (ii) we begin by defining $r_{ijk} := \psi_j(S_i;\theta_P,\hat{\eta}^{(k)})-\psi_j(S_i;\theta_P,\eta_P)$,
\begin{align*}
    &\quad\E_P\left[\big\{\psi_j(S_i;\theta_P,\hat{\eta}^{(k)})-\psi_j(S_i;\theta_P,\eta_P)\big\}^2\biggiven\SIkc\right]
    \\
    &=
    \E_P\left[r_{ijk}^2\ind_{\{|r_{ijk}|> 1\}}\biggiven\SIkc\right] + \E_P\left[r_{ijk}^2\ind_{\{|r_{ijk}|\leq 1\}}\biggiven\SIkc\right]
    \\
    &\leq
    \E_P\left[|r_{ijk}|^{2+\delta}\biggiven\SIkc\right]^{\frac{2}{2+\delta}} \PP_P\left(|r_{ijk}|>1\given\SIkc\right)^{\frac{\delta}{2+\delta}} + \E_P\left[|r_{ijk}|\given\SIkc\right]
    \\
    &\leq
    \E_P\left[|r_{ijk}|^{2+\delta}\biggiven\SIkc\right]^{\frac{2}{2+\delta}} \E_P\left[|r_{ijk}|\given\SIkc\right]^{\frac{\delta}{2+\delta}} + \E_P\left[|r_{ijk}|\given\SIkc\right]
    \\
    &\leq
    2^{2+\delta}\left(\E_P\left[\big|\psi_j(S_i;\theta_P,\hat{\eta}^{(k)})\big|^{2+\delta}\biggiven\SIkc\right] + \E_P\left[\big|\psi_j(S_i;\theta_P,\eta_P)\big|^{2+\delta}\biggiven\SIkc\right]\right)^{\frac{2}{2+\delta}} \E_P\left[|r_{ijk}|\given\SIkc\right]^{\frac{\delta}{2+\delta}}
    \\
    &\qquad + \E_P\left[|r_{ijk}|\given\SIkc\right]
    \\
    &= o_\cP(1),
\end{align*}
by H{\"o}lder's inequality and Assumption~\ref{ass:bounded-moments}. 
Similarly, for (iii) we have that
    \begin{align*}
        &\quad
        \E_P\big[|\partial_\theta\psi_j(S;\theta_P,\hat{\eta}(Z)) - \partial_\theta\psi_j(S;\theta_P,\eta_P(Z))| \given \hat{\eta}\big]
        \\
        &\leq
        \E_P\big[ \sup_{r\in[0,1]}\big|\nabla_{\eta(z)}\partial_{\theta}\psi_j\big(S;\theta_P,\eta_P(Z)+r(\hat{\eta}(Z)-\eta_P(Z))\big)^\top  \big(\hat{\eta}(Z)-\eta_P(Z)\big) \big|  \given  \hat{\eta}\big]
        \\
        &\leq
        \Big(\E_P\Big[\sup_{r\in[0,1]}\norm{\nabla_{\eta(z)}\partial_{\theta}\psi_j\big(S;\theta,\eta_P(Z)+r(\hat{\eta}(Z)-\eta_P(Z))\big)}_2^2\given\hat{\eta}\Big]\Big)^{\frac{1}{2}}
        \E_P\big[\norm{\hat{\eta}(Z)-\eta_P(Z)}_2^2\given\hat{\eta}\big]^{\frac{1}{2}}
        \\
        &= o_\cP(1),
    \end{align*}
    again by Assumptions~\ref{ass:consistency-and-DR}~and~\ref{ass:bounded-moments}. The result of (iii) can be shown similarly.
    
\end{proof}

\subsubsection{Relaxed regularity conditions}\label{appsec:regularity-conditions}

The regularity conditions of Assumption~\ref{ass:data},~\ref{ass:bounded-moments} dictates that there exists some $\delta>0$ such that
\begin{equation*}
    \sup_{\theta\in\Theta}\sup_{r\in[0,1]]}\E_P\big[\big\|D\psi_j(S;\theta,\eta_P(Z)+r\{\hat{\eta}(Z)-\eta_P(Z)\})\big\|_2^{2+\delta}\given\hat{\eta}\big] = O_\cP(1),
\end{equation*}
for all operators $D\in\{\text{id}, \nabla_\theta, \nabla_\theta^2, \partial_\theta\}$. This condition is stronger than required for the results of Theorems~\ref{thm:theta_est_asymp_norm} and~\ref{thm:rose_forest_consistency} (and only relaxed for notational simplicity). In particular, such regularity conditions can be relaxed as follows; there exists some $\delta>0$ such that
\begin{gather*}
    \supP\E_P\big[\sup_{\theta\in\Theta}\psi_j(S;\theta,\eta_P(Z))^2\big] = O(1),
    \quad
    \supP\E_P\big[|\psi_j(S;\theta_P,\eta_P(Z))|^{2+\delta}\big] = O(1),
    \\
    \supP\E_P\big[\partial_\theta\psi_j(S;\theta_P,\eta_P(Z))^2\big] = O(1),
    \quad
    \supP\E_P\big[\nabla_\theta\psi_j(S;\theta_P,\eta_P(Z))^2\big] = O(1),
    \\
    \E_P\big[\sup_{\theta\in\Theta}|\nabla_\theta^2\psi_j(S;\theta,\eta_P(Z))|\given\hat{\eta}\big] = O_\cP(1),
    \quad
    \E_P\big[\sup_{\theta\in\Theta}|\nabla_\theta^2\psi_j(S;\theta,\hat{\eta}(Z))|\given\hat{\eta}\big] = O_\cP(1),
    \\
    \E_P\big[\sup_{\theta\in\Theta}|\nabla_\theta\psi_j(S;\theta,\hat{\eta}(Z))|^{2+\delta}\given\hat{\eta}\big]=O_\cP(1),
    \\
    \E_P\Big[\sup_{\theta\in\Theta}\sup_{r\in[0,1]}\big\|\nabla_{\eta(z)}\psi_j\big(S;\theta,\eta_P(Z)+r(\hat{\eta}(Z)-\eta_P(Z))\big)\big\|_2^2\biggiven\hat{\eta}\Big] = O_\cP(1),
    \\
    \E_P\Big[\sup_{r\in[0,1]}\big\|\nabla_{\eta(z)}\nabla_\theta\psi_j\big(S;\theta_P,\eta_P(Z)+r(\hat{\eta}(Z)-\eta_P(Z))\big)\big\|_2^2\biggiven\hat{\eta}\Big] = O_\cP(1),
    \\
    \E_P\Big[\sup_{r\in[0,1]}\big\|\nabla_{\eta(z)}\partial_\theta\psi_j\big(S;\theta_P,\eta_P(Z)+r(\hat{\eta}(Z)-\eta_P(Z))\big)\big\|_2^2\biggiven\hat{\eta}\Big] = O_\cP(1).
\end{gather*}

\subsection{Comments on sufficient conditions for doubly robust rates}\label{appsec:DR}

We present relevant sufficient conditions for the DR property of Assumption~\ref{ass:data},~\ref{ass:consistency-and-DR} to hold. In particular, in the partially linear model studied in Section~\ref{sec:PLMmodel} with nuisance parameterisation $\eta_P=(f_P,m_{P,1},\ldots,m_{P,J})$ it is well known the DR condition~\ref{ass:consistency-and-DR} holds when for each $j\in[J]$
\begin{equation}\label{eq:prod-rates}
    \E_P\Big[\big(\hat{m}_j(Z)-m_{P,j}(Z)\big)^2\Biggiven\hat{m}_j\Big]\E_P\Big[\big(\hat{f}(Z)-f_P(Z)\big)^2\Biggiven\hat{f}\Big] = o_\cP\big(n^{-1}\big).
\end{equation}
This form of double robustness only asks for the product of the mean squared errors of $\hat{m}_j$ and $\hat{f}$ to converge at faster than the rate $n^{-1}$. Rather, one of $m_{P,j}$ and $f_P$ may be estimated at relatively slow rates (e.g.~with mean squared error converging at slower than $n^{-\frac{1}{2}}$ rates), provided the other estimator compensates with faster rates. 
We remark that for estimators in the partially parametric model of the form~\eqref{eq:semipara-model} with $\eta_P=(f_P,m_{P,1},\ldots,m_{P,J})$~\eqref{eq:prod-rates} is not necessarily sufficient for the doubly robust rates of Assumption~\ref{ass:consistency-and-DR}. We outline weaker sufficient conditions than those presented in Section~\ref{sec:asymp_norm} for the doubly robust rates of Assumption~\ref{ass:consistency-and-DR}, that also encapsulates this stronger case~\eqref{eq:prod-rates} for the partially linear model. Whilst some of the conditions here appear rather involved, we later explain two simple examples of the generalised partially linear model where the conditions of Proposition~\ref{prop:DR-weak-cond}(a) are satisfied (see Corollary~\ref{cor:DR-weak-cond}). 
In the following results for ease of exposition we introduce the following notation. Given an estimator $\hat{f}$ and $r\in[0,1]$ we write
\begin{equation*}
    \tilde{\mu}_r(X,Z) := \mu(X,Z;\theta_P,f_P(Z)+r\{\hat{f}(Z)-f_P(Z)\}).
\end{equation*}
We assume for all $(x,z)\in\cX\times\cZ$ the map $\gamma\mapsto\mu(x,z;\theta,\gamma)$ on the domain $\R$ is continuously differentiable, and denote for each $q\in\mathbb{N}$,
\begin{equation*}
    \tilde{\mu}_r^{(q)}(X,Z)
    :=\frac{\partial^q\mu(X,Z;\theta_P,\gamma)}{\partial \gamma^q}\bigg|_{\gamma=f_P(Z)+r\{\hat{f}(Z)-f_P(Z)\}}.
\end{equation*}
It will also help to define the functions
\begin{align*}
    \varphi_r(X,Z) &:= \frac{\tilde{\mu}^{(2)}_r(X,Z)}{\tilde{\mu}^{(1)}_r(X,Z)} 
    +
    \frac{2\{\tilde{\mu}_r^{(2)}(X,Z)\}^2-\tilde{\mu}_r^{(3)}(X,Z)\tilde{\mu}_r^{(1)}(X,Z)}{\{\tilde{\mu}_r^{(1)}(X,Z)\}^3}
    \big(\tilde{\mu}_r(X,Z)-\tilde{\mu}_0(X,Z)\big),
    \\
    \zeta_r(X,Z) &:= \frac{\tilde{\mu}_r^{(2)}(X,Z)}{\{\tilde{\mu}_r^{(1)}(X,Z)\}^2}
    \big(\tilde{\mu}_r(X,Z)-\tilde{\mu}_0(X,Z)\big).
\end{align*}
Note in particular $\varphi_r$ and $\zeta_r$ each carry dependence on $\hat{f}$. 
We also take $C_g, C_\sigma > 0$ to be finite positive constants. Note in the following conditional expectations if $\hat{f}$ and $\hat{m}_j$ are estimated using the same dataset then conditioning on $\hat{\eta}$, $\hat{f}$ and $\hat{m}_j$ are all equivalent.
\begin{proposition}\label{prop:DR-weak-cond}
    Consider the partially parametric model with $\gamma\mapsto\mu(x,z;\theta_P,\gamma)$ continuously differentiable for all $(x,z)\in\cX\times\cZ$, and with
    
    \begin{equation*}
        \big|\E_P\big[\varphi_r(X,Z)\given\hat{f},Z\big]\big|
        \vee
        \big|\E_P\big[\zeta_r(X,Z)\given\hat{f},Z\big]\big|
        \leq
        C_g,
    \end{equation*}
    almost surely. 
    Suppose also one of the following two cases holds:
    \begin{itemize}
        \item[(a)] For each $z\in\cZ$ and $r\in[0,1]$ the map $x\mapsto \varphi_r(x,z)$ on $\cX$ is constant. 
        \item[(b)] $\Var_P(M_j(X)\given Z)\leq C_\sigma$ almost surely, and $\E_P[\varphi_r(X,Z)^2\given\hat{f},Z]\leq C_g^2$.
    \end{itemize}
    Suppose the nuisance estimators $\hat{\eta}=(\hat{m}_1,\ldots,\hat{m}_j,\hat{f})$ satisfy for each $j\in[J]$,
    \begin{gather*}
        \E_P\big[\|\hat{m}_j(Z) - m_{P,j}(Z)\|^{4+\delta}\given\hat{m}_j\big]
        ,\;\,
        \E_P\big[\|\hat{f}(Z) - f_P(Z)\|^{4+\delta}\given\hat{f}\big]
        =O_\cP(1)
        ,
        \\
        \E_P\big[\big|\{\hat{m}_j(Z)-m_{P,j}(Z)\}\{\hat{f}(Z)-f_P(Z)\}\big|\biggiven\hat{\eta}\big] = o_\cP\big(n^{-\frac{1}{2}}\big)
        ,\notag
    \end{gather*}
    for some $\delta>0$. 
    Additionally, but only for case (b), suppose
    \begin{equation*}
    \makebox[0pt][l]{\hspace{-4.35cm}(b)}    \E_P\big[\{\hat{f}(Z)-f_P(Z)\}^2\given\hat{f}\,\big] = o_\cP\big(n^{-\frac{1}{2}}\big).
    \end{equation*}
    Then the DR property of Assumption~\ref{ass:data},~\ref{ass:consistency-and-DR} holds.
\end{proposition}

\begin{proof}
    Without loss of generality take $J=1$, $m_P := m_{P,1}$ and $\eta_P=(m_P,f_P)$. 
    Recall the function $\psi_1$ takes the form
    \begin{equation*}
        \psi_1\big(S;\theta,\eta(Z)\big) := \big(M_1(X)-m_P(Z)\big)\frac{Y-\mu(X,Z;\theta,\eta(Z))}{\mu'(X,Z;\theta,\eta(Z))}.
    \end{equation*}
    Define
    \begin{align*}
        \tilde{\xi}_r(X,Z) &:= (M_1(X) - m_P(Z)) - r\{\hat{m}(Z)-m_P(Z)\},
        \\
        \xi &:= \tilde{\xi}_0(X,Z) = M_1(X)-m_P(Z),
        \\
        \tilde{\epsilon}_r(S) &:= Y - \tilde{\mu}_r(X,Z),
    \end{align*}
    recalling $S:=(Y,X,Z)$, and noting we suppress a number of dependencies in the above definition e.g.~on $\hat{\eta}$ and $P$. Then direct calculations give
    \begin{multline*}
        \nabla_{\eta(z)}^2\psi_1\big(S;\theta_P,\eta_P(Z)+r(\hat{\eta}(Z)-\eta_P(Z))\big)
        =
        \begin{pmatrix}
            0 & 1 \\ 1 & 0
        \end{pmatrix}
        +
        \begin{pmatrix}
            0 & 1 \\ 1 & 0
        \end{pmatrix}
        \frac{\tilde{\mu}_r^{(2)}(X,Z)}{\big(\tilde{\mu}_r^{(1)}(X,Z)\big)^2}
        \tilde{\epsilon}_r(S)
        \\
        + 
        \begin{pmatrix}
            0 & 0 \\ 0 & 1
        \end{pmatrix}
        \bigg(\frac{\tilde{\mu}_r^{(2)}(X,Z)}{\tilde{\mu}_r^{(1)}(X,Z)} + \frac{2\{\tilde{\mu}_r^{(2)}(X,Z)\}^2-\tilde{\mu}_r^{(3)}(X,Z)\tilde{\mu}_r^{(1)}(X,Z)}{\{\tilde{\mu}_r^{(1)}(X,Z)\}^3}
        \tilde{\epsilon}_r(S)\bigg)\tilde{\xi}_r(X,Z).
    \end{multline*}
    Thus
    \small
    \begin{align*}
        &
        \sup_{r\in[0,1]}\E_P\Big[\{\hat{\eta}(Z)-\eta_P(Z)\}^\top\E\big[\nabla_{\eta(z)}^2\psi_1\big(S;\theta_P,\eta_P(Z)+r\{\hat{\eta}(Z)-\eta_P(Z)\}\big) \given \hat{\eta},Z\big]
        \{\hat{\eta}(Z)-\eta_P(Z)\}\Biggiven\hat{\eta}\Big]
        \\
        &= 
        \sup_{r\in[0,1]}\bigg\{
        2\,\underbrace{\E_P\big[\{\hat{m}(Z)-m_P(Z)\}\{\hat{f}(Z)-f_P(Z)\}\biggiven\hat{\eta}\big]}_{=:R_{\RN{1}}}
        \\
        &\quad
        +
        2\,\underbrace{\E_P\bigg[\frac{\tilde{\mu}_r^{(2)}(X,Z)}{\{\tilde{\mu}_r^{(1)}(X,Z)\}^2}\tilde{\epsilon}_r(S)\{\hat{m}(Z)-m_P(Z)\}\{\hat{f}(Z)-f_P(Z)\}\bigggiven\hat{\eta}\bigg]}_{=:R_{\RN{2}}}
        \\
        &\quad
        +
        \underbrace{\E_P\bigg[\bigg(\frac{\tilde{\mu}_r^{(2)}(X,Z)}{\tilde{\mu}_r^{(1)}(X,Z)} + \frac{2\{\tilde{\mu}_r^{(2)}(X,Z)\}^2-\tilde{\mu}_r^{(3)}(X,Z)\tilde{\mu}_r^{(1)}(X,Z)}{\{\tilde{\mu}_r^{(1)}(X,Z)\}^3}
        \tilde{\epsilon}_r(S)\bigg)
        \tilde{\xi}_r(X,Z)\{\hat{f}(Z)-f_P(Z)\}^2\bigggiven\hat{\eta}\bigg]}_{=:R_{\RN{3}}}\bigg\}
    \end{align*}
    \normalsize
    The first term $R_{\RN{1}}$ in this expansion is $o_\cP(n^{-\frac{1}{2}})$ by assumption. For the second term $R_{\RN{2}}$,
        \begin{align*}
        &\quad
        \bigg|\E_P\bigg[\frac{\tilde{\mu}_r^{(2)}(X,Z)}{\{\tilde{\mu}_r^{(1)}(X,Z)\}^2}\tilde{\epsilon}_r(S)\bigggiven\hat{\eta},Z\bigg]\bigg|
        \\
        &=
        \bigg|\E_P\bigg[\frac{\tilde{\mu}_r^{(2)}(X,Z)}{\{\tilde{\mu}_r^{(1)}(X,Z)\}^2}\Big(\underbrace{\E_P\big[Y-\tilde{\mu}_0(X,Z)\given X,Z\big]}_{=0}-\big(\tilde{\mu}_r(X,Z)-\tilde{\mu}_0(X,Z)\big)\Big)\bigggiven\hat{f},Z\bigg]\bigg|
        \\
        &=
        \big|\E_P\big[\zeta_r(X,Z)\biggiven\hat{f},Z\big]\big|
        \leq C_g ,
    \end{align*}
    and so $R_{\RN{2}} = o_\cP(n^{-\frac{1}{2}})$ in either case.

    To control the term $R_{\RN{3}}$, first notice
    \begin{align*}
        &\quad
        \E_P\bigg[\bigg(\frac{\tilde{\mu}_r^{(2)}(X,Z)}{\tilde{\mu}_r^{(1)}(X,Z)} + \frac{2\{\tilde{\mu}_r^{(2)}(X,Z)\}^2-\tilde{\mu}_r^{(3)}(X,Z)\tilde{\mu}_r^{(1)}(X,Z)}{\{\tilde{\mu}_r^{(1)}(X,Z)\}^3}
        \tilde{\epsilon}_r(S)\bigg)
        \tilde{\xi}_r(X,Z)\bigggiven\hat{\eta},Z\bigg]
        \\
        &= \E_P\Big[\varphi_r(X,Z)\tilde{\xi}_r(X,Z)\Biggiven\hat{\eta},Z\Big]
        \\
        &\qquad +
        \E_P\bigg[\frac{2(\tilde{\mu}_r^{(2)}(X,Z))^2-\tilde{\mu}_r^{(3)}(X,Z)\tilde{\mu}_r^{(1)}(X,Z)}{\{\tilde{\mu}_r^{(1)}(X,Z)\}^3}\underbrace{\E_P[Y-\tilde{\mu}_0(X,Z)\given X,Z]}_{=0}\tilde{\xi}_r(X,Z)\bigggiven\hat{\eta},Z\bigg]
        \\
        &=
        \E_P\Big[\varphi_r(X,Z)\tilde{\xi}_r(X,Z)\Biggiven\hat{\eta},Z\Big]
        ,
    \end{align*}
    where the first inequality follows due to identity $\tilde{\epsilon}_r(S) = \{Y-\tilde{\mu}_0(X,Z)\} - \{\tilde{\mu}_r(X,Z)-\tilde{\mu}_0(X,Z)\}$ and by the definition of $\varphi_r$. 
    Therefore
    \begin{align*}
        R_{\RN{3}} &=
        \E_P\Big[\varphi_r(X,Z)\tilde{\xi}_r(X,Z)\{\hat{f}(Z)-f_P(Z)\}^2\Biggiven\hat{\eta}\Big]
        \\
        &=
        \underbrace{\E_P\Big[\varphi_r(X,Z)\xi\{\hat{f}(Z)-f_P(Z)\}^2\Biggiven\hat{\eta}\Big]}_{=:R_{\RN{3}a}}
        \\
        &\qquad
        -r
        \underbrace{\E_P\Big[\varphi_r(X,Z)\{\hat{m}(Z)-m_P(Z)\}\{\hat{f}(Z)-f_P(Z)\}^2\Biggiven\hat{\eta}\Big]}_{=:R_{\RN{3}b}}.
    \end{align*} 
    For the latter term $R_{\RN{3}b}$,
    \begin{align*}
        |R_{\RN{3}b}|&=
        \big|\E_P\big[\E_P[\varphi_r(X,Z)\given\hat{f},Z]\{\hat{m}(Z)-m_P(Z)\}\{\hat{f}(Z)-f_P(Z)\}^2\given\hat{\eta}\big]\big|
        \\
        &\leq
        \E_P\big[|\E_P[\varphi_r(X,Z)\given\hat{f},Z]|\cdot|\hat{m}(Z)-m_P(Z)|\{\hat{f}(Z)-f_P(Z)\}^2\given\hat{\eta}\big]
        \\
        &\leq 
        C_g \, \E_P\big[|\hat{m}(Z)-m_P(Z)|\{\hat{f}(Z)-f_P(Z)\}^2\given\hat{\eta}\big]
        \\
        &\leq
        C_g \,
        \E_P\big[\{\hat{m}(Z)-m_P(Z)\}^2\given\hat{m}\big]^{\frac{1}{2}}
        \E_P\big[\{\hat{f}(Z)-f_P(Z)\}^2\given\hat{f}\big]^{\frac{1}{2}}
        \\
        &\qquad\cdot
        \E_P\big[\{\hat{f}(Z)-f_P(Z)\}^{4+\delta}\given\hat{f}\big]^{\frac{1}{2+\delta}}
        \\
        &= o_\cP(n^{-\frac{1}{2}}),
    \end{align*}
    where in the final inequality we make use of the following inequality: for non-negative real-valued random variables $A$ and $B$,
    \begin{equation*}
        \E_P\big[AB^2\big] 
        \leq
        \E_P\Big[A^{\frac{2+\delta}{1+\delta}}B^{\frac{\delta}{1+\delta}}\Big]^{\frac{1+\delta}{2+\delta}} \E_P\big[B^{4+\delta}\big]^{\frac{1}{2+\delta}}
        \leq
        \E_P\big[A^2\big]^{\frac{1}{2}}\E_P\big[B^2\big]^{\frac{1}{2}}\E_P\big[B^{4+\delta}\big]^{\frac{1}{2+\delta}},
    \end{equation*}
    which obtained by twice use of H\"{o}lder's inequality. 
    To control the term $R_{\RN{3}a}$ we consider cases (a) and (b) separately.

    \smallskip\noindent{\bf Case (a):}
    As $(x,z)\mapsto\varphi_r(x,z)$ carries no $x$-dependence it follows $\E_P[\varphi_r(X,Z)\given\hat{\eta},Z] = \varphi_r(X,Z)$ and so
    \begin{equation*}
        \E_P\big[\varphi_r(X,Z)\xi\given\hat{\eta},Z\big] 
        =
        \varphi_r(X,Z)\,\E_P[\xi\given Z] 
        = 0,
    \end{equation*}
    and thus $R_{\RN{3}a}=0$.

    \smallskip\noindent{\bf Case (b):}
    By the Cauchy--Schwarz inequality
    \begin{equation*}
        \E_P\big[\varphi_r(X,Z)\xi\given\hat{\eta},Z\big]^2
        \leq
        \E_P\big[\varphi_r(X,Z)^2\given\hat{\eta},Z\big]
        \,
        \Var_P(X\given Z)
        \leq
        C_g^2 C_\sigma,
    \end{equation*}
    almost surely. Thus
    \begin{equation*}
        \big|R_{\RN{3}a}\big| \leq C_gC_\sigma^{\frac{1}{2}}\,
        \E_P\big[\{\hat{f}(Z)-f_P(Z)\}^2\given\hat{f}\big]
        = o_\cP(n^{-\frac{1}{2}})
    \end{equation*}

    \smallskip\noindent
    Therefore, in both cases (a) and (b), $R_{\RN{3}a}=o_\cP(n^{-\frac{1}{2}})$ and so alongside $R_{\RN{3}b}=o_\cP(n^{-\frac{1}{2}})$ it follows $R_{\RN{3}}=o_\cP(n^{-\frac{1}{2}})$.
    
    We have shown that $R_{\RN{1}}, R_{\RN{2}}, R_{\RN{3}} = o_\cP(n^{-\frac{1}{2}})$ and thus 
    \begin{multline*}
        \sup_{r\in[0,1]}\E_P\Big[\{\hat{\eta}(Z)-\eta_P(Z)\}^\top\E\big[\nabla_{\eta(z)}^2\psi_1\big(S;\theta_P,\eta_P(Z)+r\{\hat{\eta}(Z)-\eta_P(Z)\}\big) \given \hat{\eta},Z\big]
        \\
        \cdot \{\hat{\eta}(Z)-\eta_P(Z)\}\Big]
        = o_\cP(n^{-\frac{1}{2}}),
    \end{multline*}
    as required.
\end{proof}    

The conditions of Proposition~\ref{prop:DR-weak-cond} are rather involved, but provide general conditions under which the doubly robust condition~\ref{ass:consistency-and-DR} holds. The assumptions required for (b) are all relatively weak boundedness conditions. In contrast case (a) imposes rather strong assumptions on the model; we show these indeed hold for two commonly used cases of the generalised partially linear model.

\begin{corollary}\label{cor:DR-weak-cond}
    The results of Proposition~\ref{prop:DR-weak-cond} hold for the generalised partially linear model~\eqref{eq:gplm} with assumption (a) replaced with either of the following:
    \begin{itemize}
        \item[(a.i)] The partially linear model holds i.e. $g(\mu)=\mu$.
        \item[(a.ii)] The generalised partially linear model with log link holds i.e. $g(\mu)=\log\mu$, with additionally $\hat{f}(Z)\geq f_P(Z)-C_f$ almost surely, for some positive constant $C_f$.
    \end{itemize}
\end{corollary}

\begin{proof}
    In the case of the generalised partially linear model
    \begin{equation*}
        \mu(X,Z;\theta_P,f_P(Z)) = g^{-1}(X\theta_P + f_P(Z)),
    \end{equation*}
    and so
    \begin{equation*}
        \tilde{\mu}_r(X,Z) =  g^{-1}(X\theta_P + f_P(Z)+r\{\hat{f}(Z)-f_P(Z)\})
    \end{equation*}
    Direct calculations show that
    \begin{align*}
        \varphi_r(X,Z) &= -\frac{g''(\tilde{\mu}_r(X,Z))}{\{g'(\tilde{\mu}_r(X,Z))\}^2}
        \\
        &\qquad
        +
        \frac{g'''(\tilde{\mu}_r(X,Z))g'(\tilde{\mu}_r(X,Z))-\{g''(\tilde{\mu}_r(X,Z))\}^2}{\{g'(\tilde{\mu}_r(X,Z))\}^3}
        \big(\tilde{\mu}_r(X,Z)-\tilde{\mu}_0(X,Z)\big)
        ,
        \\
        \zeta_r(X,Z) &= 
        -\frac{g''(\tilde{\mu}_r(X,Z))}{g'(\tilde{\mu}_r(X,Z))}\big(\tilde{\mu}_r(X,Z)-\tilde{\mu}_0(X,Z)\big),
    \end{align*}
    and so
    \begin{gather*}
        \varphi_r(X,Z) = \begin{cases}
            0
            &\;\text{when $g=\text{id}$ (a.i)}
            \\
            2 - \exp(-r\{\hat{f}(Z)-f_P(Z)\}) 
            &\;\text{when $g=\log$ (a.ii)}
        \end{cases},
        \\
        \zeta_r(X,Z) = \begin{cases}
            0&\;\text{when $g=\text{id}$ (a.i)}
            \\
            1-\exp(-r\{\hat{f}(Z)-f_P(Z)\})
            &\;\text{when $g=\log$ (a.ii)}
        \end{cases}.
    \end{gather*}
    
    In both cases it is then easy to see that the condition of (a) hold i.e.~$\varphi_r(x,z)$ and $\zeta_r(x,z)$ carry no $x$-dependence. Further, 
    \begin{equation*}
        \big|\E_P\big[\varphi_r(X,Z)\given\hat{f},Z\big]\big|
        \vee
        \big|\E_P\big[\zeta_r(X,Z)\given\hat{f},Z\big]\big|
        \begin{cases}
            =0 &\;\text{when $g=\text{id}$ (a.i)}
            \\
            \leq 2\vee |\exp(C_f)-1| &\;\text{when $g=\text{id}$ (a.ii)}
        \end{cases}.
    \end{equation*}
\end{proof}

\section{Supplementary Information for Section~\ref{sec:rrfs}}

In this section we cover supplementary material regarding the ROSE random forests explored in Section~\ref{sec:rrfs}. In Appendix~\ref{appsec:rose-weights} the optimal ROSE weights are verified to take the form~\eqref{eq:w_j}. Appendix~\ref{appsec:div-ratio-gplm} proves the result of Theorem~\ref{thm:div-ratio-gplm}. In Appendix~\ref{appsec:rose-complexity} we briefly discuss favourable computational properties of ROSE random forests, and prove consistency in Appendix~\ref{appsec:consistency-proof-first}. In Appendix~\ref{appsec:rose-forest-J>1} we generalised the ROSE random forest presented in Appendix~\ref{sec:rose-forests} for $J=1$ to accommodate $J>1$.

\subsection{Comments on ROSE weighting scheme}\label{appsec:rose-weights}

We prove that the weights~\eqref{eq:w_j} minimise the sandwich loss~\eqref{eq:sand-loss} as claimed.

\begin{proposition}
    The weighting scheme $w_P^{\rose} = (w_{P,1}^{\rose},\ldots,w_{P,J}^{\rose})$ that minimises the sandwich loss
    \begin{equation*}
	L_{P,\SL}(w) 
	:=\;   \frac{\E_P\left[\big(\sum_{j=1}^J w_{j}(Z)\psi_j(S;\theta_P,\eta_P(Z))\big)^2\right]}{\big(\sum_{j=1}^J\E_P\big[w_{j}(Z)\partial_\theta\psi_j(S;\theta_P,\eta_P(Z))\big]\big)^2}.
    \end{equation*}
is (up to an arbitrary positive constant of proportionality) given by
\begin{gather*}
    w_{P,j}^{\rose}(Z) := e_j^\top 
    \E_P\big[\bm{\psi}(S;\theta_P,\eta_P(Z)) \bm{\psi}(S;\theta_P,\eta_P(Z))^\top \given Z\big]^{-1}
    \E_P\big[\bm{\partial_\theta\psi}(S;\theta_P,\eta_P(Z)) \given Z\big],
    \\
    \bm{\psi} = (\psi_1,\psi_2,\ldots,\psi_J),
    \quad
    \text{where} \;\;\bm{\partial_\theta\psi} = (\partial_\theta\psi_1,\partial_\theta\psi_2,\ldots,\partial_\theta\psi_J).
\end{gather*}
\end{proposition}

\begin{proof}
We notate the vector function $\cZ\to\R^J$ by $w = (w_1, w_2, \ldots, w_J)$, and define the matrix and vector functions $A:\cZ\to\R^{J\times J}$ and $b:\cZ\to\R^J$,
\begin{align*}
    A_P(Z) &:= \E_P\big[\bm{\psi}(S;\theta_P,\eta_P(Z)) \bm{\psi}(S;\theta_P,\eta_P(Z))^\top \given Z\big],
    \\
    b_P(Z) &:= \E_P\big[\bm{\partial_\theta\psi}(S;\theta_P,\eta_P(Z)) \given Z\big].
\end{align*}
Then
\begin{multline*}
    L_{P,\SL}(w) 
    :=\;   \frac{\E_P\left[\big(\sum_{j=1}^J w_{j}(Z)\psi_j(S;\theta_P,\eta_P(Z))\big)^2\right]}{\big(\sum_{j=1}^J\E_P\big[w_{j}(Z)\partial_\theta\psi_j(S;\theta_P,\eta_P(Z))\big]\big)^2}
    = 
    \frac{\E_P\Big[\big(\bm{\psi}(S;\theta_P,\eta_P(Z))^\top w(Z)\big)^2\Big]}{\big(\E_P\big[\bm{\partial_\theta\psi}(S;\theta_P,\eta_P(Z))^\top w(Z)\big]\big)^2}
    \\
    =
    \frac{\E_P\big[w(Z)^\top A_P(Z) w(Z)\big]}{\big(\E_P\big[b_P(Z)^\top w(Z)\big]\big)^2}
    \geq
    \big(\E_P\big[b_P(Z)^\top 
    A_P(Z)^{-1}b_P(Z)\big]\big)^{-1}
\end{multline*}
with equality if and only if
\begin{equation*}
    w(Z) \propto \E_P\big[\bm{\psi}(S;\theta_P,\eta_P(Z)) \bm{\psi}(S;\theta_P,\eta_P(Z))^\top \given Z\big]^{-1}
    \E_P\big[\bm{\partial_\theta\psi}(S;\theta_P,\eta_P(Z)) \given Z\big].
\end{equation*}
This follows by the Cauchy--Schwarz inequality, obtained by noticing that the quadratic
\begin{equation*}
    t \mapsto \E_P\Big[\big\|A_P(Z)^{\frac{1}{2}}\big(w(Z) - t\,A_P(Z)^{-1}b_P(Z)\big)\big\|_2^2\Big],
\end{equation*}
has non-positive discriminant.

\end{proof}

\subsection{Proof of Theorem~\ref{thm:div-ratio-gplm}}\label{appsec:div-ratio-gplm}

\begin{proof}[Proof of Theorem~\ref{thm:div-ratio-gplm}]
    Define $\xi := X - m_0(Z)$ and $\mu_0(X,Z):=g^{-1}(X\theta_0+f_0(Z))$. 
    In the generalised linear model context of Theorem~\ref{thm:div-ratio-gplm} for any law $P$ satisfying $\E_P[X\given Z]=m_0(Z)$ and $\E_P[Y\given X,Z]=\mu_0(X,Z)$, the function $\psi_1$ takes the form
    \begin{equation*}
        \psi_1(S;\theta_0,\eta_0(Z)) = g'(\mu_0(X,Z)) \, \xi \, (Y-\mu_0(X,Z)).
    \end{equation*}
    Therefore, as in~\eqref{eq:sand-loss}, with $w=w_1$
    \begin{equation*}
        L_{P,\SL}(w) = \Big(\E_P\big[w(Z)\Var_P(X\given Z)\big]\Big)^{-2}\Big(\E_P\big[w^2(Z)\{g'(\mu_0(X,Z))\}^2\xi^2\Var_P(Y\given X,Z)\big]\Big),
    \end{equation*}
    and so
    \begin{align*}
        V_{P,\mathrm{loceff}} &:= \bigg(\E_P\bigg[\frac{\Var_P(X\given Z)}{\E_P[\{g'(g^{-1}(X\theta_0+f_0(Z)))\}^2\Var_P(Y\given X,Z)\given Z]}\bigg]\bigg)^{-2}
        \\
        &\qquad\cdot\bigg(\E_P\bigg[\frac{\E_P[\xi^2\{g'(g^{-1}(X\theta_0+f_0(Z)))\}^2\Var_P(Y\given X,Z)\given Z]}{\{\E_P[\{g'(g^{-1}(X\theta_0+f_0(Z)))\}^2 \Var_P(Y\given X,Z) \given Z]\}^2}\bigg]\bigg)
        ,
        \\
        V_{P,\mathrm{rose}} &:= \bigg(\E_P\bigg[\frac{\{\Var_P(X\given Z)\}^2}{\E_P[\xi^2\{g'(g^{-1}(X\theta_0+f_0(Z)))\}^2\Var_P(Y\given X,Z)\given Z]}\bigg]\bigg)^{-1}
        ,
        \\
        V_{P,\mathrm{unw}} &:= \Big(\E_P[\Var_P(X\given Z)]\Big)^{-2} \Big(\E_P[\xi^2\{g'(g^{-1}(X\theta_0+f_0(Z)))\}^2\Var_P(Y\given X,Z)]\Big)
        .
    \end{align*}  
We proceed to construct a distribution $P_0$ satisfying the properties in Theorem~\ref{thm:div-ratio-gplm} as follows. 
Denote the given marginal distribution for $Z$ by $P_Z$. First fix an arbitrary $\beta\in(0.5,1)$ and $\alpha\in[\beta,1)$. As $Z$ is almost surely non-constant there exists some measurable space $\cZ_1$ and $\cZ_2:=\cZ\backslash\cZ_1$ such that $q_1:=\PP_{P_Z}(Z\in\cZ_1)\in(0,1)$; also define the constant $q_2:=1-q_1$. Now define the functions
\begin{align*}
    a(Z) &:= \E_{X\given Z\sim N(m_0(Z),\,p(Z)^{-\beta}),\,Z\sim P_Z}\big[\{g'(g^{-1}(X\theta_0+f_0(Z)))\}^{-2}\given Z\big],
    \\
    b(Z) &:= \{g'(g^{-1}(m_0(Z)\theta_0+f_0(Z)))\}^{-2},
\end{align*}
noting that both are positive-valued, and so the constants
\begin{gather*}
    c := \E_{P_Z}[a(Z)\given Z\in\cZ_1],
    \qquad
    c' := \E_{P_Z}[b(Z)\given Z\in\cZ_1],
    \qquad
    c'' := \E_{P_Z}[b(Z)^{-1}\given Z\in\cZ_1],
    \\
    m_2^{(1)} := \E_{P_Z}[a(Z)^{-1}\given Z\in\cZ_2],
    \qquad
    m_2^{(-1)} := \E_{P_Z}[a(Z)\given Z\in\cZ_2]
    ,
\end{gather*}
are also all positive. 
The map
\begin{equation*}
    \tilde{\zeta}\mapsto\E_{P_Z}\big[\{a(Z)\tilde{\zeta}+b(Z)(1-\tilde{\zeta})\}^{-1}\given Z\in\cZ_1 \big],
\end{equation*}
is continuous, and so there exists some $\bar{\zeta}>0$ such that for all $\tilde{\zeta}\in[0,\bar{\zeta}]$
\begin{equation*}
    2^{-1}c'' \leq \E_{P_Z}\big[\{a(Z)\tilde{\zeta}+b(Z)(1-\tilde{\zeta})\}^{-1}\given Z\in\cZ_1 \big] \leq 2c''.
\end{equation*}
Take
\begin{multline*}
    \zeta := 
    1
    \wedge
    \bar{\zeta}
    \wedge
    \big(2^{-1}A^{-1}q_1q_2 c''m_2^{(-1)}\big)^{\frac{1}{\alpha+\beta-1}}
    \wedge
    \big(q_1^{-1}q_2(c+c')^{-1}m_2^{(-1)}\big)^{\frac{1}{\alpha-\beta}}
    \\
    \wedge
    \big(2^{-1}\{2c''q_1+m_2^{(-1)}q_2\}^{-2}A^{-1}q_2^2\big)^{\frac{1}{2(1-\alpha)}}
    ,
\end{multline*}
and $p(z) := \zeta\ind_{\cZ_1}(z) + \ind_{\cZ_2}(z)$. Define the conditional distribution $P_{X|Z}$ on $X\given Z$ by
\begin{equation}\label{eq:B|Z-X|ZB}
    B\given Z \sim \text{Ber}(p(Z)),
    \qquad
    X\given (Z,B) \sim N\big( m_0(Z) , p^{-\beta}(Z)B \big).
\end{equation}
Then define the conditional distribution for $(Y,X)\given Z$ by~\eqref{eq:B|Z-X|ZB} and
\begin{equation*}
    Y\given (X,Z,B) \sim N\big( g^{-1}(X\theta_0+f_0(Z)) ,\, \{g'(g^{-1}(X\theta_0+f_0(Z)))\}^{-2}\bar{\varphi}(Z)p^{-\alpha}(Z)B \big),
\end{equation*}
where
\begin{equation*}
    \bar{\varphi}(Z) := 
    \E_{P_{X|Z}}\big[\{g'(g^{-1}(X\theta_0+f_0(Z)))\}^{-2}\given Z\big]^{-1}
\end{equation*}
Together with $P_Z$, this implies a joint distribution $P_0$ for the triple $(Y,X,Z)$. It then follows by construction that $\E_{P_0}[Y\given X,Z]=g^{-1}(X\theta_0+f_0(Z))$ and $\E_{P_0}[X\given Z]=m_0(Z)$, in addition to
    \begin{align*}
        \Var_{P_0}(X\given Z) &= \E_{P_0}[\Var_{P_0}(X\given Z,B)\given Z] + \Var_{P_0}(\E_{P_0}[X\given Z,B]\given Z)
        \\
        &= \E_{P_0}[Bp^{-\beta}(Z)\given Z] + \Var_{P_0}(\E_{P_0}[X\given Z]\given Z)
        \\
        &= p^{1-\beta}(Z),
    \end{align*}
    and
    \begin{align*}
        \E_{P_0}\big[\xi^2\{g'(g^{-1}(X\theta_0+f_0(Z)))\}^2\Var_{P_0}(Y\given X,Z)\given Z\big]
        &=
        \E_{P_0}\big[\Var_{P_0}(X\given Z,B)\bar{\varphi}(Z)p^{-\alpha}(Z)B\given Z\big]
        \\ 
        &= \E_{P_0}\big[B^2p^{-\alpha-\beta}(Z)\bar{\varphi}(Z)\given Z\big]
        \\ 
        &= \E_{P_0}\big[Bp^{-\alpha-\beta}(Z)\bar{\varphi}(Z)\given Z\big]
        \\
        &= p^{1-\alpha-\beta}(Z)\bar{\varphi}(Z),
    \end{align*}
    and similarly
    \begin{equation*}
        \E_{P_0}\big[\{g'(g^{-1}(X\theta_0+f_0(Z)))\}^2\Var_{P_0}(Y\given X,Z)\given Z\big] = p^{1-\alpha}(Z)\bar{\varphi}(Z).
    \end{equation*}
    Therefore
    \begin{align*}
        V_{P_0,\mathrm{loceff}} &:= \big(\E_{P_0}\big[p(Z)^{\alpha-\beta}\bar{\varphi}(Z)^{-1}\big]\big)^{-2}\big(\E_{P_0}\big[p(Z)^{\alpha-\beta-1}\bar{\varphi}(Z)^{-1}\big]\big)
        ,
        \\
        V_{P_0,\mathrm{rose}} &:= \big(\E_{P_0}\big[p(Z)^{1+\alpha-\beta}\bar{\varphi}(Z)^{-1}\big]\big)^{-1}
        ,
        \\
        V_{P_0,\mathrm{unw}} &:= \big(\E_{P_0}\big[p(Z)^{1-\beta}\big]\big)^{-2} \big(\E_{P_0}\big[p(Z)^{1-\alpha-\beta}\bar{\varphi}(Z)\big]\big)
        .
    \end{align*}
    Now
    \begin{multline*}
        \bar{\varphi}(Z)^{-1} = \E_{P_{X|Z}}\big[\{g'(g^{-1}(X\theta_0+f_0(Z)))\}^{-2}\given Z\big]
        \\
        =
        a(Z)p(Z)+b(Z)(1-p(Z))
        =
        \begin{cases}
            a(Z)
            &\;\text{if $Z\in\cZ_1$}
            \\
            a(Z)\zeta + b(Z)(1-\zeta)
            &\;\text{if $Z\in\cZ_1$}
        \end{cases}.
    \end{multline*}
    Then define 
    \begin{gather*}
        m_1^{(1)}(\zeta) := \E_{P_0}[\bar{\varphi}(Z)\given Z\in\cZ_1],
        \qquad
        m_2^{(1)} = \E_{P_0}[\bar{\varphi}(Z)\given Z\in\cZ_2],
        \\
        m_1^{(-1)}(\zeta) := \E_{P_0}[\bar{\varphi}(Z)^{-1} \given Z\in\cZ_1],
        \qquad
        m_2^{(-1)} = \E_{P_0}[\bar{\varphi}(Z)^{-1} \given Z\in\cZ_2],
    \end{gather*}
    all of which we note are positive. 
    Also note that by construction $m_1^{(1)}(\zeta)\in[2^{-1}c'',\,2c'']$
    and
    \begin{equation*}
        m_1^{(-1)}(\zeta)
        =
        \E_{P_Z}\big[\bar{\varphi}(Z)^{-1}\given Z\in\cZ_1\big]
        =
        c\,\zeta + c'(1-\zeta)
    \end{equation*}
    Then
    \begin{align*}
        \frac{V_{P_0,\mathrm{unw}}}{V_{P_0,\mathrm{rose}}}
        &=
        \frac{\E_{P_0}[p^{1-\alpha-\beta}(Z)\bar{\varphi}(Z)]\E_{P_0}[p^{1+\alpha-\beta}(Z)\bar{\varphi}(Z)^{-1}]}{\big\{\E_{P_0}[p^{1-\beta}(Z)]\big\}^2}
        \\
        &= 
        \frac{\{\zeta^{1-\alpha-\beta}m_1^{(1)}(\zeta)q_1 + m_2^{(1)}q_2\}  \{\zeta^{1+\alpha-\beta}m_1^{(-1)}(\zeta)q_1 + m_2^{(-1)}q_2\}}{\{\zeta^{1-\beta}q_1 + q_2\}^2}
        \\
        &\geq
        m_1^{(1)}(\zeta)m_2^{(-1)}q_1q_2\zeta^{1-\alpha-\beta}
        \\
        &\geq
        2^{-1} c'' m_2^{(-1)}q_1q_2\zeta^{1-\alpha-\beta}
        \\
        &\geq A,
    \end{align*}
    where we make use of the inequalities $\zeta\leq1$ (so that $\zeta^{1-\beta}q_1+q_2\leq q_1+q_2=1$), $\zeta\leq\bar{\zeta}$ (so that $m_1^{(1)}(\zeta)\geq2^{-1}c''$), and $\zeta\leq \big(2^{-1}A^{-1}q_1q_2 c''m_2^{(-1)}\big)^{\frac{1}{\alpha+\beta-1}}$ for the final inequality, noting that $1-\alpha-\beta<0$.
    
    For the ratio ${V_{P_0,\mathrm{loceff}}}/{V_{P_0,\mathrm{unw}}} $,
    \begin{align*}
        \frac{V_{P_0,\mathrm{loceff}}}{V_{P_0,\mathrm{unw}}} 
        &= 
        \frac{\E_{P_0}[p^{\alpha-\beta-1}(Z)\bar{\varphi}(Z)^{-1}]\big\{\E_{P_0}[p^{1-\beta}(Z)]\big\}^2}{\big\{\E_{P_0}[p^{\alpha-\beta}(Z)\bar{\varphi}(Z)^{-1}]\big\}^2\E_{P_0}[p^{1-\alpha-\beta}(Z)\bar{\varphi}(Z)]}
        \\
        &=
        \frac{\{\zeta^{\alpha-\beta-1}m_1^{(1)}q_1 + m_2^{(1)}q_2\}  \{\zeta^{1-\beta}q_1 + q_2\}^2  }{\{\zeta^{\alpha-\beta}m_1^{(-1)}q_1 + m_2^{(-1)}q_2\}^2  \{\zeta^{1-\alpha-\beta}m_1^{(1)}q_1 + m_2^{(1)}q_2\}}
        \\
        &\geq
        \frac{q_2^2}{2\{m_1^{(-1)}(\zeta)q_1+m_2^{(-1)}q_2\}^2} \,\zeta^{2(\alpha-1)}
        \\
        &\geq
        \frac{q_2^2}{2\{2c''q_1+m_2^{(-1)}q_2\}^2} \,\zeta^{2(\alpha-1)}
        \\
        &\geq A,
    \end{align*}
    with the final inequality following as $\zeta\leq \big(2^{-1}\{2c''q_1+m_2^{(-1)}q_2\}^{-2}A^{-1}q_2^2\big)^{\frac{1}{2(1-\alpha)}}$, alongside
    \begin{equation*}
        \zeta^{\alpha-\beta}m_1^{(-1)}(\zeta)q_1 + m_2^{(-1)}q_2
        \leq
        \zeta^{\alpha-\beta}(c+c')q_1+m_2^{(-1)}q_2
        \leq
        2\zeta^{\alpha-\beta}(c+c')q_1,
     \end{equation*}
    where in the above we make use of: $\zeta\leq1$; $\alpha\geq\beta$; $m_1^{(1)}(\zeta)=c+c'\zeta \leq c+c'$; and $\zeta\leq\big(q_1^{-1}q_2(c+c')^{-1}m_2^{(-1)}\big)^{\frac{1}{\alpha-\beta}}$.
    
    Thus the required result is obtained;
    \begin{equation*}
        \frac{V_{P_0,\mathrm{loceff}}}{V_{P_0,\mathrm{unw}}}\geq A
        \quad
        \text{and}
        \quad
        \frac{V_{P_0,\mathrm{unw}}}{V_{P_0,\mathrm{rose}}}\geq A.
    \end{equation*}

\end{proof}

\subsection{Note on computational complexity of ROSE random forests}\label{appsec:rose-complexity}

Selecting the optimal split point for the $m$th iteration of the decision tree in Algorithm~\ref{alg:roseforest} requires updating at most $(1-\alpha)^{m-1} nd$ goodness-of-fit measures, the set of which can be calculated in $O((1-\alpha)^{m-1}nd)$ complexity. Therefore, the computational complexity to build a tree with $M$ splitting points is upper bounded by
\begin{equation*}
    \sum_{m=1}^M (1-\alpha)^{m-1}nd \leq nd \sum_{m=1}^{\infty} (1-\alpha)^{m-1} = \alpha^{-1}nd
\end{equation*}

\subsection{Asymptotic consistency of ROSE random forests (Theorem~\ref{thm:rose_forest_consistency})}\label{appsec:consistency-proof-first}

For notational simplicity throughout the following proof we will denote $\hat{w}_1^{\rose}$ by $\hat{w}_1$, omitting the superscript. Note in the below, conditioning on $\hat{w}_1^{\rose}$ (i.e. conditioning on the data used for this estimator) is equivalent to conditioning on $S_\cI := (S_i)_{i\in\cI}$ (i.e. the data used to construct each ROSE decision tree's splits and evaluations) in addition to $\hat{\eta}$ and $\hat{\theta}$ (i.e. the data used for the $\eta_P$ estimator and the pilot $\theta$ estimator). Also, for avoidance of doubt, given non-negative functions e.g.~$\omega_{ib}$ and for any $q\in\R$ we notate exponentiation by $\omega_{ib}^q(z)=\{\omega_{ib}(z)\}^q$.

\begin{proof}[Proof of Theorem~\ref{thm:rose_forest_consistency}]
    Define the `residual terms'
    \begin{align*}
        R_i&:=\psi_1^2(S_i;\hat{\theta},\hat{\eta}(Z_i))-\psi_1^2(S_i;\theta_P,\eta_P(Z_i)),
        \\
        R_{\theta,i}&:=\partial_\theta\psi_1(S_i;\hat{\theta},\hat{\eta}(Z_i))-\partial_\theta\psi_1(S_i;\theta_P,\eta_P(Z_i)),
    \end{align*}
    and for each $z\in\cZ$ define
    \begin{align*}
        L_i(z) &:= \psi_1^2(S_i;\theta_P,\eta_P(Z_i)) - \E_P[\psi_1^2(S_i;\theta_P,\eta_P(Z_i))|Z_i=z, \hat{\eta}],
        \\
        Q_i(z) &:= \partial_\theta\psi_1(S_i;\theta_P,\eta_P(Z_i)) - \E_P\left[\partial_\theta\psi_1(S_i;\theta_P,\eta_P(Z_i))|Z_i=z, \hat{\eta}\right],
        \\
        U_{\theta}(z)&:=\E_P\left[\partial_\theta\psi_1(S_i;\theta_P,\eta_P(Z_i))|Z_i=z\right].
    \end{align*}
    Also recall $\omega_{ib}(z)$ as defined in Algorithm~\ref{alg:roseforest} and let $$\omega_i(z):=\frac{1}{B}\sum_{b=1}^B\omega_{ib}(z)$$. Then
    \begin{equation*}
        \big( \hat{w}_1(z)-w_{P,1}^{\rose}(z)\big)^2 = \frac{\Delta^2(z)}{U^2(z)\big(U(z)+\tilde{\Delta}(z)\big)^2},
    \end{equation*}
    where
    \begin{align*}
        \tilde{\Delta}(z) &:= \sum_{i=1}^n\omega_i(z)\big\{L_i(z)+R_i\big\}
        ,
        \\
        \Delta(z) &:= U(z)\sum_{i=1}^n\omega_i(z)\big\{Q_i(z)+R_{\theta,i}\big\} - U_{\theta}(z)\tilde{\Delta}(z).
    \end{align*}
    Without loss of generality take $c_w\geq c_1^{-1}C_2$; otherwise replace $c_w$ with $c_w\vee (c_1^{-1}C_2)$. Then for an arbitrary $\kappa\in(0,1)$ and $P\in\cP$,
    \begin{align*}
        &\quad
        \E_P\left[\big( \hat{w}_1(Z)-w_{P,1}^{\rose}(Z)\big)^2\Biggiven\Strain\right]
        \\
        &=
        \E_P\left[\big( \hat{w}_1(Z)-w_{P,1}^{\rose}(Z)\big)^2\ind_{\left\{\left|\tilde{\Delta}(Z)\right|<(1-\kappa)|U(Z)|\right\}}\Biggiven\Strain\right]
        \\
        &\quad+\E_P\left[\big( \hat{w}_1(Z)-w_{P,1}^{\rose}(Z)\big)^2\ind_{\left\{\left|\tilde{\Delta}(Z)\right|\geq(1-\kappa)|U(Z)|\right\}}\Biggiven\Strain\right]
        \\
        &\leq 2c_w\E_P\left[\big| \hat{w}_1(Z)-w_{P,1}^{\rose}(Z)\big|\ind_{\left\{\left|\tilde{\Delta}(Z)\right|<(1-\kappa)|U(Z)|\right\}}\Biggiven\Strain\right]
        \\
        &\quad + 4c_w^2\PP_P\left(\big|\tilde{\Delta}(Z)\big|>(1-\kappa)|U(Z)|\Biggiven\Strain\right)
        \\
        &\leq \frac{2c_w}{\kappa}\E_P\left[\frac{|\Delta(Z)|}{U^2(Z)}\bigggiven\Strain\right] + 4c_w^2\PP_P\left(\big|\tilde{\Delta}(Z)\big|>(1-\kappa)|U(Z)|\Biggiven\Strain\right)
        \\
        &\leq \frac{2c_w}{\kappa}\E_P\left[\frac{|\Delta(Z)|}{U^2(Z)}\bigggiven\Strain\right] + \frac{4c_w^2}{1-\kappa}\E_P\left[\frac{|\tilde{\Delta}(Z)|}{|U(Z)|}\bigggiven\Strain\right]
        \\
        &\leq \frac{2c_w}{c_1^2\kappa}\E_P\big[|\Delta(Z)|\biggiven\Strain\big] + \frac{4c_w^2}{c_1(1-\kappa)}\E_P\big[|\tilde{\Delta}(Z)|\biggiven\Strain\big],
    \end{align*}
with the first inequality following by the Cauchy--Schwarz and triangle inequalities, the second by Markov's inequality, and the third by Assumption~\ref{ass:holder}. Therefore, it suffices to show that $\E_P[|\Delta(Z)|\given\Strain]$ and $\E_P[|\tilde{\Delta}(Z)|\given\Strain]$ are each $o_\cP(1)$. By the triangle inequality
\begin{multline}\label{eq:Delta}
    \E_P\big[|\Delta(X)|\biggiven\Strain\big] \leq M \Bigg(
        \underbrace{\E_P\bigg[\bigg|\sum_{i=1}^n\omega_i(Z)Q_i(Z_i)\bigg|\bigggiven\Strain\bigg]}_{=:\mathcal{Q}_{\RN{1}}}
        + \underbrace{\E_P\bigg[\bigg|\sum_{i=1}^n\omega_i(Z)L_i(Z_i)\bigg|\bigggiven\Strain\bigg]}_{_{=:\mathcal{L}_{\RN{1}}}}
        \\
        + \underbrace{\E_P\bigg[\bigg|\sum_{i=1}^n\omega_i(Z)\big(Q_i(Z)-Q_i(Z_i)\big)\bigg|\bigggiven\Strain\bigg]}_{=:\mathcal{Q}_{\RN{2}}}
        + \underbrace{\E_P\bigg[\bigg|\sum_{i=1}^n\omega_i(Z)\big(L_i(Z)-L_i(Z_i)\big)\bigg|\bigggiven\Strain\bigg]}_{=:\mathcal{L}_{\RN{2}}}
        \\
        + \underbrace{\E_P\bigg[\bigg|\sum_{i=1}^n\omega_i(Z)R_i\bigg|\bigggiven\Strain\bigg]}_{=:\cR_{\RN{2}}}
        + \underbrace{\E_P\bigg[\bigg|\sum_{i=1}^n\omega_i(Z)R_{\theta,i}\bigg|\bigggiven\Strain\bigg]}_{=:\cR_{\RN{1}}},
\end{multline}
and
\begin{equation}\label{eq:Delta-tilde}
    \E_P\big[|\tilde{\Delta}(Z)|\biggiven\Strain\big] \leq \mathcal{L}_{\RN{1}} + \mathcal{L}_{\RN{2}} + \cR_{\RN{2}}.
\end{equation}
The terms $(\mathcal{L}_{\RN{1}},\mathcal{Q}_{\RN{1}})$ and $(\mathcal{L}_{\RN{2}},\mathcal{Q}_{\RN{2}})$ act as variance and bias terms respectively, and $(\cR_{\RN{1}},\cR_{\RN{2}})$ act as error terms from the estimators of the nuisance functions.
By (a uniform version of) Slutsky's Theorem (Lemma~\ref{lem:unif_slutsky}), it suffices to show that all these terms are $o_\cP(1)$; we show each in turn.

To bound the $\mathcal{Q}_{\RN{1}}$ term
\begin{equation}\label{eq:Q1}
    \mathcal{Q}_{\RN{1}} = \E_P\left[\bigg|\frac{1}{B}\sum_{b=1}^B\sum_{i=1}^n\omega_{ib}(Z)Q_i(Z_i)\bigg|\bigggiven\Strain\right]
    \leq
    \frac{1}{B}\sum_{b=1}^B \E_P\left[\bigg|\sum_{i=1}^n\omega_{ib}(Z)Q_i(Z_i)\bigg|\bigggiven\Strain\right].
\end{equation}
Then by Jensen's inequality,
\begin{align*}
    \E_P\left[\bigg|\sum_{i=1}^n\omega_{ib}(Z)Q_i(Z_i)\bigg|\bigggiven\Strain\right]
    &\leq
    \E_P\left[\Big(\sum_{i=1}^n\omega_{ib}(Z)Q_i(Z_i)\Big)^2\bigggiven\Strain\right]^{\frac{1}{2}}
    \\
    &=
    \bigg(\sum_{i=1}^n\sum_{i'=1}^nQ_i(Z_i)Q_{i'}(Z_{i'})\E_P\left[\omega_{ib}(Z)\omega_{i'b}(Z)\given\Strain\right]\bigg)^{\frac{1}{2}}
\end{align*}
By the honesty property~\ref{P1} the data used to generate the splits of the tree $T_{jb}$ and thus $\omega_{jb}$ are independent of the data used for evaluation on the tree indexed by $b$, and thus $\omega_{jb}\independent (S_i)_{i\in\cI_{\text{eval},b}}$. Consequently, for any $i,i'\in\cI_{\text{eval},b}$ it then follows that $\big(Q_i(Z_i),Q_{i'}(Z_{i'})\big)\allowbreak\independent\allowbreak\big(\omega_{ib}(Z),\omega_{i'b}(Z)\big)\given Z_i$. In conjunction with the fact that $\E_P[Q_i(Z_i)\given Z_i]=0$ and Jensen's inequality,
\begin{align*}
    &\quad \supP\E_P\left[\bigg|\sum_{i=1}^n\omega_{ib}(Z)Q_i(Z_i)\bigg|\right]
    \\
    &\leq 
    \supP\E_P\left[\bigg(\sum_{i=1}^n\omega_{ib}(Z)Q_i(Z_i)\bigg)^2\right]^{\frac{1}{2}}
    \\
    &=
    \supP\E_P\left[\sum_{i\in\cI_{\text{eval},b}}\sum_{i'\in\cI_{\text{eval},b}}\omega_{ib}(Z)\omega_{i'jb}(Z)Q_i(Z_i)Q_{i'}(Z_{i'})\right]^{\frac{1}{2}}
    \\
    &=
    \supP\E_P\left[\sum_{i\in\cI_{\text{eval},b}}\sum_{i'\in\cI_{\text{eval},b}}\E_P\left[\omega_{ib}(Z)\omega_{i'b}(Z)\given Z,Z_{\text{train}},\hat{\eta}\right]\E_P\left[Q_i(Z_i)Q_{i'}(Z_{i'})\given Z_i,Z_{i'},\hat{\eta}\right]\right]^{\frac{1}{2}}
    \\
    &=\supP\E_P\left[\sum_{i=1}^n\E_P\left[\omega_{ib}^2(Z)\given Z,Z_{\text{train}},\hat{\eta}\right]\E_P\left[Q_i^2(Z_i)\given Z_i,\hat{\eta}\right]\right]^{\frac{1}{2}}
    \\
    &=\supP\left(\sum_{i=1}^n\E_P\left[\omega_{ib}^2(Z)Q_i^2(Z_i)\right]\right)^{\frac{1}{2}}
    \\
    &\leq
    \supP\left(\sum_{i=1}^n\E_P\Big[\omega_{ib}^{\frac{2(2+\delta)}{\delta}}(Z)\Big]^{\frac{\delta}{2+\delta}}\E_P\Big[|Q_i(Z_i)|^{2+\delta}\Big]^{\frac{2}{2+\delta}}\right)^{\frac{1}{2}}
    \\
    &\leq
    \supP\left(\bigg\{\max_{i\in[n]}\E_P\Big[|Q_i(Z_i)|^{2+\delta}\Big]^{\frac{2}{2+\delta}}\bigg\} \E_P\Big[\max_{i\in[n]}\omega_{ib}^{\frac{4+\delta}{\delta}}(Z)\Big]^{\frac{\delta}{2+\delta}} \right)^{\frac{1}{2}} 
    \\
    &\leq
    \supP\left(\bigg\{\max_{i\in[n]}\E_P\Big[|\partial_\theta\psi_1(S_i;\theta_P,\eta_P(Z_i))|^{2+\delta}\Big]^{\frac{2}{2+\delta}}\bigg\} \big(\min_{l_{b}}k(l_{b})\big)^{-\frac{4+\delta}{2+\delta}} \right)^{\frac{1}{2}}
    \\
    &= o(1),
\end{align*}
where the second inequality is a direct application of H\"{o}lder's inequality, and the third inequality uses that $\sum_{i=1}^n \omega_{ib}(z)=1$.
Thus with~\eqref{eq:Q1} it follows
\begin{equation*}
    \supP\E_P\big[\mathcal{Q}_{\RN{1}}\big]
    \leq
    \frac{1}{B}\sum_{b=1}^B \supP\E_P\left[\bigg|\sum_{i=1}^n\omega_{ib}(Z)Q_i(Z_i)\bigg|\right]
    =
    o(1).
\end{equation*}
By Markov's inequality we thus have $\mathcal{Q}_{\RN{1}}=o_\cP(1)$. 

For the term $\mathcal{L}_{\RN{1}}$ we consider the two cases $\delta\geq 2$ and $\delta\in(0,2)$ separately. If $\delta\geq 2$ using the same arguments to that used to bound the term $\mathcal{Q}_{\RN{1}}$ gives
\begin{equation*}
    \supP\E_P\left[\bigg|\sum_{i=1}^n\omega_{ib}(Z)L_i(Z_i)\bigg|\right]
    \leq
    \supP\left(\bigg\{\max_{i\in[n]}\E_P\Big[\psi_1^4(S_i;\theta_P,\eta_P(Z_i))\Big]^{\frac{1}{2}}\bigg\} \big(\min_{l_{b}}k(l_{b})\big)^{-\frac{3}{2}} \right)^{\frac{1}{2}}.
\end{equation*}
If instead $\delta\in(0,2)$ then we use the von--Bahr--Esseen inequality on the sequence of mean zero terms $\big(\omega_{ijb}(Z)L_i(Z_i)\big)_{i\in[n]}$. Note these terms are mean zero as
\begin{equation*}
    \E_P\left[\omega_{ib}(Z)L_i(Z_i)\right]
    =
    \E_P\big[\E_P\left[\omega_{ib}(Z)\given Z_i\right]\E_P\left[L_i(Z_i)\given Z_i\right]\big]
    =
    0
    ,
\end{equation*}
using the honesty property~\ref{P1}, and so we can apply the von--Bahr--Esseen inequality;
\begin{align*}
    &\quad \supP\E_P\left[\bigg|\sum_{i=1}^n\omega_{ib}(Z)L_i(Z_i)\bigg|\right]
    \\
    &\leq 
    \supP\E_P\left[\bigg|\sum_{i=1}^n\omega_{ib}(Z)L_i(Z_i)\bigg|^{1+\frac{\delta}{2}}\right]^{\frac{2}{2+\delta}}
    \\
    &\leq
    \big(2-n^{-1}\big)^{\frac{2}{2+\delta}} \supP\bigg(\sum_{i=1}^n\E_P\left[\omega_{ib}^{1+\frac{\delta}{2}}(Z)|L_i(Z_i)|^{1+\frac{\delta}{2}}\right]\bigg)^{\frac{2}{2+\delta}}
    \\
    &\leq
    \big(2-n^{-1}\big)^{\frac{2}{2+\delta}} \supP\bigg(\sum_{i=1}^n\E_P\left[\omega_{ib}^{2+\delta}(Z)\right]^{\frac{1}{2}}\E_P\left[|L_i(Z_i)|^{2+\delta}\right]^{\frac{1}{2}}\bigg)^{\frac{2}{2+\delta}}
    \\
    &\leq
    \supP\bigg(\bigg\{\max_{i\in[n]}\E_P\left[\big|\psi_1(S_i;\theta_P,\eta_P(Z_i))\big|^{2+\delta}\right]^{\frac{2}{2+\delta}}\bigg\} \E_P\left[\max_{i\in[n]}\omega_{ib}^{\frac{\delta}{2}}(Z)\right]\bigg)^{\frac{2}{2+\delta}}
    \\
    &\leq
    \supP\bigg(\bigg\{\max_{i\in[n]}\E_P\left[\big|\psi_1(S_i;\theta_P,\eta_P(Z_i))\big|^{2+\delta}\right]^{\frac{2}{2+\delta}}\bigg\} \big(\min_{l_{b}}k(l_{b})\big)^{-\frac{\delta}{2}}\bigg)^{\frac{2}{2+\delta}}
    \\
    &= o(1),
\end{align*}
and thus in either case
\begin{equation*}
    \supP\E_P\big[\mathcal{L}_{\RN{1}}\big]
    \leq
    \frac{1}{B}\sum_{b=1}^B \supP\E_P\left[\bigg|\sum_{i=1}^n\omega_{ib}(Z)L_i(Z_i)\bigg|\right]
    =
    o(1),
\end{equation*}
and so again by Markov's inequality $\mathcal{L}_{\RN{1}} = o_\cP(1)$.

    We now show that the `bias terms' $\mathcal{Q}_{\RN{2}}, \mathcal{L}_{\RN{2}}=o_\cP(1)$. Note that by Assumption~\ref{ass:holder}
    \begin{multline*}
        \mathcal{Q}_{\RN{2}} \leq \E_P\bigg[\sum_{i=1}^n\omega_i(Z)\big|Q_i(Z)-Q_i(Z_i)\big|\Biggiven\Strain\bigg]
        \leq L^2\E_P\left[\sum_{i=1}^n\omega_i(Z)\norm{Z-Z_i}^{\beta}\bigggiven\Strain\right]
        \\
        \leq \frac{L^2}{B}\sum_{b=1}^B\E_P\Big[\big(\diam\,l_{b}(Z)\big)^{\beta}\biggiven\Strain\Big]
        ,
    \end{multline*}
    and so it suffices to show for a single tree $T_{b}$ that $\E_P\big[(\diam\,l_{b}(Z))^{\beta}\given\Strain\big]=o_\cP(1)$. This follows by Lemma~\ref{lem:diam},
    and so $\mathcal{Q}_{\RN{2}}=o_\cP(1)$. Analogous arguments show that $\mathcal{L}_{\RN{2}}=o_\cP(1)$.

Finally we show that the error terms $\cR_{\RN{1}}, \cR_{\RN{2}} = o_\cP(1)$. For the $\cR_{\RN{1}}$ term
\begin{align*}
    \cR_{\RN{1}} &\leq \sum_{i=1}^n \E_P\big[\omega_i(Z)\given\Strain\big] |R_{\theta,i}|
    \\
    & \leq \left(\sum_{i=1}^n\E_P[\omega_i(Z)\given\Strain]^p\right)^{\frac{1}{p}} \left(\sum_{i=1}^n|R_{\theta,i}|^q\right)^{\frac{1}{q}}
    \\
    &\leq \left(\Big\{\max_{i\in[n]}\E_P[\omega_i(Z)\given\Strain]\Big\}\sum_{i=1}^n|R_{\theta,i}|^q\right)^{\frac{1}{q}}
    \\
    & = O_\cP\big(n^{\frac{1}{q}-\lambda}\big) = o_\cP(1),
\end{align*}
using H\"{o}lder's inequality (with $q$ such that $q>\lambda^{-1}$ and $p=q/(q-1)$, and taking any $\lambda\in(0,\frac{1}{2})$), the fact that (by definition) $\sum_{i=1}^n\omega_i(Z)=1$, and the last step following due to
\begin{align*}
    &\qquad
    \E_P\left[n^{\lambda}\big|\partial_\theta\psi(S_i;\hat{\theta},\hat{\eta}(Z_i))-\partial_\theta\psi(S_i;\theta_P,\eta_P(Z_i))\big|\biggiven\hat{\eta},\hat{\theta}\right]
    \\
    &\leq
    \E_P\left[n^{\lambda}\big|\partial_\theta\psi(S_i;\hat{\theta},\hat{\eta}(Z_i))-\partial_\theta\psi(S_i;\hat{\theta},\eta_P(Z_i))\big|\biggiven\hat{\eta},\hat{\theta}\right]
    \\
    &\qquad
    + \E_P\left[n^{\lambda}\big|\partial_\theta\psi(S;\hat{\theta},\eta_P(Z_i))-\partial_\theta\psi(S_i;\theta_P,\eta_P(Z_i))\big|\biggiven\hat{\eta},\hat{\theta}\right]
    \\
    &\leq 
    \underbrace{\sup_{\theta\in\Theta}\E_P\left[n^{\lambda}\big|\partial_\theta\psi(S_i;\theta,\hat{\eta}(Z_i))-\partial_\theta\psi(S;\theta,\eta_P(Z_i))\big|\biggiven\hat{\eta}\right]}_{=o_\cP(1)}
    \\
    &\qquad + \underbrace{n^{\lambda}|\hat{\theta}-\theta_P|}_{=O_\cP\big(n^{\lambda-\frac{1}{2}}\big)}\cdot\sup_{\theta\in\Theta}\E_P\left[\big|\nabla_\theta\partial_\theta\psi(S_i;\theta,\eta_P(Z_i))\big|\right]
    \\
    & = o_\cP(1),
\end{align*}
following analogous arguments to the proof of Lemma~\ref{lem:eta-consistency}, and utilising Assumption~\ref{ass:bounded-moments}, and the asymptotic normality of the unweighted $\theta$ estimator (on a single fold) as shown in the proof of Theorem~\ref{thm:rose_forest_consistency} (and thus, by Markov's inequality, it follows that $|R_{\theta,i}|=o_\cP\big(n^{-\lambda}\big)$). Similar arguments show that $\cR_{\RN{2}}=o_\cP(1)$.

Therefore all the terms on the right hand side of~\eqref{eq:Delta}~and~\eqref{eq:Delta-tilde} are $o_\cP(1)$ and so applying a uniform version of Slutsky's Theorem (Lemma~\ref{lem:unif_slutsky}) completes the proof;
\begin{equation*}
    \E_P\left[\big( \hat{w}_1^{\rose}(Z)-w_{P,1}^{\rose}(Z)\big)^2\Biggiven\hat{w}_1^{\rose}\right]^{\frac{1}{2}} = o_\cP(1).
\end{equation*}

\end{proof}

\begin{lemma}\label{lem:diam}
    Under the setup of Theorem~\ref{thm:rose_forest_consistency} we have that
    \begin{equation*}
        \E_P\left[\left(\diam\, l_{b}(Z)\right)^\gamma\,\Big|\,\Strain\right] = o_\cP(1),
    \end{equation*}
    for any $\gamma>0$.
\end{lemma}
\begin{proof}
    Note \citet{app-meinshausen} proves a result to this effect for a fixed distribution $P$, pointwise in $z\in\mathcal{Z}$, and with $\gamma=1$; our proof here follows their method closely.
    
    Given some $z\in\cZ$ the rectangular region given by the leaf $l_{b}(z)$ that contains $z$ for the $b$th tree can be represented by $d$ intervals $I_{b}(z;q)$ (indexed by $q\in[d]$) as
    \begin{equation*}
        l_{b}(z) = \otimes_{q=1}^{d} I_{b}(z;q).
    \end{equation*}
    Define the terms $N_{b}(z) := \left\{i\in\cI_{\text{eval},b} : Z_i\in l_{b}(z)\right\}$ and $N_{b}(z;q) := \left\{i\in\cI_{\text{eval},b} : Z_{iq} \in I_{b}(z;q) \right\}$, so $k(l_{b}(z)) = N_{b}(z')$ for any $z'\in l_{b}(z)$. Further, let $S_{b}(z;q)$ to be the number of times the $q$th variable is split when $z\in\cZ$ is passed through the tree $T_{b}$. Similarly define $S_{b}(z)$ to be the number of nodes passed when $z\in\cZ$ is passed through $T_{b}$, so $S_{b}(z)=\sum_{q=1}^pS_{b}(z;q)$. Also define $S_{b}^{\min} := \min_{z\in\cZ}S_{b}(z)$. Then
    \begin{equation*}
        n\alpha^{S_{b}^{\min}}
        = \max_{z\in\cZ}\big(n\alpha^{S_{b}(z)}\big)
        \leq \max_{z\in\cZ}N_{b}(z)
        = \max_{l_{b}}k(l_{b}(z))
        = o(1),
    \end{equation*}
    by Assumptions~\ref{P2}~and~\ref{P3}, and so also
    \begin{equation*}
        \alpha^{S_{b }^{\min}} = o(1).
    \end{equation*}
    Now, for any $\phi\in(0,1)$,
    \begin{align*}
        \E_P\left[\phi^{\underset{q\in[d]}{\min} S_{b}(Z;q) }\right]
        &\leq \sum_{q=1}^{d} \E_P\left[\phi^{S_{b}(Z;q)}\right]
        \\
        &= \sum_{q=1}^{d} \E_P\left[\E_P\left[\phi^{S_{b}(Z;q)} \biggiven S_{b}(Z) \right]\right]
        \\
        &\leq d \,\E_P\left[\left(1-(1-\phi)\frac{\pi}{p}\right)^{S_{b}(Z)}\right]
        \\
        & \leq d \,\E_P\left[\big(\alpha^{S_{b}^{\min}}\big)^{\frac{\log\left(1-(1-\phi)\frac{\pi}{p}\right)}{\log\alpha}}\right],
    \end{align*}
    where the second inequality follows as  $S_{b}(Z;q)\given Z \overset{d}{=}\sum_{i=1}^{S_{b}(Z)}Y_i\given Z$ where $Y_i\sim\text{Ber}(p_i)$ for $p_i\geq\frac{\pi}{d}$ 
    and thus $S_{b}(Z;q)\given Z$ is stochastically dominated by $\text{Bin}\big(S_{b}(Z),\frac{\pi}{d}\big)\given Z$, and where we also make use of $y\mapsto\phi^y$ as a decreasing function on $y\in\mathbb{N}$.

    Now, we claim that
    \begin{equation}\label{eq:need-to-show}
        A_{b} := \E_P\left[(1-\alpha)^{\gamma\underset{q\in[d]}{\min}S_{b}(Z;q)} \Biggiven T_{b}\right] = o_\cP(1).
    \end{equation}
    To prove~\eqref{eq:need-to-show}  take $\phi=(1-\alpha)^2$ in the above result, so for any $\epsilon>0$
    \begin{equation*}
        \sup_{P\in\cP}\PP_P\left(A_{b}>\epsilon\right)
        \leq \epsilon^{-1}\sup_{P\in\cP}\E_P\left[A_{b}\right]
        \\
        \leq d \, \supP\E_P\left[\left(\alpha^{S_{b}^{\min}}\right)^{\frac{\log\left(1-\left(1-(1-\alpha)^{\gamma}\right)\frac{\pi}{p}\right)}{\log\alpha}}\right] = o(1),
    \end{equation*}
    by applying Lemma~\ref{lem:bounded_op1} to the term $\alpha^{S_{b}^{\min}}$, and noting that $1-\big(1-(1-\alpha)^{\gamma}\big)\frac{\pi}{p}<1$. 
    Thus $A_{b}=o_\cP(1)$. Now, for any $z\in\cZ$ and $q\in [d]$
    \begin{equation*}
        \frac{1}{n}\sum_{i=1}^n\ind_{I_{b}(z;q)}(Z_{iq}) \leq (1-\alpha)^{S_{b}(z;q)},
    \end{equation*}
    and so
    \begin{equation*}
        \max_{q\in[d]}\left(\frac{1}{n}\sum_{i=1}^n\ind_{I_{b}(z;q)}(Z_{iq})\right)^{\gamma} \leq (1-\alpha)^{\gamma\underset{q\in[d]}{\min} S_{b}(z;q)}.
    \end{equation*}
    Thus
    \begin{equation}\label{eq:pt1}
        \E_P\left[\max_{q\in[d]}\left(\frac{1}{n}\sum_{i=1}^n\ind_{I(Z;q)}(Z_{iq})\right)^{\gamma} \bigggiven T_{b}\right] \leq A_{b} = o_\cP(1).
    \end{equation}
    Define
    \begin{equation*}
        F^{(q)}(t) := \PP_P\big(Z_{iq}\leq t\big)
        ,\quad\text{and}\quad
        F_n^{(q)}(t) := \frac{1}{n}\sum_{i=1}^n\ind_{[Z_{iq},\infty)}(t)
        .
    \end{equation*}
    Then
    \begin{align*}
        \E_P\left[\left(\diam\, l_{b}(Z)\right)^{\gamma}\biggiven\Strain\right] 
        &= \E_P\left[\bigg(\max_{q\in[d]}\left|I_{b}(Z;q)\right|\bigg)^{\gamma} \Biggiven T_{b}\right]
        \\
        &\lesssim 
            \E_P\left[\max_{q\in[d]}\sup_{t\in[0,1]}\big|F_n^{(q)}(t)-F^{(q)}(t)\big|^{\gamma} \Biggiven T_{b}\right]
            \\
            &\quad
            + \E_P\left[\max_{q\in[d]}\bigg(\frac{1}{n}\sum_{i=1}^n\ind_{I_{b}(Z;q)}(Z_{iq})\bigg)^{\gamma} \Biggiven T_{b}\right]
            + o_\cP(1)
        \\
        &\leq
            \underbrace{\E_P\Big[\max_{q\in[d]}\sup_{t\in[0,1]}\big|F_n^{(q)}(t)-F^{(q)}(t)\big|^{\gamma} \biggiven T_{b}\Big]}_{=:B_{b}} + o_\cP(1),
    \end{align*}
    and so it suffices to show that $B_{b}=o_\cP(1)$. Firstly, note that for every $\epsilon>0$
    \begin{align*}
        \sup_{P\in\cP}\PP_P\bigg(\max_{q\in[d]}\sup_{t\in[0,1]}\big|F_n^{(q)}(t)-F^{(q)}(t)\big|^{\gamma}>\epsilon\bigg)
        &\leq \sup_{P\in\cP}\sum_{q=1}^{d}\PP_P\bigg(\sup_{t\in[0,1]}\big|F_n^{(q)}(t)-F^{(q)}(t)\big|>\epsilon^{\frac{1}{\gamma}}\bigg)
        \\
        &\leq 2d\,\text{exp}\big(-2n\epsilon^{\frac{2}{\gamma}}\big) = o(1)
    \end{align*}
    using the Dvoretzky--Kiefer--Wolfowitz inequality (the upper bound of which is independent of $P\in\cP$).
    In conjunction with Lemma~\ref{lem:bounded_op1} we have
    \begin{equation*}
        \sup_{P\in\cP}\E_P\left[B_{b}\right] = \sup_{P\in\cP}\E_P\bigg[\max_{q\in[d]}\sup_{t\in[0,1]}\big|F_n^{(q)}(t)-F^{(q)}(t)\big|^{\gamma}\bigg] = o(1)
    \end{equation*}
    and so by Markov's inequality $B_{b}=o_\cP(1)$, completing the proof.
\end{proof}

\subsection{ROSE random forest plus (\texorpdfstring{$J>1$}{TEXT})}\label{appsec:rose-forest-J>1}

Algorithm~\ref{alg:roseforest} develops a ROSE random forest for estimators with $J=1$. Here we present an extension for the case $J>1$, designed to estimate the optimal $J$ weight functions consistently (see Theorem~\ref{thm:rose_forest_plus_consistency}) in a computationally fast manner. This consistency results is based predominantly on the evaluation criterion being implemented and not the splitting rule. The only reliance on the tree splitting in the consistency results are through Assumption~\ref{ass:forest}. Specifically, the splitting rule we employ is likely suboptimal in finite finite samples, but regardless our extension provides a computationally lean pilot estimator that still demonstrates valid asymptotic (Theorem~\ref{thm:rose_forest_plus_consistency}) and promising practical (Section~\ref{sec:sim3}) performance.

We now motivate the ROSE random forest plus algorithm. Given an estimator $\hat{\eta}$ of $\eta_P$ and a pilot estimator $\hat{\theta}_{\text{init}}$ of $\theta_P$ (achieved via e.g.~the zero of the unweighted $\psi$ function), and data indexed by the index set $\mathcal{I}$, the empirical sandwich loss~\eqref{eq:sand-loss} is
\begin{multline*}
    \hat{L}_{\SL}(w) := 
    \Bigg(\sum_{j=1}^J\sum_{i\in\cI}w_j(Z_i)\partial_\theta\psi_j(S_i;\hat{\theta}_{\text{init}},\hat{\eta}(Z_i))\Bigg)^{-2}
    \\
    \cdot
    \Bigg(\sum_{j=1}^J\sum_{j'=1}^J\sum_{i\in\cI}w_j(Z_i)w_{j'}(Z_i)\psi_j(S_i;\hat{\theta}_{\text{init}},\hat{\eta}(Z_i))\psi_{j'}(S_i;\hat{\theta}_{\text{init}},\hat{\eta}(Z_i))\Bigg).
\end{multline*}
Suppose for each $j\in[J]$ the state space $\cZ$ for $Z$ is partitioned into disjoint rectangular regions $R\in\cR_j$ i.e.~$\cZ = \uplus_{R\in\cR_j}R$ for all $j\in[J]$, and suppose $w_j$ is piece-wise constant given this partition i.e.~for all $j\in[J]$, and for $R\in\cR_j$, $w_j(Z_i) = w_{j,R}$ if $i\in I_{j,R} := \{i\in\cI:Z_i\in R\}$. Under this weight model the sandwich loss becomes
\begin{align*}
    \hat{L}_{\SL}(w) &:= 
    \Bigg(\sum_{j=1}^J\sum_{R\in\cR_j}w_{j,R}\sum_{i\in I_{j,R}}\partial_\theta\psi_j(S_i;\hat{\theta}_{\text{init}},\hat{\eta}(Z_i))\Bigg)^{-2}
    \\
    &\qquad\cdot
    \Bigg(\sum_{j=1}^J\sum_{R\in\cR_j}\sum_{j'=1}^J\sum_{R'\in\cR_{j'}}w_{j,R}w_{j',R'}\sum_{i\in I_{j,R}\cap I_{j',R'}}\psi_j(S_i;\hat{\theta}_{\text{init}},\hat{\eta}(Z_i))\psi_{j'}(S_i;\hat{\theta}_{\text{init}},\hat{\eta}(Z_i))\Bigg)
    \\
    &= \big({\bf w}^\top  {\bf a}\big)^{-2} \big({\bf w}^\top  {\bf{F}} {\bf w}\big),
\end{align*}
where
\begin{align}
    {\bf w} &:= \text{vec}\big((w_{j,R})_{j\in[J],R\in\cR_j}\big), \label{eq:bfw}
    \\
    {\bf a} &:= \text{vec}\big((a_{j,R})_{j\in[J],R\in\cR_j}\big),
    \qquad
    a_{j,R} := \sum_{i\in I_{j,R}} \partial_\theta\psi_j(S_i;\hat{\theta}_{\text{init}},\hat{\eta}(Z_i)), \label{eq:aa}
    \\
    {\bf F} &:= \begin{pmatrix}(F_{11,RR'})_{R,R'\in\cR_1^2}&\cdots&(F_{1J,RR'})_{R,R'\in\cR_1\times\cR_J}\\\vdots & \ddots & \vdots\\(F_{J1,RR'})_{R,R'\in\cR_J\times\cR_1}&\cdots&(F_{JJ,RR'})_{R,R'\in\cR_J^2}\end{pmatrix}, \label{eq:F}
    \\
    F_{jj',RR'} &:= \sum_{i\in I_{j,R}\cap I_{j',R'}} \psi_j(S_i;\hat{\theta}_{\text{init}},\hat{\eta}(Z_i))\psi_{j'}(S_i;\hat{\theta}_{\text{init}},\hat{\eta}(Z_i)). \notag
\end{align}
The optimal weights for this partition are
\begin{equation}\label{eq:w_P-J}
    {\bf w} := {\bf F}^{-1} {\bf a},
\end{equation}
up to a constant multiplicative factor, following by the Cauchy--Schwarz inequality. The minimising variance is thus
$\hat{L}_{\SL}({\bf w}) = \big({\bf a}^\top {\bf F}^{-1}{\bf a}\big)^{-1}$. As such the quantity ${\bf a}^\top {\bf F}^{-1}{\bf a}$ can be used to motivate a splitting rule for decision trees, albeit a computationally more involved splitting rule. Note in particular that ${\bf F}$ takes the form of a $J\times J$ block matrix, with diagonal blocks being diagonal matrices. As such, for the special case of $J=1$ the minimising weights take the form ${\bf F}^{-1}{\bf a}=(F_{11,RR}^{-1}a_{1R})_{R\in\cR_1}$ and similarly the reciprocal of the minimising variance takes a simpler form ${\bf a}^\top {\bf F}^{-1}{\bf a}=\sum_{R\in\cR_1}F_{11,RR}^{-1}a_{1R}^2$, so that the goodness of fit metric as in ROSE decision trees for $J=1$ can be written as a function of the data only contained within the candidate region to be split; this property drives the computational leanness of ROSE decision trees matching that of classical CART random forests.

For a general $J>1$ however such quantities used for splitting are of a greater computational complexity. To trade off some finite sample optimality for favourable computation, we propose a computationally lean alternative to the ROSE random forest (ROSE random forest plus). The (potentially somewhat suboptimal) splits are generated using the ROSE decision tree algorithm (Algorithm~\ref{alg:rosetree}) as if only that individual $\psi_j$ were to be used. Given such splits (obtained from $J$ runs of Algorithm~\ref{alg:rosetree}) the joint optimal weight functions $(w_j)_{j\in[J]}$ can be estimated using the evaluation criterion~\eqref{eq:w_P-J}. A forest can then be generated from such trees by a mean aggregation. The full algorithm is given as Algorithm~\ref{alg:rose-random-forest-J}.

\RestyleAlgo{ruled}
\begin{algorithm}[ht]
\KwIn{Data $\left(S_i\right)_{i\in\cI}$ associated with the $\psi$ function~\eqref{eq:w-psi}; estimator $\hat{\eta}$ of nuisance functions; pilot estimator $\hat{\theta}$ of $\theta_P$; $c_{\text{split}}>0$ proportion of data chosen for each tree in the forest; ROSE tree hyperparameters.}

\For{$b\in[B]$} {

Generate random subsets of the data $\cI_{\text{split},b}$ and $\cI_{\text{eval},b}$ as disjoint subsets of $\cI$ with $|\cI_{\text{split},b}|, |\cI_{\text{eval},b}| \approx 2^{-1}c_{\text{split}}|\cI|$ (by bootstrapping or sampling without replacement).

\For{$j\in[J]$} {
    Grow a ROSE decision tree $T_{jb}$ as in Algorithm~\ref{alg:rosetree} (using only the $b$th split and evaluation data $\cI_{\text{split},b}$ and $\cI_{\text{eval},b}$), outputting leaves $\cL_{jb}$ and leaf index sets $(I_{jb,l})_{l\in\cL_{jb}}$.
}

    \For{$j\in[J]$ and $l\in\cL_{jb}$}{
        Calculate $a_{jl,b} = \sum_{i\in I_{jl,b}} \partial_\theta\psi_j(S_i;\hat{\theta}_{\text{init}},\hat{\eta}(Z_i))$.
        \\
        \For{$j'\in[J]$ and $l'\in\cL_{j'b}$}{
        Calculate $F_{jj',ll'} = \sum_{i\in I_{jl,b}\cap I_{j'l',b}} \psi_j(S_i;\hat{\theta}_{\text{init}},\hat{\eta}(Z_i))\psi_{j'}(S_i;\hat{\theta}_{\text{init}},\hat{\eta}(Z_i))$.
        
    }
    }

    Construct ${\bf a}_b = \text{vec}\big((a_{jl,b})_{j\in[J],l\in\cL_{jb}}\big)$.

    Construct ${\bf F}_b = \begin{pmatrix}(F_{11,ll'})_{l,l'\in\cL_{1b}^2}&\cdots&(F_{1J,ll'})_{l,l'\in\cL_{1b}\times\cL_{Jb}}\\\vdots & \ddots & \vdots\\(F_{J1,ll'})_{l,l'\in\cL_{Jb}\times\cL_{1b}}&\cdots&(F_{JJ,ll'})_{l,l'\in\cL_{Jb}^2}\end{pmatrix}$.

    Calculate ${\bf \hat{w}}_{b} = {\bf F}_b^{-1}{\bf a}_b$, giving weight values $(\hat{w}_{jlb})_{j\in[J],l\in\cL_{jb}}$ as in~\eqref{eq:bfw}

}

\For{$j\in[J]$} {
Calculate
$\hat{w}_j^{\roseplus}(z) = B^{-1}\sum_{b=1}^B \sum_{l\in\cL_{jb}} \ind_{(z\in l)} \hat{w}_{jlb}(z)$
}

\KwOut{ROSE random forest weight functions $\hat{w}_j^{\roseplus}:\R^{d}\to\R$ for each $j\in[J]$.}
\caption{ROSE random forests ($J>1$ finite)}
\label{alg:rose-random-forest-J}
\end{algorithm}

We present consistency of ROSE random forest plus (Algorithm~\ref{alg:rose-random-forest-J}) analogous to that of~Theorem~\ref{thm:rose_forest_consistency} for $J=1$ ROSE random forests. Note however that whilst Theorem~\ref{thm:rose_forest_plus_consistency} provides theoretical justification on the use of Algorithm~\ref{alg:rose-random-forest-J} through the asymptotic optimality of the evaluation criterion used in Algorithm~\ref{alg:rose-random-forest-J}, the proof only relies on the splitting rule for trees through the properties of Assumption~\ref{ass:forest}.

\begin{theorem}[Consistency of ROSE random forests plus]\label{thm:rose_forest_plus_consistency}
Let Assumptions~\ref{ass:data} and~\ref{ass:data-2} hold for a class of distributions  $\mathcal{P}$. 
Then, for $J\geq1$ finite, the ROSE random forest plus estimator $\hat{w}^{\roseplus}$ (Algorithm~\ref{alg:rose-random-forest-J}, satisfying Assumption~\ref{ass:forest}) is uniformly consistent in estimating the optimal weights~\eqref{eq:w_j}: for each $j\in[J]$,
    \begin{equation*}
        \E_P\Big[\big(\hat{w}_j^{\roseplus}(Z)-w_{P,j}^{\rose}(Z)\big)^2 \Biggiven \hat{w}_j^{\roseplus}\Big]
        = o_\cP(1).
    \end{equation*}
\end{theorem}

\begin{proof}
    Note that as the splitting construction of $\hat{w}_j^{\roseplus}$ (as in Algorithm~\ref{alg:rose-random-forest-J}) is separable over $j\in[J]$, it suffices (by Lemma~\ref{lem:unif_max}) to prove that for a single fixed $j\in[J]$
    \begin{equation*}
        \E_P\Big[\big(\hat{w}_j^{\roseplus}(Z)-w_{P,j}^{\rose}(Z)\big)^2\biggiven\hat{w}_j^{\roseplus}\Big] = o_\cP(1).
    \end{equation*}
    The proof then follows by analogous arguments to that of Theorem~\ref{thm:rose_forest_consistency} (see Appendix~\ref{appsec:consistency-proof-first}).
\end{proof}

\section{Further details of numerical experiments in Section~\ref{sec:numerical-results}}\label{appsec:sims}

We present further details and additional simulation resutls for the numerical experiments of Section~\ref{sec:numerical-results}.

\subsection{Further details for Figure~\ref{fig:EIF_performing_poorly}}\label{appsec:Fig2a}
For the experiment of Figure~\ref{fig:EIF_performing_poorly}a, we study the following partially linear model. Take $n=20\,000$ iid instances of the partially linear model with $Z\sim N_{10}\big({\bf 0},(0.9^{|j-k|})_{(j,k)\in[10]^2}\big)$ and
\begin{gather*}
    p(Z) := \text{expit}(3Z_1) \vee 0.01,
    \quad
    B\given Z \sim \text{Ber}(p(Z)),
    \quad
    X\given (Z,B) \sim N\bigg(\sum_{j=1}^5\text{expit}(Z_j) \;,\; \frac{B+0.01}{p(Z)}\bigg),
    \\
    Y \given (X,Z,B) \sim N\bigg(\theta_0X + \sum_{j=1}^5\text{expit}(Z_j) \;,\; \frac{B+0.01}{\sqrt{p(Z)}}\bigg).
\end{gather*}
The results are presented in Figure~\ref{fig:EIF_performing_poorly}a.

\subsection{Simulation 1: Further details}\label{appsec:sim1}

\subsubsection{Hyperparameter tuning details}
For both simulations, the nuisance functions $(l_0,m_0)$ are fitted using random forests with the \texttt{ranger} package~\citep{app-ranger}. Hyperparameter tuning is performed for the parameters in the grid $\texttt{sample.fraction}\allowbreak\in\{0.01,\allowbreak0.05,0.1\allowbreak,0.2\allowbreak,0.5\allowbreak,0.8\allowbreak,1.0\}$ and $\texttt{min.node.size}\in\{10\allowbreak,20\allowbreak,30\allowbreak,40\allowbreak,50\allowbreak,100\allowbreak,200\allowbreak,500\allowbreak,1000\allowbreak,2000\}$ (note that, except in the case if $\texttt{sample.fraction}=1.0$ no tuning selects parameters on the extremes of this grid). All other parameters are taken as default as per the \texttt{ranger} package, in particular: 500 trees per forest; the number of variables considered for a split at each node $\texttt{mtry}=\big\lceil\sqrt{d}\big\rceil=4$); and data sampling for each tree performed with replacement. For the hyperparameter tuning cross-validation was performed (with CV criterion of the out-of-bag squared error) with both ROSE random forests and locally efficient CART random forests having optimal maximum depth of two and three in Simulations 1a and 1b respectively.

For Simulation 1a, performing cross-validation selected for both $\E[Y\given Z]$ and $\E[X\given Z]$ estimators $(\texttt{sample.fraction},\allowbreak\texttt{min.node.size})\allowbreak=(0.05,20)$. For the semiparametric efficient estimator, the additional nuisance functions $(\sigma_0, h_0)$ (as in Section~\ref{sec:PLMmodel}) must be estimated, and are estimated via the same cross validation scheme, selecting optimal hyperparameter values of $(\texttt{sample.fraction},\texttt{min.node.size})=(1.0,100)$ for the $\sigma_0$-learner and $(\texttt{sample.fraction}, \texttt{min.node.size})=(1.0,20)$ for the $h_0$-learner. Note that the optimal \texttt{sample.fraction} parameter is much larger for the $h_0$-learner than the $m_0$-learner; a practical indication of the informal (motivating) notion that the $h_0$ nuisance function may be a `harder' nuisance function to estimate than the $m_0$ nuisance function (see e.g.~Figure~\ref{fig:surface}).

For Simulation 1b, cross validation is performed via the same scheme as for Simulation 1a, in this case for the three nuisance functions $\big(\E[X\given Z],\allowbreak \E[Y\given X,Z],\allowbreak \E[g(\E[Y\given X,Z])\given Z]\big)$, selecting optimal hyperparameters of $\texttt{min.node.size}=10$ for all three nuisance functions, and $\texttt{sample.fraction}=0.05, 1.0$ and $0.1$ for the three nuisance functions respectively. For the semiparametric efficient estimator, the two additional nuisance functions $\sigma_0(X,Z) := \E[g'(\mu(X,Z))^2(Y-\mu(X,Z))^2\given X,Z]$ and $h_0(Z) := \E[\sigma_0^{-2}(X,Z)\given Z]^{-1}\E[\sigma_0^{-2}(X,Z) X \given Z]$ are estimatied by the same cross-validation scheme, selecting $(\texttt{min.node.size},\allowbreak\texttt{sample.fraction})\allowbreak =\allowbreak (100,\allowbreak1.0)$ and $(20,\allowbreak0.5)$ respectively.

\subsection{Simulation 2: Further details}\label{appsec:sim2}

\subsubsection{Hyperparameter tuning details}
To estimate weights in both cases we use minimum node size 10, and select data for each tree by random selection with replacement (with \texttt{sample.fraction}=1). We consider trees of depth 1 to 15 (note that the optimal trees in every simulation are not at the extrema in both cases) with the optimal tree depth selected via cross-validation. For the ROSE random forests we use the sandwich loss cross-validation criterion. For the CART random forests we use the standard prediction error CART cross-validation criterion.

\subsubsection{Additional details on simulation results}

Figure~\ref{fig:sim2-app} displays the mean squared error of the $\theta_0$ estimators (weighted with ROSE and CART random forests) for each tree depth (ranging from 0 to 15). In particular, the suboptimal performance of CART random forests in this simulation are likely a result of two properties: a suboptimal splitting rule; and the CART cross-validation criterion not targetting the mean squared error of the resulting $\theta_0$ estimator as accurately as the sandwich loss cross-validation criterion. In the example given in Figure~\ref{fig:sim2-app} any choice of fixed tree depth from 5 to 9 would actually perform better than carrying out cross-validation to find the supposed `optimal' tree depth. This suggests perhaps a wider concern for CART random forests in this setting; it appears performing hyperparameter tuning over a wider range of hyperparameters for the CART random forests can result in poorer performance for the inferential task of interest.

\begin{figure}[ht]
   \centering
   \includegraphics[width=0.9\textwidth]{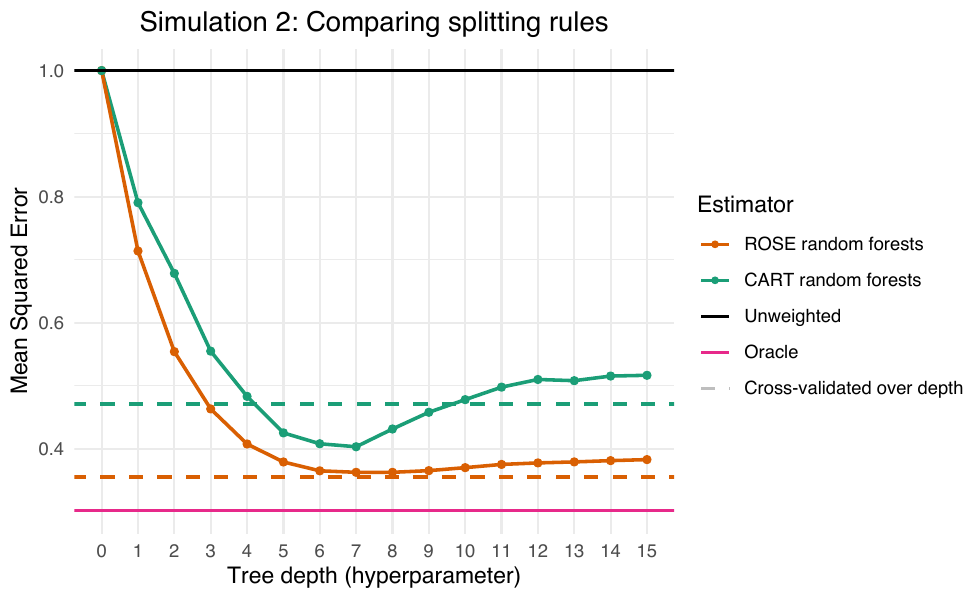}
    \caption{Mean squared error (as a ratio to the unweighted estimator) of estimators weighted by ROSE and CART random forests with varying tree depths for the simulation of Section~\ref{sec:sim2} with $n=6400$ (1000 simulations).}
   \label{fig:sim2-app}
\end{figure}

\subsection{Simulation 3: Further details}\label{appsec:sim3}

\subsubsection{Hyperparameter tuning details}

As in Simulations~1 and~2 all nuisance functions except for $\PP(X\given Z)$ are fitted with random forests using the \texttt{ranger} package, with hyperparameter tuning for  $\texttt{sample.fraction}\in\{0.01,\allowbreak 0.05,\allowbreak 0.1,\allowbreak 0.2,\allowbreak 0.5,\allowbreak 0.8,\allowbreak 1.0\}$ and $\texttt{min.node.size}\in\{10, \allowbreak 20, \allowbreak 30, \allowbreak 40, \allowbreak 50, \allowbreak 100, \allowbreak 200, \allowbreak 500, \allowbreak 1000, \allowbreak 2000\}$. All other hyperparameters are taken as the \texttt{ranger} package default parameters (for example 500 trees per forest). For each sample size $n$, to determine the optimal hyperparameters a grid search is used, with cross validation-criterion for all nuisance function the squared error (or in the case of estimating the nuisance function $h_0$ the weighted least squared error). The nuisance function $\PP(X\given Z)$ is fitted using probability random forests via the \texttt{grf} package~\citep{app-grf-code} with hyperparameter tuning as above. For the two ROSE random forests the maximum depth was taken to be 5 for every sample size; this was the optimal hyperparameter obtained via cross-validation with the sandwich loss criterion.

\subsubsection{Additional details on simulation results}

Additional details on the estimators of Simulation~3 (Section~\ref{sec:sim3}) are presented in Table~\ref{tab:sim3-appendix}. For a discussion see Section~\ref{sec:sim3}.

\begin{table}
	\begin{center}
		\begin{tabular}{c|ccccc}
			\toprule
                \multirow{3}{*}{Sample Size, $n$} & \multicolumn{5}{c}{Simulation 3}
			\\
                & \begin{tabular}{@{}c@{}c@{}} Squared Bias\\$(\times 10^{-5})$\end{tabular} & \begin{tabular}{@{}c@{}c@{}}Variance\\$(\times 10^{-5})$\end{tabular} &\begin{tabular}{@{}c@{}c@{}} MSE (ratio to\\unweighted) \end{tabular}
                &
                \begin{tabular}{@{}c@{}c@{}} \underline{Squared Bias}\\MSE \end{tabular}
                &
                \begin{tabular}{@{}c@{}c@{}} Coverage \\($95\%$) \end{tabular}
                \\
			\midrule
            \multicolumn{5}{l}{\,Unweighted estimator} 
            \\
            \midrule
			$10\,000$ & 0.001 & 8.146 & 1 & 0.016\% & 95.4\% \\
			$20\,000$ & 0.002 & 4.171 & 1 & 0.049\% & 94.7\% \\
			$40\,000$ & 0.0004 & 2.080 & 1 & 0.018\% & 94.8\% \\
                $80\,000$ & 0.0006 & 0.997 & 1 & 0.065\% & 95.2\% \\
		\midrule
            \multicolumn{5}{l}{\,ROSE random forest ($J=1$) estimator} 
            \\
            \midrule
			$10\,000$ & 0.018 & 7.984 & 0.982 & 0.224\% & 95.3\% \\
			$20\,000$ & 0.009 & 4.035 & 0.969 & 0.228\% & 94.7\% \\
			$40\,000$ & 0.004 & 1.998 & 0.963 & 0.220\% & 94.8\% \\
                $80\,000$ & 0.003 & 0.955 & 0.960 & 0.330\% & 95.3\% \\
		\midrule
  \multicolumn{5}{l}{\,ROSE random forest ($J=2$) estimator} 
            \\
            \midrule
			$10\,000$ & 0.006 & 7.425 & 0.912 & 0.079\% & 95.4\% \\
			$20\,000$ & 0.005 & 3.770 & 0.905 & 0.120\% & 94.4\% \\
			$40\,000$ & 0.002 & 1.839 & 0.885 & 0.132\% & 94.9\% \\
                $80\,000$ & 0.0005 & 0.876 & 0.879  & 0.062\% & 95.4\% \\
		\midrule
  \multicolumn{5}{l}{\,Semiparametric efficient estimator} 
            \\
            \midrule
			$10\,000$ & 1.089 & 6.659 & 0.951 & 14.05\% & 93.2\% \\
			$20\,000$ & 0.855 & 3.429 & 1.027 & 19.96\% & 91.2\% \\
			$40\,000$ & 0.405 & 1.676 & 1.001 & 19.45\% & 91.5\% \\
                $80\,000$ & 0.286 & 0.786 & 1.075 & 26.65\% & 90.4\% \\
			\bottomrule
		\end{tabular}                                                                      
		\caption{Results of Simulation 3 for different sample sizes (4000 simulations).}\label{tab:sim3-appendix}
	\end{center}
\end{table}

\subsection{Real-world data: Further details}\label{appsec:real-world}

For the real-world data example, all nuisance functions are fit with random forests using the \texttt{ranger} package~\citep{app-ranger} and consist of 500 trees, with hyperparameters selected for each random forest using the \texttt{tuneRanger} package~\citep{app-tuneRanger}. In the case of the ROSE random forest weight function, each forest consists of 500 trees, each of maximum depth 20 and minimum node size 5. All estimators employ cross-fitting (DML2) with $K=10$ folds.

\section{Auxiliary lemmas for uniform convergence results}

For the uniform convergence results we prove we require the following lemmas.

\begin{lemma}\label{lem:uniform-simpleones}
    Let $(X_n)_{n\in\mathbb{N}}$ and $(Y_n)_{n\in\mathbb{N}}$ be sequences of real-valued random variables. Then
    \begin{enumerate}[label=(\roman*)]
        \item If $X_n=o_\cP(1)$ and $Y_n=o_\cP(1)$ then $X_n+Y_n=o_\cP(1)$.
        \item If $X_n=o_\cP(1)$ and $Y_n=O_\cP(1)$ then $X_nY_n=o_\cP(1)$.
        \item If $X_n=o_\cP(1)$ and $Y_n\leq X_n$ for all $n\in\mathbb{N}$ uniformly over $P\in\cP$ then $Y_n=o_\cP(1)$.
        \item If there exists some $\alpha\geq1$ such that $\E_P\left[\left|X_n\right|^\alpha\big|Y_n\right] = o_\cP(1)$ then $X_n=o_\cP(1)$.
    \end{enumerate}
\end{lemma}

\begin{lemma}\label{lem:unif_max}
    Let $(X_n^{(1)})_{n\in\mathbb{N}}, \ldots, (X_n^{(K)})_{n\in\mathbb{N}}$ be sequences of real-valued random variables satisfying $X_n^{(k)}=o_{\mathcal{P}}(1)$ for all $k\in[K]$, where $K\in\mathbb{N}$ is finite. Then $\max_{k\in[K]}X_n^{(k)} = o_{\mathcal{P}}(1)$.
\end{lemma}
\begin{proof}
    For any $\epsilon>0$
    \begin{equation*}
        \sup_{P\in\mathcal{P}}\PP_P\Big(
        \big|\max_{k\in[K]}X_n^{(k)}\big| > \epsilon\Big)
        \leq \sup_{P\in\mathcal{P}}\PP_P\Big(\max_{k\in[K]}\big|X_n^{(k)}\big|>\epsilon\Big)
        \leq\sum_{k=1}^K\sup_{P\in\mathcal{P}}\PP_P\Big(\big|X_n^{(k)}\big|>\epsilon\Big)
        = o(1).
    \end{equation*}
\end{proof}

\begin{lemma}[\citet{app-shahpeters}, Lemma 25]\label{lem:bounded_op1}
    Let $(X_n)_{n\in\mathbb{N}}$ be a sequence of non-negative real-valued random variables such that $X_n = o_\cP(1)$. Suppose $X_n\leq C$ for all $n\in\mathbb{N}$ for some $C>0$. Suppose $\alpha\geq 1$. Then
    \begin{equation*}
        \supP \E_P\left[X_n^\alpha\right] = o(1).
    \end{equation*}
\end{lemma}
\begin{proof}
    \citet{app-shahpeters} proves the following result for $\alpha=1$ (the extension to $\alpha\geq1$ given here follows trivially).
    For any $\epsilon>0$,
    \begin{equation*}
        X_n^\alpha = X_n^\alpha\ind_{\left\{\left|X_n\right|>\epsilon\right\}} + X_n^\alpha\ind_{\left\{\left|X_n\right|\leq\epsilon\right\}}
        \leq C^\alpha\ind_{\left\{\left|X_n\right|>\epsilon\right\}} + \epsilon^\alpha,
    \end{equation*}
    and so
    \begin{equation*}
        \supP\E_P\left[X_n^\alpha\right] \leq C^\alpha \PP_P(X_n>\epsilon) + \epsilon^\alpha = o(1) + \epsilon^\alpha,
    \end{equation*}
    as $X_n=o_\cP(1)$. Taking $\epsilon\downarrow 0$ completes the proof.
\end{proof}

\begin{lemma}[\citet{app-shahpeters}, Lemma 18]\label{lem:unif_clt}
Let $(X_{n,i})_{n\in\mathbb{N},i\in[n]}$ be a triangular array of real-valued random variables satisfying: (i) $X_{n,1},\ldots,X_{n,n}$ are independent; (ii) $\E_P[X_{n,i}]=0$ for all $n\in\mathbb{N}, i\in[n], P\in\cP$; (iii) there exists some $\Delta>0$ such that $$\lim_{n\to\infty}\supP\bigg(\sum_{i=1}^n\E_P\big[X_{n,i}^2\big]\bigg)^{-\big(1+\frac{\Delta}{2}\big)} \bigg(\sum_{i=1}^n\E_P\big[|X_{n,i}|^{2+\Delta}\big]\bigg) = 0.$$ Then $S_n := \big(\sum_{i=1}^n \E_P[X_{n,i}^2]\big)^{-\frac{1}{2}}\big(\sum_{i=1}^n X_{n,i}\big)$ converges uniformly in distribution to $N(0,1)$ i.e.
\begin{equation*}
    \lim_{I\to\infty}\sup_{P\in\mathcal{P}_I}\sup_{t\in\R}\left|\PP_{P}\left(S_I\leq t\right)-\Phi(t)\right|=0.
\end{equation*}

\end{lemma}

\begin{lemma}[\citet{app-shahpeters}, Lemma 20]\label{lem:unif_slutsky}
Let \(\left(X_n\right)_{n\in\mathbb{N}}\) and \(\left(Y_n\right)_{n\in\mathbb{N}}\) be sequences of real-valued random variables. Suppose
\begin{equation*}
    \lim_{n\to\infty}\sup_{P\in\cP}\sup_{t\in\R}\left|\PP_{P}\left(X_n\leq t\right)-\Phi(t)\right| = 0.
\end{equation*}
Then we have:
\begin{enumerate}
    \item[(a)] \[\text{If } Y_I=o_\cP(1) \text{ then } \lim_{n\to\infty}\sup_{P\in\mathcal{P}}\sup_{t\in\R}\left|\PP_{P}\left(X_n+Y_n\leq t\right)-\Phi(t)\right| = 0;\]
    \item[(b)] \[\text{If } Y_n=1+o_{\mathcal{P}}(1) \text{ then } \lim_{n\to\infty}\sup_{P\in\mathcal{P}}\sup_{t\in\R}\left|\PP_{P}\left(X_n/Y_n\leq t\right)-\Phi(t)\right| = 0.\]
\end{enumerate}
\end{lemma}

\begin{lemma}\label{lem:unif_root-n-theta-moments-bounded}
    Let $(X_n)_{n\in\mathbb{N}}$ be a sequence of real-valued random variables that is uniformly asymptotically Gaussian
    \begin{equation}\label{eq:lemass}
        \lim_{n\to\infty}\sup_{P\in\cP}\sup_{t\in\R}\left|\PP_P\big(X_n\geq t\big)-\Phi(t)\right| = 0.
    \end{equation}
    Then for any $q\geq 1$,
    \begin{equation*}
        \supP\E_P\big[|X_n|^{q}\big] = \E_{Z\sim N(0,1)}\big[|Z|^q\big] + o(1) = O(1),
    \end{equation*}
    where $N(0,1)$ denotes a standard normal random variable.
\end{lemma}
\begin{proof}
    Using the `layer-cake' formula
    \begin{align*}
        \E\big[|X_n|^q\big] &=
        \int_0^\infty \PP_P\big(|X_n|^q>t\big) dt
        \\
        &= \int_0^\infty \PP_P\big(|X_n|\geq t^{\frac{1}{q}}\big) dt
        \\
        &= \int_0^{\infty} \Big\{1-\Phi\big(t^{\frac{1}{q}}\big)+\Phi\big(-t^{\frac{1}{q}}\big)\Big\}dt
        - \int_0^{\infty} \Big\{\PP_P\big(X_n\leq t^{\frac{1}{q}}\big)-\Phi\big(t^{\frac{1}{q}}\big)\Big\}dt
        \\
        &\qquad + \int_0^{\infty} \Big\{\PP_P\big(X_n\leq -t^{\frac{1}{q}}\big)-\Phi\big(-t^{\frac{1}{q}}\big)\Big\}dt
        \\
        &=\E\big[|Z|^q\big] - q\int_0^{\infty} t^{q-1}\Big\{\PP_P(X_n\leq t)-\Phi(t)\Big\}dt
        \\
        &\qquad + q\int_0^{\infty} t^{q-1}\Big\{\PP_P(X_n\leq -t)-\Phi(-t)\Big\}dt.
    \end{align*}
    where $Z\sim N(0,1)$. Now, we claim that there exists some $N\in\mathbb{N}$ such that for all $n\geq N$ and all $P\in\cP$
    \begin{equation*}
        \PP_P(X_n\geq t) \leq 2 (1-\Phi(t))
        \qquad
        (\forall t\in\R).
    \end{equation*}
    We show this claim holds by contradiction now. Suppose the claim does not hold; for all $N\in\mathbb{N}$ there exists some $n_0\geq N$, some $P_0\in\cP$ and some $t_0\in\R$ such that
    \begin{equation*}
        \PP_P(X_n>t_0) > 2(1-\Phi(t)).
    \end{equation*}
    Then
    \begin{multline*}
        \sup_{P\in\cP}\sup_{t\in\R}\left|\PP_P(X_n\leq t)-\Phi(t)\right|
        \geq
        \left|\PP_{P_0}(X_n\leq t_0)-\Phi(t_0)\right|
        \\
        \geq
        \left|\PP_{P_0}(X_n>t_0)\right| - \left|1-\Phi(t_0)\right|
        >
        |1-\Phi(t_0)|,
    \end{multline*}
    and thus
    \begin{equation*}
        \lim_{n\to\infty}\sup_{P\in\cP}\sup_{t\in\R}\left|\PP_P(X_n\leq t)-\Phi(t)\right| \geq |1-\Phi(t_0)|,
    \end{equation*}
    contradicting~\eqref{eq:lemass}. 
    Therefore, there exists some $N'\in\mathbb{N}$ such that
    \begin{equation*}
        \sup_{n\geq N}\supP\left|\PP_P(X_n\geq t)-(1-\Phi(t))\right| \leq (1-\Phi(t)).
    \end{equation*}
    
    Given $\epsilon>0$, let $N_\epsilon\in\mathbb{N}$ be such that for all $n\geq N_\epsilon$ and all $P\in\cP$
    \begin{equation*}
        \sup_{t\in\R} \left|\PP_P(X_n\leq t)-\Phi(t)\right| \leq \epsilon.
    \end{equation*}
    Then for all $n\geq N_\epsilon \vee N'$
    \begin{align*}
        &\left|q\int_0^{\infty}t^{q-1}\Big\{\PP_P(X_n\leq t)-\Phi(t)\Big\}dt\right|
        \\
        &\leq q\int_0^{\infty}t^{q-1}\left|\PP_P(X_n\leq t)-\Phi(t)\right|dt
        \\
        &\leq 
        q\int_0^M t^{q-1}\left|\PP_P(X_n\leq t)-\Phi(t)\right|dt
        +
        q\int_M^{\infty} t^{q-1}\left|\PP_P(X_n\leq t)-\Phi(t)\right|dt
        \\
        &\leq M^q\epsilon + q \int_M^{\infty} t^{q-1} (1-\Phi(t)) dt
        \\
        &= M^q\epsilon + I(M),
    \end{align*}
    where
    \begin{equation*}
        I(M) := \int_M^\infty \frac{t^q}{\sqrt{2\pi}}e^{-\frac{1}{2}t^2}dt.
    \end{equation*}
    Note that $\lim_{M\to\infty}I(M)=0$. Then taking $M:=\epsilon^{-\frac{1}{q+1}}$
    \begin{equation*}
        \left|q\int_0^{\infty}t^{q-1}\Big\{\PP_P(X_n\leq t)-\Phi(t)\Big\}dt\right| \leq \epsilon^{\frac{1}{q+1}} + I(\epsilon^{-\frac{1}{q+1}}).
    \end{equation*}
    As $\epsilon>0$ was taken to be arbitrary, taking $\epsilon\downarrow 0$ yields
    \begin{equation*}
        \supP q\int_0^{\infty}t^{q-1}\Big\{\PP_P(X_n\leq t)-\Phi(t)\Big\}dt = o(1).
    \end{equation*}
    Analogous arguments similarly show
    \begin{equation*}
        \supP q\int_0^\infty t^{q-1} \Big\{\PP_P(X_n\leq -t)+\Phi(-t)\Big\}
        =
        o(1),
    \end{equation*}
    and thus
    \begin{equation*}
        \supP\E_P\big[|X_n|^q\big] = \E_{Z\sim N(0,1)}\big[|Z|^q\big] + o(1).
    \end{equation*}
    The result then follows as $\E_{Z\sim N(0,1)}\big[|Z|^q\big]=2^{\frac{q-1}{2}}\pi^{-\frac{1}{2}}\Gamma\big(\frac{q}{2}+1\big)<\infty$ for any finite $q\geq 1$.
\end{proof}

\begin{lemma}\label{lem:unif_cons}
    Let $(\hat{\Psi}_{P,n}:\Theta\to\R)_{P\in\cP,n\in\mathbb{N}}$ be sequences of random real-valued functions and $(\Psi_P:\Theta\to\R)_{P\in\cP}$ be a set of deterministic functions such that
    \begin{equation*}
        \sup_{\theta\in\Theta}\big|\hat{\Psi}_{P,n}(\theta)-\Psi_P(\theta)\big| = o_{\cP}(1),
    \end{equation*}
    and for every $\epsilon>0$
    \begin{equation*}
        \inf_{P\in\cP}\inf_{\underset{|\theta-\theta_P|\geq\epsilon}{\theta\in\Theta}}\Psi_P(\theta) - \Psi_P(\theta_P) 
         > 0.
    \end{equation*}
    Then any sequence of estimators $\hat{\theta}_n$ with $\hat{\Psi}_{P,n}(\hat{\theta}_n)\leq \hat{\Psi}_{P,n}(\theta_P) + o_{\cP}(1)$ converges uniformly in probability to $\theta_P$;
    \begin{equation*}
        \hat{\theta}_n = \theta_P + o_{\cP}(1).
    \end{equation*}
\end{lemma}
\begin{proof}
    We extend the proof of Theorem~5.7 of \citet{app-vandervaart} to the analogous result with uniform convergence in probability. Using the property of the sequence of estimators $(\hat{\theta}_n)_{n\in\mathbb{N}}$ as well as the first of the assumptions of the theorem, we have

    \begin{align*}
        \Psi_P(\hat{\theta}_n) - \Psi_P(\theta_P) 
        &=
        \big(\Psi_P(\hat{\theta}_n) - \hat{\Psi}_{P,n}(\hat{\theta}_n)\big) + \big(\hat{\Psi}_{P,n}(\theta_P) - \Psi_P(\theta_P)\big) + \big(\hat{\Psi}_{P,n}(\hat{\theta}_n) - \hat{\Psi}_{P,n}(\theta_P)\big) \\
        &\leq
        2\sup_{\theta\in\Theta}\big|\hat{\Psi}_{P,n}(\theta)-\Psi_P(\theta_P)\big| + o_{\mathcal{P}}(1) \\
        &= o_{\mathcal{P}}(1).
    \end{align*}
    By the second part of the assumption, for any $\epsilon>0$ there exists some $\nu_\epsilon>0$ such that $\Psi_P(\theta) \geq \Psi_P(\theta_P) + \nu_\epsilon$ for any $\theta\in\Theta$ satisfying $|\theta-\theta_P|>\epsilon$. 
    Hence,
    \begin{equation*}
        \sup_{P\in\mathcal{P}}\PP_P\big(|\hat{\theta}_n-\theta_P|>\epsilon\big)
        \leq
        \sup_{P\in\mathcal{P}}\PP_P\big(\Psi_P(\theta) \geq \Psi_P(\theta_P) + \nu_\epsilon\big)
        = o(1),
    \end{equation*}
    thus $\hat{\theta}_n = \theta_P + o_\cP(1)$.
\end{proof}


\begin{thebibliography}{}

\bibitem[Athey et~al., 2019]{grfannals}
Athey, S., Tibshirani, J., and Wager, S. (2019).
\newblock Generalized random forests.
\newblock {\em The Annals of Statistics}, 47(2):1148--1178.

\bibitem[Biau, 2012]{biau}
Biau, G. (2012).
\newblock Analysis of a random forests model.
\newblock {\em Journal of Machine Learning Research}, 13(38):1063--1095.

\bibitem[Bickel et~al., 1998]{bickel}
Bickel, P.~J., Klaassen, C. A.~J., Ritov, Y., and Wellner, J.~A. (1998).
\newblock {\em Efficient and Adaptive Estimation for Semiparametric Models}.
\newblock Springer.

\bibitem[Breiman, 2001]{randomforest}
Breiman, L. (2001).
\newblock Random forests.
\newblock {\em Machine Learning}, 45(1):5--32.

\bibitem[Cevid et~al., 2022]{distrandfor}
Cevid, D., Michel, L., N{\"a}f, J., B{\"u}hlmann, P., and Meinshausen, N.
  (2022).
\newblock Distributional random forests: Heterogeneity adjustment and
  multivariate distributional regression.
\newblock {\em Journal of Machine Learning Research}, 23(333):1--79.

\bibitem[Chen, 1988]{chen}
Chen, H. (1988).
\newblock Convergence rates for parametric components in a partly linear model.
\newblock {\em The Annals of Statistics}, 16(1):136--146.

\bibitem[Chen, 2007]{chen-odds}
Chen, H.~Y. (2007).
\newblock A semiparametric odds ratio model for measuring association.
\newblock {\em Biometrics}, 63(2):413--21.

\bibitem[Chernozhukov et~al., 2018]{chern}
Chernozhukov, V., Chetverikov, D., Demirer, M., Duflo, E., Hansen, C., Newey,
  W., and Robins, J. (2018).
\newblock {Double/debiased machine learning for treatment and structural
  parameters}.
\newblock {\em The Econometrics Journal}, 21(1):C1--C68.

\bibitem[Emmenegger and B{\"u}hlmann, 2023]{emmenegger}
Emmenegger, C. and B{\"u}hlmann, P. (2023).
\newblock Plug-in machine learning for partially linear mixed-effects models
  with repeated measurements.
\newblock {\em Scandinavian Journal of Statistics}, 50(4):1553--1567.

\bibitem[Engle et~al., 1986]{engle}
Engle, R.~F., Granger, C. W.~J., Rice, J., and Weiss, A. (1986).
\newblock Semiparametric estimates of the relation between weather and
  electricity sales.
\newblock {\em Journal of the American Statistical Association},
  81(394):310--320.

\bibitem[Eric J. Tchetgen~Tchetgen, 2010]{tchetgen}
Eric J. Tchetgen~Tchetgen, James M.~Robins, A.~R. (2010).
\newblock On doubly robust estimation in a semiparametric odds ratio model.
\newblock {\em Biometrika}, 97(1):171--180.

\bibitem[Green et~al., 1985]{green}
Green, P., Jennison, C., and Seheult, A. (1985).
\newblock Analysis of field experiments by least squares smoothing.
\newblock {\em Journal of the Royal Statistical Society. Series B
  (Methodological)}, 47(2):299--315.

\bibitem[Guo and Shah, 2024]{guo}
Guo, F.~R. and Shah, R.~D. (2024).
\newblock Rank-transformed subsampling: inference for multiple data splitting
  and exchangeable p-values.
\newblock {\em Journal of the Royal Statistical Society Series B: Statistical
  Methodology}.

\bibitem[Hansen, 1982]{gmm}
Hansen, L. (1982).
\newblock Large sample properties of generalized method of moments estimators.
\newblock {\em Econometrica}, 50(4):1029--1054.

\bibitem[Hardin and Hilbe, 2003]{geehardin}
Hardin, J.~W. and Hilbe, J.~M. (2003).
\newblock {\em Generalized estimating equations}.
\newblock Chapman and Hall.

\bibitem[H{\"a}rdle, 2004]{hardle}
H{\"a}rdle, W.~K. (2004).
\newblock {\em Nonparametric and Semiparametric Models}.
\newblock Springer Series in Statistics. Springer.

\bibitem[Huber, 1967]{huber}
Huber, P.~J. (1967).
\newblock The behaviour of maximum likelihood estimates under nonstandard
  conditions.
\newblock {\em Proceedings of the Fifth Berkeley Symposium}, 1:221--223.

\bibitem[Kosorok, 2008]{kosorok}
Kosorok, M.~R. (2008).
\newblock {\em Introduction to empirical processes and semiparametric
  inference}.
\newblock Springer.

\bibitem[Le~Cam, 1953]{le-cam-thesis}
Le~Cam, L.~M. (1953).
\newblock {\em On some asymptotic properties of maximum likelihood estimates
  and related Bayes' estimates}.
\newblock PhD thesis, University of California, Berkeley.

\bibitem[Liang and Zhou, 1998]{liang-zhou}
Liang, H. and Zhou, Y. (1998).
\newblock A modified estimator of error variance in a partly linear model.
\newblock {\em Communications in Statistics - Theory and Methods},
  27(4):819--825.

\bibitem[Liang and Zeger, 1986]{liangzeger}
Liang, K.-Y. and Zeger, S.~L. (1986).
\newblock Longitudinal data analysis using generalized linear models.
\newblock {\em Biometrika}, 73(1):13--22.

\bibitem[Liu et~al., 2021]{liu}
Liu, M., Zhang, Y., and Zhou, D. (2021).
\newblock Double/debiased machine learning for logistic partially linear model.
\newblock {\em The Econometrics Journal}, 24(3):559--588.

\bibitem[Ma et~al., 2006]{ma}
Ma, Y., Chiou, J.-M., and Wang, N. (2006).
\newblock {Efficient semiparametric estimator for heteroscedastic partially
  linear models}.
\newblock {\em Biometrika}, 93(1):75--84.

\bibitem[Meinshausen, 2006]{meinshausen}
Meinshausen, N. (2006).
\newblock Quantile regression forests.
\newblock {\em Journal of Machine Learning Research}, 7(35):983--999.

\bibitem[Neyman, 1959]{neyman}
Neyman, J. (1959).
\newblock Optimal asymptotic tests of composite statistical hypotheses.
\newblock {\em Probability and Statistics}, pages 416--44.

\bibitem[Neyman, 1979]{neyman2}
Neyman, J. (1979).
\newblock C($\alpha$) tests and their use.
\newblock {\em Sankhy{\=a}: The Indian Journal of Statistics, Series A},
  41(1/2):1--21.

\bibitem[Rice, 1986]{rice}
Rice, J. (1986).
\newblock Convergence rates for partially splined models.
\newblock {\em Statistics \& Probability Letters}, 4(4):203--208.

\bibitem[Robinson, 1988]{robinson}
Robinson, P. (1988).
\newblock Root-n-consistent semiparametric regression.
\newblock {\em Econometrica}, 56(4):931--54.

\bibitem[Rosenblum and van~der Laan, 2010]{locallyeff2}
Rosenblum, M. and van~der Laan, M.~J. (2010).
\newblock Simple, efficient estimators of treatment effects in randomized
  trials using generalized linear models to leverage baseline variables.
\newblock {\em International Journal of Biostatistics}, 6(1):13.

\bibitem[Rubin and van~der Laan, 2008]{locallyeff1}
Rubin, D.~B. and van~der Laan, M.~J. (2008).
\newblock Empirical efficiency maximization: improved locally efficient
  covariate adjustment in randomized experiments and survival analysis.
\newblock {\em The International Journal of Biostatistics}, 4(1):4--5.

\bibitem[Sathishkumar et~al., 2020]{seoulbikes}
Sathishkumar, V.~E., Park, J., and Cho, Y. (2020).
\newblock Using data mining techniques for bike sharing demand prediction in
  metropolitan city.
\newblock {\em Computer Communications}, 153:353--366.

\bibitem[Schick, 1996]{schick-plr}
Schick, A. (1996).
\newblock Root-n consistent estimation in partly linear regression models.
\newblock {\em Statistics \& Probability Letters}, 28(4):353--358.

\bibitem[Shah and Peters, 2020]{shahpeters}
Shah, R.~D. and Peters, J. (2020).
\newblock The hardness of conditional independence testing and the generalised
  covariance measure.
\newblock {\em The Annals of Statistics}, 48(3).

\bibitem[Speckman, 1988]{speckman}
Speckman, P. (1988).
\newblock Kernel smoothing in partially linear models.
\newblock {\em Journal of the Royal Statistical Society. Series B
  (Methodological)}, 50(3):413--436.

\bibitem[Tan, 2019]{tan}
Tan, Z. (2019).
\newblock On doubly robust estimation for logistic partially linear models.
\newblock {\em Statistics \& Probability Letters}, 155:108577.

\bibitem[Therneau et~al., 2013]{rpart}
Therneau, T., Atkinson, B., and Ripley, B. (2013).
\newblock {\em Rpart: Recursive Partitioning}.
\newblock R-package available on CRAN.

\bibitem[Tibshirani et~al., 2024]{grf-code}
Tibshirani, J., Athey, S., Sverdrup, E., and Wager, S. (2024).
\newblock {\em grf: Generalized Random Forests}.
\newblock R package version 2.3.2.

\bibitem[Tsiatis, 2006]{tsiatis}
Tsiatis, A.~A. (2006).
\newblock {\em Semiparametric theory and missing data}, volume~1 of {\em
  Springer series in statistics}.
\newblock Springer, New York.

\bibitem[van~der Laan and Rose, 2018]{targeted-learning}
van~der Laan, M.~J. and Rose, S. (2018).
\newblock {\em Targeted Learning in Data Science Causal Inference for Complex
  Longitudinal Studies / by Mark J. van der Laan, Sherri Rose.}
\newblock Springer series in statistics. Springer New York, NY, 1 edition.

\bibitem[van~der Vaart, 1998]{vandervaart}
van~der Vaart, A.~W. (1998).
\newblock {\em Asymptotic Statistics}.
\newblock Cambridge University Press.

\bibitem[Vansteelandt and Dukes, 2022a]{dukes}
Vansteelandt, S. and Dukes, O. (2022a).
\newblock Assumption‐lean inference for generalised linear model parameters.
\newblock {\em Journal of the Royal Statistical Society Series B},
  84(3):657--685.

\bibitem[Vansteelandt and Dukes, 2022b]{vansteelandt-comments}
Vansteelandt, S. and Dukes, O. (2022b).
\newblock Authors' reply to the discussion of `assumption-lean inference for
  generalised linear model parameters' by vansteelandt and dukes.
\newblock {\em Journal of the Royal Statistical Society Series B: Statistical
  Methodology}, 84(3):729--739.

\bibitem[Wager and Athey, 2018]{grfs}
Wager, S. and Athey, S. (2018).
\newblock Estimation and inference of heterogeneous treatment effects using
  random forests.
\newblock {\em Journal of the American Statistical Association},
  113(523):1228--1242.

\bibitem[Wager and Walther, 2016]{wagerwalther}
Wager, S. and Walther, G. (2016).
\newblock Adaptive concentration of regression trees, with application to
  random forests.
\newblock {\em arXiv preprint arXiv:1503.06388}.

\bibitem[Wright and Ziegler, 2017]{ranger}
Wright, M.~N. and Ziegler, A. (2017).
\newblock {ranger}: A fast implementation of random forests for high
  dimensional data in {C++} and {R}.
\newblock {\em Journal of Statistical Software}, 77(1):1--17.

\bibitem[You et~al., 2007]{you}
You, J., Chen, G., and Zhou, Y. (2007).
\newblock Statistical inference of partially linear regression models with
  heteroscedastic errors.
\newblock {\em Journal of Multivariate Analysis}, 98(8):1539--1557.

\bibitem[Young and Shah, 2024]{sandwich-boosting}
Young, E.~H. and Shah, R.~D. (2024).
\newblock Sandwich boosting for accurate estimation in partially linear models
  for grouped data.
\newblock {\em Journal of the Royal Statistical Society Series B: Statistical
  Methodology}.
  
 \bibitem[Young and Shah, 2024]{supp}
Young, E.~H. and Shah, R.~D. (2024).
\newblock Supplement to ``ROSE Random Forests for robust semiparametric efficient estimation''.

\bibitem[Ziegler, 2011]{ziegler}
Ziegler, A. (2011).
\newblock {\em Generalized estimating equations}.
\newblock Lecture notes in statistics. Springer, New York.

\end{thebibliography}

\begin{thebibliography}{}

\bibitem[Chen et~al., 2021]{app-chen-IV-semieff}
Chen, J., Chen, D.~L., and Lewis, G. (2021).
\newblock Mostly harmless machine learning: Learning optimal instruments in
  linear iv models.
\newblock {\em arXiv preprint arXiv:2011.06158}.

\bibitem[Chernozhukov et~al., 2018]{app-chern}
Chernozhukov, V., Chetverikov, D., Demirer, M., Duflo, E., Hansen, C., Newey,
  W., and Robins, J. (2018).
\newblock {Double/debiased machine learning for treatment and structural
  parameters}.
\newblock {\em The Econometrics Journal}, 21(1):C1--C68.

\bibitem[Cover, 1968]{app-cover}
Cover, T.~M. (1968).
\newblock Rates of convergence for nearest neighbor procedures.
\newblock {\em IEEE Transactions on Information Theory}, 14:50--55.

\bibitem[Devroye, 1982]{app-devroye}
Devroye, L. (1982).
\newblock Necessary and sufficient conditions for the almost everywhere
  convergence of nearest neighbor regression function estimates.
\newblock {\em Zeitschrift f\"{u}r Wahrscheinlichkeitstheorie und verwandte
  Gebiete}, 61:467--481.

\bibitem[Gy\"{o}rfi et~al., 1998]{app-gyorfi}
Gy\"{o}rfi, L., Kohler, M., Krzyzak, A., and Walk, H. (1998).
\newblock {\em A Distribution-Free Theory of Nonparametric Regression}.
\newblock Number~1 in Springer series in statistics. Springer New York.

\bibitem[Hansen, 1982]{app-gmm}
Hansen, L. (1982).
\newblock Large sample properties of generalized method of moments estimators.
\newblock {\em Econometrica}, 50(4):1029--1054.

\bibitem[Liu et~al., 2021]{app-liu}
Liu, M., Zhang, Y., and Zhou, D. (2021).
\newblock Double/debiased machine learning for logistic partially linear model.
\newblock {\em The Econometrics Journal}, 24(3):559--588.

\bibitem[Mackey et~al., 2018]{app-mackey}
Mackey, L., Syrgkanis, V., and Zadik, I. (2018).
\newblock Orthogonal machine learning: Power and limitations.
\newblock In {\em Proceedings of the 35th International Conference on Machine
  Learning}, volume~80 of {\em Proceedings of Machine Learning Research}, pages
  3375--3383.

\bibitem[Meinshausen, 2006]{app-meinshausen}
Meinshausen, N. (2006).
\newblock Quantile regression forests.
\newblock {\em Journal of Machine Learning Research}, 7(35):983--999.

\bibitem[Neyman, 1959]{app-neyman}
Neyman, J. (1959).
\newblock Optimal asymptotic tests of composite statistical hypotheses.
\newblock {\em Probability and Statistics}, pages 416--44.

\bibitem[Neyman, 1979]{app-neyman2}
Neyman, J. (1979).
\newblock C($\alpha$) tests and their use.
\newblock {\em Sankhy{\=a}: The Indian Journal of Statistics, Series A},
  41(1/2):1--21.

\bibitem[Probst et~al., 2018]{app-tuneRanger}
Probst, P., Wright, M., and Boulesteix, A.-L. (2018).
\newblock Hyperparameters and tuning strategies for random forest.
\newblock {\em Wiley Interdisciplinary Reviews: Data Mining and Knowledge
  Discovery}.

\bibitem[Shah and Peters, 2020]{app-shahpeters}
Shah, R.~D. and Peters, J. (2020).
\newblock The hardness of conditional independence testing and the generalised
  covariance measure.
\newblock {\em The Annals of Statistics}, 48(3).

\bibitem[Tibshirani et~al., 2024]{app-grf-code}
Tibshirani, J., Athey, S., Sverdrup, E., and Wager, S. (2024).
\newblock {\em grf: Generalized Random Forests}.
\newblock R package version 2.3.2.

\bibitem[van~der Vaart, 1998]{app-vandervaart}
van~der Vaart, A.~W. (1998).
\newblock {\em Asymptotic Statistics}.
\newblock Cambridge University Press.

\bibitem[Wright and Ziegler, 2017]{app-ranger}
Wright, M.~N. and Ziegler, A. (2017).
\newblock {ranger}: A fast implementation of random forests for high
  dimensional data in {C++} and {R}.
\newblock {\em Journal of Statistical Software}, 77(1):1--17.

\end{thebibliography}
\end{document}